\documentclass{acta}

\usepackage{harvard_acta}
\usepackage{amssymb}
\usepackage{amsmath}
\usepackage{verbatim}
\usepackage{graphicx}
\usepackage{float}\usepackage{bm}
\usepackage{tikz} 
\usepackage{multirow}
\newtheorem{theorem}{Theorem}[section]
\newtheorem{remark}{Remark}[section]
\newtheorem{proposition}{Theorem}[section]

\usetikzlibrary{arrows}
\usetikzlibrary{patterns}

\newcommand{\CS}{{\mathcal{S}}}
\newcommand{\CI}{{\mathcal{I}}} 
\newcommand{\CG}{{\mathcal{G}}}
\newcommand{\CF}{{\mathcal{F}}}
\newcommand{\CP}{{\mathcal{P}}}
\newcommand{\CO}{{\mathcal{O}}}
\newcommand{\CK}{{\mathcal{K}}}
\newcommand{\CM}{{\mathcal{M}}}
\newcommand{\dg}{diagonalization }
\newcommand{\pd}{preconditioner }
\colorlet{shadecolor}{blue!20}
\newcommand{\ds}{discretization }
\definecolor{darkgreen}{rgb}{0,0.6,0.1}
\usepackage[linkcolor=blue,citecolor=cyan,urlcolor=cyan]{hyperref}
\renewcommand{\vec}[1]{\mbox{\boldmath $#1$}}
\usepackage{color}
\usepackage{framed}
\definecolor{shadecolor}{rgb}{0.92,0.92,0.92}

\title[]{Time parallelization for hyperbolic and
  parabolic problems}

\author[{M. J. Gander, S. L. Wu and T. Zhou}]{%
  Martin J. Gander\\
  {\it Department of Mathematics,}\\
  {\it University of Geneva, CP64, 1211 Geneva 4, Switzerland}\\
  {\tt martin.gander@unige.ch}\\
  \and
  Shu-Lin Wu\\
    {\rm \textbf{(corresponding author)} }\\
  {\it School of Mathematics and Statistics,}\\
  {\it Northeast Normal University,
Changchun 130024, China }\\
  {\tt wushulin84@hotmail.com}\\
    \and
    Tao Zhou\\
{\it Institute of Computational Mathematics and}\\
{\it Scientific/Engineering Computing,  Academy of
Mathematics and Systems Science,}\\
{\it Chinese Academy of Sciences, China}\\
{\tt tzhou@lsec.cc.ac.cn}  }
\begin{document}

\maketitle
\vspace{12em}

\begin{abstract}
Time parallelization, also known as PinT (Parallel-in-Time) is a new research direction for the development of algorithms
used for solving very large scale evolution problems on highly parallel computing architectures. Despite the fact that interesting
theoretical work on PinT appeared as early 1964, it was not until 2004, when processor clock speeds reached their physical limit, that research in PinT
took off. A distinctive characteristic of parallelization in time is that information flow only goes
forward in time, meaning that time evolution processes seem necessarily to be sequential. Nevertheless, many
algorithms have been developed over the last two decades to do PinT
computations, and they are often grouped into four basic classes according
to how the techniques work and are used: shooting-type methods; waveform relaxation methods based on domain
decomposition; multigrid methods in space-time; and direct time
parallel methods. However, over the past few years, it
has been recognized that highly successful PinT algorithms for
parabolic problems struggle when applied to hyperbolic problems.
We focus in this survey therefore on this important aspect, by first providing a
summary of the fundamental differences
between parabolic and hyperbolic problems for time
parallelization. We then group PinT algorithms into two basic groups: the first group contains four effective PinT techniques for
hyperbolic problems, namely Schwarz Waveform Relaxation with its
relation to Tent Pitching; Parallel Integral Deferred Correction;
ParaExp; and ParaDiag. While the methods in the first group also work well for parabolic
problems, we then present PinT methods especially designed for
parabolic problems in the second group: Parareal: the Parallel Full Approximation
Scheme in Space-Time; Multigrid Reduction in Time; and Space-Time
Multigrid. We complement our analysis with numerical illustrations
using four time-dependent PDEs: the heat equation; the
advection-diffusion equation; Burgers' equation; and the
second-order wave equation. 
\end{abstract}

\tableofcontents

\section{Introduction}

Time parallelization has been a very active field of research over the
past two decades. The reason for this is that hardware development has
reached its physical limit for clock speed, and faster computation is
only possible using more and more cores. We see this development even
in small-scale computing devices like smartphones that have become
multicore, and high-performance computers now have millions of cores.
Time parallelization methods, also referred to as Parallel in Time
(PinT) methods, are methods that allow one to use for evolution
problems more cores than when one would only parallelize the
in space. Already 60 years ago, Nievergelt proposed in a visionary
paper \cite{Nievergelt:1964:PMI} such an approach, concluding with:
\begin{quote}
``{\em The integration methods introduced in this paper are to be regarded as tentative examples of a much wider class of numerical procedures in which parallelism is introduced at the expense of redundancy of computation. As such, their merits lie not so much in their usefulness as numerical algorithms as in their potential as prototypes of better methods based on the same principle. It is believed that more general and improved versions of these methods will be of great importance when computers capable of executing many computations in parallel become available.}''
\end{quote}
Several new methods like these were then developed over the decades
that followed Nievergelt, until PinT methods were brought to the
forefront of research with the advent of the Parareal algorithm
\cite{Lions:2001:PTD}; see the historical review \cite{gander50years},
the review focusing on PinT applications \cite{ong2020applications},
and also the recent research monograph \cite{Gander:TPTI:2024}. 

Parallelization in time for evolution problems may, at first glance,
seem impossible due to the causality principle: solutions at later
times are determined by solutions at earlier times, and not vice
versa. Evolution problems thus have an inherent sequential
nature. This becomes clear when we consider a simple ordinary
differential equation as our evolution problem, along with its forward
Euler discretization,
\begin{equation}
  \partial_t u=f(u), \quad u(0)=u_0,\qquad
  u_{n+1}=u_n+\Delta tf(u_n).
\end{equation}
The recurrence formula of Forward Euler clearly shows that we must
know $u_n$  before we can compute $u_{n+1}$,  as illustrated in
Figure \ref{FEFig}.
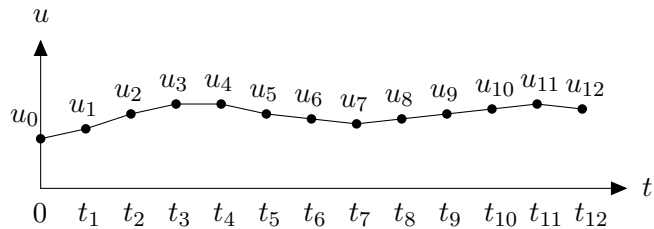
\begin{figure}
  \centering
    \begin{tikzpicture}[line cap=round,line join=round,>=triangle 45,x=1.0cm,y=1.1cm,scale=0.6]  
\clip(-0.7,-0.8) rectangle (13.8,3.7);
\draw [->] (0,0) -- (13,0);
\draw [->] (0,0) -- (0,3);
\draw (13.1,0.4) node[anchor=north west] {$t$};
\draw (-0.4,3.8) node[anchor=north west] {$u$};
\draw (-0.9,1.8) node[anchor=north west] {$u_0$};
\draw (0.4,2) node[anchor=north west] {$u_1$};
\draw (1.4,2.3) node[anchor=north west] {$u_2$};
\draw (2.4,2.5) node[anchor=north west] {$u_3$};
\draw (3.4,2.5) node[anchor=north west] {$u_4$};
\draw (4.4,2.3) node[anchor=north west] {$u_5$};
\draw (5.4,2.2) node[anchor=north west] {$u_6$};
\draw (6.4,2.1) node[anchor=north west] {$u_7$};
\draw (7.4,2.2) node[anchor=north west] {$u_8$};
\draw (8.4,2.3) node[anchor=north west] {$u_9$};
\draw (9.4,2.4) node[anchor=north west] {$u_{10}$};
\draw (10.4,2.5) node[anchor=north west] {$u_{11}$};
\draw (11.4,2.4) node[anchor=north west] {$u_{12}$};
\draw (-0.4,-0.1) node[anchor=north west] {$0$};
\draw (0.6,-0.1) node[anchor=north west] {$t_1$};
\draw (1.6,-0.1) node[anchor=north west] {$t_2$};
\draw (2.6,-0.1) node[anchor=north west] {$t_3$};
\draw (3.6,-0.1) node[anchor=north west] {$t_4$};
\draw (4.6,-0.1) node[anchor=north west] {$t_5$};
\draw (5.6,-0.1) node[anchor=north west] {$t_6$};
\draw (6.6,-0.1) node[anchor=north west] {$t_7$};
\draw (7.6,-0.1) node[anchor=north west] {$t_8$};
\draw (8.6,-0.1) node[anchor=north west] {$t_9$};
\draw (9.6,-0.1) node[anchor=north west] {$t_{10}$};
\draw (10.6,-0.1) node[anchor=north west] {$t_{11}$};
\draw (11.6,-0.1) node[anchor=north west] {$t_{12}$};
\fill [color=black] (0,1) circle (3pt);
\fill [color=black] (1,1.2) circle (3pt);
\draw (0,1)-- (1,1.2);
\fill [color=black] (2,1.5) circle (3pt);
\draw (1,1.2)-- (2,1.5);
\fill [color=black] (3,1.7) circle (3pt);
\draw (2,1.5)-- (3,1.7);
\fill [color=black] (4,1.7) circle (3pt);
\draw (3,1.7)-- (4,1.7);
\fill [color=black] (5,1.5) circle (3pt);
\draw (4,1.7)-- (5,1.5);
\fill [color=black] (6,1.4) circle (3pt);
\draw (5,1.5)-- (6,1.4);
\fill [color=black] (7,1.3) circle (3pt);
\draw (6,1.4)-- (7,1.3);
\fill [color=black] (8,1.4) circle (3pt);
\draw (7,1.3)-- (8,1.4);
\fill [color=black] (9,1.5) circle (3pt);
\draw (8,1.4)-- (9,1.5);
\fill [color=black] (10,1.6) circle (3pt);
\draw (9,1.5)-- (10,1.6);
\fill [color=black] (11,1.7) circle (3pt);
\draw (10,1.6)-- (11,1.7);
\fill [color=black] (12,1.6) circle (3pt);
\draw (11,1.7)-- (12,1.6);
\end{tikzpicture}
  \caption{Sequential nature of time integration using Forward Euler.}
  \label{FEFig}
\end{figure}
It is not clear for example if one can do useful computational work
for the approximations $u_{10}$ to $u_{12}$ before knowing the
approximation $u_9$.

Nevertheless, many new PinT methods have been developed since 2001,
and they are often classified into four groups based on the
algorithmic techniques used; see \cite{gander50years} and also the
recent research monograph \cite{Gander:TPTI:2024}, and later text
  for more complete references:
\begin{enumerate}
\item Methods based on {\em Multiple Shooting} going back to the work
  of \cite{Nievergelt:1964:PMI,Bellen:1989:PAI,Chartier:1993:APS},
  leading to \cite{Saha:1996:API} and culminating in the {\em
    Parareal} algorithm \cite{Lions:2001:PTD} and many variants.
\item Methods based on {\em Domain Decomposition}
  \cite{Schwarz1870} and {\em Waveform Relaxation} 
  \cite{Lelarasmee:1982:WRM} that were combined in
  \cite{Bjorhus:1995:DDS}, resulting in {\em Schwarz Waveform
    Relaxation (SWR)} \cite{Gander:1999:OCO}.
\item Methods based on {\em Multigrid} going back to the parabolic
  multigrid method \cite{Hackbusch:1984:PMG}, and developed into fully
  parallel {\em Space-Time MultiGrid (STMG)} methods
  \cite{Gander:2016:AOANST}.
\item {\em Direct time parallel methods}, which started with parallel
  time stepping techniques \cite{Miranker:1967:PMF}, and led to the
  modern {\em Revisionist Integral Deferred Correction (RIDC)} method
  \cite{CMO10}. Currently very successful methods in this class we
  will see are {\em ParaExp} \cite{gander2013paraexp}, and
  parallelization by diagonalization \cite{maday2008parallelization},
  which led to {\em ParaDiag} \cite{gander:2021:ParaDiag}.
\end{enumerate}
The first three groups contain iterative methods, whereas the last one
contains non-iterative ones, but the boundaries in this classification
are not strict. A good example are the ParaDiag methods, which were
first exclusively direct solvers based on the diagonalization of the
time stepping matrix, but then rapidly iterative variants appeared,
within WR methods from the second group \cite{gander2019convergence}
or within parareal from the first group
\cite{gander2020diagonalization}, and approximate ParaDiag methods
were also applied as stationary iterations or preconditioners for
Krylov methods directly to the all-at once system derived from the
space time discretization in the third group
\cite{mcdonald2018preconditioning,liu2020fast}. Parareal from the
first group can also be interpreted as a multigrid method from the
third group with aggressive coarsening \cite{gander2007analysis}, and
in turn MGRiT as a Parareal algorithm with overlap
\cite{gander2018MIP}.

We would like to adopt a different approach here, however, to classify
PinT methods, specifically based on the types of problems that they
can effectively  solve:
\begin{enumerate} 
\item Effective PinT methods for hyperbolic problems.
\item PinT methods designed  for parabolic problems.
\end{enumerate}
To achieve this, we explain in Section \ref{Sec2} intuitively why
there must be a fundamental distinction in PinT methods when solving
hyperbolic or parabolic problems, and introduce model problems that
will later serve to illustrate this for PinT methods.  We then
describe effective PinT methods for hyperbolic problems in Section
\ref{Sec3}, which in general work even better for parabolic
  problems, and in Section \ref{Sec4}, we present PinT methods
  designed for parabolic problems, which in general struggle when
  applied to hyperbolic problems. We will draw conclusions
in Section \ref{Sec5}. The Matlab codes for the numerical results in
Sections \ref{Sec2}-\ref{Sec4} can be obtained from
\url{https://github.com/wushulin/ActaPinT}.

\section{Model Problems linking the Parabolic and Hyperbolic world}
\label{Sec2}

We will often use as our test problems Partial Differential Equations
(PDEs) that allow us to link the parabolic and hyperbolic worlds. A
typical example is the linear advection-diffusion equation we will
first see in subsection \ref{AdvectionReactionDiffusionSec}, which
contains both parabolic and hyperbolic components.  We will also
frequently use the system of ordinary differential equations (ODEs)
\begin{equation}\label{linearODE}
  \begin{array}{rcll}
   {\bm u}'(t)&=&A{\bm u}(t)+{\bm g}(t),~t\in(0, T], \\
     {\bm u}(0)&=&{\bm u}_0,
  \end{array}
\end{equation}
where $A\in\mathbb{R}^{N_x\times N_x}$ is the discrete matrix arising
from semi-discretizing the PDE in space, because many time parallel
methods (except the domain decomposition based methods) are described
and analyzed for ODEs. A nonlinear important PDE variant of
advection-diffusion is the so-called Burgers' equation that we will
first see in subsection \ref{BurgersSec}.  Similarly to the linear
case, in order to discuss time parallel methods in the nonlinear
setting, we will use the nonlinear system of ODEs
\begin{equation}\label{nonlinearODE}
  \begin{array}{rcll}
   {\bm u}'(t)&=&f({\bm u}(t), t),~t\in(0, T], \\
     {\bm u}(0)&=&{\bm u}_0,
  \end{array}
\end{equation}
where $f: \mathbb{R}^{N_x}\times\mathbb{R}\rightarrow\mathbb{R}^{N_x}$
depends on its first variable in a nonlinear manner, such as $f({\bm
  u}(t), t)=A{\bm u}(t)+B{\bm u}^2(t)+{\bm g}(t)$ for Burgers'
equation.

We introduce however also simpler equations that are either of
parabolic or hyperbolic nature, like the heat equation and the second
order wave equation.  To avoid complicated notation, we will only
consider model problems in one spatial dimension on the unit interval
$\Omega=(0, 1)$, which is not really a restriction, since the
applicability and convergence properties of PinT methods do in general
not depend on the space dimension.

\subsection{Heat equation}

As parabolic model problem, we consider the one dimensional heat
equation, 
\begin{equation}\label{heatequation}
  \begin{array}{rcll}
      \partial_tu(x,t)=\partial_{xx} u(x,t)+g(x,t)&\mbox{in $\Omega\times(0,T]$},\\
  \end{array}
\end{equation}
with initial value $u(x,0)=u_0(x)$ and either homogeneous Dirichlet or
Neumann boundary conditions.  An example solution with $u_0(x)=0$ and
a forcing function that heats at four different time instances
$t_1=0.1, t_2=0.6$, $t_3=1.35$ and $t_4=1.85$ in the middle of the space domain
$\Omega=(0, 1)$,
\begin{equation}\label{fxt}
g(x,t)= 10{\sum}_{j=1}^{4}{\rm exp}(-\sigma[(t-t_j)^2+(x-0.5)^2]),  
\end{equation}
is shown for $\sigma=200$ in Figure \ref{ParabolicExampleFig} in the
first three panels, for homogeneous Dirichlet, Neumann, and also
periodic boundary conditions.
\begin{figure}
  \centering
  \includegraphics[width=1.2in,height=2.6in,angle=0]{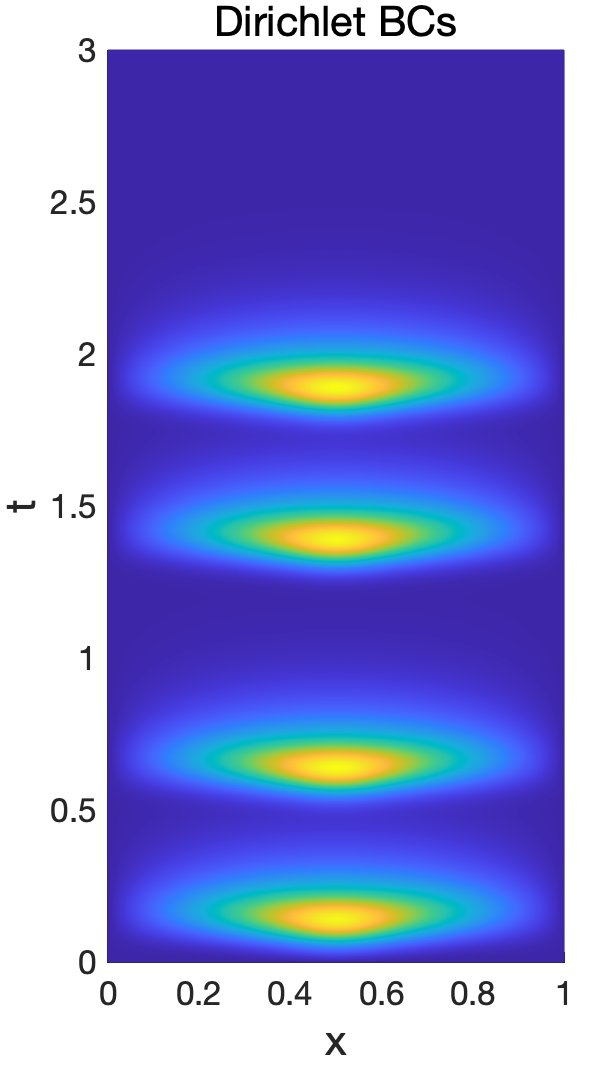}~
    \includegraphics[width=1.2in,height=2.6in,angle=0]{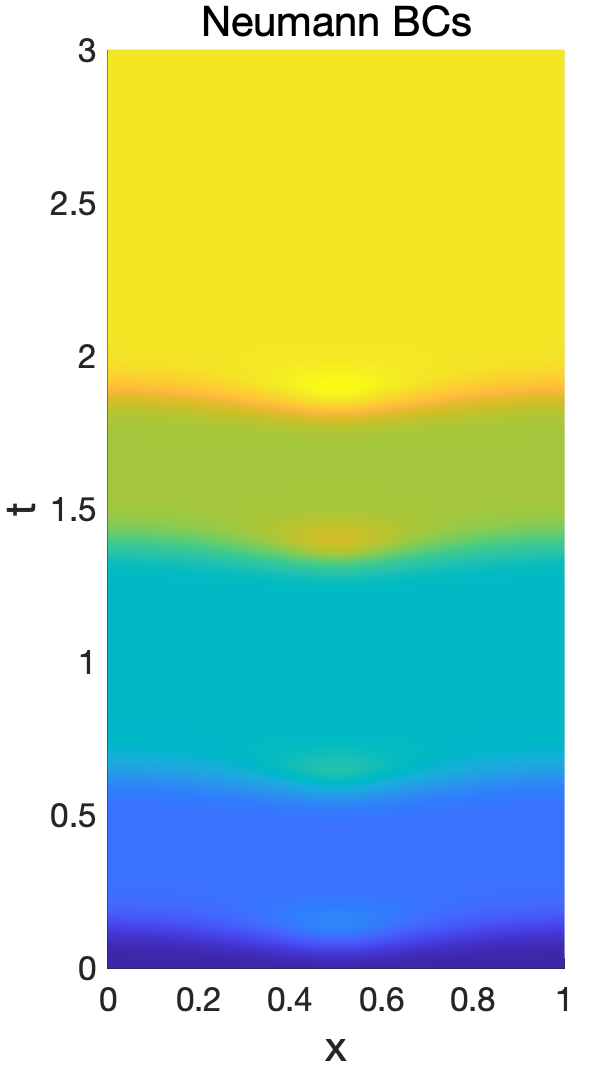}~
        \includegraphics[width=1.2in,height=2.6in,angle=0]{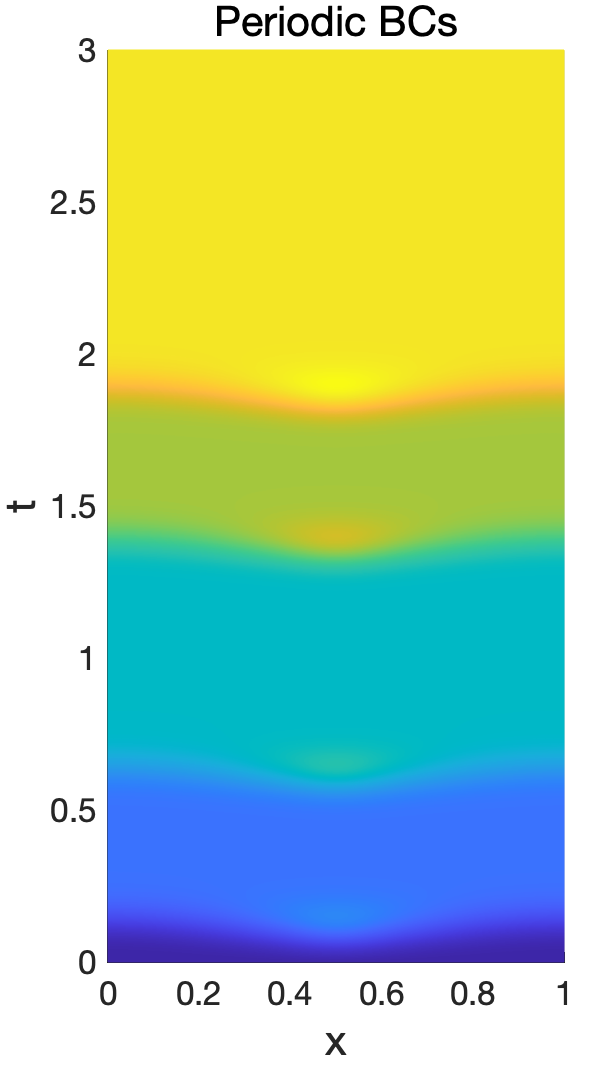}~   \includegraphics[width=1.2in,height=2.6in,angle=0]{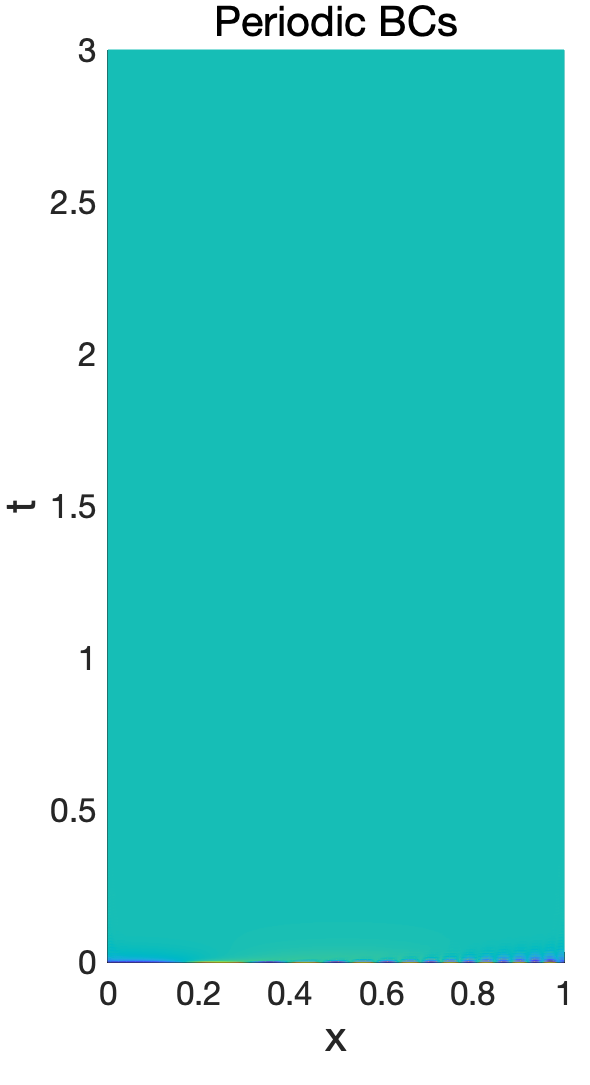}
  \caption{Heat equation with homogeneous Dirichlet, Neumann, and
      periodic boundary conditions (first three panels), using
    the source term \eqref{fxt} and zero initial condition
    $u_0(x)=0$. In the last panel we still use periodic boundary
    conditions, but a zero source term, with initial 
    condition $u_0(x)=\sin^2(8\pi(1-x)^2)$.}
  \label{ParabolicExampleFig}
\end{figure}
We observe that with Dirichlet conditions, the solution does not
propagate far in time, and thus we can compute for example the
solution for $t \in (1.7, 2.2)$ for the fourth source term
independently from the solution at earlier times! This is a prime
example where, despite the causality principle, we can perform useful
computations for later time instances before knowing the earlier
ones. This concept can be naturally understood from daily life
experience: it is straightforward to predict the temperature in
your living room in winter a week or a month in advance; you simply
need to know if the heater will be on and the windows closed.

However, this changes significantly in our simple model problem when
Neumann conditions are applied, as shown in the second plot of Figure
\ref{ParabolicExampleFig}. Here, the solution for $t \in (1.7,
  2.2)$ is influenced by the first, second and third source term
at earlier times, since heat is now accumulating, nicely illustrating
the causality principle. This scenario corresponds to a perfectly
insulated room where heat cannot escape, and in this situation it is
crucial to know how frequently or for how long the heating was on,
since this heat will stay forever in the perfectly insulated room.
Note however that in practice it is difficult to have a perfectly
insulated room, and heat will always eventually escape, which one
  would model with a Robin boundary condition.

The situation in the third panel of Figure
  \ref{ParabolicExampleFig} with periodic boundary conditions is
  similar to the case with Neumann boundary conditions in the second
  panel, the solution for $t \in (1.7, 2.2)$ is also influenced by the
  first, second and third source term at earlier times, and with
  periodic conditions, heat can never escape.

 In the fourth panel of Figure \ref{ParabolicExampleFig} we show a
  solution with zero source term and periodic boundary conditions, but
  now imposing an initial condition with a precise, oscillating
  signal, namely $u_0(x)=\sin^2(8\pi(1-x)^2)$. We see that the only
  information left from this signal after a very short time already
  is a constant, about the same constant as from the first two source
  terms in the second and third panel of Figure
  \ref{ParabolicExampleFig}.

In spite of the causality principle, time parallelization and thus
  PinT computations for a heat equation, and also more general
  parabolic problems, should thus be rather easily possible in the
  case of Dirichlet boundary conditions, since then the solutions are
  completely local in time \cite{gander:2024:PararealNoCoarse}, as it
  is the case in space with solvation models in computational
  chemistry, see
  \cite{ciaramella2017analysis,ciaramella2018analysis,ciaramella2018analysis3}.
  With Neumann or periodic boundary conditions, it should still be
  well possible to do PinT computations, provided one can propagate low
  frequency solution components, like the constant in our example,
  effectively over long time, using a coarse grid for example.

\subsection{Advection-diffusion equation}\label{AdvectionReactionDiffusionSec}

We now consider the advection-diffusion equation with homogeneous
Dirichlet and periodic boundary conditions\footnote{We would not learn
  anything new with Neumann conditions.} on the unit domain
$\Omega=(0, 1)$,
\begin{equation}\label{ADE}
    \partial_tu(x,t) + \partial_xu(x,t) - \nu\partial_{xx}u(x,t) = g(x,t) \quad \text{in } \Omega\times(0,T],
\end{equation}
with initial condition $u(x,0)=u_0(x)$, where $\nu>0$ is the diffusion
parameter.  We show in Figure \ref{ADEFig} in the top row 
the solution obtained with zero Dirichlet boundary conditions, and in
the bottom row with periodic boundary conditions.
\begin{figure}
  \centering
  \includegraphics[width=1.2in,height=2.6in]{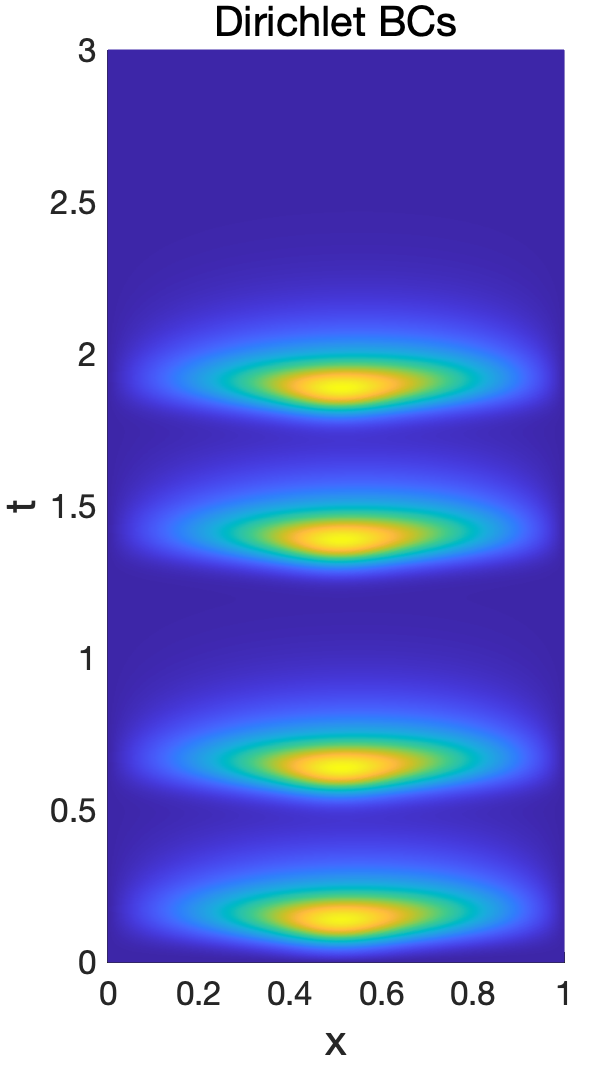}~
  \includegraphics[width=1.2in,height=2.6in]{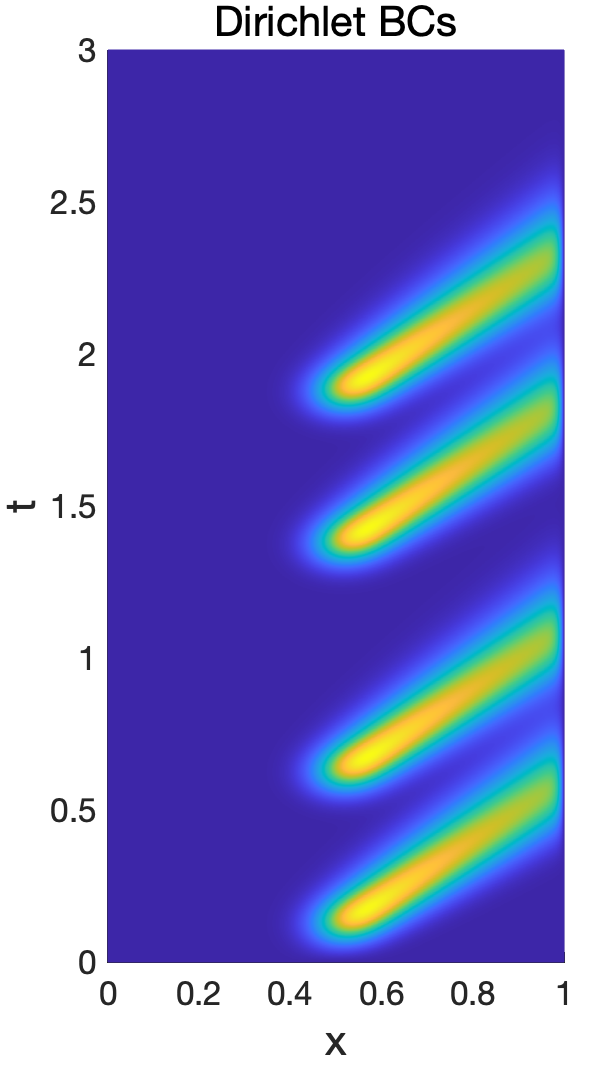}~
  \includegraphics[width=1.2in,height=2.6in]{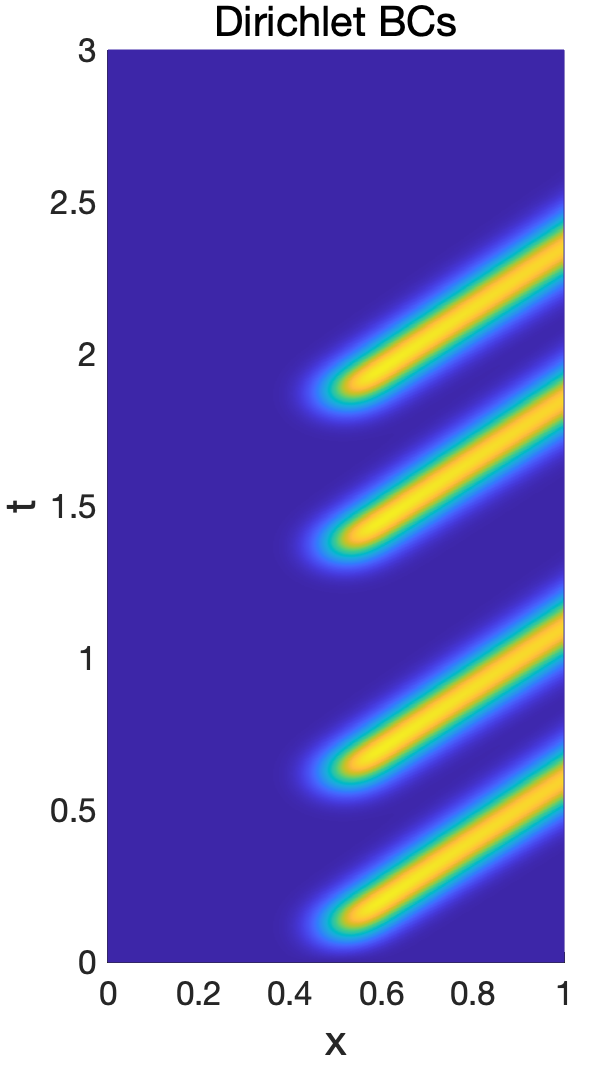}~
    \includegraphics[width=1.2in,height=2.6in]{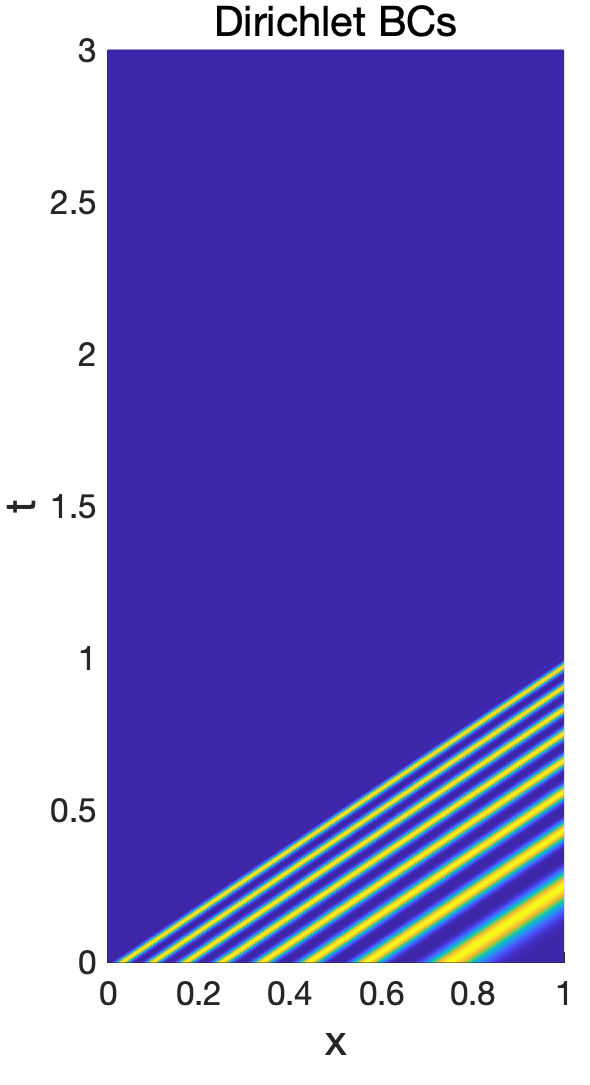}  \\
  \includegraphics[width=1.2in,height=2.6in]{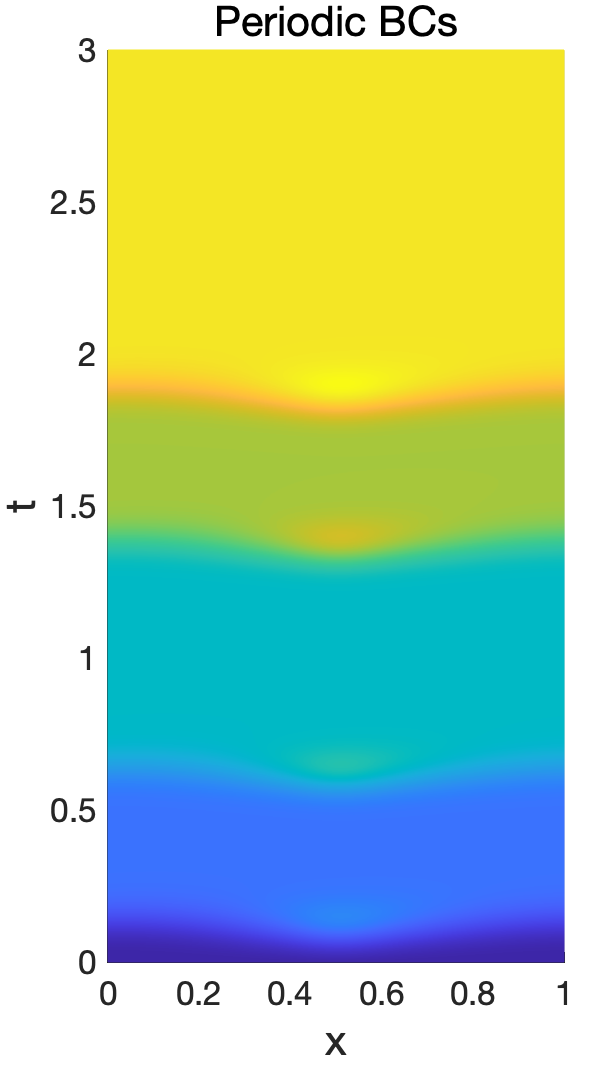}~
  \includegraphics[width=1.2in,height=2.6in]{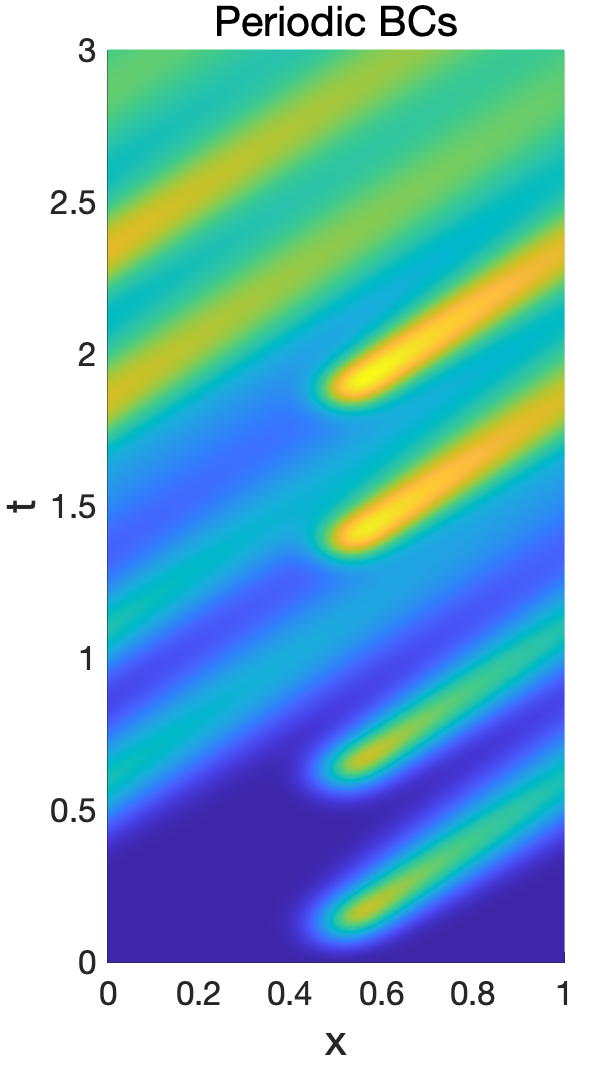}~
  \includegraphics[width=1.2in,height=2.6in]{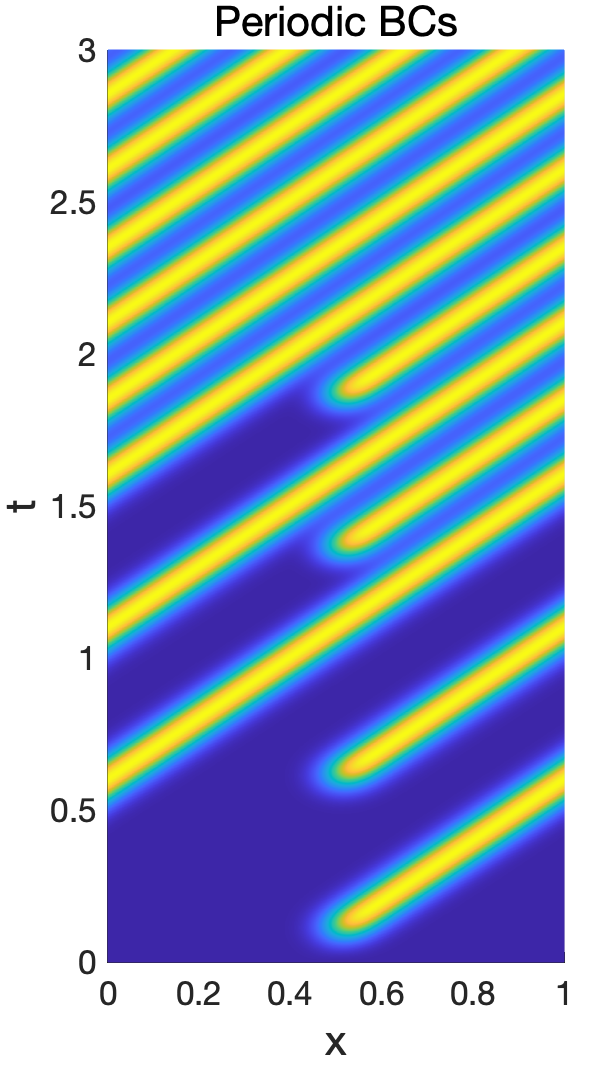}~
      \includegraphics[width=1.2in,height=2.6in]{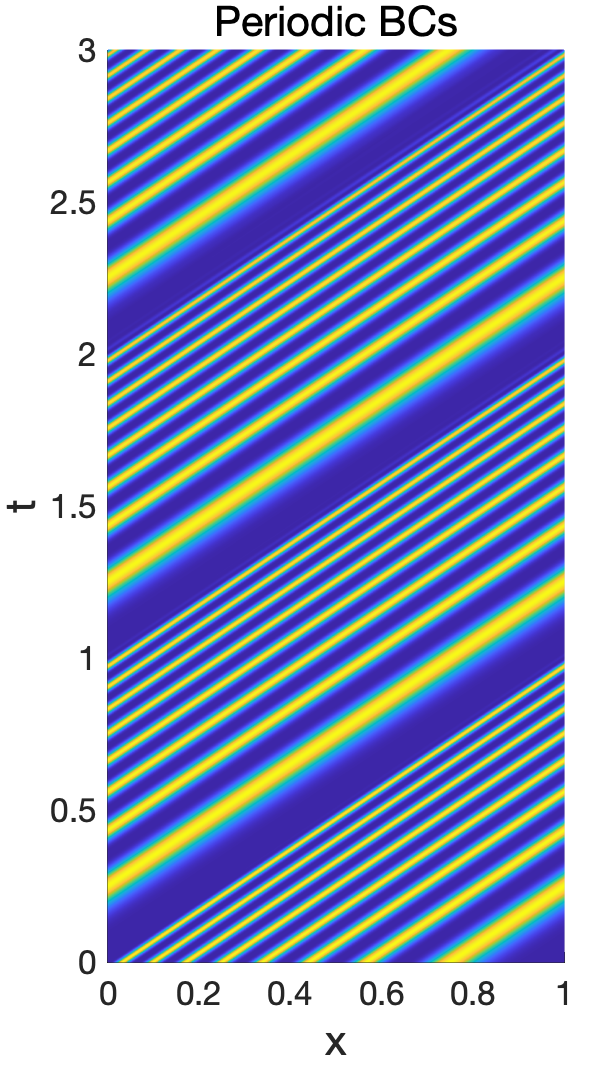}
  \caption{Advection-diffusion equation with zero Dirichlet
      boundary conditions (top row) and periodic boundary conditions
      (bottom row). In the first three pannels, we use a zero initial
      condition, $u_0(x) = 0$, and the same source term as in Figure
      \ref{ParabolicExampleFig} for the heat equation, and a
      smaller and smaller diffusion parameter $\nu=1, 10^{-2}$ and
      $5\times10^{-4}$. Last column: solution for zero source term and
      initial condition $u_0(x)=\sin^2(8\pi(1-x)^2)$ with small
      diffusion $\nu=5\times10^{-4}$.}
\label{ADEFig}
\end{figure}
 In the first three panels, we use a zero initial condition,
  $u_0(x)=0$ and the same source term \eqref{fxt} we had used for the
  heat equation for three different values of the diffusion parameter,
  $\nu=1, 10^{-2}$ and $5\times10^{-4}$. We see that when $\nu$ is
  large, then the diffusion part dominates and the solution has
similar properties as the solution of the heat equation. If $\nu$ is
small however, i.e., the advection part plays a dominant role, then
the solution is transported from left to right over much longer time,
as we see in  the top middle two panels in Figure \ref{ADEFig}. In
  the top right panel, we use a zero source term, but a non-zero
  initial condition $u_0(x)=\sin^2(8\pi(1-x)^2)$ and again the small
  diffusion parameter $5\times10^{-4}$. We see that now all the fine
  features present in the high frequency components of the initial
  condition are transported far in time.  Nevertheless, for both
$\nu$ large and $\nu$ small, we can still compute the solution for
 $t\in (1.25, 2.5)$  before we obtain the solution earlier in time,
because all solution components eventually are diffused or leave the
domain.

For periodic boundary conditions however, we see in Figure
  \ref{ADEFig} in the bottom row that the advection-diffusion equation
  transports information over long time: in the first panel with large
  diffusion this information is only low frequency, a constant, like
  for the heat equation, and PinT computations are still possible if
  one has a way of transporting coarse solution components far in
  time, using a coarse grid for example. In the next two panels
  however, we see that when the diffusion parameter becomes small,
  more and more fine information is transported very far in time, and
  for successful PinT computation there must be a mechanism to
  propagate this information effectively far in time.  The last panel
  without source and just a non-zero initial condition shows that for
  small diffusion, a lot of fine, high frequency information
  propagates very far in time, and we cannot pre-compute the solution
  later in time any more without knowing the solution earlier in time
  when the diffusion parameter becomes small. It is therefore
  difficult to do PinT computations, especially when
  $\nu\rightarrow0$, in the hyperbolic limit. This is fundamentally
different from the heat equation case, and only becomes manifest with
periodic boundary conditions and small diffusion, an important
  point when testing the performance of PinT methods on advection
  dominated problems.
  
\subsection{Burgers' equation}\label{BurgersSec}
To illustrate the difference between the various PinT methods in a
  non-linear setting, we will use Burgers' equation,
\begin{equation}\label{Burgers}
    \begin{array}{rcll}
        \partial_tu(x,t) - \nu\partial_{xx}u(x,t) + \frac{1}{2}\partial_x(u^2(x,t)) & = & g(x,t) & \text{in } \Omega\times(0,T],\\
        u(x,0) & = & u_0(x) & \text{in } \Omega,
    \end{array}
\end{equation}
with $\nu > 0$. We show in Figure \ref{BurgersFig} in the top row 
the solution obtained with zero Dirichlet boundary conditions, and in
the bottom row with periodic boundary conditions.
\begin{figure}
    \centering
  \includegraphics[width=1.2in,height=2.6in]{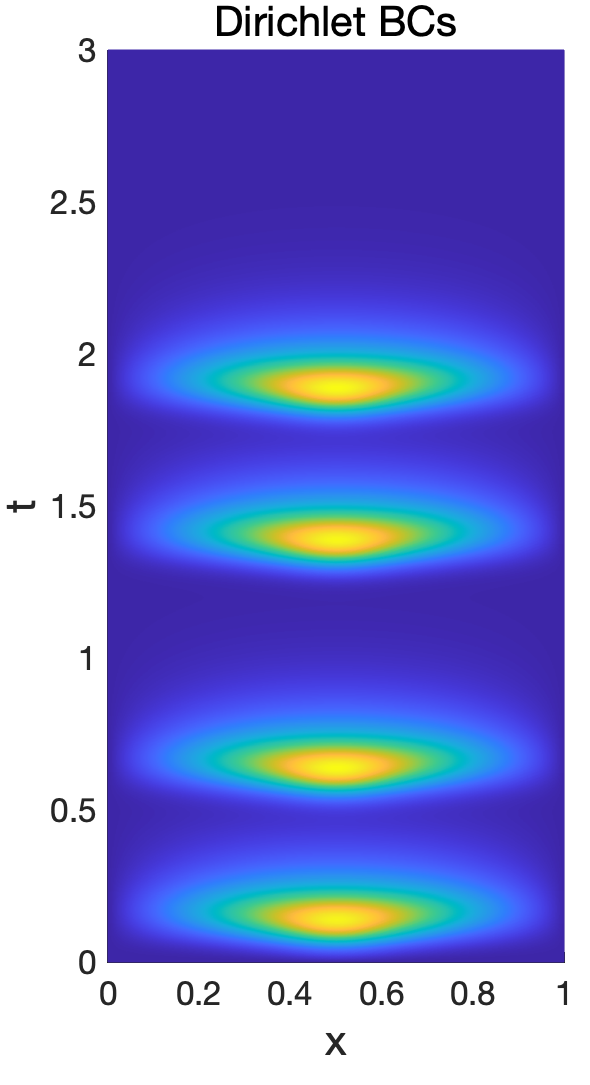}~
  \includegraphics[width=1.2in,height=2.6in]{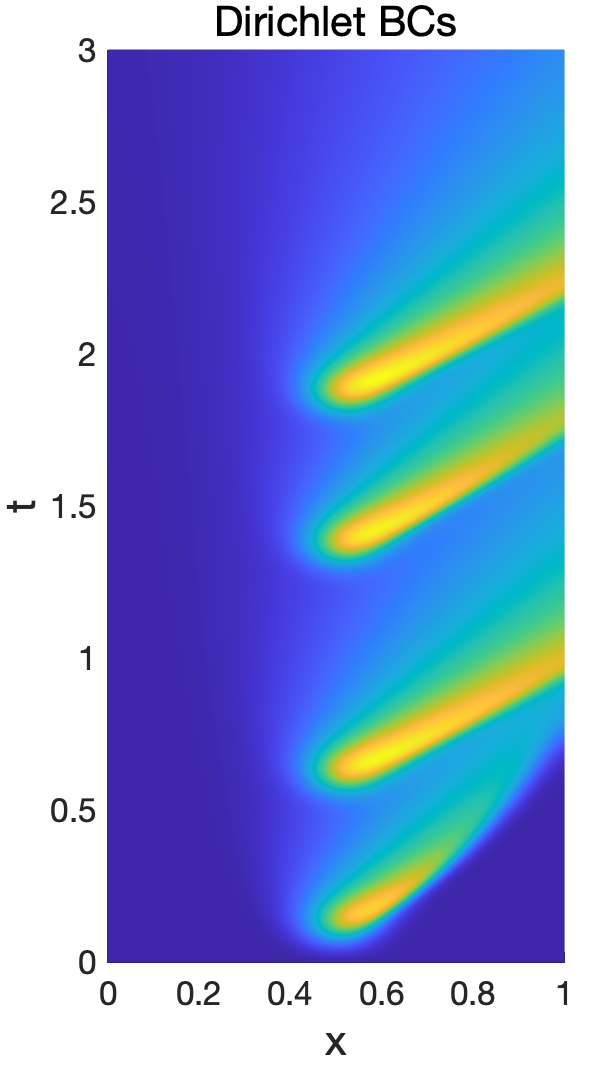}~
  \includegraphics[width=1.2in,height=2.6in]{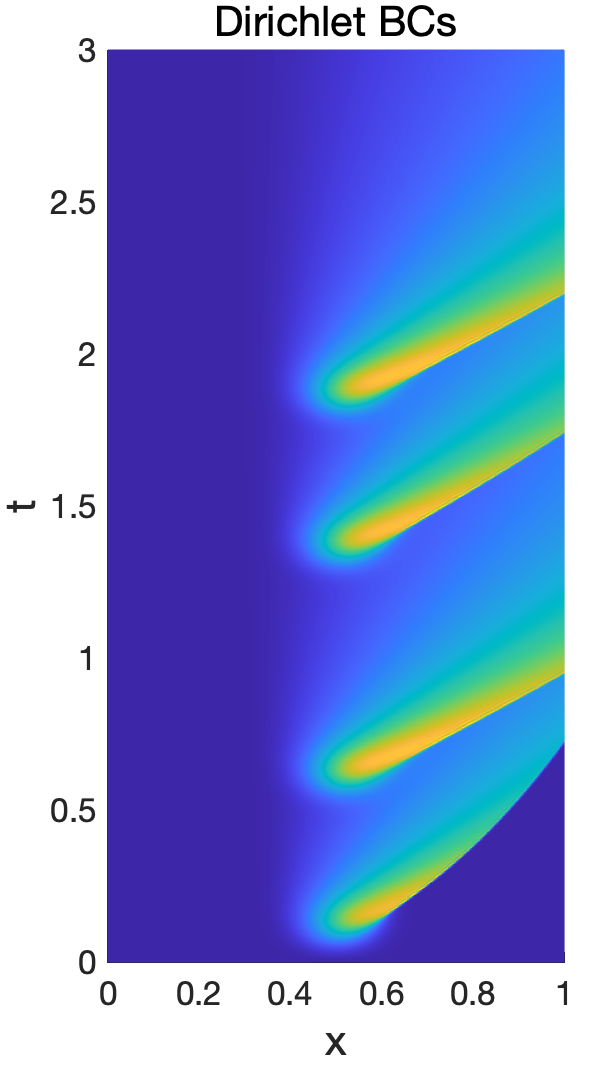}~
    \includegraphics[width=1.2in,height=2.6in]{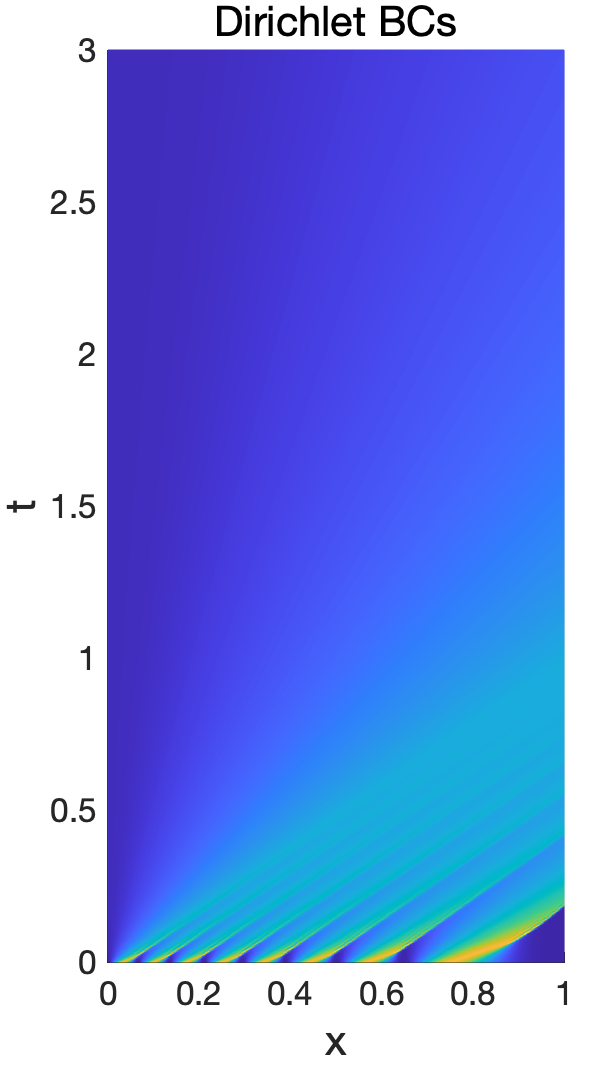}  \\
  \includegraphics[width=1.2in,height=2.6in]{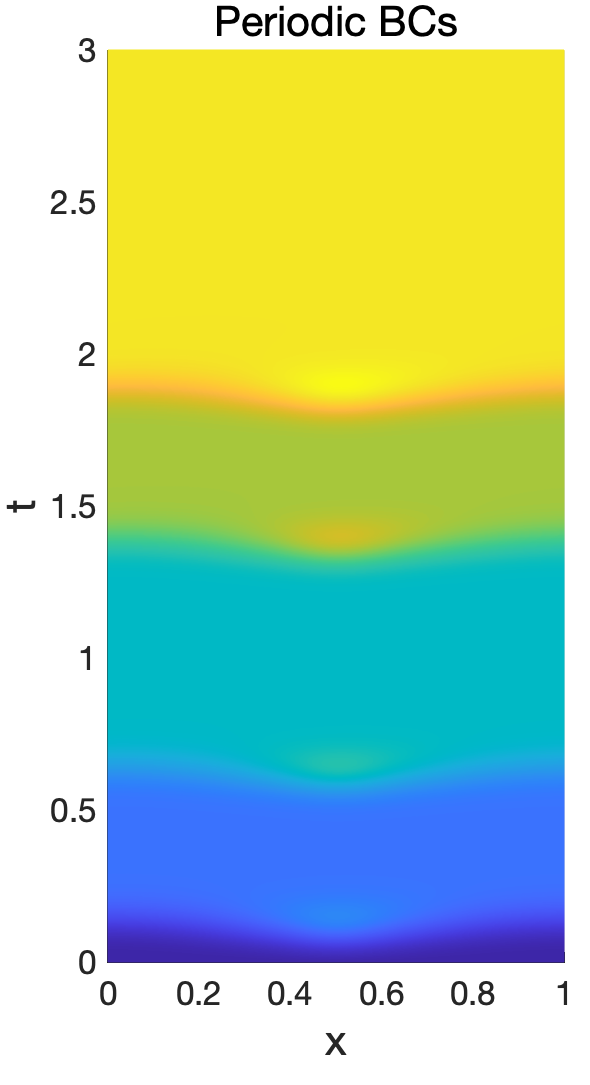}~
  \includegraphics[width=1.2in,height=2.6in]{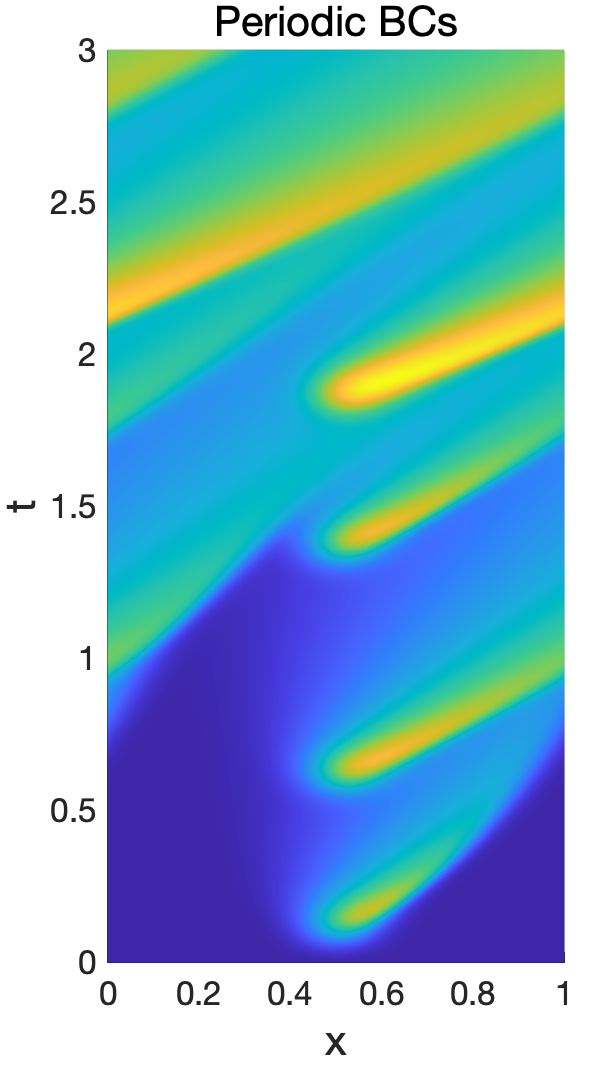}~
  \includegraphics[width=1.2in,height=2.6in]{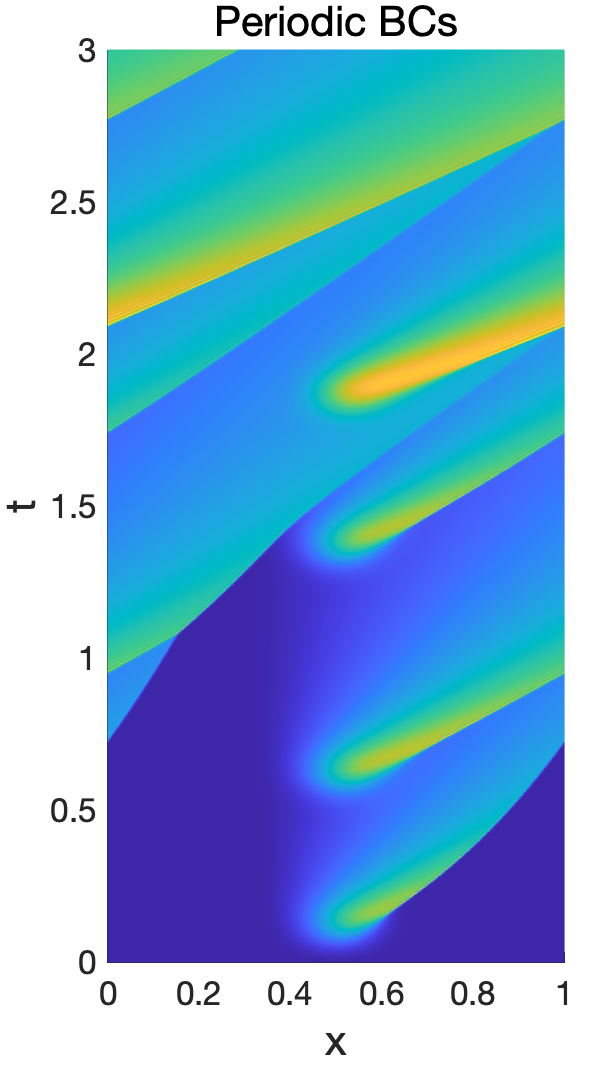}~
      \includegraphics[width=1.2in,height=2.6in]{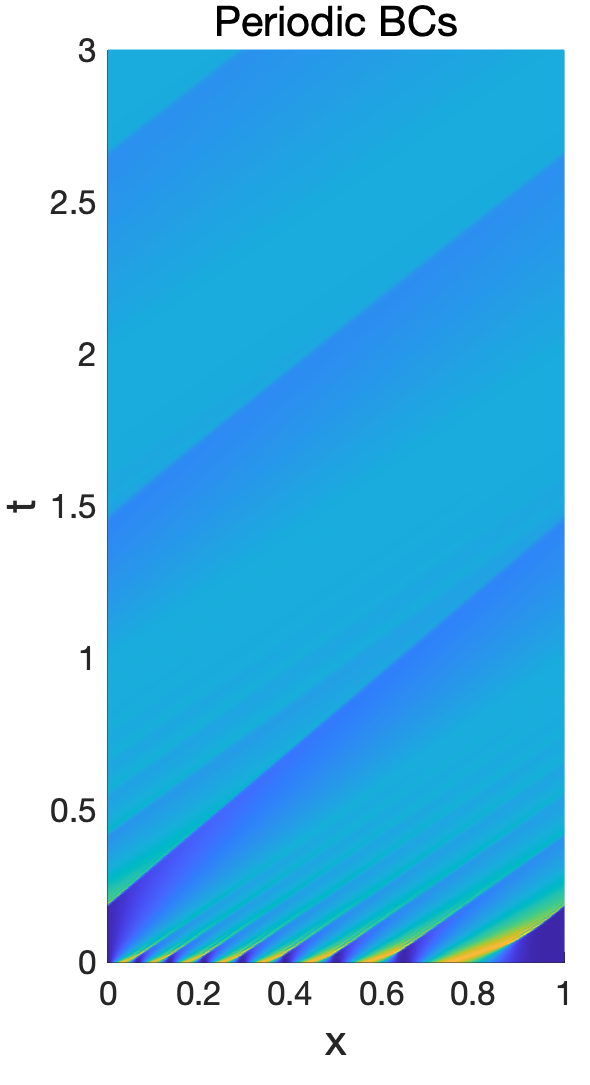}
    \caption{Burgers' equation with zero Dirichlet boundary
        conditions (top row) and periodic boundary conditions (bottom
        row). In the first three pannels, we use a zero initial
      condition, $u_0(x) = 0$, and the same source term as in Figure
      \ref{ParabolicExampleFig} for the heat equation, and a
      smaller and smaller diffusion parameter $\nu=1, 10^{-2}$ and
      $5\times10^{-4}$. Last column: solution for zero source term and
      initial condition $u_0(x)=\sin^2(8\pi(1-x)^2)$ with small
      diffusion $\nu=5\times10^{-4}$.}
    \label{BurgersFig}
\end{figure}
 In the first three panels, we use a zero initial condition,
  $u_0(x)=0$ and the same source term \eqref{fxt} we had used for the
  heat and advection-diffusion equation, also for three different
  values of the diffusion parameter, $\nu=1, 10^{-2}$ and
  $5\times10^{-4}$. We see that when $\nu$ is large, then the
diffusion part dominates and the solution has similar properties as
the solution of the heat equation. If $\nu$ is small and the
non-linear advection part starts playing a dominant role, then the
solution is transported from left to right over much longer time, as
we see in the top middle two panels in Figure \ref{ADEFig}, like
  for advection-diffusion. However, we see a further very important
  new phenomenon in the non-linear case: the solution shape changes as
  well, and even with the smooth source term, very sharp edges are
  forming in the solution, so called shock waves, containing very high
  frequency components that travel fare in space and time. In the top
  right panel, we use a zero source term, but a non-zero initial
  condition $u_0(x)=\sin^2(8\pi(1-x)^2)$ as for advection-diffusion
  before, and again the small diffusion parameter $5\times10^{-4}$. We
  see that also from the already fine features present in the high
  frequency components of the initial condition even sharper edges are
  formed in shock waves, and all are transported far in time.
Nevertheless, for both $\nu$ large and $\nu$ small, we can still
compute the solution for $t\in (1.25, 2.5)$ before we obtain the
solution earlier in time, because all solution components eventually
are diffused or leave the domain, as in the advection-diffusion
  case in the top row in Figure \ref{ADEFig}.

In contrast, for periodic boundary conditions, we see in the
  bottom row of Figure \ref{BurgersFig} that our observations from the
  bottom row for the advection-diffusion equation in Figure
  \ref{ADEFig} are further accentuated, as soon as the diffusion
  parameter becomes small: more and more fine information is generated
  and transported very far in time, through shock waves that are
  forming. For successful PinT computation, such high frequency shock
  waves must be transported by a mechanism that propagate them
  effectively far in space and time, which is very difficult using a
  coarse grid for example.  The last panel without source and just a
  non-zero initial condition shows the same effect for an initial
  condition transported over very long time: one cannot pre-compute
  the solution later in time any more without knowing the solution
  earlier in time when the diffusion parameter becomes small. It is
  therefore even harder to do PinT computations in such non-linear
  problems when $\nu\rightarrow0$, in the hyperbolic limit, where
  shock waves are natural in the solutions and all frequency
  components in it travel very far in space and time. Note that again
  we need periodic boundary conditions and small diffusion to
  encounter these difficulties, an important point when testing the
  performance of PinT methods on such problems. In the next subsection
  we will see that for hyperbolic problems, these difficulties already
  appear no matter what boundary conditions are used.

\subsection{Second-order wave equation}
For hyperbolic problems, we will use the second-order wave equation as
our model problem,
\begin{equation}\label{WaveEquation1d}
    \begin{array}{rcll}
      \partial_{tt}u(x,t) & = &c^2\partial_{xx}u(x,t) + g(x,t)
      &\text{in } (0,1)\times(0,T],\\
      u(x,0) & = &u_0(x) &\text{in } (0,1),\\
      \partial_tu(x,0) & = & 0 &\text{in } (0,1),\\
    \end{array}
\end{equation}
with a constant wave speed $c>0$. We show in Figure
\ref{WaveExampleFig}
\begin{figure}
  \centering
 \includegraphics[width=1.2in,height=2.6in,angle=0]{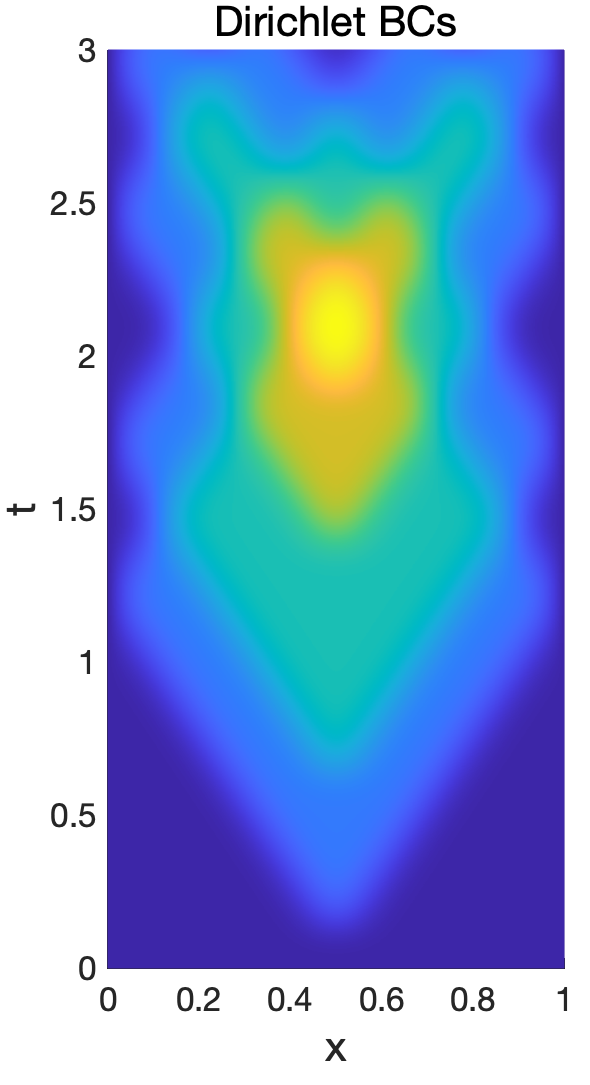}~
     \includegraphics[width=1.2in,height=2.6in,angle=0]{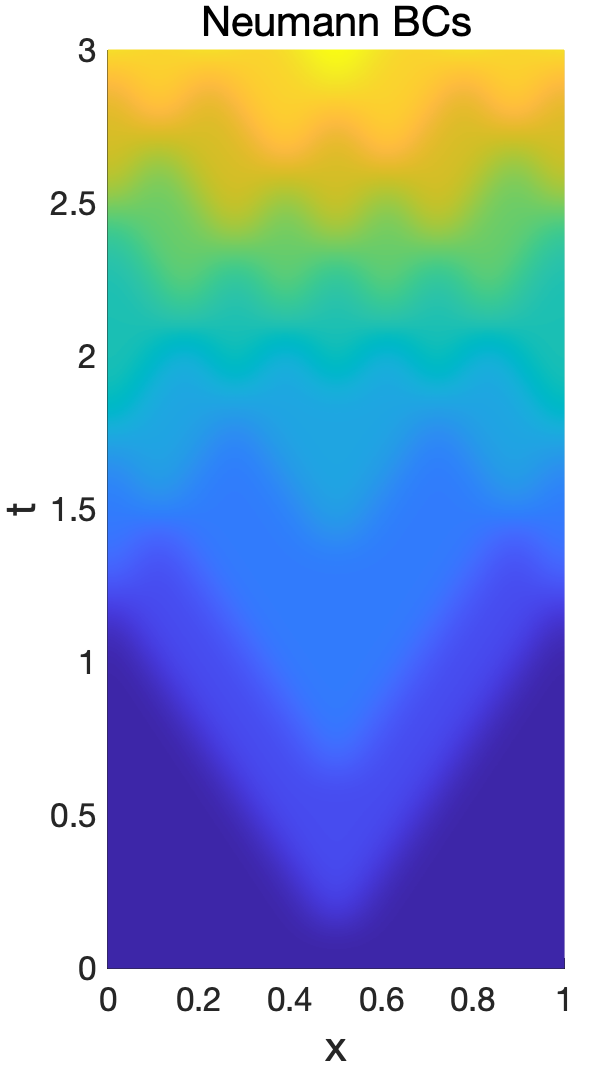}~
    \includegraphics[width=1.2in,height=2.6in,angle=0]{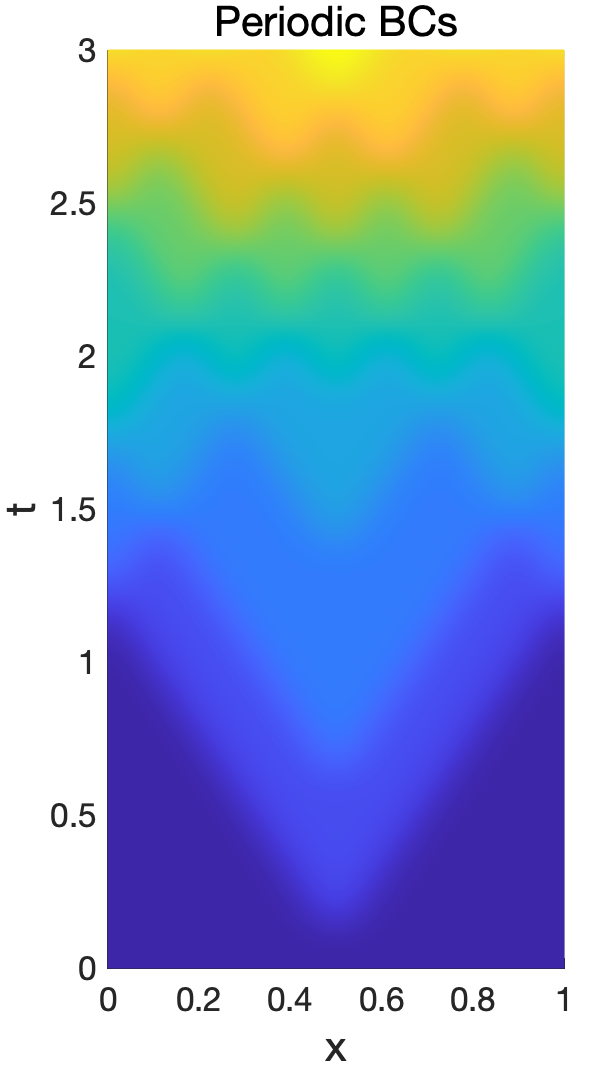}~    \includegraphics[width=1.2in,height=2.6in,angle=0]{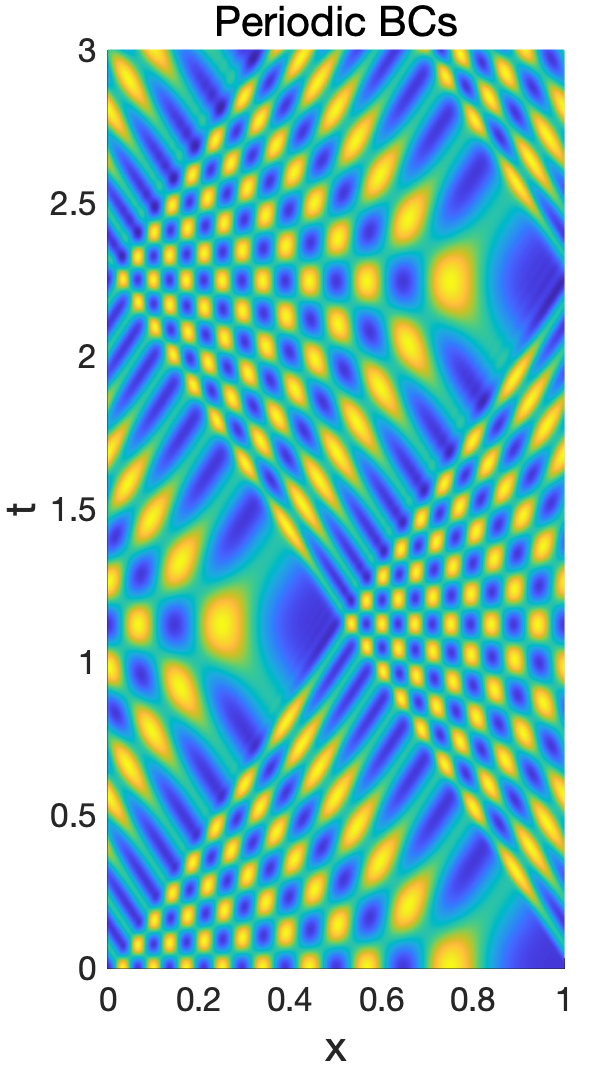}
  \caption{Solution of the second-order wave equation, with $c^2 =
    0.2$, zero initial condition $u_0(x)=0$ and the same source
      term as in Figure \ref{ParabolicExampleFig} for Dirichlet,
      Neumann and periodic boundary conditions (first three panels),
      and in the last panel the solution with zero source term and
      non-zero initial condition $u_0(x)=\sin^2(8\pi(1-x)^2)$.}
  \label{WaveExampleFig}
\end{figure}
in the first three panels the solution of the wave equation with the
same source term \eqref{fxt} we used before, with Dirichlet, Neumann
and periodic boundary conditions. In all cases, we observe that for
the wave equation, the solution depends in a complex and detailed
manner over a long time on the various source terms in space and
time.  In the last panel in Figure \ref{WaveExampleFig}, we show
  the solution of the wave equation with zero source term, but using
  the initial condition $u_0(x)=\sin^2(8\pi(1-x)^2)$ and zero first
  derivative in time. We see that the solution of this hyperbolic
  problem depends in a very detailed manner on all frequency
  components present in the initial condition, and this would be the
  same also for the other boundary conditions. This is typical for
hyperbolic problems and all boundary conditions; one only needs
periodic boundary conditions in the advection-diffusion and Burgers'
equation case to make this difficulty appear for small $\nu$, because
the advection term is first order and the transport has a
direction. For the wave equation and other hyperbolic problems, this
detailed and long time propagation in several directions with
reflections is present already for Dirichlet and Neumann boundary
conditions. It is this propagation of fine information over long time
in hyperbolic problems that makes time parallelization more
challenging than for parabolic problems, and requires different PinT
techniques to address it.

\section{Effective PinT methods for hyperbolic problems}\label{Sec3}

We have seen in Section \ref{Sec2} that parabolic problems have
solutions which are rather local in time, see Figure
\ref{ParabolicExampleFig}, where with Dirichlet conditions all
information is forgotten very rapidly over time, and with Neumann
and periodic boundary conditions only the lowest frequency
component, namely the constant, remains over long time. This changes
when transport terms are present and become dominant, see Figures
\ref{ADEFig}, \ref{BurgersFig}, and in the hyperbolic limit, and for
hyperbolic problems in general, all frequency components can travel
arbitrarily far in space and time, see Figure \ref{WaveExampleFig} for
the second order wave equation. PinT methods must take this into
account to be effective. It is interesting that many methods designed
specifically for hyperbolic problems also work well (or even better)
for parabolic problems. An exception are the mapped tent pitching
methods introduced at the end of Subsection \ref{sec3.2} which use the
finite speed of propagation in hyperbolic problems in their
construction. On the other hand, PinT methods designed for
parabolic problems (see Section \ref{Sec4}) do not perform well in
general for hyperbolic problems.

\subsection{Historical development}

We present four PinT methods that have proven effective for
hyperbolic problems. For each method, we show the main theoretical
properties and demonstrate these properties using the four PDEs
introduced in Section~\ref{Sec2}.

The first methods are rooted in solving overlapping or
non-overlapping space-time continuous subproblems, an approach
initially proposed for parabolic problems in \cite{G98} and
independently introduced in \cite{GanK02}. This strategy incorporates
elements of both Domain Decomposition (DD) methods, a long-established
technique for parallel PDE solving going back to \cite{Schwarz1870},
and Waveform Relaxation (WR) methods, which originated in circuit
simulations \cite{LRS82}.  These methods have been developed and
analyzed both for parabolic and hyperbolic problems in
\cite{Gander:1997:PhD}, and the name Schwarz Waveform Relaxation (SWR)
methods was coined in \cite{Gander:1999:OCO}. Further results for
non-linear parabolic problems can be found in
\cite{Gander:1998:WRA,gander2005overlapping}. Optimized Schwarz
Waveform Relaxation (OSWR) methods using more effective transmission
conditions were developed for parabolic problems in
\cite{Gander:2007:OSW,bennequin2009homographic,bennequin2016optimized},
and for hyperbolic problems in
\cite{Gander:2003:OSWW,Gander:2004:ABC}, see also \cite{gander2023non}
for non-linear advection-diffusion equations. The recently developed
Unmapped Tent Pitching (UTP) technique \cite{Ciaramella:2023:UTP} is
based on SWR. There are also Dirichlet-Neumann and Neumann-Neumann
Waveform relaxation variants, see
\cite{Gander:2014:DNNWR,gander2021dirichlet}.

The third method is based on the time parallelization of the Integral
Deferred Correction (IDC) technique. IDC for evolution problems was
first introduced in \cite{BS84}, and was later identified as a
specialized time-integrator in \cite{DGR00}, which theoretically has
the capability to generate numerical solutions of arbitrarily high
order by accurately treating the associated integral. Revisionist
Integral Deferred Correction (RIDC) is one such technique \cite{CMO10}
that can be used parallel in time, and there is another recent
parallel version (PIDC) from \cite{GT07}, which we will introduce in
detail in Section \ref{Sec3.4}.

The fourth time-parallel method that we will introduce in Section
\ref{Sec3.5} is the ParaExp method, proposed a decade ago by
\cite{gander2013paraexp}, which relies on a new strategy of
separately handling the initial value and source term,  see also
  \cite{merkel2017paraexp,kooij2017block}, and
  \cite{GGP18} for a non-linear variant. 

Finally, in Section \ref{Sec3.6}, we will present the ParaDiag family
of methods. Time-parallel methods based on diagonalization were first
proposed in \cite{maday2008parallelization} as direct time parallel
solvers, without iteration, and they were studied in more detail in
\cite{GHR16} for parabolic problems, with a non-linear variant in
\cite{GH17}, and in \cite{gander2019direct} for hyperbolic
problems. Rapidly then iterative variants appeared, within WR methods
\cite{gander2019convergence} or within Parareal
\cite{gander2020diagonalization}. Approximate ParaDiag methods were
also applied as preconditioners for Krylov methods directly to the
all-at once system in \cite{mcdonald2018preconditioning} and
\cite{liu2020fast}.  A comprehensive study of ParaDiag methods
appeared in \cite{gander2019direct}.  Since then, ParaDiag methods
have gained widespread traction in the PinT field, with new techniques
enhancing these methods, see for example \cite{Kressner2022} for a
direct ParaDiag technique using interpolation, and \cite{GP24}
for a new ParaDiag variant combining the Sherman-Morrison-Woodbury
formula and Krylov techniques.

\subsection{Schwarz waveform relaxation (SWR) methods}\label{sec3.2}

SWR combines the strengths of the classical Schwarz DD method and WR,
while overcoming some of their inherent limitations. In the context of
evolution PDEs, the Schwarz DD method typically involves a uniform
implicit time discretization, followed by the application of the DD
technique to solve the resulting elliptic problems at each time step
sequentially; see e.g.  \cite{Cai91,Meu91,Cai94}.  The DD iterations
need to converge at each time step before proceeding to the next,
across all subdomains, and one has to use the same time discretization
across subdomains, which undermines a key advantage of DD methods,
namely to tailor numerical treatments for each subdomain individually.

On the other hand, the classical WR method begins with  a system of
  ODEs, often obtained from a spatial discretization of an evolution
  PDE, which is  then solved using a dynamic iteration  similar to
  the Picard iteration, but using an  appropriate system
partitioning. For instance, in the case of the linear  system of
  ODEs \eqref{linearODE}, the WR iteration can be expressed as 
$$
\frac{d{\bm u^k}(t)}{dt}-M{\bm u}^k(t)=N{\bm u}^{k-1}(t)+f(t), ~t\in(0, T),
$$
where $k\geq1$ denotes the iteration index, ${\bm u}^k(0)={\bm
  u}_0$ for all $k\geq0$, and $(M, N)$ represents a consistent
splitting of $A$ such that $A=M+N$. For Jacobi (diagonal) or
Gauss-Seidel (triangular) type splittings, solving for ${\bm u}^k(t)$
boils down to solving a series of scalar ODEs.  In the Jacobi case,
all these ODEs can be solved in parallel, making this into a PinT
method in the sense that the future of all unknowns is approximated
before the future of connected unknowns is already known. Similarly in
the Gauss-Seidel case, one can obtain such parallelism using red-black
or other colorings. Even more parallelism can be introduced
using the cyclic reduction technique, see
\cite{worley1991parallelizing,HVW95,SV00}. However, a significant
challenge of WR lies in finding an effective system splitting
to ensure rapid convergence. As Nevanlinna remarked in \cite{Nev89}:
\begin{quote}
  ``{\em In practice, one is interested in knowing what subdivisions
    yield fast convergence for the iterations... The splitting into
    subsystems is assumed to be given. How to split in such a way that
    the coupling remains weak is an important question.}''
\end{quote}
A bad splitting can lead to arbitrarily slow convergence or even
divergence, rendering WR impractical.

SWR circumvents these limitations by initially decoupling the spatial
domain (rather than performing a spatial discretization first) and
then independently solving the space-time continuous PDEs on these
subdomains, similar to the WR approach. This approach allows for the
use of tailored space and time discretizations for each subdomain
problem; but more importantly, knowing that the space-time subdomain
problems are coupled for a particular PDE, one can design transmission
conditions which decouple the problems such that the methods converge
very rapidly, completely addressing the difficulty identified by
Nevanlinna above. This leads to the class of Optimized Schwarz
Waveform Relaxation (OSWR) methods, which have been studied for many
different types of PDEs, see e.g. \cite{martin2009schwarz} for the
shallow water equations, \cite{courvoisier2013time} for the time
domain Maxwell equations,
\cite{halpern2010optimized,besse2017schwarz,antoine2017analysis} for
Schr\"odinger equations, \cite{audusse2010optimized} for the primitive
equations of the ocean, \cite{antoine2016lagrange} for quantum wave
problems, \cite{wu2017optimized} for fractional diffusion problems,
\cite{thery2022analysis} the coupled Ekman boundary layer problem, and
also the many references therein.  OSWR methods have distinct
convergence characteristics for first-order parabolic problems (such
as the advection-diffusion equation \eqref{ADE} and the nonlinear
Burgers' equation \eqref{Burgers}) compared to second-order hyperbolic
problems like the wave equation \eqref{WaveEquation1d}. In the
following, we present these two types of problems separately.

\subsubsection{First-order  parabolic problems}

For the advection-diffusion equation \eqref{ADE}  with  homogeneous
Dirichlet boundary conditions, $u(0, t)=u(L,t)=0$, and an initial
condition, $u(x,0)=u_0(x)$, the OSWR  method with the two
  overlapping subdomains $\Omega_1:=(0, \beta L)$ and
  $\Omega_2:=(\alpha L, L)$, $\alpha<\beta$,
  and Robin transmission conditions is given by 
\begin{equation}\label{ADESWR}
\begin{split}
  & \begin{cases}
    \partial_tu_1^k(x,t)+\mathcal{L}u_1^k(x,t)=0    ~~(x,t)\in\Omega_1\times(0,T],\\
    u_1^k(0, t)=0,   \\
    \frac{1}{p}\partial_xu_1^k(\beta L, t)+ u_1^k(\beta L,t)= \frac{1}{p} \partial_xu_2^{k-1}(\beta L, t)+u_2^{k-1}(\beta L,t),  
   \end{cases}\\
&\begin{cases}
   \partial_tu_2^k(x,t)+\mathcal{L}u_2^k(x,t)=0  ~~ (x,t)\in\Omega_2\times(0,T],\\
   \frac{1}{p}\partial_xu_2^k(\alpha L, t)- u_2^k(\alpha L,t)= \frac{1}{p} \partial_xu_1^{k-1}(\alpha L, t)- u_1^{k-1}(\alpha L,t),   \\
   u_2^k(L, t)=0,  
  \end{cases}
\end{split}
\end{equation}
with $\mathcal{L}=\partial_x-\nu\partial_{xx}$, and initial conditions
$u_1^k(x,0)=u_0(x)$ for $x\in\Omega_1$ and $u_2^k(x,0)=u_0(x)$ for
$x\in\Omega_2$. Here, $k\geq1$ represents the iteration index,
$\{u_1^0(\alpha L, t), u^0(\beta L,t)\}$ are initial guesses for
$t\in(0, T]$, and $0<\alpha\leq \beta<1$, and $(\beta-\alpha) L$
  denotes the overlap size. We use Robin Transmission Conditions (TCs)
  in \eqref{ADESWR} with parameter $p>0$ at $x=\alpha L$ and $x=\beta
  L$ to transmit information between the subdomains, and the classical
  Dirichlet TCs correspond to the limit $p\rightarrow\infty$, i.e.,
  $u^k_1(\beta L, t)=u^{k-1}_2(\beta L, t)$ and $u^k_2(\alpha L,
  t)=u^{k-1}_1(\alpha L, t)$. Generalizing this two-subdomain case in
  \eqref{ADESWR} to a multi-subdomain scenario is straightforward, and
  the OSWR method for nonlinear problems is obtained by simply
  replacing the linear operator ${\cal L}$ by the corresponding
  non-linear one in \eqref{ADESWR}.

For the OSWR iteration given in \eqref{ADESWR}, the optimized choice
of the parameter $p$ and the corresponding convergence factor were
analyzed by Gander and Halpern in 2007 under the simplified assumption
that the space domain is unbounded.

\begin{theorem}\cite{Gander:2007:OSW}\label{proSWR}
\em{For    OSWR   \eqref{ADESWR} with
  $\mathcal{L}=\partial_x-\nu\partial_{xx}$ and an  overlap  size of
  $l>0$, the  optimized  choice for the Robin parameter, denoted by $p^*$,
  is given by $p^*=\frac{\tilde{p}^*\nu}{s}$, where $\tilde{p}^*$
  is the unique solution of the nonlinear equation
\begin{subequations}
\begin{equation}\label{Opta}
R_0(y_0, \tilde{p}^*, y_0)=R_0(\bar{y}(y_0, \tilde{p}^*), \tilde{p}^*, y_0),
\end{equation}
provided that $y_0:=\frac{l}{\nu}<y_c$, where $y_c$ is a constant equal to $1.618386576...$, and
\begin{equation*}
\begin{split}
&R_0(y, \tilde{p}, y_0)=\frac{(y-\tilde{p})^2+y^2-y_0^2}{(y+\tilde{p})^2+y^2-y_0^2}e^{-y}, \\
&\bar{y}(y_0, \tilde{p})=\sqrt{\frac{y_0^2+2\tilde{p}+\sqrt{\tilde{p}(-\tilde{p}^3-4\tilde{p}^2+(4+2y_0^2)\tilde{p}+8y_0^2)}}{2}}.
\end{split}
\end{equation*}
The function $R_0(y, \tilde{p}, y_0)$ is the convergence
factor of the  OSWR  iteration, obtained in Fourier space, where $y$
corresponds to a single Fourier mode $\omega \in [\frac{\pi}{T},
  \frac{\pi}{\Delta t}]$, i.e., $y=\frac{l}{\nu}\omega$.

If, on the other hand, $y_0\geq y_c$, then $\tilde{p}^*$ is the unique
solution of
\begin{equation}\label{Optb}
  y_0=\tilde{p}^*\sqrt{\frac{\tilde{p}^*}{4+\tilde{p}^*}}.
\end{equation}
With the  optimized  Robin parameter $p^*$, the convergence factor
$\rho$  over all relevant  Fourier modes can be bounded as
\begin{equation}\label{Optc}
\rho:={\max}_{y\in[y_{\text{min}}, y_{\text{max}}]}R_0(y, \tilde{p}^*, y_0)\leq R_0(\bar{y}(y_0, \tilde{p}^*), \tilde{p}^*, y_0),
\end{equation}
where $y_{\text{min}}=\frac{l\pi}{\nu T}$ and $y_{\text{max}}=\frac{l\pi}{\nu \Delta t}$.
\end{subequations}}
\end{theorem}
As we mentioned above, classical SWR with Dirichlet TCs corresponds to
the case $p=\infty$ in \eqref{ADESWR}.  By setting $p=\infty$ in the
function $R_0$, we obtain
\begin{equation}\label{DirichletABC}
\rho \leq e^{-y_{\text{min}}}=e^{-\frac{l\pi}{\nu T}},
\end{equation}
 which was  analyzed in \cite[Section 3]{Gander:2007:OSW}. 

We now show a numerical experiment to illustrate OSWR. We discretize
the advection-diffusion operator $\mathcal{L}$ using a centered finite
difference method and Backward Euler in time.  Let $L=8.2$, $T=5$,
$\Delta t=0.01$, $\Delta x=0.02$, and $l=2\Delta x$. The initial value
for the advection-diffusion equation is $u_0(x)=e^{-10(x-L/2)^2}$. We
show in Figure \ref{SWRADEFig} (left) the theoretical convergence
factor of the SWR method with Dirichlet and optimized Robin TCs as
function of the diffusion parameter $\nu$. We see that the convergence
factor becomes small when the advection term becomes dominant, and the
method converges faster and faster. To test this numerically, we
decompose the spatial domain $(0, L)$ into 4 subdomains and solve the
advection-diffusion equation using the OSWR method for several values
of the parameter $\nu$.  The method starts from a random initial
guess, and we stop the iteration when the error between the iterate
and the converged solution is less than $10^{-8}$. The iteration
number for Dirichlet and optimized Robin TCs is shown in Figure
\ref{SWRADEFig} (right), and we see that indeed fewer iterations are
required when $\nu$ is small, as predicted by the theoretical results
in the left panel.
\begin{figure}
\centering
\includegraphics[width=2.3in,height=1.85in,angle=0]{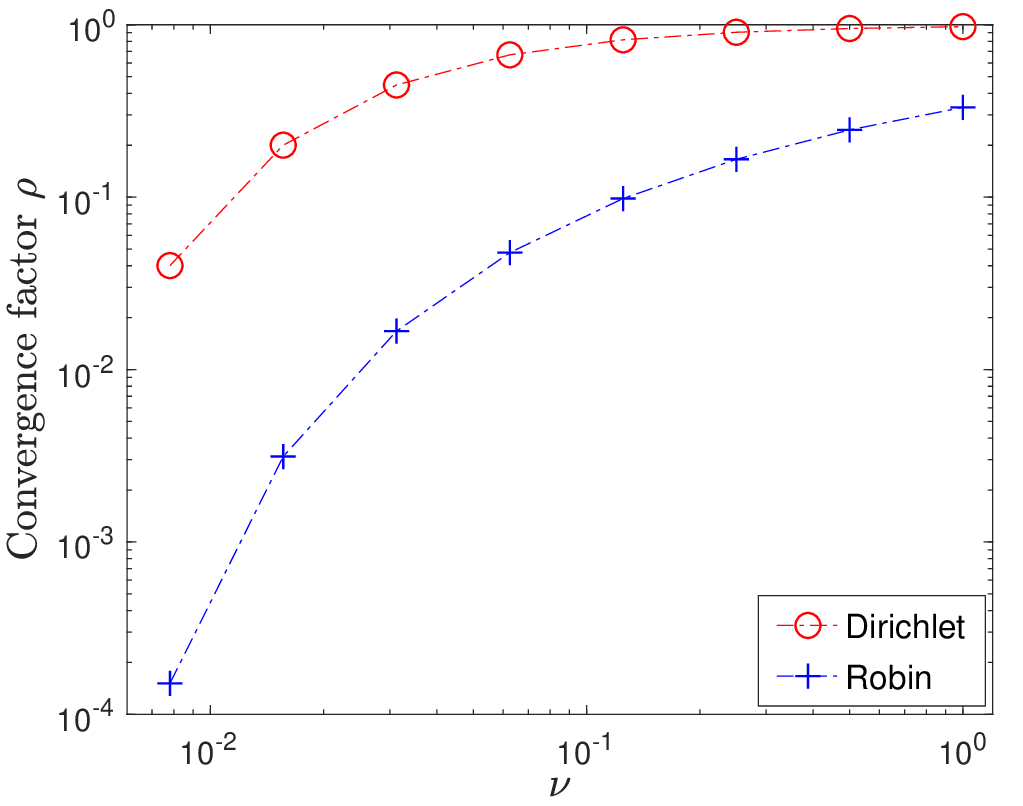}
\includegraphics[width=2.3in,height=1.85in,angle=0]{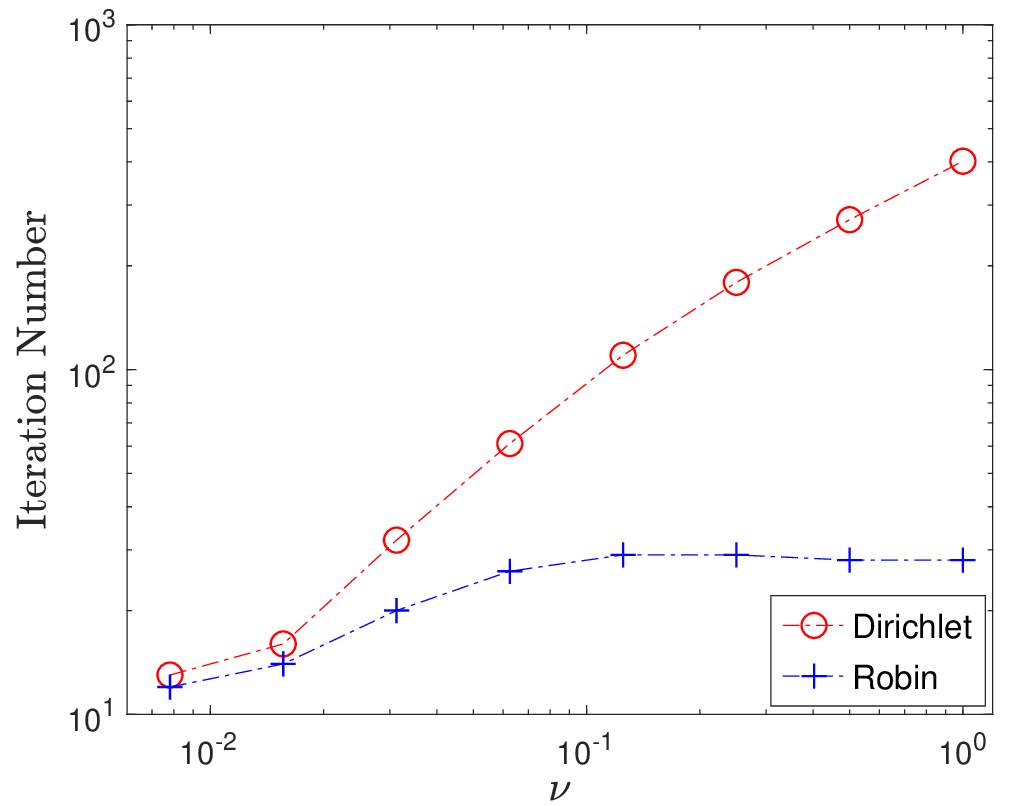}\\
\caption{Left: the theoretical convergence factor of the OSWR method
  when applied to the advection-diffusion equation \eqref{ADE}
  decreases when the diffusive parameter $\nu$ decreases. Right: the
  iteration, measured in a 4-subdomain numerical experiment with
  tolerance of $10^{-8}$, supports this prediction very well.}
\label{SWRADEFig}
\end{figure}

For a given $\nu$, the discretized OSWR method with four subdomains
may however not converge as rapidly as predicted by the theoretical
convergence factor $\rho$ obtained for a two subdomain decomposition
at the continuous level on an unbounded domain. For instance, when
$\nu=0.1$, the iteration numbers shown in Figure \ref{SWRADEFig} are
92 for Dirichlet TCs and 28 for Robin TCs, but iteration numbers
predicted by $\rho$ are 32 for Dirichlet TCs and 4 for Robin TCs,
which are significantly smaller. The reason for this discrepancy is
that the convergence factor $\rho$ in Theorem \ref{proSWR} is analyzed
for the 2-subdomain case at the space-time continuous level on an
unbounded domain, and we tested the discretized SWR method in the
4-subdomain case on a bounded domain. Convergence analyses of SWR with
Dirichlet TCs in the multi-subdomain case can be found in
\cite{GStuart98,WHH12}. However, a comprehensive convergence analysis
for Robin TCs in the multi-subdomain case is still missing so far. For
Robin TCs, the convergence of SWR in the two-subdomain case at the
semi-discrete level can be found in \cite{WK14}, see also the detailed
studies in \cite{gander2018optimized,gander2021discrete} for the
steady state case between continuous and discrete analyses on bounded
and unbounded domains.

 In addition to  Dirichlet and Robin TCs, there are efforts to
further accelerate   OSWR  using Ventcel TCs
\cite{BGGH16}. Essentially, these TCs serve as local approximations of
the {\em optimal} TCs analyzed in Fourier (or Laplace) space in
\cite[Section 3]{Gander:2007:OSW}, namely
$$
\partial_x-\frac{1}{2\nu}\mathcal{F}^{-1} \left(1+\sqrt{1+{\rm i}4\nu\omega}\right),
$$
where ${\rm i}=\sqrt{-1}$ and $\mathcal{F}^{-1}$ denotes the inverse
Fourier transform with $\omega$ representing the Fourier mode. In an
asymptotic sense, i.e., $l=C_1\Delta x$, $\Delta t=C_1\Delta x^\beta$,
and with $\Delta x$ being small, the convergence factor $\rho$
satisfies $\rho=1-\mathcal{O}(\Delta x^\gamma)$, where $\gamma>0$ is a
quantity that depends on $\beta$; see
\cite{Gander:2007:OSW,BGGH16,bennequin2009homographic}. Convolution
TCs  analyzed in \cite{WX17}  result in a mesh-independent constant
convergence factor $\rho=1-C$, where $C\in(0, 1)$. This is
particularly useful for handling  evolution  PDEs with nonlocal
terms, such as Volterra partial integro-differential equations.

\subsubsection{Second-order hyperbolic  problems}
{Unlike for first-order  parabolic  problems, for second-order
 hyperbolic  problems (e.g., the wave equation
  \eqref{WaveEquation1d}), the SWR method converges in a finite
  number of iterations, even when using simple Dirichlet transmission
  conditions. Applying SWR to the wave equation for a two
    subdomain decomposition leads to the algorithm 
\begin{equation}\label{WaveSWR}
\begin{split}
  & \begin{cases}
    \partial_{tt}u^k_1(x,t)=c^2 \partial_{xx}u_1^k(x,t)+g(x,t)  &(x, t)\in\Omega_1\times(0, T],\\
    u_1^k(x,0)=u_0(x),~\partial_tu_1^k(x,0)=\tilde{u}_0(x)  &x\in(0, \beta L),\\
    u_1^k(0, t)=0, ~  {u_1^k(\beta L, t)=u_2^{k-1}(\beta L, t)} &t\in(0, T),\\     
      \end{cases}\\
&\begin{cases}
    \partial_{tt}u^k_2(x,t)=c^2 \partial_{xx}u_2^k(x,t)+g(x,t)  &(x, t)\in\Omega_2\times(0, T],\\
    u_2^k(x,0)=u_0(x),~\partial_tu_2^k(x,0)=\tilde{u}_0(x)  &x\in(\alpha L, L),\\
     {u_2^k(\alpha L, t)=u_1^{k-1}(\alpha L, t)}, ~ u_2^k(L, t)=0 &t\in(0, T),\\   
      \end{cases}
      \end{split}
\end{equation}
where $c>0$ and $0<\alpha<\beta<1$.
\begin{theorem}\cite[Theorem 6.3.3]{Gander:1997:PhD} \label{Prop3.2}
  For the SWR method \eqref{WaveSWR}, the errors at the interfaces
  $x=\alpha L$ and $x=\beta L$ become zero after $k$ iterations, i.e.,
  $u^k_1(\alpha L, t)-u(\alpha L, t)=0$ and $u^k_2(\beta L, t)-u(\beta
  L, t)=0$, {provided that} $k>\frac{Tc}{\beta-\alpha}$.
\end{theorem}
The reason for   this  convergence result lies in the finite speed
of propagation inherent to hyperbolic problems. By exploiting this
property, similar results can be achieved in the case of many
subdomains, as well as for more general decompositions in higher
dimensions \cite{Gander:2004:ABC} and for other hyperbolic equations,
see for example \cite{GR2005} for 1D nonlinear conservation laws. 

Theorem \ref{Prop3.2} shows a very important property of SWR applied
to hyperbolic problems, already indicated in \cite[Figure
  3.1]{Gander:2003:OSWW}: the subdomains are computing the exact
solution in the cone within the subdomain which is only influenced by
the initial condition, and not by the transmission conditions
where possibly incorrect data is still coming from the neighboring
subdomains. Using this property, one can choose the space-time
subdomains in SWR applied to hyperbolic problems in order to avoid
iterations and advance directly with space-time subdomain solves in
parallel. To illustrate this, it is best to consider again the
one-dimensional wave equation (\ref{WaveEquation1d}) and a {\em
  red-black} domain decomposition with generous overlap, as it was
done in \cite{Ciaramella:2023:UTP}, see Figure \ref{RBSWRFig}.
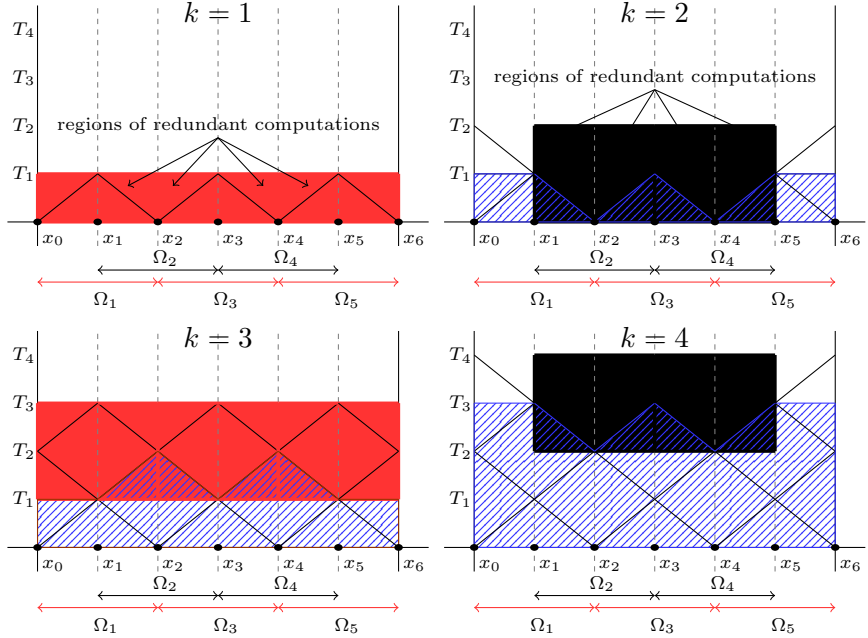
\begin{figure}
  \centering
\definecolor{zzttqq}{rgb}{0.6,0.2,0}
\definecolor{ttttff}{rgb}{0.2,0.2,1}
\definecolor{fftttt}{rgb}{1,0.2,0.2}
\mbox{\begin{tikzpicture}[xscale=0.20,yscale=0.16]
\fill[line width=1.2pt,color=fftttt,fill=fftttt,fill opacity=0.5] (4,2) -- (4,6) -- (12,6) -- (12,2) -- cycle;
\fill[line width=1.2pt,color=fftttt,fill=fftttt,fill opacity=0.5] (12,6) -- (20,6) -- (20,2) -- (12,2) -- cycle;
\fill[line width=1.2pt,color=fftttt,fill=fftttt,fill opacity=0.5] (20,6) -- (20,2) -- (28,2) -- (28,6) -- cycle;
\draw (2,2)-- (30,2);
\draw (4,0)-- (4,20);
\draw [gray,dash pattern=on 2pt off 2pt] (8,0)-- (8,20);
\draw [gray,dash pattern=on 2pt off 2pt] (12,0)-- (12,20);
\draw [gray,dash pattern=on 2pt off 2pt] (16,0)-- (16,20);
\draw [gray,dash pattern=on 2pt off 2pt] (20,0)-- (20,20);
\draw [gray,dash pattern=on 2pt off 2pt] (24,0)-- (24,20);
\draw (28,0)-- (28,20);
\draw [line width=1.2pt,color=fftttt] (4,2)-- (4,6);
\draw [line width=1.2pt,color=fftttt] (4,6)-- (12,6);
\draw [line width=1.2pt,color=fftttt] (12,6)-- (12,2);
\draw [line width=1.2pt,color=fftttt] (12,2)-- (4,2);
\draw [line width=1.2pt,color=fftttt] (12,6)-- (20,6);
\draw [line width=1.2pt,color=fftttt] (20,6)-- (20,2);
\draw [line width=1.2pt,color=fftttt] (20,2)-- (12,2);
\draw [line width=1.2pt,color=fftttt] (12,2)-- (12,6);
\draw [line width=1.2pt,color=fftttt] (20,6)-- (20,2);
\draw [line width=1.2pt,color=fftttt] (20,2)-- (28,2);
\draw [line width=1.2pt,color=fftttt] (28,2)-- (28,6);
\draw [line width=1.2pt,color=fftttt] (28,6)-- (20,6);
\draw (4,2)-- (8,6);
\draw (8,6)-- (12,2);
\draw (12,2)-- (16,6);
\draw (16,6)-- (20,2);
\draw (20,2)-- (24,6);
\draw (24,6)-- (28,2);

\draw[color=black] (16,19.5) node {$k=1$};

\begin{scriptsize}
\draw[color=black] (16,10) node {regions of redundant computations};
\end{scriptsize}
\draw [->,color=black] (16,9)-- (10,5);
\draw [->,color=black] (16,9)-- (13,5);
\draw [->,color=black] (16,9)-- (19,5);
\draw [->,color=black] (16,9)-- (22,5);

\draw [<->,color=fftttt] (4,-3)-- (12,-3);
\draw [<->,color=fftttt] (12,-3)-- (20,-3);
\draw [<->,color=fftttt] (20,-3)-- (28,-3);
\draw[<->] (8,-2.0)-- (16,-2.0);
\draw[<->] (16,-2.0)-- (24,-2.0);

\begin{scriptsize}
\draw[color=black] (8.56,-4.5) node {$\Omega_1$};
\draw[color=black] (16.54,-4.5) node {$\Omega_3$};
\draw[color=black] (24.58,-4.5) node {$\Omega_5$};

\draw[color=black] (12.52,-1.0) node {$\Omega_2$};
\draw[color=black] (20.56,-1.0) node {$\Omega_4$};

\begin{scope}[xshift=0.5cm,yshift=-0.5cm]
\draw[color=black] (4.54,1.0) node {$x_0$};
\draw[color=black] (8.56,1.0) node {$x_1$};
\draw[color=black] (12.52,1.0) node {$x_2$};
\draw[color=black] (16.54,1.0) node {$x_3$};
\draw[color=black] (20.56,1.0) node {$x_4$};
\draw[color=black] (24.58,1.0) node {$x_5$};
\draw[color=black] (28.54,1.0) node {$x_6$};
\draw[color=black] (2.6,6.5) node {$T_1$};
\draw[color=black] (2.6,10.5) node {$T_2$};
\draw[color=black] (2.6,14.5) node {$T_3$};
\draw[color=black] (2.6,18.5) node {$T_4$};
\end{scope}

\fill [color=black] (4,2) circle (8.0pt);
\fill [color=black] (8,2) circle (8.0pt);
\fill [color=black] (12,2) circle (8.0pt);
\fill [color=black] (16,2) circle (8.0pt);
\fill [color=black] (20,2) circle (8.0pt);
\fill [color=black] (24,2) circle (8.0pt);
\fill [color=black] (28,2) circle (8.0pt);
\end{scriptsize}
\end{tikzpicture}
\begin{tikzpicture}[xscale=0.20,yscale=0.16]
\begin{scope}[xshift=30cm]
\definecolor{ttttff}{rgb}{0.2,0.2,1}
\fill[line width=1.2pt,fill=black,fill opacity=0.4] (8,2) -- (8,10) -- (16,10) -- (16,2) -- cycle;
\fill[line width=1.2pt,fill=black,fill opacity=0.4] (16,2) -- (16,10) -- (24,10) -- (24,2) -- cycle;
\fill[pattern color=ttttff,fill=ttttff,pattern=north east lines] (4,6) -- (4,2) -- (12,2) -- (8,6) -- cycle;
\fill[pattern color=ttttff,fill=ttttff,pattern=north east lines] (12,2) -- (16,6) -- (20,2) -- cycle;
\fill[pattern color=ttttff,fill=ttttff,pattern=north east lines] (20,2) -- (24,6) -- (28,6) -- (28,2) -- cycle;
\draw (2,2)-- (30,2);
\draw (4,0)-- (4,20);
\draw [gray,dash pattern=on 2pt off 2pt] (8,0)-- (8,20);
\draw [gray,dash pattern=on 2pt off 2pt] (12,0)-- (12,20);
\draw [gray,dash pattern=on 2pt off 2pt] (16,0)-- (16,20);
\draw [gray,dash pattern=on 2pt off 2pt] (20,0)-- (20,20);
\draw [gray,dash pattern=on 2pt off 2pt] (24,0)-- (24,20);
\draw (28,0)-- (28,20);
\draw (8,6)-- (12,2);
\draw (12,2)-- (16,6);
\draw (16,6)-- (20,2);
\draw (20,2)-- (24,6);
\draw (8,6)-- (12,10);
\draw (12,10)-- (16,6);
\draw (16,6)-- (20,10);
\draw (20,10)-- (24,6);
\draw [line width=1.2pt] (8,2)-- (8,10);
\draw [line width=1.2pt] (8,10)-- (16,10);
\draw [line width=1.2pt] (16,10)-- (16,2);
\draw [line width=1.2pt] (16,2)-- (8,2);
\draw [line width=1.2pt] (16,2)-- (16,10);
\draw [line width=1.2pt] (16,10)-- (24,10);
\draw [line width=1.2pt] (24,10)-- (24,2);
\draw [line width=1.2pt] (24,2)-- (16,2);
\draw (8,6)-- (4,6);
\draw (24,6)-- (28,6);
\draw [color=ttttff] (4,6)-- (4,2);
\draw [color=ttttff] (4,2)-- (12,2);
\draw [color=ttttff] (12,2)-- (8,6);
\draw [color=ttttff] (8,6)-- (4,6);
\draw [color=ttttff] (12,2)-- (16,6);
\draw [color=ttttff] (16,6)-- (20,2);
\draw [color=ttttff] (20,2)-- (12,2);
\draw [color=ttttff] (20,2)-- (24,6);
\draw [color=ttttff] (24,6)-- (28,6);
\draw [color=ttttff] (28,6)-- (28,2);
\draw [color=ttttff] (28,2)-- (20,2);
\draw (4,2)-- (8,6);
\draw (24,6)-- (28,2);
\draw (8,6)-- (4,10);
\draw (24,6)-- (28,10);

\draw[color=black] (16,19.5) node {$k=2$};

\begin{scriptsize}
\draw[color=black] (16,14) node {regions of redundant computations};
\end{scriptsize}
\draw [->,color=black] (16,13)-- (9,9);
\draw [->,color=black] (16,13)-- (14,9);
\draw [->,color=black] (16,13)-- (18,9);
\draw [->,color=black] (16,13)-- (23,9);

\draw [<->,color=fftttt] (4,-3)-- (12,-3);
\draw [<->,color=fftttt] (12,-3)-- (20,-3);
\draw [<->,color=fftttt] (20,-3)-- (28,-3);
\draw[<->] (8,-2.0)-- (16,-2.0);
\draw[<->] (16,-2.0)-- (24,-2.0);

\begin{scriptsize}
\draw[color=black] (8.56,-4.5) node {$\Omega_1$};
\draw[color=black] (16.54,-4.5) node {$\Omega_3$};
\draw[color=black] (24.58,-4.5) node {$\Omega_5$};

\draw[color=black] (12.52,-1.0) node {$\Omega_2$};
\draw[color=black] (20.56,-1.0) node {$\Omega_4$};

\begin{scope}[xshift=0.5cm,yshift=-0.5cm]
\draw[color=black] (4.54,1.0) node {$x_0$};
\draw[color=black] (8.56,1.0) node {$x_1$};
\draw[color=black] (12.52,1.0) node {$x_2$};
\draw[color=black] (16.54,1.0) node {$x_3$};
\draw[color=black] (20.56,1.0) node {$x_4$};
\draw[color=black] (24.58,1.0) node {$x_5$};
\draw[color=black] (28.54,1.0) node {$x_6$};
\draw[color=black] (2.6,6.5) node {$T_1$};
\draw[color=black] (2.6,10.5) node {$T_2$};
\draw[color=black] (2.6,14.5) node {$T_3$};
\draw[color=black] (2.6,18.5) node {$T_4$};
\end{scope}

\fill [color=black] (4,2) circle (8.0pt);
\fill [color=black] (8,2) circle (8.0pt);
\fill [color=black] (12,2) circle (8.0pt);
\fill [color=black] (16,2) circle (8.0pt);
\fill [color=black] (20,2) circle (8.0pt);
\fill [color=black] (24,2) circle (8.0pt);
\fill [color=black] (28,2) circle (8.0pt);
\end{scriptsize}
\end{scope}

\end{tikzpicture}}
\mbox{\begin{tikzpicture}[xscale=0.20,yscale=0.16]
\fill[line width=1.2pt,color=fftttt,fill=fftttt,fill opacity=0.4] (4,14) -- (4,6) -- (12,6) -- (12,14) -- cycle;
\fill[line width=1.2pt,color=fftttt,fill=fftttt,fill opacity=0.4] (12,6) -- (20,6) -- (20,14) -- (12,14) -- cycle;
\fill[line width=1.2pt,color=fftttt,fill=fftttt,fill opacity=0.4] (20,6) -- (28,6) -- (28,14) -- (20,14) -- cycle;
\fill[pattern color=ttttff,fill=ttttff,pattern=north east lines] (4,2) -- (4,6) -- (8,6) -- (12,10) -- (16,6) -- (20,10) -- (24,6) -- (28,6) -- (28,2) -- cycle;
\draw (2,2)-- (30,2);
\draw (4,0)-- (4,20);
\draw [gray,dash pattern=on 2pt off 2pt] (8,0)-- (8,20);
\draw [gray,dash pattern=on 2pt off 2pt] (12,0)-- (12,20);
\draw [gray,dash pattern=on 2pt off 2pt] (16,0)-- (16,20);
\draw [gray,dash pattern=on 2pt off 2pt] (20,0)-- (20,20);
\draw [gray,dash pattern=on 2pt off 2pt] (24,0)-- (24,20);
\draw (28,0)-- (28,20);
\draw (8,6)-- (12,2);
\draw (12,2)-- (16,6);
\draw (16,6)-- (20,2);
\draw (20,2)-- (24,6);
\draw (8,6)-- (12,10);
\draw (12,10)-- (16,6);
\draw (16,6)-- (20,10);
\draw (20,10)-- (24,6);
\draw (8,6)-- (4,6);
\draw (24,6)-- (28,6);
\draw (12,10)-- (8,14);
\draw (12,10)-- (16,14);
\draw (16,14)-- (20,10);
\draw (20,10)-- (24,14);
\draw [line width=1.2pt,color=fftttt] (4,14)-- (4,6);
\draw [line width=1.2pt,color=fftttt] (4,6)-- (12,6);
\draw [line width=1.2pt,color=fftttt] (12,6)-- (12,14);
\draw [line width=1.2pt,color=fftttt] (12,14)-- (4,14);
\draw [line width=1.2pt,color=fftttt] (12,6)-- (20,6);
\draw [line width=1.2pt,color=fftttt] (20,6)-- (20,14);
\draw [line width=1.2pt,color=fftttt] (20,14)-- (12,14);
\draw [line width=1.2pt,color=fftttt] (12,14)-- (12,6);
\draw [line width=1.2pt,color=fftttt] (20,6)-- (28,6);
\draw [line width=1.2pt,color=fftttt] (28,6)-- (28,14);
\draw [line width=1.2pt,color=fftttt] (28,14)-- (20,14);
\draw [line width=1.2pt,color=fftttt] (20,14)-- (20,6);
\draw (8,6)-- (4,2);
\draw (24,6)-- (28,2);
\draw (8,6)-- (4,10);
\draw (4,10)-- (8,14);
\draw (24,6)-- (28,10);
\draw (28,10)-- (24,14);
\draw [color=zzttqq] (4,2)-- (4,6);
\draw [color=zzttqq] (4,6)-- (8,6);
\draw [color=zzttqq] (8,6)-- (12,10);
\draw [color=zzttqq] (12,10)-- (16,6);
\draw [color=zzttqq] (16,6)-- (20,10);
\draw [color=zzttqq] (20,10)-- (24,6);
\draw [color=zzttqq] (24,6)-- (28,6);
\draw [color=zzttqq] (28,6)-- (28,2);
\draw [color=zzttqq] (28,2)-- (4,2);

\draw[color=black] (16,19.5) node {$k=3$};

\draw [<->,color=fftttt] (4,-3)-- (12,-3);
\draw [<->,color=fftttt] (12,-3)-- (20,-3);
\draw [<->,color=fftttt] (20,-3)-- (28,-3);
\draw[<->] (8,-2.0)-- (16,-2.0);
\draw[<->] (16,-2.0)-- (24,-2.0);

\begin{scriptsize}
\draw[color=black] (8.56,-4.5) node {$\Omega_1$};
\draw[color=black] (16.54,-4.5) node {$\Omega_3$};
\draw[color=black] (24.58,-4.5) node {$\Omega_5$};

\draw[color=black] (12.52,-1.0) node {$\Omega_2$};
\draw[color=black] (20.56,-1.0) node {$\Omega_4$};

\begin{scope}[xshift=0.5cm,yshift=-0.5cm]
\draw[color=black] (4.54,1.0) node {$x_0$};
\draw[color=black] (8.56,1.0) node {$x_1$};
\draw[color=black] (12.52,1.0) node {$x_2$};
\draw[color=black] (16.54,1.0) node {$x_3$};
\draw[color=black] (20.56,1.0) node {$x_4$};
\draw[color=black] (24.58,1.0) node {$x_5$};
\draw[color=black] (28.54,1.0) node {$x_6$};
\draw[color=black] (2.6,6.5) node {$T_1$};
\draw[color=black] (2.6,10.5) node {$T_2$};
\draw[color=black] (2.6,14.5) node {$T_3$};
\draw[color=black] (2.6,18.5) node {$T_4$};
\end{scope}

\fill [color=black] (4,2) circle (8.0pt);
\fill [color=black] (8,2) circle (8.0pt);
\fill [color=black] (12,2) circle (8.0pt);
\fill [color=black] (16,2) circle (8.0pt);
\fill [color=black] (20,2) circle (8.0pt);
\fill [color=black] (24,2) circle (8.0pt);
\fill [color=black] (28,2) circle (8.0pt);
\end{scriptsize}
\end{tikzpicture}
\begin{tikzpicture}[xscale=0.20,yscale=0.16]
  \begin{scope}[xshift=30cm]
\definecolor{ttttff}{rgb}{0.2,0.2,1}
\fill[line width=1.2pt,fill=black,fill opacity=0.4] (8,10) -- (16,10) -- (16,18) -- (8,18) -- cycle;
\fill[line width=1.2pt,fill=black,fill opacity=0.4] (16,10) -- (24,10) -- (24,18) -- (16,18) -- cycle;
\fill[pattern color=ttttff,fill=ttttff,pattern=north east lines] (4,2) -- (4,14) -- (8,14) -- (12,10) -- (16,14) -- (20,10) -- (24,14) -- (28,14) -- (28,2) -- cycle;
\draw (2,2)-- (30,2);
\draw (4,0)-- (4,20);
\draw [gray,dash pattern=on 2pt off 2pt] (8,0)-- (8,20);
\draw [gray,dash pattern=on 2pt off 2pt] (12,0)-- (12,20);
\draw [gray,dash pattern=on 2pt off 2pt] (16,0)-- (16,20);
\draw [gray,dash pattern=on 2pt off 2pt] (20,0)-- (20,20);
\draw [gray,dash pattern=on 2pt off 2pt] (24,0)-- (24,20);
\draw (28,0)-- (28,20);
\draw (8,6)-- (12,2);
\draw (12,2)-- (16,6);
\draw (16,6)-- (20,2);
\draw (20,2)-- (24,6);
\draw (8,6)-- (12,10);
\draw (12,10)-- (16,6);
\draw (16,6)-- (20,10);
\draw (20,10)-- (24,6);
\draw (12,10)-- (8,14);
\draw (12,10)-- (16,14);
\draw (16,14)-- (20,10);
\draw (20,10)-- (24,14);
\draw (8,6)-- (4,2);
\draw (24,6)-- (28,2);
\draw (8,6)-- (4,10);
\draw (4,10)-- (8,14);
\draw (24,6)-- (28,10);
\draw (28,10)-- (24,14);
\draw [line width=1.2pt] (8,10)-- (16,10);
\draw [line width=1.2pt] (16,10)-- (16,18);
\draw [line width=1.2pt] (16,18)-- (8,18);
\draw [line width=1.2pt] (8,18)-- (8,10);
\draw [line width=1.2pt] (16,10)-- (24,10);
\draw [line width=1.2pt] (24,10)-- (24,18);
\draw [line width=1.2pt] (24,18)-- (16,18);
\draw [line width=1.2pt] (16,18)-- (16,10);
\draw [color=ttttff] (4,2)-- (4,14);
\draw [color=ttttff] (4,14)-- (8,14);
\draw [color=ttttff] (8,14)-- (12,10);
\draw [color=ttttff] (12,10)-- (16,14);
\draw [color=ttttff] (16,14)-- (20,10);
\draw [color=ttttff] (20,10)-- (24,14);
\draw [color=ttttff] (24,14)-- (28,14);
\draw [color=ttttff] (28,14)-- (28,2);
\draw [color=ttttff] (28,2)-- (4,2);
\draw (8,14)-- (12,18);
\draw (12,18)-- (16,14);
\draw (16,14)-- (20,18);
\draw (20,18)-- (24,14);
\draw (24,14)-- (28,18);
\draw (8,14)-- (4,18);

\draw[color=black] (16,19.5) node {$k=4$};

\draw [<->,color=fftttt] (4,-3)-- (12,-3);
\draw [<->,color=fftttt] (12,-3)-- (20,-3);
\draw [<->,color=fftttt] (20,-3)-- (28,-3);
\draw[<->] (8,-2.0)-- (16,-2.0);
\draw[<->] (16,-2.0)-- (24,-2.0);

\begin{scriptsize}
\draw[color=black] (8.56,-4.5) node {$\Omega_1$};
\draw[color=black] (16.54,-4.5) node {$\Omega_3$};
\draw[color=black] (24.58,-4.5) node {$\Omega_5$};

\draw[color=black] (12.52,-1.0) node {$\Omega_2$};
\draw[color=black] (20.56,-1.0) node {$\Omega_4$};

\begin{scope}[xshift=0.5cm,yshift=-0.5cm]
\draw[color=black] (4.54,1.0) node {$x_0$};
\draw[color=black] (8.56,1.0) node {$x_1$};
\draw[color=black] (12.52,1.0) node {$x_2$};
\draw[color=black] (16.54,1.0) node {$x_3$};
\draw[color=black] (20.56,1.0) node {$x_4$};
\draw[color=black] (24.58,1.0) node {$x_5$};
\draw[color=black] (28.54,1.0) node {$x_6$};
\draw[color=black] (2.6,6.5) node {$T_1$};
\draw[color=black] (2.6,10.5) node {$T_2$};
\draw[color=black] (2.6,14.5) node {$T_3$};
\draw[color=black] (2.6,18.5) node {$T_4$};
\end{scope}

\fill [color=black] (4,2) circle (8.0pt);
\fill [color=black] (8,2) circle (8.0pt);
\fill [color=black] (12,2) circle (8.0pt);
\fill [color=black] (16,2) circle (8.0pt);
\fill [color=black] (20,2) circle (8.0pt);
\fill [color=black] (24,2) circle (8.0pt);
\fill [color=black] (28,2) circle (8.0pt);
\end{scriptsize}
\end{scope}
\end{tikzpicture}}
\caption{Illustration of Red-Black SWR with
  generous overlap.}
  \label{RBSWRFig}
\end{figure}
In the first panel, we solve the wave equation in parallel within the
three red subdomains in space-time, i.e. in $\Omega_j\times(0,T_1)$
for $j=1,3,5$. We use arbitrary interface data at $x=x_2$ and $x=x_4$
because the solution there is not yet known. Due to the finite speed
of propagation, we obtain the exact solution within triangular tents
in space-time, bounded by the characteristic lines of the wave
equation. These tents are marked in blue in the second panel, and the
solution is also correct in two additional small zones on the left and
right, since the outer boundary conditions are known. In this first
red solve, SWR also computes a not yet correct approximation in the
regions above the correct tents, as indicated in the first panel: SWR
performs redundant computations, as advocated by Nievergelt already in
order to obtain more parallelism. In the next step of this red-black
SWR, one computes wave equation solutions in the black subdomains
$\Omega_j\times(0,T_2)$ for $j=2,4$ in space-time, as indicated in the
second panel of Figure \ref{RBSWRFig}. Since we already have the
correct solution in the blue region, the exact solution is now
obtained in the two rhomboid blue tents indicated in the third panel,
again at the cost of some redundant computations. Next, one solves
again in the red subdomains, but now further in time in the interval
$(T_1,T_3)$. Continuing this red-black SWR algorithm, we obtain the
exact solution further and further advanced in time, as indicated in
the last panel in Figure \ref{RBSWRFig}.

This red-black SWR algorithm is an effective and simple way to
implement one of the most powerful current space-time solvers for
hyperbolic problems, namely the Mapped Tent Pitching (MTP) algorithm
from \cite{GSW2017}, see \cite{gopalakrishnan2020explicit} for its
application to time domain Maxwell equations. In MTP, one maps the
tent shape which we have seen in red-black SWR to space-time cylinders
in which the solution is then computed by classical time stepping, and
then the solution is mapped back, thus avoiding redundant
computations.  However, one has the extra cost of computing the
mapping, and also after the mapping the computational domains have the
same size as the space-time subdomains in the red-black SWR above, and
thus comparable computational cost. In addition, in MTP, order
reduction was observed due to the mapping, and specialized time
integrators were developed and need to be used to avoid this. In
contrast, in red-black SWR, also called now Unmapped Tent Pitching
(UTP), no order reduction occurs, and red-black SWR can be very easily
implemented, also for higher spatial dimensions, using Restricted
Additive Schwarz (RAS) techniques from DD directly applied to the
all-at-once space-time system, see \cite{gander2008schwarz} for an
explanation.  We show in Figure \ref{UTPFig} an example of using
  red-black SWR or equivalently UTP to solve our wave equation model
  problem \ref{WaveEquation1d} whose solution is shown in Figure
  \ref{WaveExampleFig} on the right.
\begin{figure}
  \centering
 \includegraphics[width=1.2in,height=2.6in,angle=0]{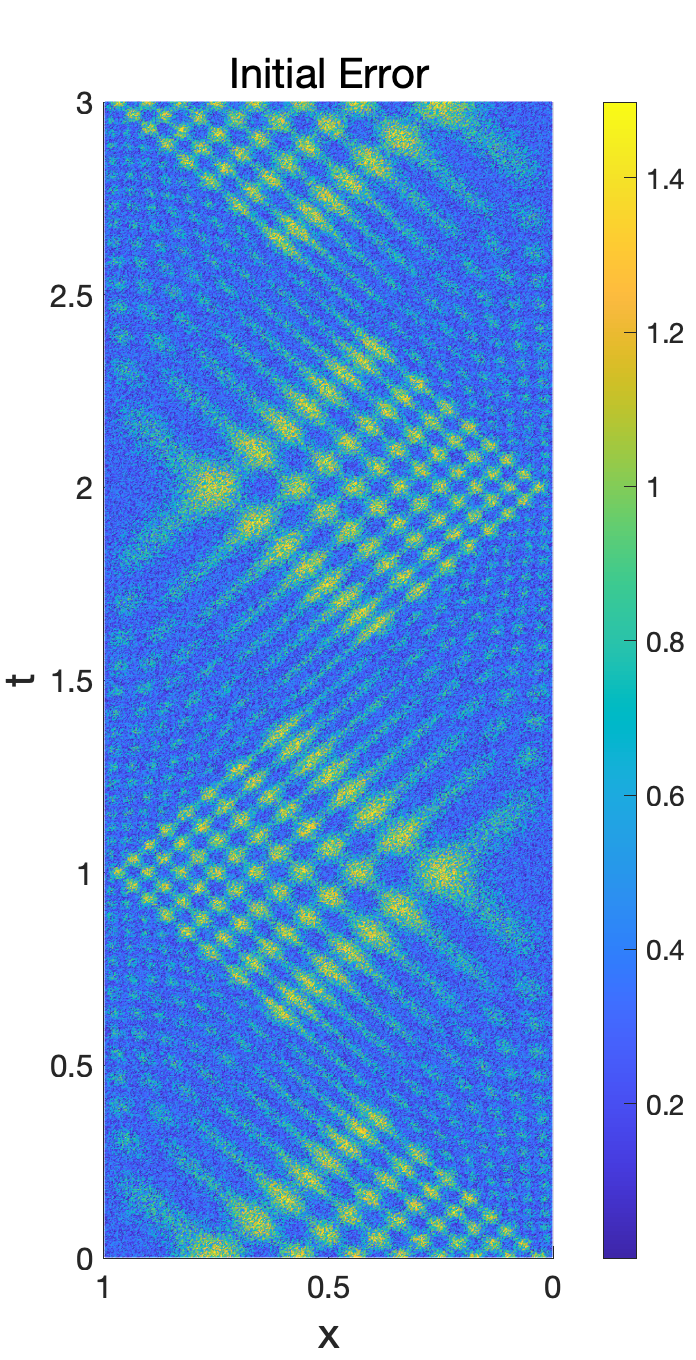}~
     \includegraphics[width=1.2in,height=2.6in,angle=0]{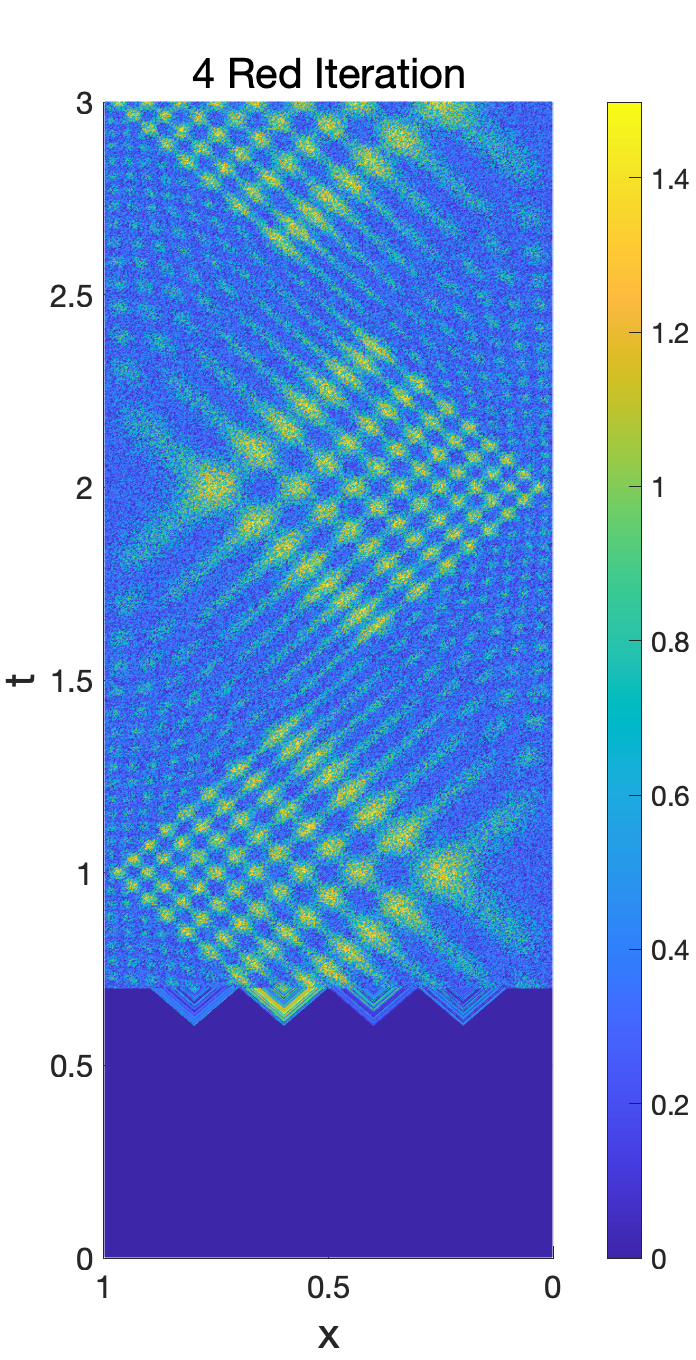}~
    \includegraphics[width=1.2in,height=2.6in,angle=0]{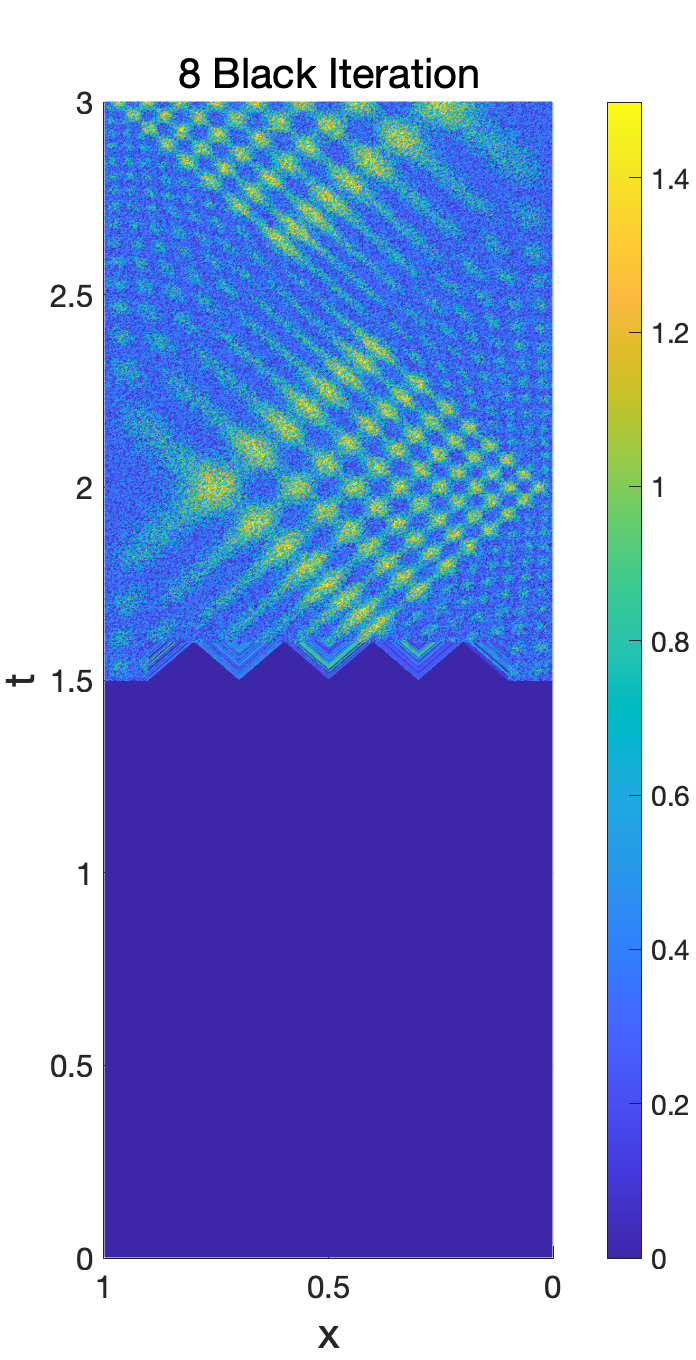}~    \includegraphics[width=1.2in,height=2.6in,angle=0]{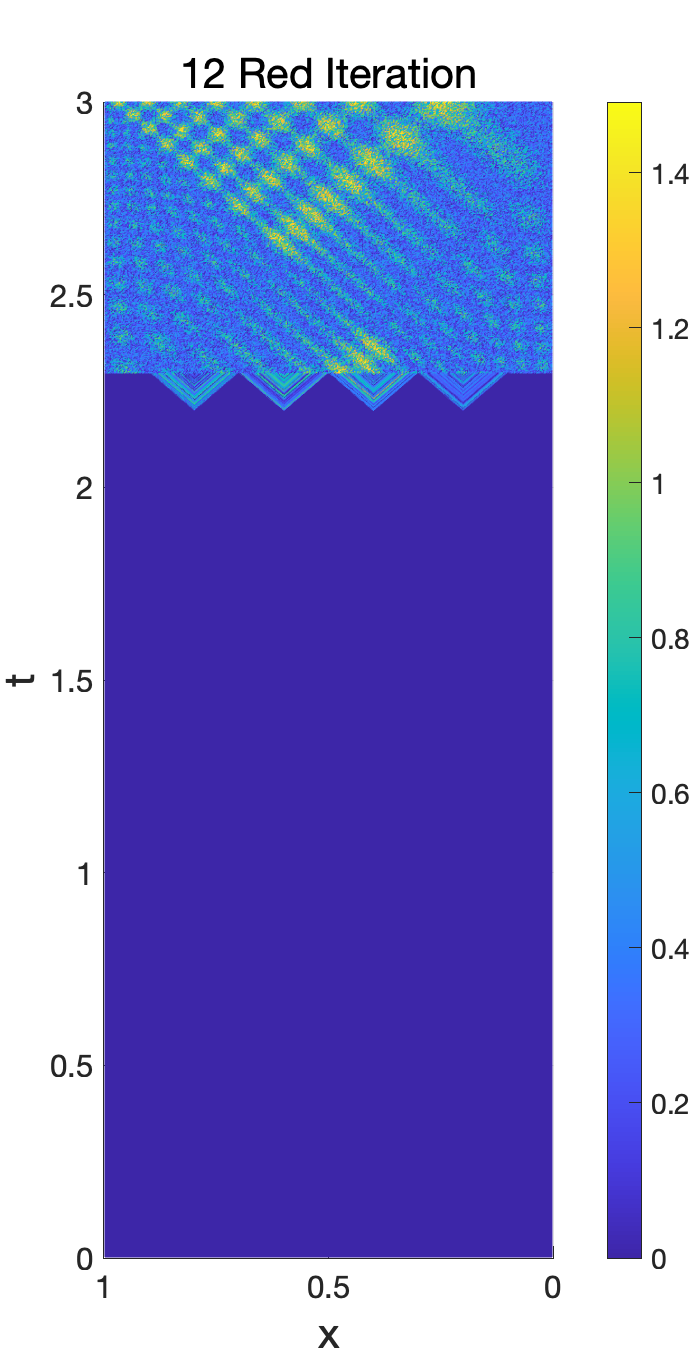}
  \caption{Red-black SWR or equivalently UTP applied to the second
    order wave equation. From left to right: initial error with random
  initial guess, and then 4th red iteration, 8th black iteration and
  12th red iteration. }
  \label{UTPFig}
\end{figure}
We see that without knowing anything about the tent structure, UTP
constructs the exact solution in the red and black tents, and advances
exactly like MTP. Note, that UTP can also as easily be applied to
non-linear hyperbolic problems, and if one does not know the tent
height, it suffices to look at the residual in the computed solution
which indicated naturally the tent height by how far the residual has
become zero in time, and one can adapt the time domain length
$T_i-T_{i-1}$ accordingly!

Clearly, the original MTP is not appropriate for parabolic problems,
since for such problems the speed of propagation is infinite, and
hence there are no tents in which the solution would be correct. SWR
and thus UTP however can be very effective also for parabolic
problems, especially the optimized SWR variants, see e.g
\cite{Gander:2007:OSW,bennequin2009homographic}, and in case of
slightly diffusive problems like our advection dominated diffusion
model problem one could consider to apply UTP, maybe with one or two
additional iterations in each time slab.

\subsection{Time Parallel IDC}\label{Sec3.4}

{\emph{Integral Deferred Correction}} (IDC), introduced by
\cite{DGR00}, serves as a technique to obtain high-order numerical
solutions through an iterative correction procedure. While the
original IDC is sequential in time, there are two more
  recent techniques to parallelize IDC in time: the Pipeline IDC
(PIDC) method by \cite{GT07} and the Revisionist IDC (RIDC)
method by \cite{CMO10}. Both PIDC and RIDC fundamentally differ from
the original IDC, but to understand them, it is necessary to explain
 first  how IDC works. To this end, we rewrite the nonlinear ODE
\eqref{nonlinearODE} as an integral equation,
\begin{equation}\label{RIDC1}
\bm{u}(t) = \bm{u}_0 + \int_{0}^t f(\bm{u}(\tau), \tau) d\tau,\quad  t \in (0, T].
\end{equation}
Suppose we already have a rough approximation \(\tilde{\bm{u}}(t)\) of
the desired solution \(\bm{u}(t)\),  for example  by simply setting
\(\tilde{\bm{u}}(t) \equiv \bm{u}_0\) for \(t \in [0, T]\) or by
solving the ODE with lower accuracy.  To improve the approximation
  \(\tilde{\bm{u}}(t)\), we introduce the error \(\bm{e}(t) :=
  \bm{u}(t) - \tilde{\bm{u}}(t)\), and the residual
\begin{equation}\label{RIDC2}
\begin{array}{rcll}
\bm{r}(t) & := &\bm{u}_0 + \int_{0}^t f(\tilde{\bm{u}}(\tau), \tau) d\tau - \tilde{\bm{u}}(t), & t \in (0, T].
\end{array}
\end{equation}
By substituting \(\bm{u}(t) = \bm{e}(t) + \tilde{\bm{u}}(t)\) into
\eqref{RIDC1} and using  \eqref{RIDC2}, we can express the error
\(\bm{e}(t)\) in terms of the residual \(\bm{r}(t)\),
\begin{equation}\label{RIDC3}
\begin{array}{rcll}
\bm{e}(t) & = &\displaystyle\bm{u}_0 + \int_{0}^t f(\tilde{\bm{u}}(\tau) + \bm{e}(\tau), \tau) d\tau - \tilde{\bm{u}}(t) \\
& = &\displaystyle\bm{r}(t) + \int_{0}^t [f(\tilde{\bm{u}}(\tau) + \bm{e}(\tau), \tau) - f(\tilde{\bm{u}}(\tau), \tau)] d\tau, & t \in (0, T].
\end{array}
\end{equation}
 Taking a derivative,  this is equivalent to the differential
equation
\begin{equation}\label{RIDC4}
\begin{array}{rcll}
\bm{e}'(t) - \bm{r}'(t) = f(\tilde{\bm{u}}(t) + \bm{e}(t), t) - f(\tilde{\bm{u}}(t), t), & t \in (0, T].
\end{array}
\end{equation}
Let the current approximate solution $\bm{\tilde{u}}$ be known at
specific time points $0=t_0<t_1<t_2<\cdots<t_{M}=T$, $\{{\bm
  u}^k_m\}:=\bm{\tilde{u}}(t_m)$.  The procedure to obtain the next
approximate solution $\{{\bm u}_m^{k+1}\}$ involves discretizing
\eqref{RIDC4} and using a quadrature rule to approximate ${\bm r}(t)$
at the discrete time nodes.  Applying the linear-$\theta$ method (with
$\theta\in[0, 1]$) to \eqref{RIDC4} yields
\begin{equation}\label{RIDC5}
\begin{split}
{\bm e}_{m+1} -{\bm e}_m=&{\bm r}_{m+1} - {\bm r}_{m}+\Delta t_m(1-\theta)[f({\bm u}^{k+1}_m, t_m)-f({\bm u}^k_m, t_m)]+\\
&\Delta t_m \theta[f({\bm u}^{k+1}_{m+1}, t_{m+1})-f({\bm u}^k_{m+1}, t_{m+1})],
\end{split}
\end{equation}
where $m=0,1,\dots, M-1$ and $\Delta t_m=t_{m+1}-t_m$.
From \eqref{RIDC2}, ${\bm r}_{m+1} - {\bm r}_{m}=\int_{t_m}^{t_{m+1}}f({\bm u}^k(\tau), \tau)d\tau-({\bm u}^k_{m+1}-{\bm u}^k_m)$. Substituting this into \eqref{RIDC5} and then using ${\bm u}_{m}^{k+1}={\bm u}_m^k+{\bm e}^k_m$, we obtain 
\begin{equation*}
\begin{split}
{\bm u}_{m+1}^{k+1}=&{\bm u}_m^{k+1}+ \Delta t_m(1-\theta)[f({\bm u}^{k+1}_m, t_m)-f({\bm u}^k_m, t_m)]+\\
&\Delta t_m \theta[f({\bm u}^{k+1}_{m+1}, t_{m+1})-f({\bm u}^k_{m+1}, t_{m+1})]+\int_{t_m}^{t_{m+1}}f({\bm u}^k(\tau), \tau)d\tau,
\end{split}
\end{equation*}
where the integral is computed  using a quadrature rule,
\begin{subequations}
\begin{equation}\label{RIDC7a}
\begin{split}
\int_{t_m}^{t_{m+1}}f({\bm u}^k(\tau), \tau)d\tau\approx {\sum}_{j=1}^M\omega_{m,j} f({\bm u}_j^k, t_j).
\end{split}
\end{equation}
The quadrature weights are determined by integrating the Lagrange polynomials,
\begin{equation}\label{RIDC7b}
\begin{split}
\omega_{m,j}=\int_{t_m}^{t_{m+1}}\left(\prod_{i=1, i\neq j}^M\frac{\tau-t_i}{t_s-t_i}\right)d\tau.
\end{split}
\end{equation}
\end{subequations}
In summary, with  the discrete time points $\{t_m\}_{m=0}^M$ spanning
the time interval $[0, T]$,  IDC generates the approximate
solution as
\begin{equation}\label{RIDC8}
\begin{split}
{\bm u}_{m+1}^{k+1}=&{\bm u}_m^{k+1}+ \Delta t_m(1-\theta)[f({\bm u}^{k+1}_m, t_m)-f({\bm u}^k_m, t_m)]+\\
&\Delta t_m \theta[f({\bm u}^{k+1}_{m+1}, t_{m+1})-f({\bm u}^k_{m+1}, t_{m+1})]+\sum_{j=1}^M\omega_{m,j} f({\bm u}_j^k, t_j),
\end{split}
\end{equation}
where $k=0,1,\dots, k_{\max}-1$,  and for each correction index $k$ we
sweep from left to right, i.e., $m=0,1,\dots, M-1$.  The choice of
quadrature rule determines the maximum order of accuracy achievable in
practice.  If we use $M$ uniformly distributed nodes with distance
$\Delta t$, the maximal order of accuracy is of $\CO(\Delta t^M)$ and,
more specifically, we have
\begin{proposition}[\cite{DGR00}]\label{ProIDC}
{\em Suppose the time-integrator is of order $p$, such as $p=1$ for
  $\theta=1$ (Backward  Euler) and $p=2$ for
  $\theta=\frac{1}{2}$ (Trapezoidal Rule) in \eqref{RIDC8}. Then,
  the approximate solution $\{{\bm u}^k_m\}$ after $k$ corrections is
  of order $\CO(\Delta t^{\min\{M, (k+1)p\}})$.  }
\end{proposition}

The original IDC method \cite{DGR00}  used  Gauss nodes for the
quadrature rule, resulting in a higher maximal order of accuracy. For
instance, using Gauss-Lobatto nodes achieves an order of $2J-1$. This
IDC variant is called {\em Spectral Deferred Correction} (SDC) and
serves as the key component of the PFASST algorithm, introduced in
Section \ref{Sec4.3} later.

For long-time computations, where $T$ is large, creating a uniformly
high-order numerical approximation across the entire time  interval
  is  challenging, because it is difficult to  accurately
approximate a function over a large interval with a single high-order
polynomial. In such scenarios, it is natural to segment the time
interval $[0, T]$ into multiple {\em windows}, denoted  by
$\{I_n:=[T_{n-1}, T_{n}]\}_{n=1}^{N_t}$ with $T_0=0$ and
$T_{N_t}=T$. Then, IDC can be applied to each window
individually. Within each time window, a lower-order polynomial often
provides precise quadrature, especially when the window size is
small. However, this process is entirely sequential, since the
computations for the $(n+1)$-st window $I_{n+1}$ must await the
completion of computations for $I_n$. This dependency arises because
the initial value for $I_{n+1}$ remains unknown until the preceding
window's computations are finalized. Additionally, within each time
window, we use a basic IDC and the computation proceeds
step-by-step.

\subsubsection{Pipeline IDC (PIDC)}\label{PIDC}

The first parallel version of IDC, known as PIDC, was introduced in
\cite{GT07}. PIDC uses a \emph{pipeline} parallelization approach for
IDC. It is based on a simple concept applicable to any time evolution
computation, already proposed by \cite{Womble:1990:ATS}. Specifically,
the computation for the next window $I_{n+1}=[T_n, T_{n+1}]$ can start
when a preliminary initial value from the current computation on
$I_n=[T_{n-1}, T_n]$ at $t=T_n$ becomes available. For instance,
following the first sweep on $I_n$, an approximation ${\bm
  u}_{n, M}^1$ at $t=T_n$, i.e.  the rightmost solution on window
$I_n$ after one sweep, is obtained, and one can compute the first
sweep on $I_{n+1}$ using ${\bm u}_{n, M}^1$ as the initial value at
the same time as performing the second sweep on $I_n$. After this
computation, one can already start on $I_{n+2}$ while continuing on
$I_{n+1}$ and $I_n$, computing three sweeps in parallel. In general,
for $N_t$ time windows, when conducting the $k$-th sweep on the
$n$-th window $I_n$, we simultaneously perform the $(k-1)$-st sweep on
$I_{n+1}$, the $(k-2)$-nd sweep on $I_{n+2}$, and so on, up to the
first sweep on $I_{n+k-1}$. This procedure is illustrated in Figure
\ref{PipelineIDCFig} with $M=6$ and $k_{\max}=4$, where the first 4
sweeps represent the initial state. Following this stage, sweeps on
$I_n$, $I_{n+1}$, $I_{n+2}$, and $I_{n+3}$ are executed in parallel.
\begin{figure}
  \centering
 \includegraphics[width=2.25in,height=1.2in,angle=0]{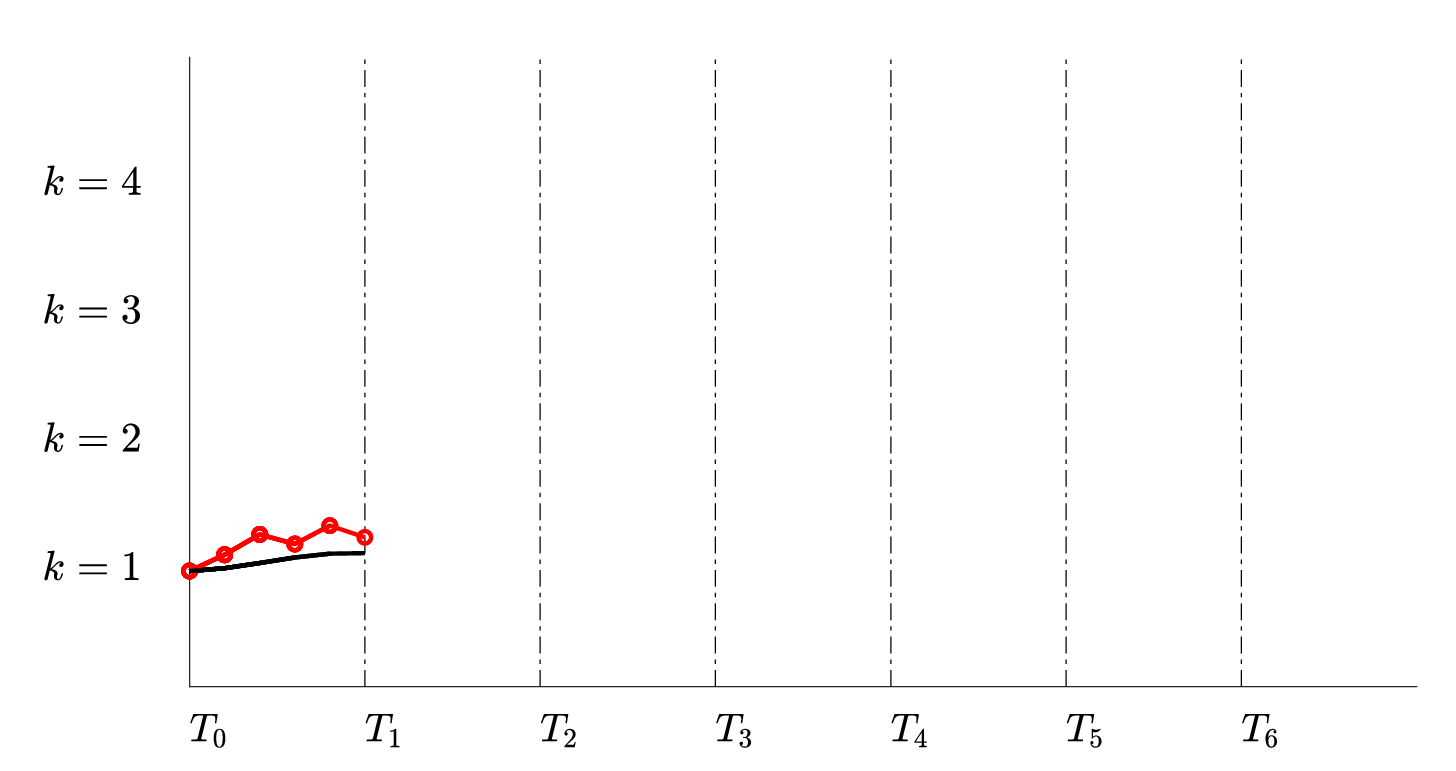}
 \includegraphics[width=2.25in,height=1.2in,angle=0]{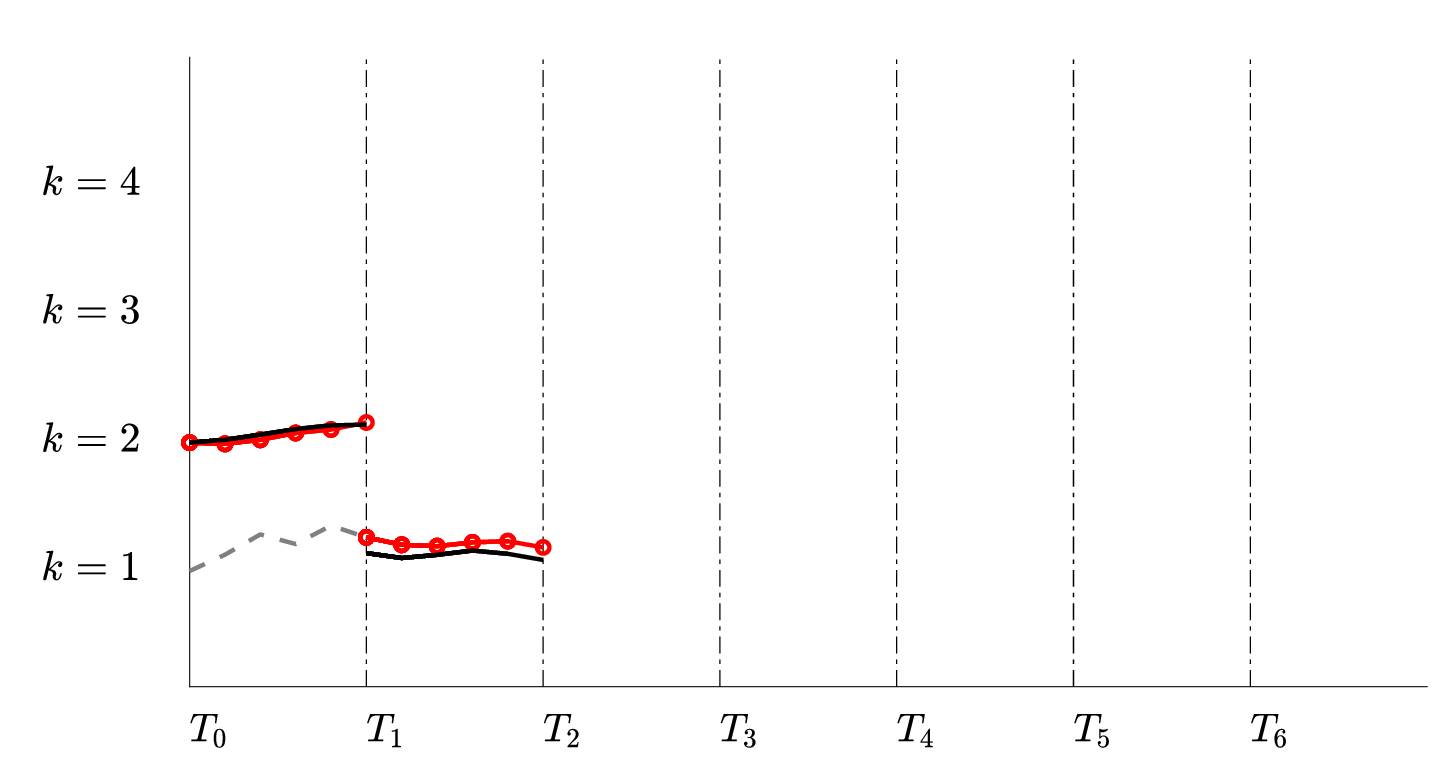}\\
 \includegraphics[width=2.25in,height=1.2in,angle=0]{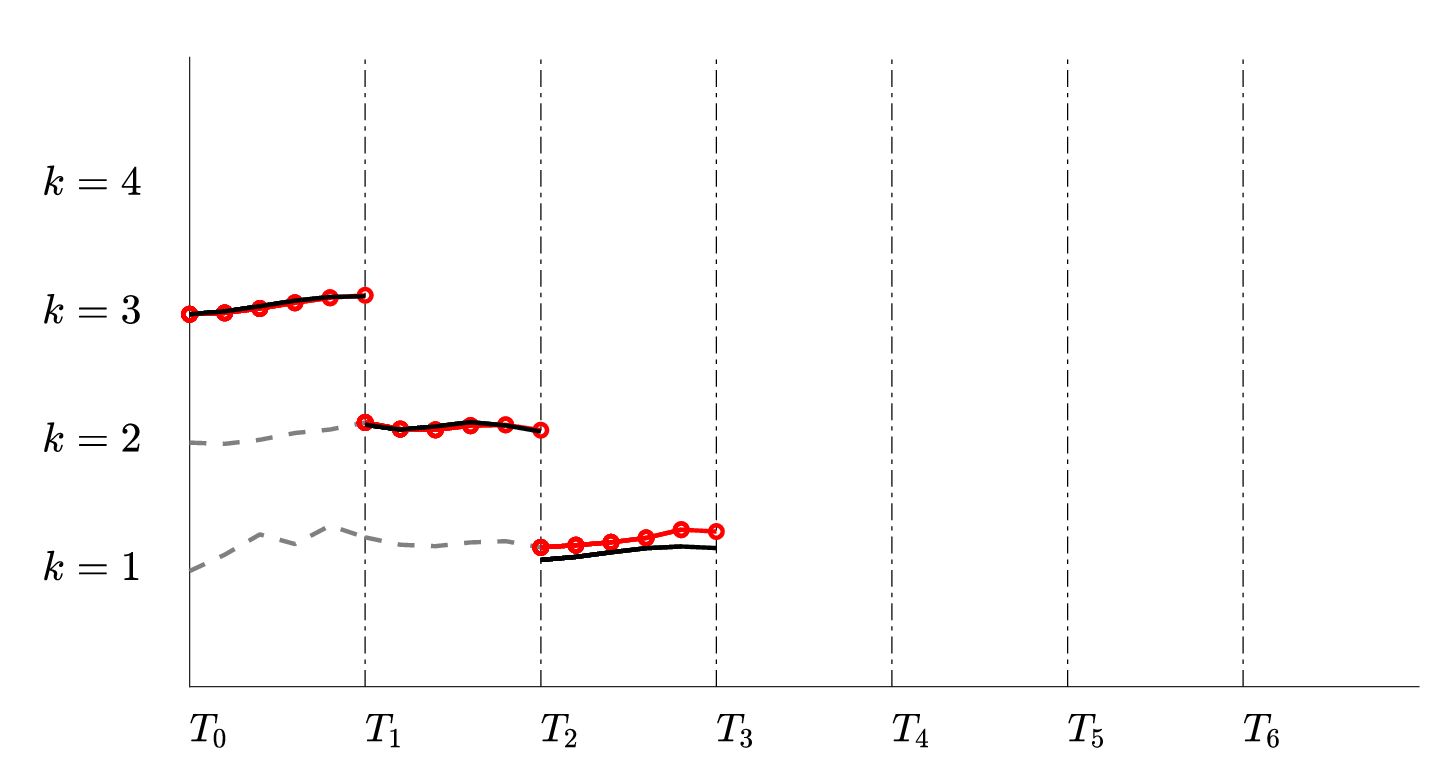}
 \includegraphics[width=2.25in,height=1.2in,angle=0]{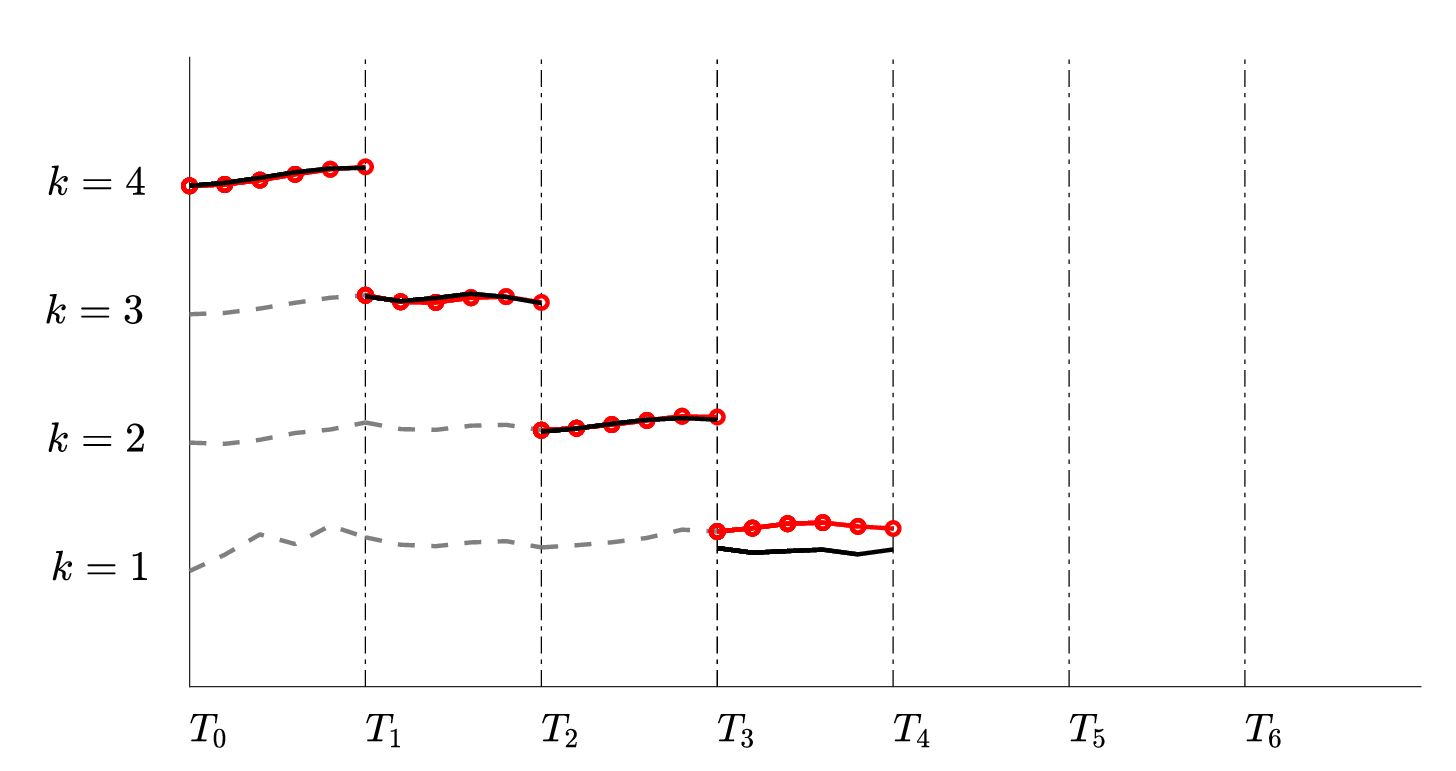}\\
  \includegraphics[width=2.25in,height=1.2in,angle=0]{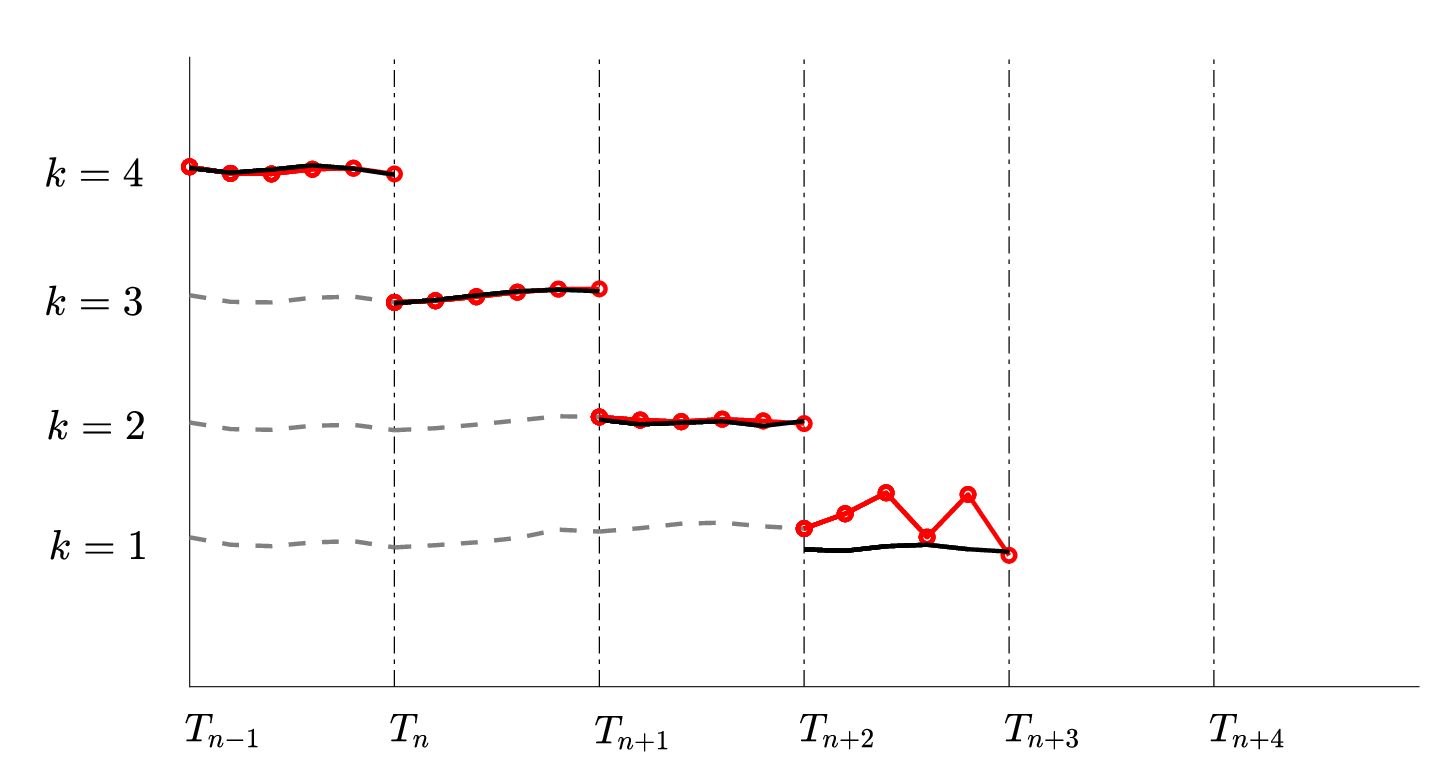}
   \includegraphics[width=2.25in,height=1.2in,angle=0]{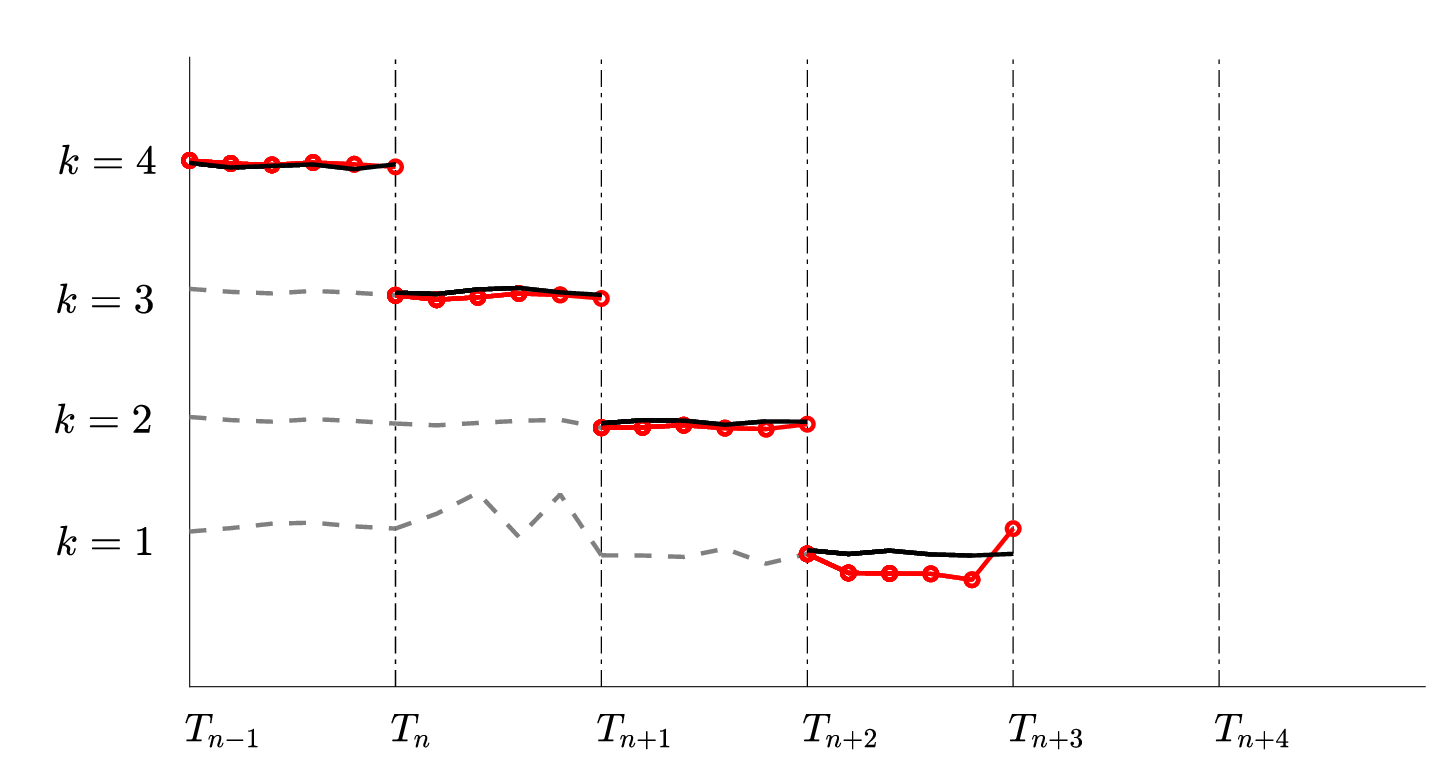}
  \caption{In PIDC with $k_{\max}=4$, the sweeps on the 4 time
    windows can run simultaneously (bottom row), once the
      initialization phase in the first four windows (top and middle
      rows) is completed. The black dashed lines represent sweep
    histories, while the red solid lines marked with a red circle
    indicate current sweeps run in parallel. The black solid
    lines show the exact solution.}
  \label{PipelineIDCFig}
\end{figure}
On $I_n$ (with $n\geq1$), because each sweep starts from a rough and
changing initial value, there is no guarantee that the accuracy of the
generated solution increases as we proceed with the corrections.  We
illustrate this point by applying IDC and PIDC to the
advection-diffusion equation \eqref{ADE} with two values of the
diffusion parameter, $\nu=1$ and $\nu=10^{-3}$. We consider periodic
boundary conditions and discretize with centered finite differences
with a mesh size $\Delta x=\frac{1}{64}$, which leads to the
linear system of ODEs \eqref{linearODE}, i.e., ${\bm u}'(t)=A{\bm
  u}(t)$ with $A=\frac{\nu}{\Delta x^2}A_{\rm xx}-\frac{1}{2\Delta
  x}A_{\rm x}$, where $\frac{\nu}{\Delta x^2}A_{\rm xx}\approx
\nu\partial_{xx}$ and $\frac{1}{2\Delta x}A_{\rm x}\approx \partial_x$
are the discretization matrices given by
\begin{equation}\label{AxAxx}
{\small  \begin{split}
 &A_{\rm xx}=
  \begin{bmatrix}
  -2 &1 & & &1\\
  1 &-2 &1 & &\\
  &\ddots &\ddots &\ddots &\\
  & &1 &-2 &1\\
  1 & & &1 &-2
  \end{bmatrix},~ A_{\rm x}= \begin{bmatrix}
  0 &1 & & &-1\\
  -1 &0 &1 & &\\
  &\ddots &\ddots &\ddots &\\
  & &-1 &0 &1\\
  1 & & &-1 &0
  \end{bmatrix}. 
  \end{split}}
\end{equation}
Let $T=3$ and the window size be $\Delta T=\frac{1}{10}$. Then, using
Backward Euler as time-integrator, we show in Figure \ref{ADEPIDCFig}
\begin{figure}
  \centering
 \includegraphics[width=2.3in,height=1.85in,angle=0]{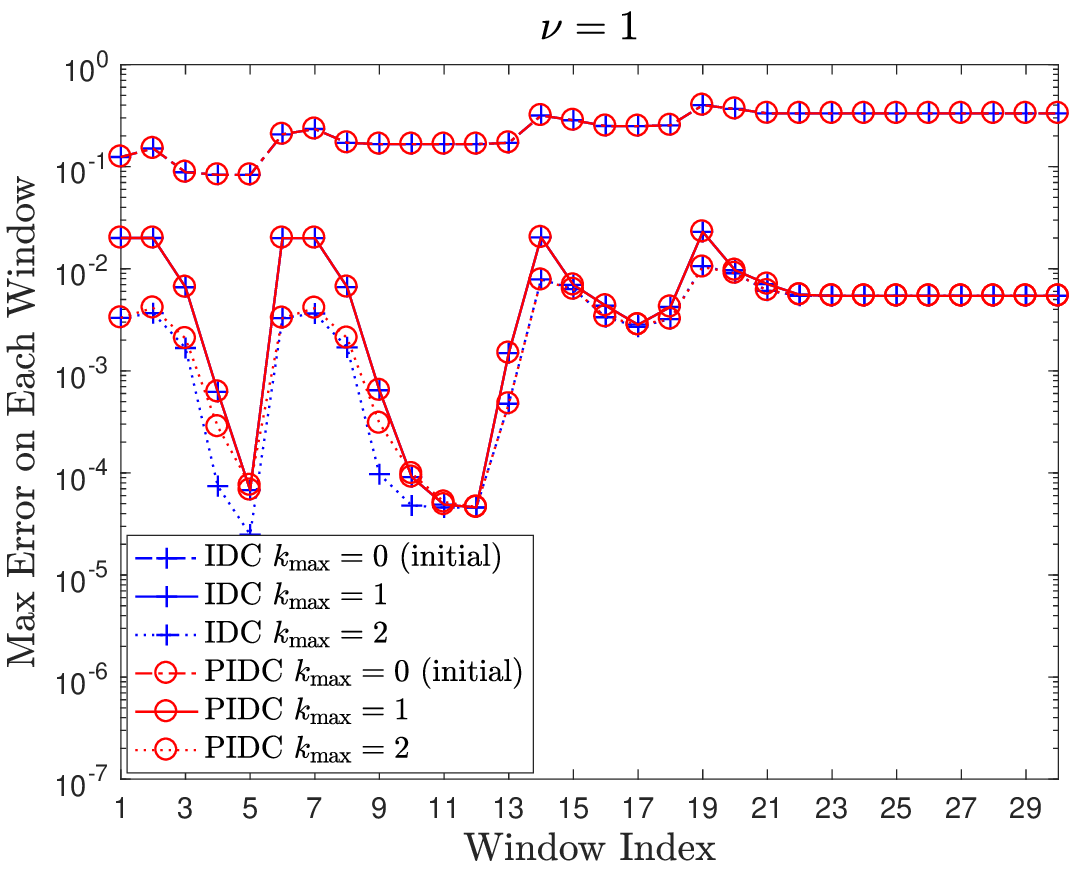}
 \includegraphics[width=2.3in,height=1.85in,angle=0]{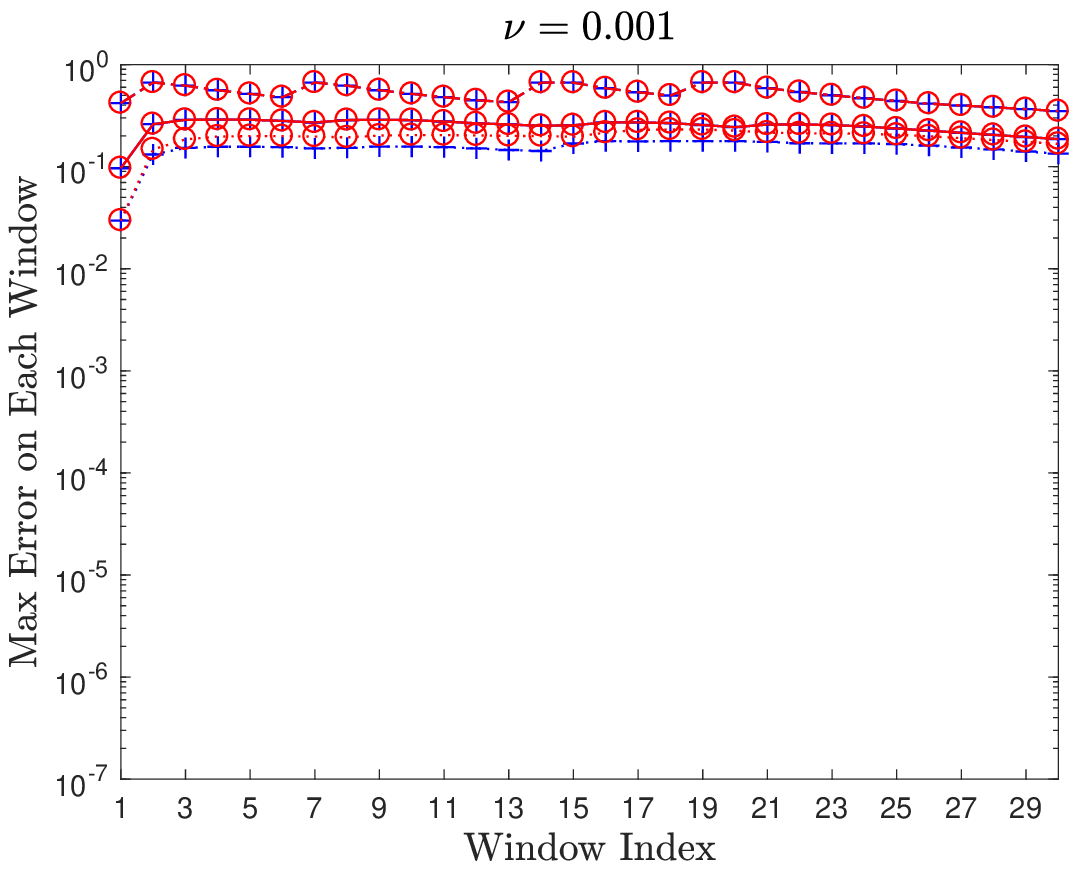}\\
 \includegraphics[width=2.3in,height=1.85in,angle=0]{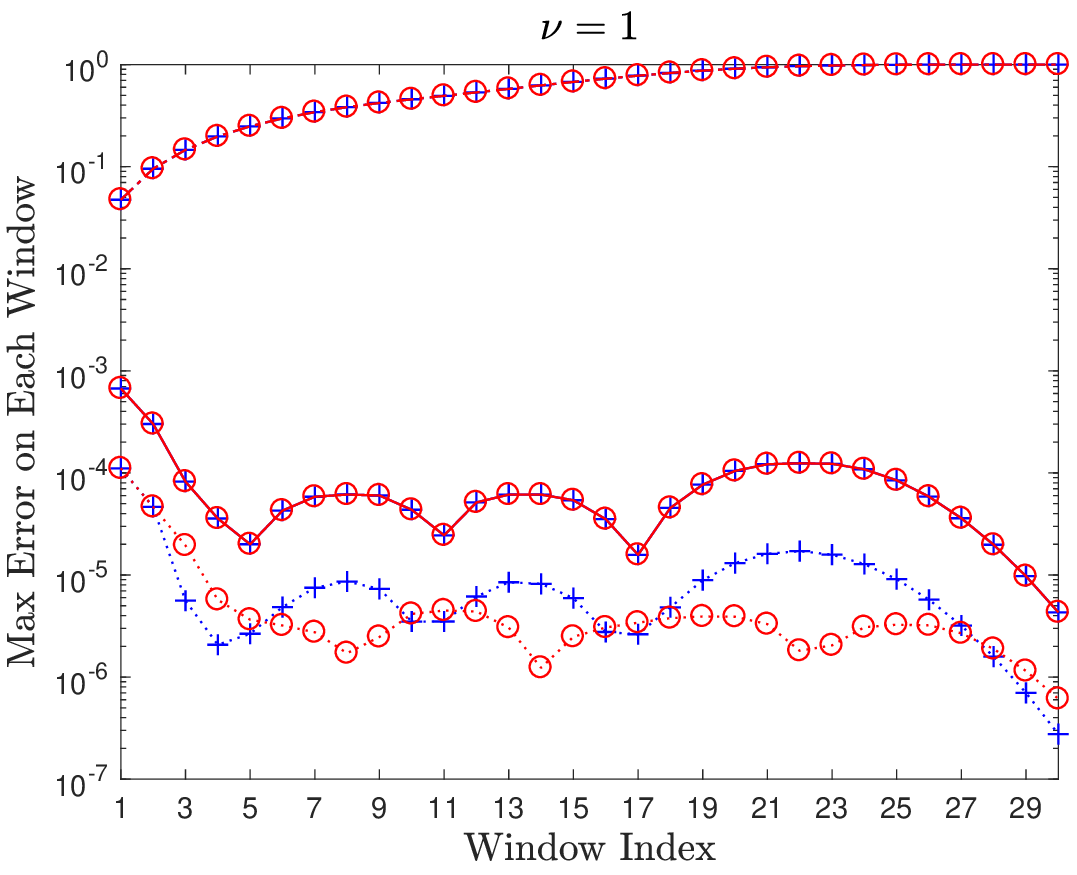}
 \includegraphics[width=2.3in,height=1.85in,angle=0]{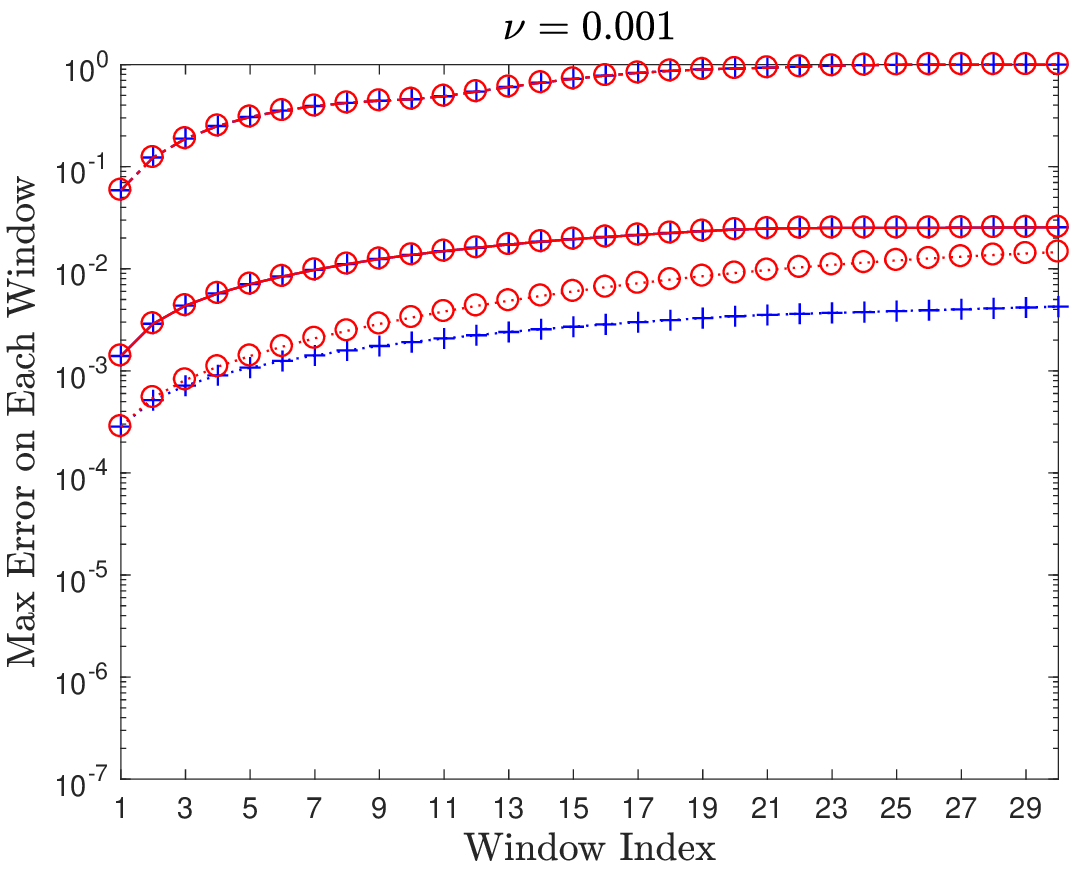}
  \caption{Maximum relative error on each time window for original IDC and
      its parallel version PIDC for the advection diffusion equation
      \eqref{ADE} with  source function
      $g(x,t)$ in \eqref{fxt} with $\sigma=1000$ (low regularity, top)
      and $\sigma=5$ (higher regularity, bottom), and large diffusion
      parameter (left), and small diffusion parameter (right).  The
      legend in the top left panel is also valid for the other
      panels.}
  \label{ADEPIDCFig}
\end{figure}
for IDC and PIDC with $M=5$ the maximal relative error for each
  time window measured as
$$
{\rm err}^k_n=\frac{\max_{m}\|{\bm u}^{\rm ref}_{n,m}-{\bm u}^{k}_{n,m}\|_\infty}{\max_{n,m}\|{\bm u}_{n,m}^{\rm ref}\|_\infty}, 
$$
where the reference solution ${\bm u}_{n,m}^{\rm ref}$ is
computed by the built-in solver \texttt{ODE45} in Matlab, using for
both the relative and absolute tolerance $1e-13 $. In each panel, for
both IDC and PIDC, we show the initial error and the errors after $1$
and $2$ sweeps. The initial guess on the $(n+1)$-th window
$I_{n+1}=[T_n, T_{n+1}]$ is fixed simply as ${\bm u}_{n+1, m}^0\equiv
{\bm u}^1_{n, M}$ for $j=0,1,\dots, M$.

The results in Figure \ref{ADEPIDCFig} show that for good performance
of IDC and PIDC, the solution of the problem needs to be regular. In
the first row, we used the source function $g(x,t)$ from \eqref{fxt}
with parameter $\sigma=1000$, which implies a $\delta$-function type
source, such that the solution is not regular enough. We see in the
first panel that both IDC and PIDC perform similarly, and after the
first correction the errors are not further reduced, the solution is
not regular enough for a higher order approximation to perform
well. In the second panel in the top row, we see that when the
diffusion is becoming small, the improvement of the first IDC
iteration is much worse than in the left panel, and a further
iteration does also not help much, and similarly for PIDC. In the
bottom row on the left, we see that if we use a very regular source,
\eqref{fxt} with parameter $\sigma=5$, and thus the solution has
enough regularity, both IDC and PIDC improve now for large diffusion
in the second iteration as well, and PIDC is comparable to IDC. For
small diffusion however at the bottom right, again performance is not
as good, and PIDC performs clearly less well at the second iteration
compared to IDC. These results indicate that for hyperbolic problems,
if the solution is not regular enough, PIDC will not be very suitable
for PinT computations.

\subsubsection{Revisionist IDC (RIDC)}

The RIDC method proposed by \cite{CMO10} is using a sliding IDC
interval as a main new idea for more fine grained parallelization. To
do so, consider a quadrature rule with $M$ equidistant nodes. In RIDC,
a first processor computes the initial approximation using a low order
time stepper, like in IDC, but once it arrives at the end of the IDC
interval after $M$ steps, it does not stop, it just continues
progressing in time computing step $M+1$, $M+2$ and so on. With the
first $M$ values of the first processor available, the second
processor has now enough information to start the first IDC
correction. Once it arrives at the end of the first IDC interval
computing the correction for the $M$-th step, the third processor can
start, but the second processor does not stop, it just continues by
moving its IDC interval and associated quadrature formula one fine
time step to the right, i.e. instead of using the approximations from
the steps $1,2,\ldots,M$ of the first processor, it considers the
approximations from the steps $2,3,\ldots,M+1$ of the first processor
as its IDC interval and quadrature nodes, and computes with these its
approximation for the $M+1$-st step. And then it considers the
approximations from the steps $3,4,\ldots,M+2$-nd of the first
processor as its IDC interval, and computes with these as quadrature
nodes the step $M+2$, and so on. Similarly, the third processor will
also continue with a sliding IDC interval, and so on.  Like PIDC,
  RIDC needs regularity to be effective, since it is a high order
  approximation technique, and thus for hyperbolic problems in the
  case of low regularity solutions, RIDC risks not to be very
  effective for PinT computations.

Like PIDC, RIDC needs regularity to be effective, since it is a
high-order approximation technique, and thus for hyperbolic problems
in the case of low-regularity solutions, RIDC risks not being very
effective for PinT computations. This is illustrated in Figure
  \ref{Fig_IDC_RIDC}, where we apply IDC and RIDC to the advection
  diffusion equation (2.5) with the same data used for Figure 3.5.
 \begin{figure}[H]
  \centering
 \includegraphics[width=2.3in,height=1.85in,angle=0]{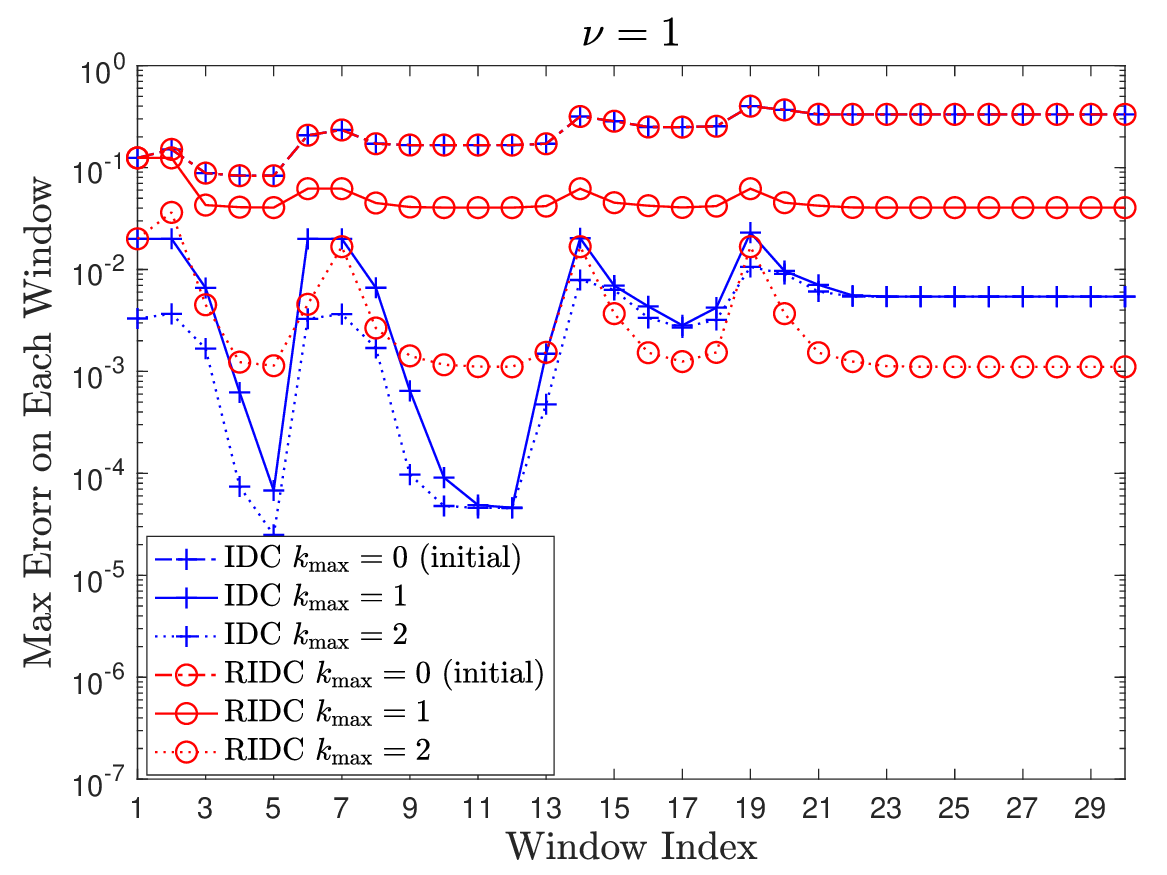}
 \includegraphics[width=2.3in,height=1.85in,angle=0]{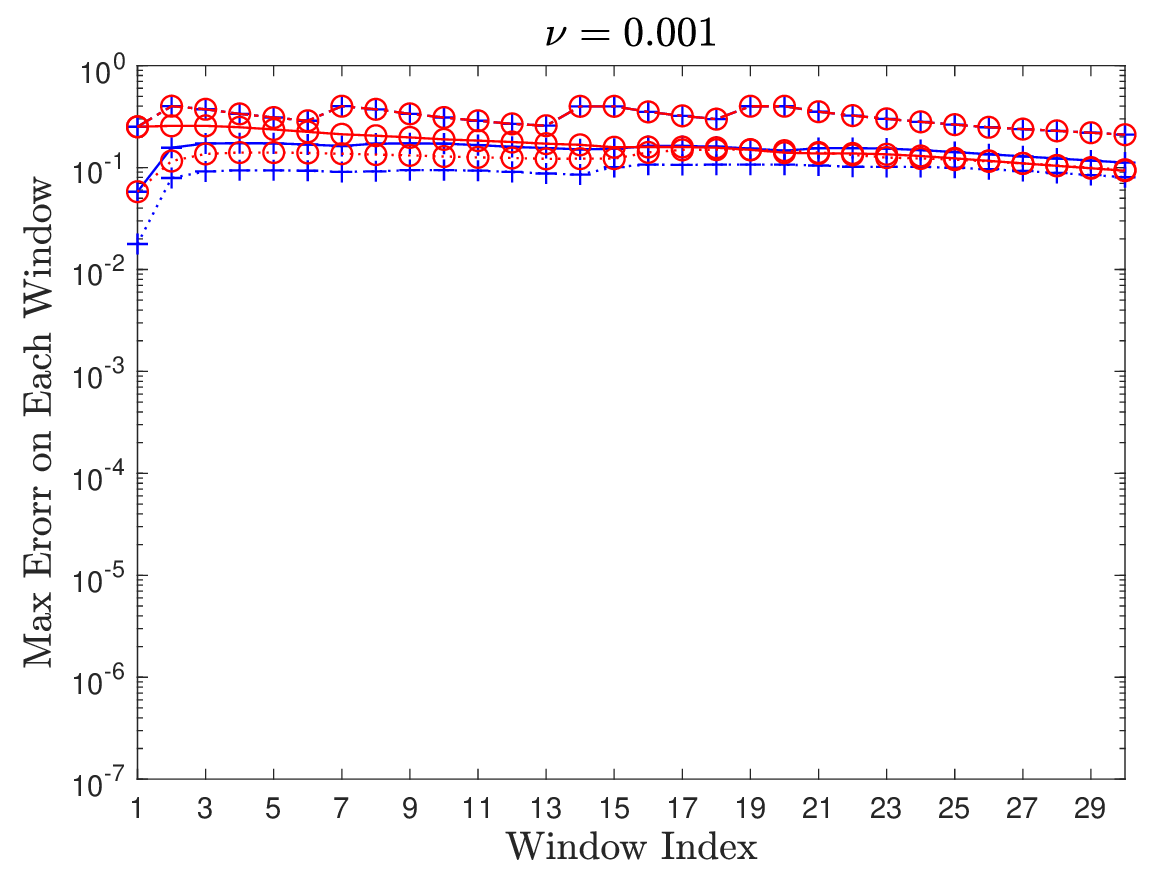}\\
 \includegraphics[width=2.3in,height=1.85in,angle=0]{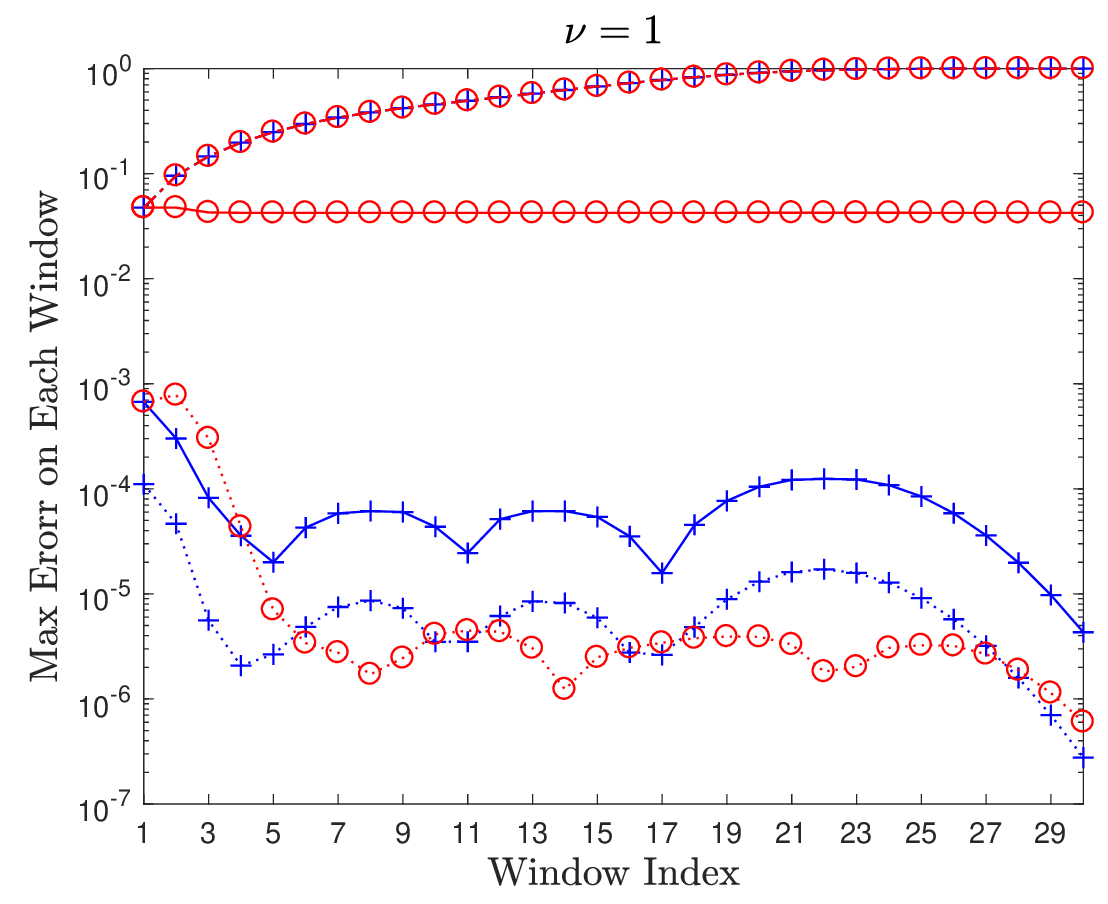}
 \includegraphics[width=2.3in,height=1.85in,angle=0]{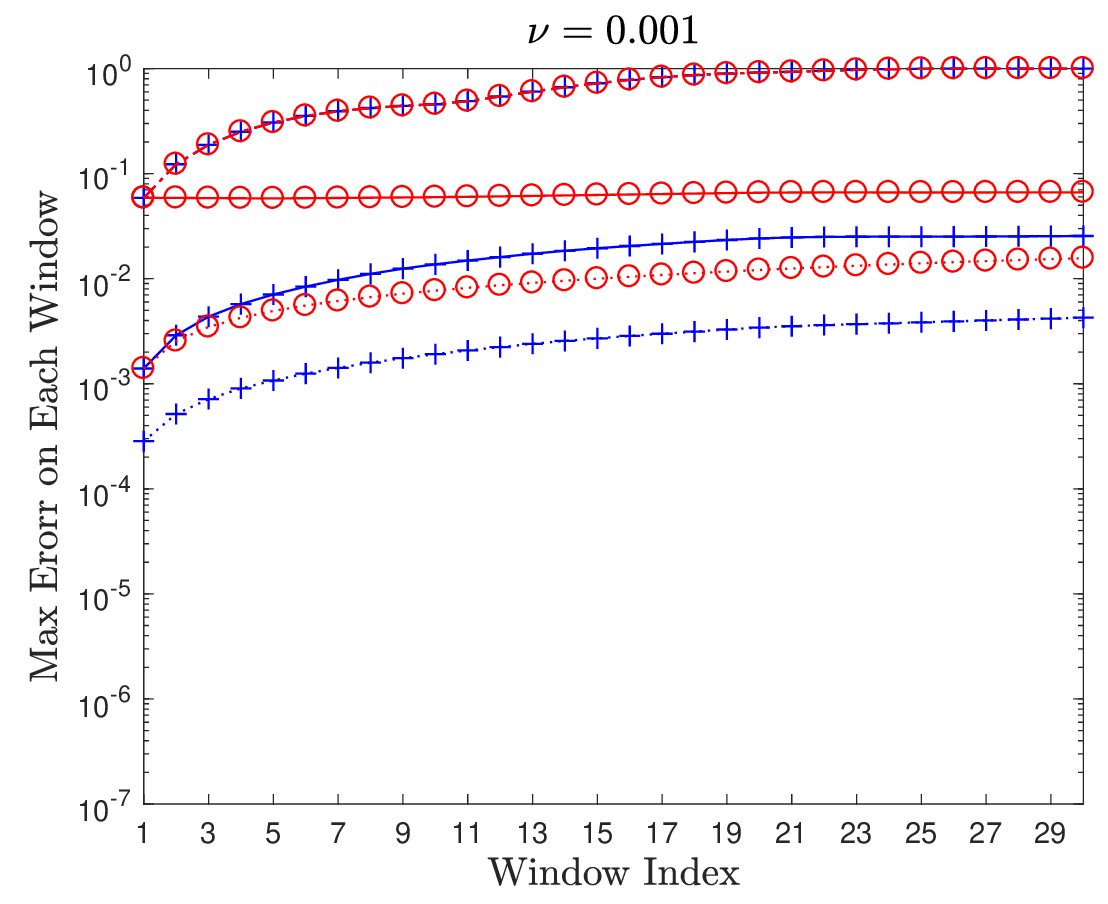}
  \caption{Maximum relative error on each time window for original IDC and
its parallel version RIDC for the advection diffusion equation (2.5) with source term $g(x,t)$ in (2.4) and initial condition $u(x,0)=0$.}
  \label{Fig_IDC_RIDC}
\end{figure}

\subsection{ParaExp}\label{Sec3.5}

The fourth time parallel method we want to present is the ParaExp
algorithm \cite{gander2013paraexp}, which is a direct time parallel
method that solves linear problems such as \eqref{linearODE}, which is
the semi-discrete version of PDEs like the advection-diffusion
equation \eqref{ADE} or the wave equation (\ref{WaveEquation1d}), see
also \cite{merkel2017paraexp,kooij2017block}. ParaExp uses special
  approximations of the matrix exponential function, and there are
  also other such techniques, like REXI \cite{schreiber2018beyond},
  see also the early PinT methods based on Laplace transforms REXI
  \cite{schreiber2018beyond}.

ParaExp is based on a time decomposition, and performs two steps
in order to construct the solution.  First, on each time interval, the
equation is solved in parallel with a source term but zero initial
condition (red problems in Figure \ref{ParaExpFig}),
\begin{equation}\label{ParaExp1}
    \begin{array}{rcll}
      {\bm v}'_n(t) = A {\bm v}_n(t) +  {\bm g}(t), &
      t \in (T_{n-1}, T_n], &
        {\bm v}_n(T_{n-1}) = 0,
    \end{array}
\end{equation}
where $n = 1, 2, \dots, N_t$ and $T_{N_t} = T$.
\begin{figure}
  \centering
  \includegraphics[width=0.8\textwidth]{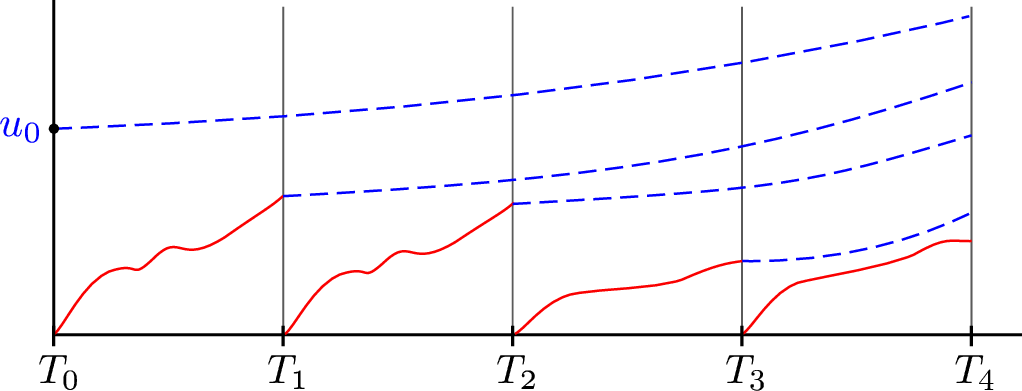}
  \caption{Two steps in the ParaExp solver.}
  \label{ParaExpFig}
\end{figure}
One then solves in parallel the linear equation \eqref{linearODE}
without source terms (blue problems in Figure \ref{ParaExpFig}), using
as initial conditions the results from \eqref{ParaExp1},
\begin{equation}\label{ParaExp2}
    \begin{array}{rcll}
      {\bm w}'_n(t) = A{\bm w}_n(t), & t \in (T_{n-1}, T], &
      {\bm w}_n(T_{n-1}) = {\bm v}_{n-1}(T_{n-1}),
    \end{array}
\end{equation}
where $n = 1, 2, \dots, N_t$ and ${\bm v}_0(T_0) = {\bm u}_0$.  The
exact solution ${\bm u}(t)$ can then by linearity be constructed from
the decoupled red and blue solutions,
\begin{equation}\label{ParaExp3}
{\bm u}(t) = {\bm v}_n(t) + {\sum}_{j=1}^n {\bm w}_j(t), \quad t \in [T_{n-1}, T_n], \quad n = 1, 2, \dots, N_t,
\end{equation}
as one can see as follows: for $n = 1$, by adding \eqref{ParaExp1} to
\eqref{ParaExp2} we have
$$
  ({\bm v}_1(t) + {\bm w}_1(t))' = A({\bm v}_1(t) + {\bm w}_1(t)) +
      {\bm g}(t),\ t \in (T_0, T_1],\ 
      ({\bm v}_1(0) + {\bm w}_1(0)) = {\bm u}_0.
$$
This proves \eqref{ParaExp3} for $n = 1$. Now, suppose
\eqref{ParaExp3} holds for $n$, and we thus have 
$$
  {\bm u}(T_{n}) = {\bm v}_n(T_n) + {\sum}_{j=1}^n {\bm w}_j(T_n).
$$
Then, in the next time interval $[T_{n}, T_{n+1}]$, since ${\bm
  w}_{n+1}(T_n) = {\bm v}_{n}(T_n)$ we have
$$
{\bm u}(T_{n}) = {\bm w}_{n+1}(T_n) + {\sum}_{j=1}^n {\bm w}_j(T_n) = {\sum}_{j=1}^{n+1} {\bm w}_j(T_n).
$$
Now, the function ${\bm w}(t)$ consisting of the first $n+1$ blue
solutions, i.e., ${\bm w}(t) = \sum_{j=1}^{n+1} {\bm w}_j(t)$
satisfies ${\bm w}'(t) = A{\bm w}(t)$ for $t \in (T_n, T_{n+1}]$
  and ${\bm w}(T_n) = {\bm u}(T_n)$, and hence ${\bm w}(t) + {\bm
    v}_{n+1}(t)$ satisfies the underlying problem \eqref{linearODE}
  for $t \in [T_n, T_{n+1}]$, which proves \eqref{ParaExp3} for $n+1$.

As illustrated in Figure \ref{ParaExpFig}, the computation of the red
and blue solutions can be done in parallel for all time intervals. But
the computation of the blue problems \eqref{ParaExp2} over longer and
longer time intervals seems at first sight to be as expensive as the
original problem \eqref{linearODE}. This is however not the case,
since the blue problems are homogeneous, i.e., without a source term,
and their solution is given by a matrix exponential,
\begin{equation}\label{ParaExp4}
    \begin{array}{rcll}
   {\bm w}_n(t) = \exp((t-T_{n-1})A){\bm v}_{n-1}(T_{n-1}), & t\in[T_{n-1}, T],
    \end{array}
\end{equation}
where the computation time of the product between the matrix
exponential and the vector ${\bm v}_{n-1}(T_{n-1})$ is independent of
the value of $t$. There are many efficient and mature computational
tools to approximate such solutions over long time
\cite{Higham2008,MV02}, such as rational Krylov methods and Chebyshev
expansions, and also the scaling and squaring algorithm with a
Pad\'{e} approximation (i.e., the built-in command `\texttt{expmv}' in
MATLAB\_R2023b or later versions). This latter approach is however
more suitable for smaller matrices, for large sparse matrices the
former approaches should be used. With efficient computations of the
matrix exponential, using ParaExp can achieve high parallel
efficiencies, up to 80\% for the time parallelization of the wave
equation \eqref{WaveEquation1d}, see \cite{gander2013paraexp}. ParaExp
is therefore an excellent time parallelization method for linear
hyperbolic problems.

The ParaExp method described above is restricted to linear
problems. An extension to nonlinear problems \eqref{nonlinearODE} was
presented in \cite{GGP18}, assuming that there is a linear part
of the nonlinear term such as
\begin{equation}\label{ParaExp5}
f({\bm u}(t), t) = A{\bm u}(t) + B({\bm u}(t)) + {\bm g}(t).
\end{equation}
Following the idea in the linear case, we decouple the nonlinear
problem \eqref{ParaExp5} into a linear problem ${\bm w}'(t) = A{\bm
  w}(t)$ with ${\bm w}(0) = {\bm u}_0$ and a nonlinear problem ${\bm
  v}'(t) = B({\bm v}(t) + {\bm w}(t)) + {\bm g}(t)$ with zero initial
value ${\bm v}(0) = 0$.   The sum ${\bm u}(t) = {\bm w}(t) + {\bm
    v}(t)$ then still solves \eqref{ParaExp5}, but the problems on
  the time intervals are now coupled: in $\{[T_{n-1},
    T_n]\}_{n=1}^{N_t}$, the initial value of ${\bm w}(t)$ at $t =
  T_{n-1}$ depends on ${\bm v}(T_{n-1})$. To obtain parallelism in
  time, we need to iterate by first solving in parallel the linear
problems
\begin{equation*}
    \begin{array}{rcll}
    ({\bm w}_n^k)'(t) & = & A{\bm w}_n^k(t), & t\in[T_{n-1}, T],\\
    {\bm w}_n^k(T_{n-1}) & = & {\bm v}_{n-1}^{k-1}(T_{n-1}), & {\bm w}_1^k(T_0) = {\bm u}_0,
    \end{array}
\end{equation*}
and then solving in parallel the nonlinear problems
\begin{equation*}
    \begin{array}{rcll}
    ({\bm v}_n^k)'(t) & = & A{\bm u}_n^k(t) + B({\bm v}_n^k(t) + \sum_{j=1}^{n}{\bm w}_{j}^{k}(t)) + {\bm g}(t), & t\in[T_{n-1}, T_n],\\
    {\bm u}_n^k(T_{n-1}) & = & 0, &
    \end{array}
\end{equation*}
where $n = 1,2,\dots, N_t$. The $k$-th iterate solution is then
defined by ${\bm u}_n^k(t) = {\bm v}_n^k(t) + \sum_{j=1}^n{\bm
  w}_j^k(t)$ for $t\in[T_{n-1}, T_n]$.

In the above nonlinear problems, the explicit dependence of $B$ on
$\sum_{j=1}^{n}{\bm w}_{j}^{k}(t)$ implies that we have to solve the
linear problems on the entire interval $[T_{n-1}, T_n]$. This would be
redundant and expensive if $A$ is large. To avoid this, we reformulate the
iteration by  rewriting ${\bm v}_n^k(t)$ as  ${\bm v}_n^k(t)={\bm
  u}_n^k(t)-{\sum}_{j=1}^n{\bm w}_j^k(t)$. 
In this new nonlinear version of the ParaExp algorithm we then solve
for $n=1,2,\dots, N_t$ sequentially
\begin{equation}\label{ParaExp6}
    \begin{array}{rcll}
({\bm w}^k_n)'(t) & = &A{\bm w}_n^k(t), 
      &t\in[T_{n-1}, T],\\
     {\bm w}_n^k(T_{n-1})& = &{\bm u}_{n-1}^{k-1}(T_{n-1})-{\sum}_{j=1}^{n-1}{\bm w}_{j}^{k-1}(T_{n-1}) &{\bm w}_1^k(T_0)={\bm u}_0,\\ 
    \end{array}
\end{equation}
followed by solving in parallel the nonlinear problems
\begin{equation}\label{ParaExp7}
    \begin{array}{rcll}
({\bm u}^k_n)'(t) & = &A{\bm u}_n^k(t)+B({\bm u}_n^k(t))+{\bm g}(t), 
      &t\in[T_{n-1}, T_n],\\
     {\bm u}_n^k(T_{n-1})& = &{\sum}_{j=1}^{n}{\bm w}_{j}^{k}(T_{n-1}),  &\\ 
    \end{array}
\end{equation}
and finally we form the approximate solution at the $k$-th iteration as 
$$
{\bm u}^k(t)={\bm u}_n^k(t),~t\in[T_{n-1}, T_n]. 
$$
The non-linear ParaExp algorithm \eqref{ParaExp6}-\eqref{ParaExp7} has
a finite step convergence property and a very interesting relation to
the Parareal algorithm:
\begin{theorem}[\cite{GGP18}]\label{pro3.1}
{\em The iterate ${\bm u}^k(t)$ at the $k$-th iteration coincides with
  the exact solution ${\bm u}(t)$ for $t\in[0, T_k]$, i.e., the
  iterative ParaExp method converges in a finite number of steps. Moreover, at
  each time point $T_n$, the iterate ${\bm u}^k(t)$ also coincides with
  the solution generated by the Parareal algorithm
\begin{subequations}
  \begin{equation}\label{ParaExp8a}
{\bm U}_n^k=\CG(T_{n-1}, T_n, {\bm U}_{n-1}^{k})+\CF(T_{n-1}, T_n, {\bm U}_{n-1}^{k-1})
-\CG(T_{n-1}, T_n, {\bm U}_{n-1}^{k-1}), 
\end{equation}
i.e., ${\bm u}^k_n={\bm U}_n^k$ for $n=0,1,\dots, N_t$, where the
coarse propagator
$\CG(T_{n-1}, T_n, {\bm U})$ solves the linear problem
\begin{equation}\label{ParaExp8b}
{\bm u}'(t)=A{\bm u}(t),~{\bm u}(T_{n-1})={\bm U},~t\in[T_{n-1}, T_n], 
\end{equation}
and the fine propagator $\CF(T_{n-1}, T_n, {\bm U})$ solves the
nonlinear problem
\begin{equation}\label{ParaExp8c}
{\bm u}'(t)=A{\bm u}(t)+B({\bm u}(t))+{\bm g}(t),~{\bm u}(T_{n-1})={\bm U},~t\in[T_{n-1}, T_n]. 
\end{equation}
\end{subequations}
}
\end{theorem}

This is the first time that we see the {\em Parareal} algorithm, which we
will discuss in detail in Section \ref{Sec4}. The Parareal algorithm
\eqref{ParaExp8a}-\eqref{ParaExp8c} is a simplified version since, for
the standard version, the coarse propagator $\CG$ also solves
\eqref{ParaExp8c}.  As will be discussed in Section \ref{Sec4}, also
the standard Parareal algorithm does not perform well for hyperbolic
problems, and thus we cannot expect the simplified version to work
well in this case.  An illustration of this aspect is shown in Figure
\ref{NonlinearParaExpFig}
\begin{figure}
  \centering
 \includegraphics[width=1.6in,height=1.4in,angle=0]{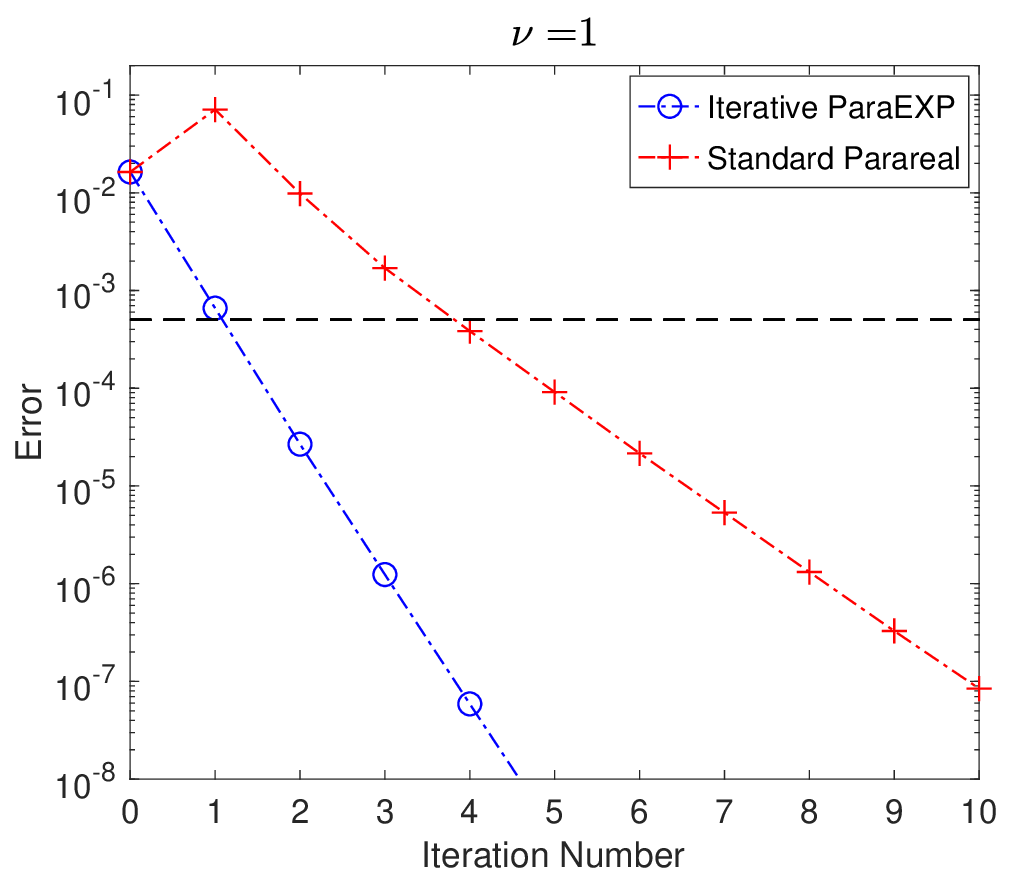}
 \includegraphics[width=1.6in,height=1.4in,angle=0]{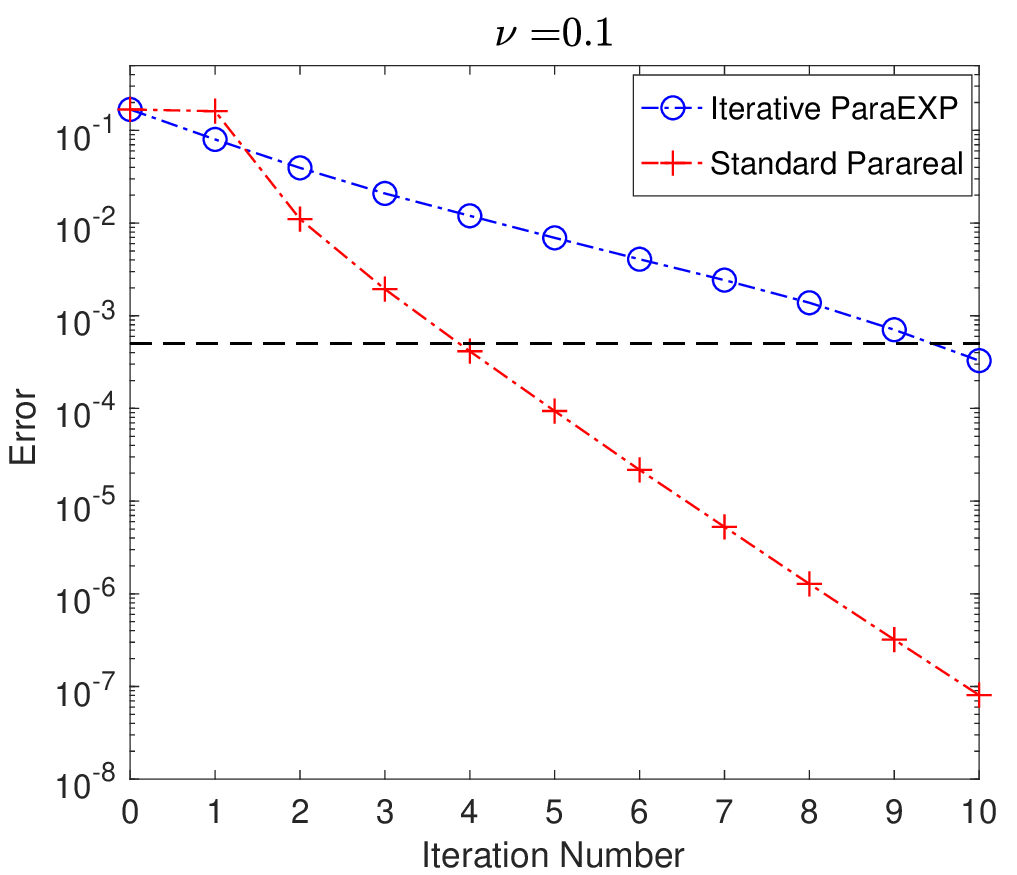}
  \includegraphics[width=1.6in,height=1.4in,angle=0]{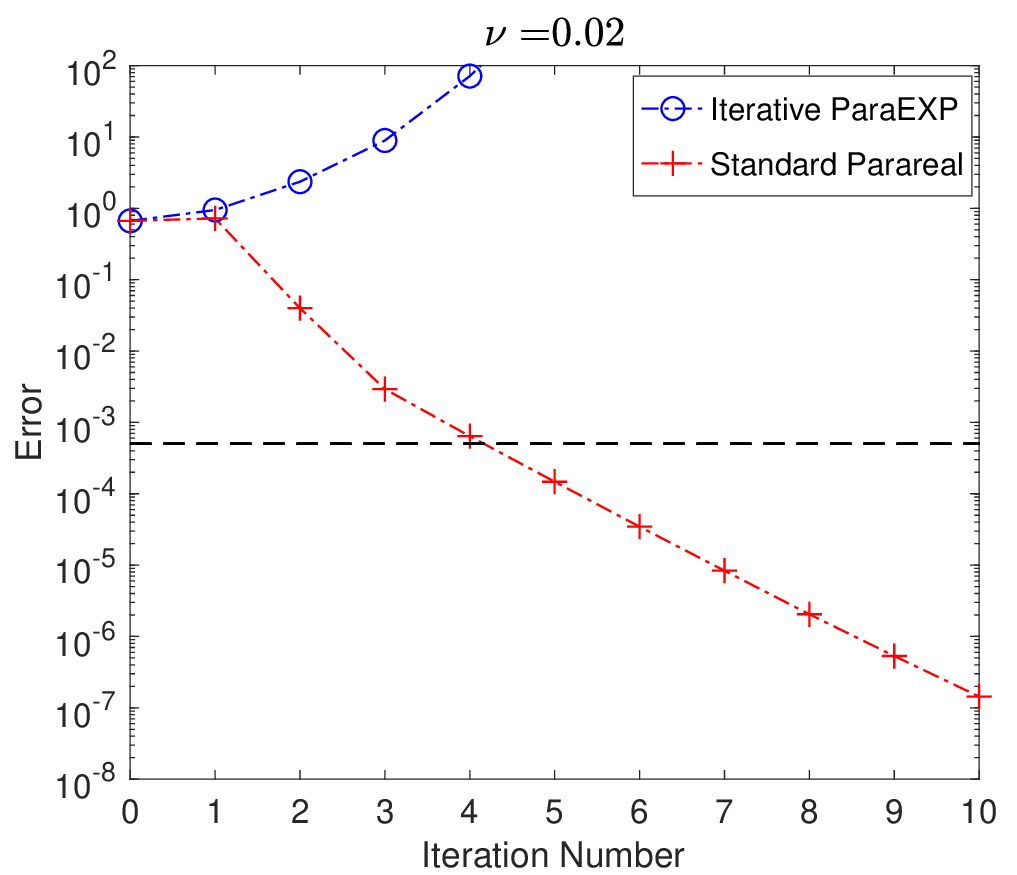}
  \caption{Convergence behavior of ParaExp and standard Parareal
    for Burgers' equation with diffusion parameter $\nu$
    changing from large to small. In each panel, the transverse line
    indicates  the order of the truncation error, $\max\{\Delta t,
    \Delta x^2\}$, where in practice one would stop the iteration.}
  \label{NonlinearParaExpFig}
\end{figure}
for \eqref{nonlinearODE} with
\begin{equation}\label{ParaExp9a}
  f({\bm u}(t), t)=A{\bm u}(t)+ B{\bm u}^2(t),~t\in(0, 2), 
\end{equation} 
arising from semi-discretizing the 1D Burgers' equation with periodic
boundary conditions using centered finite differences with a mesh size
$\Delta x=\frac{1}{100}$, where $A=A_{\rm xx}$ and
$B=-\frac{1}{2}A_{\rm x}$ with $A_{\rm xx}$ and $A_{\rm x}$ given in
\eqref{AxAxx}.  For both ParaExp and Parareal, we use for the fine
solver $\CF$ Backward Euler with a small step size $\Delta
t=\frac{0.01}{20}$. For Parareal, we use for the coarse solver $\CG$
Backward Euler as well but with a larger step size $\Delta
T=0.01$. For ParaExp, we use the built-in solver \texttt{expmv} in
Matlab for the coarse propagator.

Clearly, for strongly diffusive problems, i.e., when $\nu$ is large,
ParaExp converges very fast and the convergence rate is better than
for standard Parareal. As $\nu$ decreases, standard Parareal converges
faster than ParaExp, and particularly for $\nu=0.02$, the latter
diverges as shown in Figure \ref{NonlinearParaExpFig} on the right. If
we decrease $\nu$ further, then also standard Parareal will eventually
diverge, and we will delve into this more in Section \ref{Sec4}.

\subsection{ParaDiag}\label{Sec3.6}

The last technique we would like to explain is the ParaDiag family of
methods, which is based on diagonalizing the time stepping matrix (or
its approximation). There are two variants of ParaDiag depending on
how we treat the time stepping matrix.

In the ParaDiag I family, which also represents a direct time parallel
solver like ParaExp, one diagonalizes the time stepping matrix and
then can solve each time step in parallel after diagonalization
\cite{maday2008parallelization}. The diagonalization in ParaDiag I is
however only possible when either using variable time step sizes, or
by using a different time-integrator for the last step compared to the
other ones, like in boundary value methods. For the case of variable
step sizes, a detailed error analysis \cite{gander2019direct} shows
that one can only use a limited number of time steps to parallelize,
in double precision about 20, since one has to balance roundoff error
with truncation error. Another shortcoming is that this direct
ParaDiag method has only been explored for a few low-order
time-integrators, such as Backward Euler and the trapezoidal
rule. ParaDiag I with variable time step sizes is not easy to
generalize to higher-order time-integrators, such as Runge-Kutta
methods. When using a boundary value method type discretization, the
number of time steps one can parallelize in a single time window is
greatly improved, but again only Backward Euler and the trapezoidal
rule are applicable \cite{LWWZ22}.

Both the limitations on the number of time steps and the
time-integrator are overcome by the ParaDiag II family
\cite{gander:2021:ParaDiag}. The key idea is to design a suitable
approximation of the time stepping matrix and then to use it in a
stationary iteration or as a preconditioner for a Krylov method, so
one has to pay with iterations. The design principles for constructing
such a preconditioner are twofold: its diagonalization should be
well-conditioned in contrast to the ParaDiag I family of methods
(i.e., the condition number of its eigenvector matrix should be
small), and the iterations should converge fast, i.e. have small
spectral radius, or equivalently the spectrum of the preconditioned
matrix should be tightly clustered around 1 for Krylov
acceleration. The  first design principle ensures that roundoff
error arising from solving the preconditioning step via
diagonalization is well controlled. The second design principle
guarantees fast convergence of the preconditioned iteration. This
iterative ParaDiag method was proposed in
\cite{mcdonald2018preconditioning} and independently in
\cite{gander2019convergence}. ParaDiag II techniques have been used as
important components in new variants of Parareal and MGRiT which
  we will see in Section \ref{Sec4.1} and \ref{Sec4.4}, and which in
  their original form work only well for parabolic problems. ParaDiag
  II techniques enhance the new Parareal and MGRiT variants in two
  directions: they improve the speedup by making the coarse grid
  correction parallel \cite{WSiSC18,WZSiSC19}, and they can also make
  Parareal and MGRiT work well for hyperbolic problems
  \cite{gander2020diagonalization} by allowing the coarse and fine
  propagators using the same grids, as we will show in Section
  \ref{Sec4.5}. The application of ParaDiag II to solve the
forward-backward system arising in PDE constrained optimization can be
found in \cite{WWZSIMAX23,WLSiSC20,HPer2024}, where ParaDiag II
produces a parallel version of the matching Schur complement (MSC)
preconditioner \cite{PSW12}. Modifications and improvements of
ParaDiag II can be found in \cite{Gander:TPTI:2024} and
\cite{LWSiSC22}.

ParaDiag methods are applicable to both parabolic and hyperbolic
problems, but the mechanisms for the direct and iterative versions are
completely different. In the following, we introduce the main theory
for these two versions and illustrate them with numerical results
for the advection-diffusion equation \eqref{ADE}, Burgers'
equation \eqref{Burgers}, and the wave equation
\eqref{WaveEquation1d}.

\subsubsection{Direct ParaDiag Methods (ParaDiag I)}\label{sec3.6.1}

Parallelization by diagonalization of the time stepping matrix,
originally introduced in \cite{maday2008parallelization}, is based on
a very simple idea: consider solving the initial value problem
\eqref{linearODE}, i.e. ${\bm u}'=A{\bm u}+{\bm g}(t)$ with
$A\in\mathbb{R}^{N_x\times N_x}$, by Backward Euler with variable step
sizes,
\begin{equation}\label{IEdtn}
\frac{{\bm u}_{n}-{\bm u}_{n-1}}{\Delta t_n}=A{\bm u}_n+{\bm g}_n, \quad n=1,2,\dots, N_t, 
\end{equation}
with ${\sum}_{n=1}^{N_t}\Delta t_n=T$. Instead of solving these $N_t$ difference equations one by one, we reformulate them as an  {\em all-at-once} system, i.e., we solve all the solution vectors collected in ${\bm U}:=({\bm u}_1^{\top}, {\bm u}_2^{\top}, \dots, {\bm u}_{N_t}^{\top})^{\top}$ in one shot,
\begin{subequations}
\begin{equation}\label{AAA1a}
\CK{\bm U}={\bm b}, \quad \CK:=B\otimes I_x-I_t\otimes A, 
\end{equation}
where $\otimes$ is the Kronecker product, $I_x\in\mathbb{R}^{N_x\times
  N_x}$ and $I_t\in\mathbb{R}^{N_t\times N_t}$ are identity matrices,
and $B$ is the time stepping matrix,
\begin{equation}\label{AAA1b}
B=\begin{bmatrix}
\frac{1}{\Delta t_1} & & &\\
-\frac{1}{\Delta t_2} &\frac{1}{\Delta t_2} & &\\
&\ddots &\ddots  &\\
& &-\frac{1}{\Delta t_{N_t}} &\frac{1}{\Delta t_{N_t}}
\end{bmatrix}, \quad {\bm b}=\begin{bmatrix}
\frac{1}{\Delta t_1}{\bm u}_0+{\bm g}_1\\
{\bm g}_2\\
\vdots\\
{\bm g}_{N_t}
\end{bmatrix}. 
\end{equation}
\end{subequations}
Since the time steps $\{\Delta t_n\}$ are all different from each
other, we can diagonalize $B$,
\begin{equation}\label{diag1}
B=VDV^{-1}, \quad D={\rm diag}\left(\frac{1}{\Delta t_1},\frac{1}{\Delta t_2},\dots, \frac{1}{\Delta t_{N_t}}\right). 
\end{equation}
Then, we can factor $\CK$ in a block-wise manner as
$$
  \CK=(V\otimes I_x)(D\otimes I_x-I_t\otimes A)(V^{-1}\otimes I_x).
$$
This allows us to solve ${\bm U}$ from \eqref{AAA1a} by the following
three steps:
\begin{equation}\label{3steps}
\begin{cases}
{\bm U}^a=(V^{-1}\otimes I_x){\bm b}, &\text{(step-a)}\\
\left(\frac{1}{\Delta t_n} I_x-A\right){\bm u}^b_n={\bm u}^a_n, \quad n=1,2,\dots, N_t, &\text{(step-b)}\\
{\bm U}=(V\otimes I_x){\bm u}^b, &\text{(step-c)}\\
\end{cases}
\end{equation}
where ${\bm U}^a:=(({\bm u}^a_1)^{\top}, \dots, ({\bm
  u}^a_{N_t})^{\top})^{\top}$ and ${\bm U}^b:=(({\bm u}^b_1)^{\top},
\dots, ({\bm u}^b_{N_t})^{\top})^{\top}$. Note that the first and last
steps only involve matrix-vector multiplications and thus the
computation is cheap. The major computation is step-b, but
interestingly, all the $N_t$ linear systems stemming from the time
steps are independent and can be solved in parallel.

For an arbitrary choice of the step sizes $\{\Delta t_n\}$, we have to
rely on numerical methods (e.g., \texttt{eig} in Matlab) to obtain the
eigenvector matrix $V$. This does not bring significant computational
burden since $N_t$ does not need to be very large in practice, but it
prevents us from performing a complete analysis of the method, such as
studying the roundoff error and the selection of the parameters
involved. The time mesh used in \cite{maday2008parallelization} is a
geometric mesh $\Delta t_n=\mu^{n-1}\Delta t_1$ with $\mu>1$. They
tested this direct time-parallel method with $\mu=1.2$ for the heat
equation in 1D and obtained close to perfect speedup.

For prescribed $T$, the constraint $\sum_{n=1}^{N_t}\Delta t_n =
\sum_{n=1}^{N_t}\mu^{n-1}\Delta t_1 = T$ specifies the first step size
$\Delta t_1$ as $\Delta t_1 =
\frac{T}{\sum_{n=1}^{N_t}\mu^{n-1}}$. This gives
\begin{equation}\label{dtn}
\Delta t_n = \frac{\mu^{n-1}}{\sum_{n=1}^{N_t}\mu^{n-1}}T. 
\end{equation}
A large $\mu$ will produce large step sizes and thus large
discretization error, while a small $\mu$ close to 1 produces large
roundoff error when diagonalizing the time stepping matrix $B$, which
can be understood by noticing that as $\mu$ approaches 1, the time
stepping matrix $B$ is close to a Jordan block, and diagonalizing such
matrices results in large roundoff error. Therefore, it is important
to know how to fix $\mu$ by balancing the discretization and roundoff
error. This was carefully studied in \cite{GHR16} for first-order
parabolic problems, and in \cite{gander2019direct} for the
second-order wave equation. In what follows, we let $\mu = 1 +
\varrho$ with $\varrho > 0$ being a small value and we revisit the
main existing results for fixing $\varrho$.

\begin{theorem}[first-order problem]\label{pro1st}
{\em For the system of ODEs $\bm{u}' = A\bm{u} + \bm{g}$ with initial
  value $\bm{u}(0) = \bm{u}_0$ and $t \in [0, T]$, suppose $\sigma(A)
  \subset \mathbb{R}^-$ with $|\lambda(A)| \leq
  \lambda_{\max}$\footnote{Here and hereafter, $\lambda(\cdot)$ and
    $\sigma(\cdot)$ denote an arbitrary eigenvalue and the spectrum of
    the involved matrix.}. Let $\bm{u}_{N_t}(\varrho)$ and
  $\bm{u}_{N_t}(0)$ be the numerical solutions at $t = T$ obtained by
  using Backward Euler with geometric step sizes and the
  uniform step size, respectively. Let
  $\{\tilde{\bm{u}}_{n}(\varrho)\}$ be the numerical solution computed
  by the diagonalization method \eqref{3steps}. Then, it holds that
\begin{equation}\label{error1st}
\begin{split}
& \|\bm{u}_{N_t}(\varrho) - \bm{u}_{N_t}(0)\| \lesssim C(\lambda_{*}T, N_t)\varrho^2,\\
& \|\tilde{\bm{u}}_{n}(\varrho) - \bm{u}_{n}(\varrho)\| \lesssim \epsilon \frac{N_t^2(2N_t+1)(N_t+\lambda_{\max}T)}{\phi(N_t)}\varrho^{-(N_t-1)},
\end{split}
\end{equation}
where $C(x, N_t) := \frac{N_t(N_t^2-1)}{24}r(x/N_t, N_t)$ with
$r(\tilde{x},N_t) :=
\left(\frac{\tilde{x}}{1+\tilde{x}}\right)^2(1+\tilde{x})^{-N_t}$. Here
$\epsilon$ is the machine precision\footnote[1]{In the ISO C Standard,
  $\epsilon = 1.19 \times 10^{-7}$ for single precision and $\epsilon
  = 2.22 \times 10^{-16}$ for double precision.} and
$$
\phi(N_t) := 
\begin{cases}
\frac{N_t}{2}! \left(\frac{N_t}{2} - 1\right)!, & \text{if } N_t \text{ is even},\\
\left(\frac{N_t - 1}{2}!\right)^2, & \text{if } N_t \text{ is odd}. 
\end{cases}
$$
The quantity $\lambda_* :=
\frac{N_t\tilde{x}_*}{T}$ with $\tilde{x}_*$ being the maximizer of
the function $r(\tilde{x},N_t)$ for $\tilde{x} \in [0, \infty)$. The
  best choice of $\varrho$, denoted by $\varrho_{\rm opt}$, is the
  quantity balancing the two error bounds in \eqref{error1st}, i.e.,
\begin{equation}\label{varrho_opt}
\varrho_{\rm opt} = \left(\epsilon \frac{N_t^2(2N_t+1)(N_t+\lambda_{\max}T)}{\phi(N_t)C(\lambda_{*}T, N_t)}\right)^{\frac{1}{N_t+1}}. 
\end{equation}
}
\end{theorem}   

\begin{proof}
Let $\lambda \in \sigma(A)$ be an arbitrary eigenvalue of $A$ and
consider the Dahlquist test equation $y' = \lambda y$. Then, the
first estimate in \eqref{error1st} follows from the analysis in
\cite[Theorem 2]{GHR16}. For this test equation, the error due to
diagonalization follows from the analysis in \cite[Theorem
  6]{GHR16} and the bound of the error reaches its
maximum when $|\lambda| = \lambda_{\max}$.
\end{proof}

The first estimate in \eqref{error1st} presents the truncation error
between the use of a geometric time mesh and a uniform time
mesh. From this, we can estimate the truncation error between ${\bm
  u}_{N_t}(\varrho)$ and the exact solution ${\bm u}(T)$ as $ \|{\bm
  u}_{N_t}(\varrho)-{\bm u}(T)\|\leq \|{\bm u}_{N_t}(\varrho)-{\bm
  u}_{N_t}(0)\|+\|{\bm u}_{N_t}(0)-{\bm u}(T)\|, $ where the estimate
of the last term is well understood and does not play a dominant
role. The second estimate in \eqref{error1st} is the roundoff error
due to diagonalization of the time stepping matrix $B$. The analysis of
this error is closely related to the condition number of the
eigenvector matrix $V$. With the geometric step sizes in \eqref{dtn}, $V$
and $V^{-1}$ are lower triangular Toeplitz matrices, see \cite{GHR16},
\begin{subequations}
\begin{equation}\label{VV1a}
\begin{split}
&V=\mathbb{T}(p_1,p_2,\dots, p_{N_t-1}),~p_n:=\frac{1}{\prod_{j=1}^n(1-\varrho^j)}, \\
&V^{-1}=\mathbb{T}(q_1,q_2,\dots, q_{N_t-1}), ~q_n=(-1)^n\varrho^{\frac{n(n-1)}{2}}p_n,
\end{split}
\end{equation}
where $\mathbb{T}$ is the lower triangular Toeplitz operator
\begin{equation}\label{VV1b}
\mathbb{T}(a_1, a_2,\dots, a_{N_t})=
\begin{bmatrix}
1 & & &\\
a_1 &1 & &\\
\vdots &\ddots &\ddots &\\
a_{N_t-1} &\dots &a_1 &1
\end{bmatrix}.
\end{equation}
\end{subequations}
The closed form formula for $V$ and $V^{-1}$ in \eqref{VV1a} is useful
to estimate Cond$(V)$, and then the roundoff error in
\eqref{error1st}. However, in practice, we do not use these formulas
for $V$ and $V^{-1}$ in \eqref{3steps}. Instead, we use the command
\texttt{eig} in MATLAB to get $V$ and $V^{-1}$, since it automatically
optimizes the condition number by scaling the eigenvectors.

We now study the error of the ParaDiag I method for two PDEs with
homogeneous Dirichlet boundary conditions and the initial value
$u(x,0)=\sin(2\pi x)$ for $x\in(0, 1)$, the heat equation
\eqref{heatequation} and the advection-diffusion equation \eqref{ADE}
with $\nu=10^{-2}$. Both PDEs are discretized by centered finite
differences with mesh size $\Delta x=\frac{1}{50}$. With $T=0.2$ and
five values of $N_t$, we show in Figure \ref{ParaDiag1Fig1} the error
of ParaDiag I for $\varrho\in[10^{-2}, 1]$. The error is measured as
$\max_{n=1,2,\dots, N_t}\|\tilde{\bm u}_n(\varrho)-{\bm
  u}(t_n)\|_\infty$,
where ${\bm u}(t_n)$ is the reference solution
computed by the exponential integrator, i.e., ${\bm
  u}(t_n)=e^{-At_n}{\bm u}_0$ and $\tilde{\bm u}_n(\varrho)$ is the
solution at $t=t_n$ by ParaDiag I. Clearly, there exists an optimal
choice of $\varrho$ that minimizes the error.   We also show by a star
for each $N_t$ the theoretically estimated $\varrho_{\rm opt}$
from \eqref{varrho_opt}. For the advection-diffusion equation, this
$\varrho_{\rm opt}$ predicts the optimal choice very well, and also
quite well for the heat equation, except when $N_t$ is small, even
  though the theoretical estimate was just obtained by balancing
  roundoff and truncation error estimates asymptotically.
\begin{figure}
\centering
\includegraphics[width=2.3in,height=1.85in,angle=0]{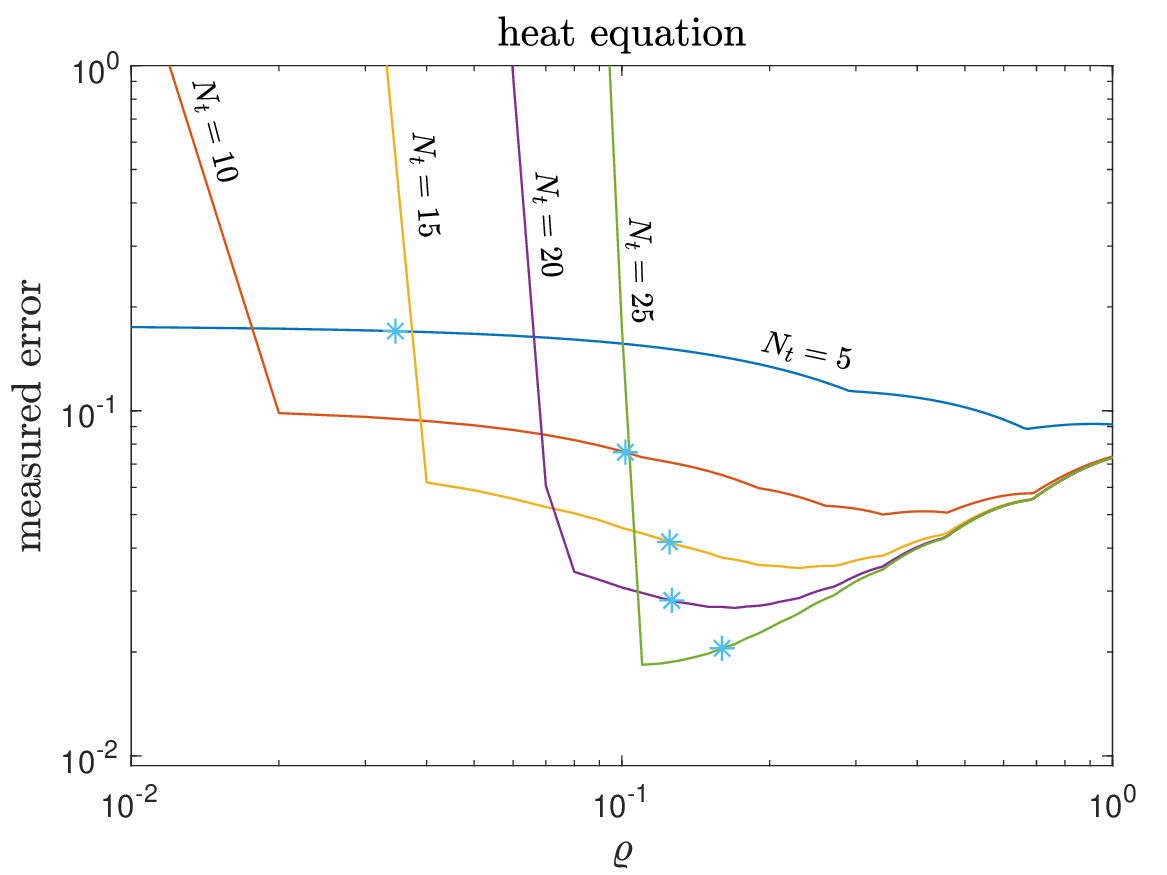}
\includegraphics[width=2.3in,height=1.85in,angle=0]{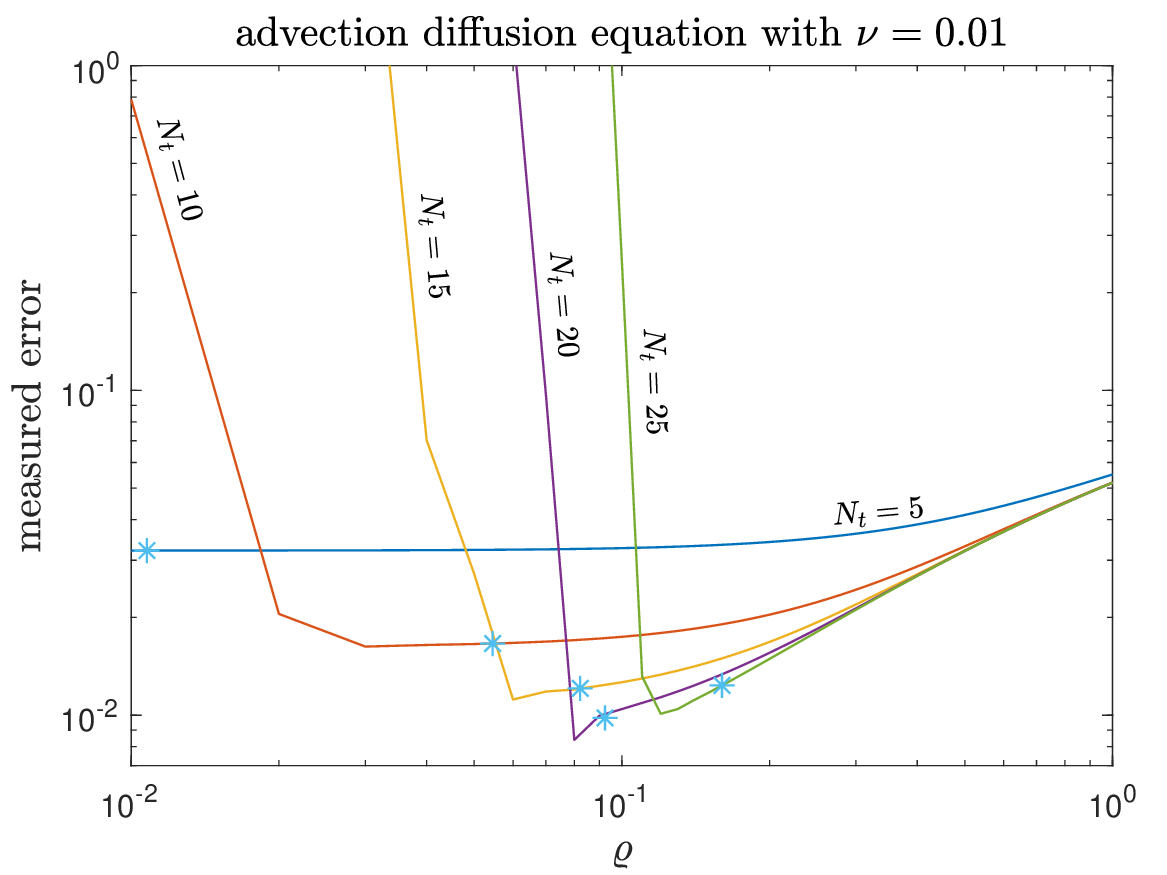}
\caption{Measured error of ParaDiag I for five values of $N_t$ using
  geometric step sizes like in \eqref{dtn} with $\mu=1+\varrho$ and
  $\varrho\in[10^{-2}, 1]$. The star denotes the theoretically
  estimated $\varrho_{\rm opt}$ from \eqref{varrho_opt}.}
\label{ParaDiag1Fig1}
\end{figure}
Let $\varrho_{num}$ be the minimizer determined numerically as shown in
Figure \ref{ParaDiag1Fig1}. We show in Figure \ref{ParaDiag1Fig2} the
error for  Backward  Euler with uniform step size $\Delta
t=\frac{T}{N_t}$ and the ParaDiag I method using variable step sizes
$\Delta t_n$, i.e., \eqref{dtn} with $\mu=1+\varrho_{num}$.  Here $T=0.5$
and $N_t=2^{4:10}$.  As $N_t$ grows, the error for ParaDiag I
decreases first and then increases rapidly when $N_t$ exceeds
some threshold less than 100.
\begin{figure} 
  \centering
 \includegraphics[width=2.3in,height=1.85in,angle=0]{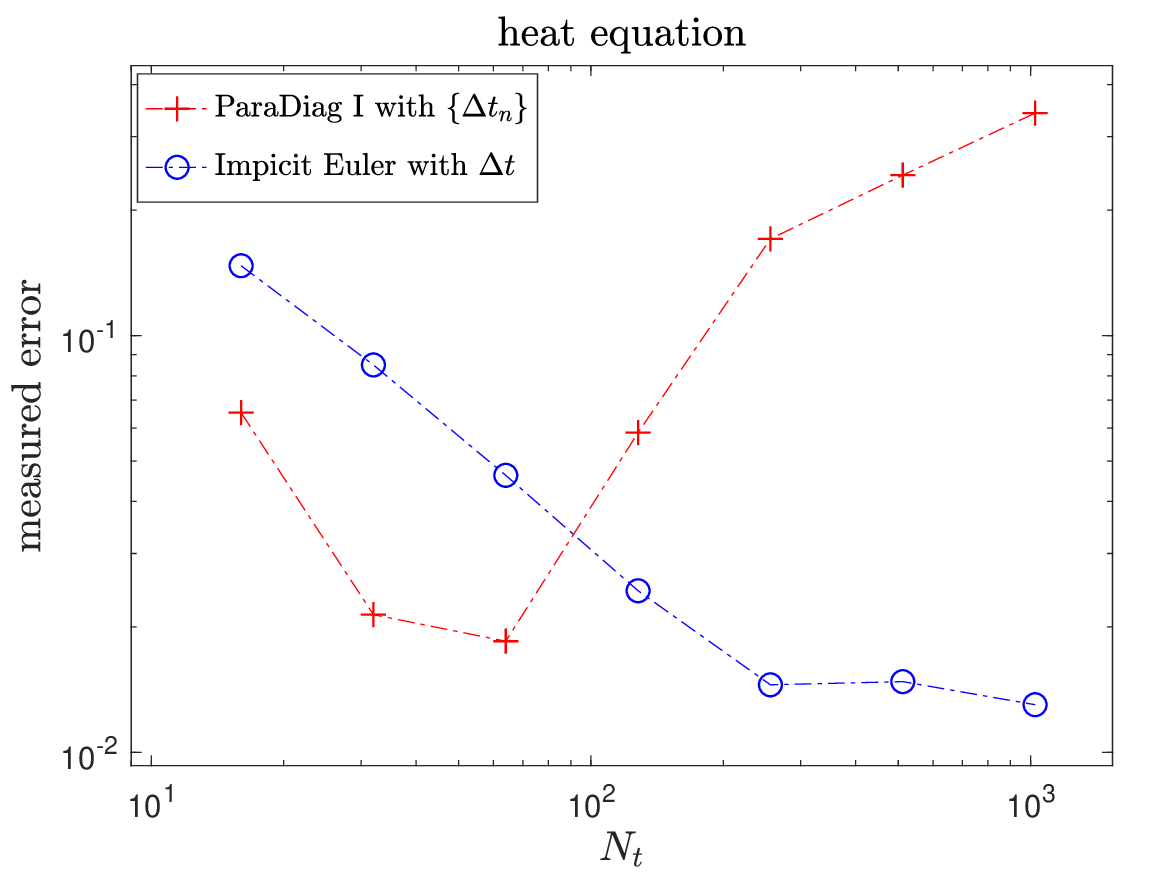}
 \includegraphics[width=2.3in,height=1.85in,angle=0]{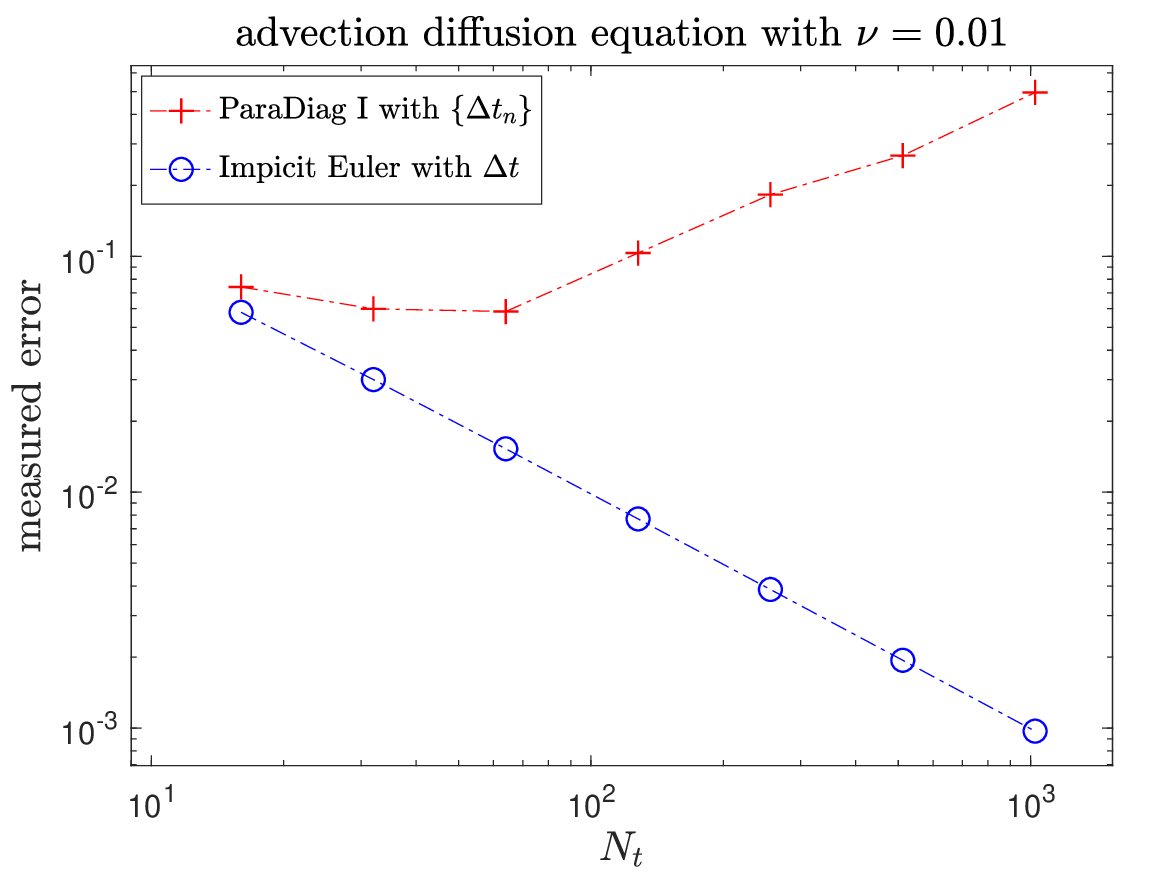} 
  \caption{The error of ParaDiag I increases rapidly for $N_t$
    exceeding a threshold when using  $\mu=1+\varrho_{num}$.}
  \label{ParaDiag1Fig2}
\end{figure}

We next consider a second-order problem, e.g., the wave equation
\eqref{WaveEquation1d} after space discretization,
\begin{equation}\label{2ndODE}
{\bm u}''(t)=A{\bm u}(t) \text{~for~}t\in(0, T],~{\bm u}(0)={\bm u}_0,~{\bm u}'(0)=\tilde{\bm u}_0, 
\end{equation} 
where $A\in\mathbb{R}^{N_x\times N_x}$. For the wave equation
\eqref{WaveEquation1d}, $A$ is a discretized Laplacian. To use
ParaDiag I, we transform this equation into a first order system,
\begin{equation}\label{two2ndODE}
{\bm w}'(t)={\bm A}{\bm w}(t) \text{~for~}t\in(0, T], {\bm w}(0)=({\bm u}_0^{\top}, \tilde{\bm u}_0^{\top})^{\top}, 
\end{equation} 
where ${\bm w}(t)=({\bm u}^{\top}(t), ({\bm u}'(t))^{\top})^{\top}$ and 
$$
{\bm A}=
\begin{bmatrix}
& I_x\\
A &
\end{bmatrix}. 
$$
To avoid numerical dispersion, we use the Trapezoidal Rule as the
time-integrator,
\begin{equation}\label{Wave_TR}
\frac{{\bm w}_{n}-{\bm w}_{n-1}}{\Delta t_n}=\frac{\bm A}{2}({{\bm w}_{n}+{\bm w}_{n-1}}),~n=1,2,\dots, N_t, 
\end{equation}
which is energy preserving in the sense that $\|{\bm w}_n\|_2=\|{\bm
  w}_0\|_2$. The step sizes $\{\Delta t_n\}$ are again geometric as in
\eqref{dtn} with some parameter $\mu=1+\varrho>1$.  Similar to
\eqref{AAA1a}, we can represent \eqref{Wave_TR} as an all-at-once
system,
\begin{subequations}
\begin{equation}\label{AAA2a}
\CK{\bm W}={\bm b},~\CK=B\otimes I_x-\tilde{B}\otimes A, 
\end{equation}
with some suitable vector ${\bm b}$, where $B$ is the matrix in
\eqref{AAA1b} and
\begin{equation}\label{AAA2b}
\tilde{B}=\frac{1}{2}
\begin{bmatrix}
1 & & &\\
1 &1 & &\\
&\ddots &\ddots &\\
& &1 &1
\end{bmatrix}. 
\end{equation}
\end{subequations}
To apply the diagonalization technique, we rewrite \eqref{AAA2a} as 
\begin{equation}\label{AAA2a1}
\CK{\bm W}=\tilde{\bm b},~\CK=\tilde{B}^{-1}B\otimes I_x- I_t\otimes A, ~\tilde{\bm b}=(\tilde{B}^{-1}\otimes I_x){\bm b}. 
\end{equation}
The matrix $\tilde{B}^{-1}B$ can be diagonalized, see
\cite{gander2019direct},
\begin{subequations}
\begin{equation}\label{VV12a}
\tilde{B}^{-1}B=V{\rm diag}\left(\frac{2}{\Delta t_1}, \dots,\frac{2}{\Delta t_{N_t}}\right)V^{-1}, 
\end{equation}
where $V$ and $V^{-1}$ are given by 
\begin{equation}\label{VV12b}
\begin{split}
&V= \mathbb{T}(p_1,p_2,\dots, p_{N_t-1}),~p_n:=\prod_{j=1}^n\frac{1+\mu^j}{1-\mu^j},\\
&V^{-1}= \mathbb{T}(q_1,q_2,\dots, q_{N_t-1}), ~q_n:=\mu^{-n}\prod_{j=1}^n\frac{1+\mu^{-j+2}}{1-\mu^{-j}}. 
\end{split}
\end{equation}
\end{subequations}
Then, the all-at-once system \eqref{AAA2a1} can be solved via ParaDiag
I (cf. \eqref{3steps}) as well.

Similar to the first-order equation studied above, by balancing the
truncation error (between the geometric mesh and the uniform mesh) and
the roundoff error, the best mesh parameter $\mu=1+\varrho_{\rm opt}$
is obtained as follows.

\begin{theorem}[second-order problem]\label{pro2nd}
{\em For the 2nd-order system of ODEs ${\bm u}''=A{\bm u}$ with
  $\lambda(A)\leq0$, let ${\bm u}_{N_t}(\varrho)$ and ${\bm
    u}_{N_t}(0)$ denote the numerical solutions at $t=T$ obtained by
  using the Trapezoidal Rule with geometric time step sizes and the
  uniform time step size. Let $\{\tilde{\bm u}_{n}(\varrho)\}$ denote the
  numerical solution computed by the diagonalization method
  \eqref{3steps}. Then, it holds that
\begin{equation}\label{error2st}
\begin{split}
&\|{\bm u}_{N_t}(\varrho)-{\bm u}_{N_t}(0)\|\lesssim \frac{N_t(N_t^2-1)}{15}\varrho^2,\\
&\|\tilde{\bm u}_{n}(\varrho)-{\bm u}_{n}(\varrho)\|\lesssim \epsilon\frac{2^{2N_t-\frac{1}{2}}N_t}{(N_t-1)!} \varrho^{-(N_t-1)}.
\end{split}
\end{equation} 
The best choice of $\varrho$, denoted by $\varrho_{\rm opt}$, is the quantity balancing the two error bounds in \eqref{error2st}, i.e.,
\begin{equation}\label{varrho_opt2}
\varrho_{\rm opt}=\left(\epsilon \frac{15\times 2^{2N_t-\frac{1}{2}}}{(N_t^2-1)(N_t-1)!} \right)^{\frac{1}{N_t+1}}.
\end{equation}
}
\end{theorem}
\begin{proof}
Let $\lambda>0$ be an arbitrary eigenvalue of $-A$ and consider the
scalar equation $u''+\lambda u=0$. Then, for small $\varrho$,
according to \cite[Theorem 2.1]{gander2019direct}, the truncation
error between using the geometric mesh and the uniform mesh is
of order $\CO(\frac{N_t(N_t^2-1)}{6}r_1(\frac{\lambda
  T}{2N_t})\varrho^2)$ with $r_1(s)=\frac{s^3}{(1+s^2)^2}$. For
$s\geq0$, it holds that $r_1(s)\leq\frac{2}{5}$ and this gives the
first estimate in \eqref{error2st}. The second estimate follows from
the roundoff error analysis in \cite[Theorem 2.11]{gander2019direct},
which is $\frac{2^{2N_t-\frac{1}{2}}N_t}{(N_t-1)!}r_2(\frac{\lambda
  T}{2N_t})$ with $r_2(s)=\frac{1}{1+s^2}$. For $s\geq0$ it holds that
$r_2(s)\leq 1$.
\end{proof}
We show in Figure \ref{ParaDiag1Fig3}
\begin{figure}
\centering
\includegraphics[width=2.3in,height=1.85in,angle=0]{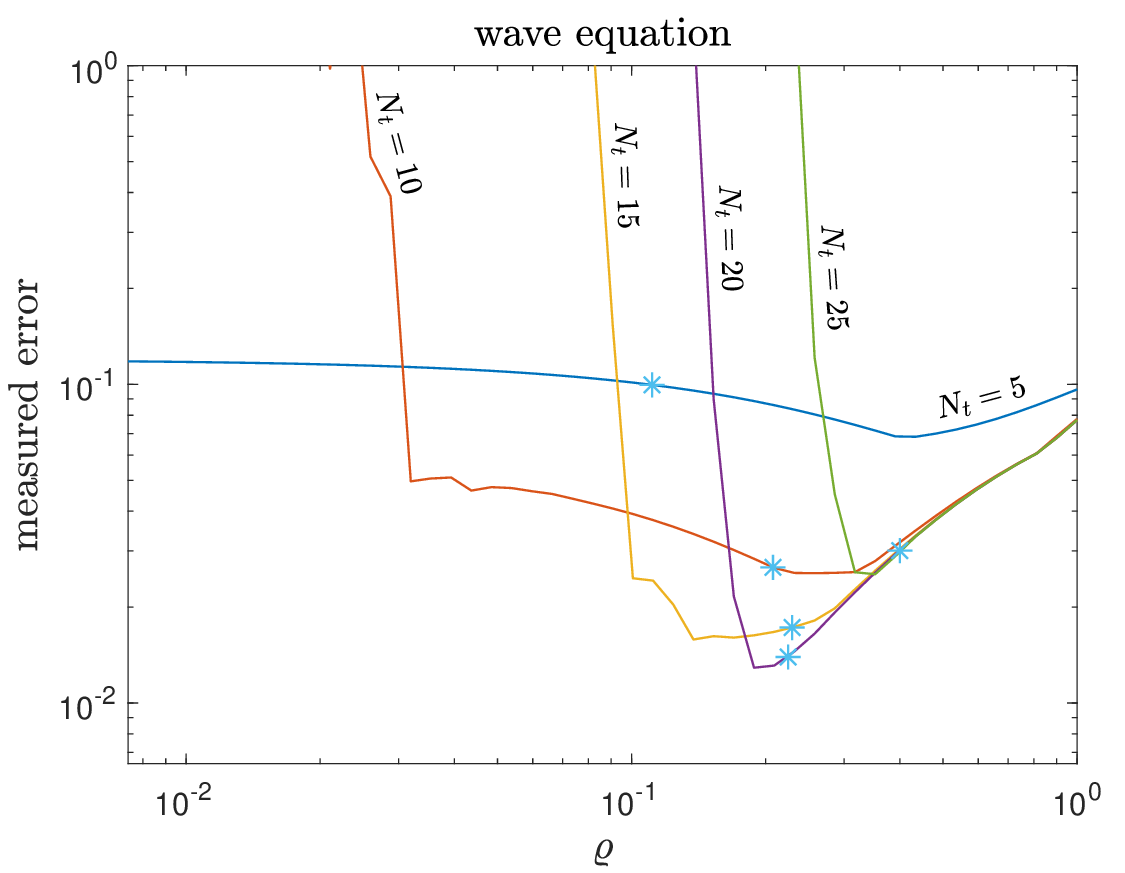}
\includegraphics[width=2.3in,height=1.85in,angle=0]{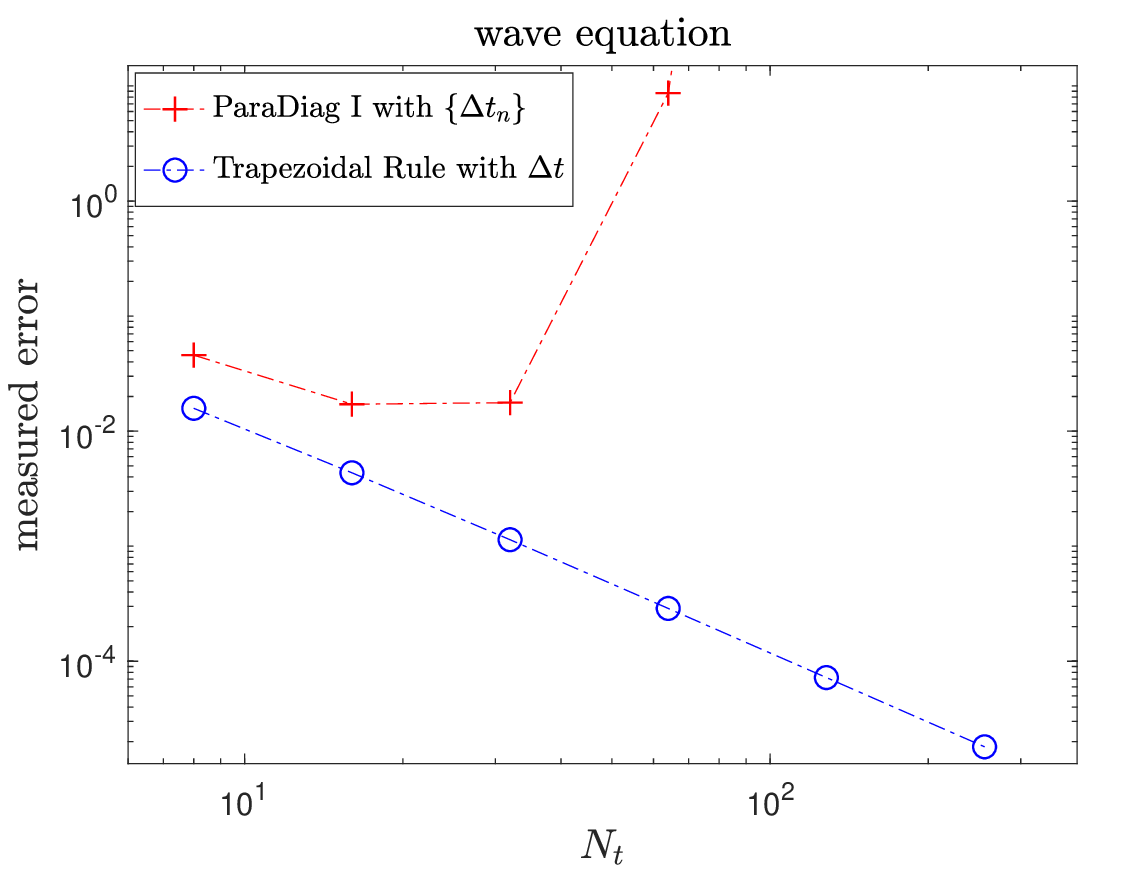}
\caption{Measured error of ParaDiag I and the Trapezoidal Rule using
  uniform step sizes for the wave equation \eqref{WaveEquation1d}.
  The star in the left panel is the parameter $\varrho_{\rm opt}$
  obtained in theory (cf. \eqref{varrho_opt2}).}
\label{ParaDiag1Fig3}
\end{figure}
(left) the error of ParaDiag I for 5 values of $N_t$ when applied to
the wave equation \eqref{WaveEquation1d} with homogeneous Dirichlet
boundary conditions. For each $N_t$, the parameter $\varrho$ varies
from $10^{-2}$ to 1. Here, $\Delta x=\frac{1}{20}$ and
$T=0.2$. Similar to the first-order parabolic problems (cf. Figure
\ref{ParaDiag1Fig1}), there exists an optimal choice of $\varrho$
which minimizes the error, and the theoretical estimate $\varrho_{\rm
  opt}$ is close to this choice.  With $\varrho_{num}$ denoting the
best working parameter determined numerically for each $N_t$, we show
in Figure \ref{ParaDiag1Fig3} (right) the error for the geometric time
mesh and the uniform time mesh. For the former, the error grows
rapidly when $N_t>32$.

For ParaDiag I, the increase in the error shown in Figures
\ref{ParaDiag1Fig2} and \ref{ParaDiag1Fig3} (right) can be
attributed to the poor condition number of the eigenvector matrix $V$
of the time step matrices $B$ and $\tilde{B}^{-1}B$. In Table
\ref{TabCondV},
\begin{table}
\caption{Cond$(V)$ for $B$ ( Backward  Euler) and $\tilde{B}^{-1}B$
  (Trapezoidal Rule)}\label{TabCondV} \centering
\begin{tabular}{rcccccc}
\hline
\small{$N_t$} & 5 & 10 & 20 & 30 & 60 & 100 \\
\hline
\small{$B$} & \small{$1.7\times10^3$} & \small{$8.4\times10^4$} & \small{$1.3\times10^6$} & \small{$2.8\times10^6$} & \small{$4.4\times10^6$} & \small{$4.8\times10^6$}\ \\
\small{$\tilde{B}^{-1}B$} & \small{$4.7\times10^3$} & \small{$7.9\times10^5$} & \small{$6.9\times10^7$} & \small{$3.8\times10^8$} & \small{$1.9\times10^9$} & \small{$4.1\times10^9$} \\
\hline
\end{tabular}
\end{table}
we show this condition number for several values of $N_t$, where
$\mu=1+\varrho_{num}$, and $\varrho_{num}$ is determined numerically
for each $N_t$ by minimizing the error shown in Figures
\ref{ParaDiag1Fig1} and \ref{ParaDiag1Fig3} on the left.  As $N_t$
increases, the condition number rises rapidly, which confirms our
analysis of roundoff error very well (cf. \eqref{error1st} and
\eqref{error2st}), but then reaches a plateau when using the
  numerically optimized parameter, and not the theoretically
  determined one, an observation that merits further study.

We now introduce another ParaDiag I method from \cite{LWWZ22}, which
addresses the limitation associated with $N_t$. Instead of using
 Backward  Euler or the Trapezoidal Rule with geometric time step
sizes, we use the same time step size $\Delta t$, but different
methods, an idea which goes back to the boundary value technique
\cite{AV85}. In this approach, we take for example centered finite
differences for the first $(N_t-1)$ time steps, followed by a final
 Backward  Euler step. For the system of ODEs \eqref{linearODE},
this gives
\begin{equation}\label{hybridDis}
\begin{cases}
\frac{{\bm u}_{n+1}-{\bm u}_{n-1}}{2\Delta t}=A{\bm u}_n+{\bm g}_n,   n=1,2,\dots, N_t-1, \\
\frac{{\bm u}_{N_t}-{\bm u}_{N_t-1}}{\Delta t}=A{\bm u}_{N_t}+{\bm g}_{N_t}.
\end{cases}
\end{equation}
It is important to note that this implicit boundary value technique
time discretization has quite different stability properties from
traditional time discretizations; see e.g. \cite[Section
  5.2]{gander50years}. For the boundary value technique discretization
\eqref{hybridDis}, the all-at-once system is
\begin{subequations}
\begin{equation}\label{hybridAAAa}
\CK{\bm U}={\bm b}, \quad \CK=B\otimes I_x- I_t\otimes A,
\end{equation}
where
\begin{equation}\label{hybridAAAb}
B=\frac{1}{\Delta t}\begin{bmatrix}
0 & \frac{1}{2} & & & \\
-\frac{1}{2} & 0 & \frac{1}{2} & & \\
& \ddots & \ddots & \ddots & \\
& & -\frac{1}{2} & 0 & \frac{1}{2} \\
& & & -1 & 1 \\
\end{bmatrix}, \quad {\bm b}=\begin{bmatrix} \frac{{\bm u}_0}{2\Delta t}+{\bm g}_1 \\ {\bm g}_2 \\ \vdots \\ {\bm g}_{N_t} \end{bmatrix}.
\end{equation}
\end{subequations}
For the all-at-once system given by equation \eqref{hybridAAAa}, only
the initial value ${\bm u}_0$ is required, and all time steps are
solved simultaneously in one-shot.

The authors in \cite{AV85} explored boundary value techniques to
circumvent the well-known Dahlquist barriers between convergence and
stability, which arise when using \eqref{hybridDis} in a time-stepping
mode. In a general nonlinear case, they proved that the numerical
solutions obtained simultaneously are of uniform second-order accuracy
\cite[Theorem 4]{AV85}, even though the last step is a first-order
scheme. Even earlier, in \cite{F54} and also \cite{FM57}, such
boundary value technique discretizations appeared already: instead of
Backward  Euler, the authors used the BDF2 method for
the last step in \eqref{hybridDis},
$$
  \frac{3{\bm u}_{N_t}-4{\bm u}_{N_t-1}+{\bm u}_{N_t-2}}{2\Delta
    t}=A{\bm u}_{N_t}+{\bm g}_{N_t}.
$$
The method \eqref{hybridDis} is a prime example of the so-called
\textit{boundary value methods} (BVMs) developed a bit later, and
the all-at-once system \eqref{hybridAAAa} was carefully justified in
\cite{BMT93}, see also \cite{BT03}. In BVMs, the resulting all-at-once
system is typically solved iteratively by constructing effective
preconditioners.

A mathematical analysis of ParaDiag I based on BVM discretization like
\eqref{hybridDis} can be found in \cite{LWWZ22}:
\begin{theorem}\label{prohybrid}
  {\em The time stepping matrix $B$ given by \eqref{hybridAAAb} can be
    factored as $B=VDV^{-1}$ with {\rm
      Cond}$(V)=\CO(N_t^2)$\footnote[1]{Closed form formulas for $V$,
      $V^{-1}$, and $D$ are provided in \cite[Section 3]{LWWZ22}.}.}  
\end{theorem}
ParaDiag I with BVM discretization can also be applied to second-order
problems of the form ${\bm u}''=A{\bm u}$, with initial values ${\bm
  u}(0)={\bm u}_0$ and ${\bm u}'(0)=\tilde{\bm u}_0$. By setting ${\bm
  v}(t):={\bm u}'(t)$ and ${\bm w}(t):=({\bm u}^{\top}(t), {\bm
  v}^{\top}(t))^{\top}$, we can rewrite this equation as
$$
{\bm w}'(t) = 
{\bm A}{\bm w}(t), ~{\bm A}:=\begin{bmatrix}
&I_x\\
A &
\end{bmatrix},~
{\bm w}(0):=\begin{bmatrix}
{\bm u}_0\\
\tilde{\bm u}_0
\end{bmatrix}.
$$
Then, similar to \eqref{hybridDis}  the same time discretization scheme leads to
\begin{equation}\label{hybridDis2}
\begin{cases}
\frac{{\bm w}_{n+1}-{\bm w}_{n-1}}{2\Delta t}={\bm A}{\bm w}_n,~n=1,2,\dots, N_t-1, \\
\frac{{\bm w}_{N_t}-{\bm w}_{N_t-1}}{\Delta t}={\bm A}{\bm w}_{N_t}.
\end{cases}
\end{equation}
Rewriting the second order problem as a first order system doubles the
storage requirement for the space variables at each time point, which
is not desirable, especially if the second-order problem arises from a
semi-discretization of a PDE in high dimensions or with small mesh
sizes. To avoid this, one can write the all-at-once system for
\eqref{hybridDis2} using only ${\bm U}:=({\bm u}_1, {\bm u}_2,\dots,
      {\bm u}_{N_t})^{\top}$, which leads to
\begin{equation}\label{hybridAAA3}
(B^2\otimes I_x-I_t\otimes A){\bm U}={\bm b},
\end{equation}
where $B$ is the matrix defined in \eqref{hybridAAAb}, and ${\bm b}:=
\left(\frac{\tilde{\bm u}_0^{\top}}{2\Delta t}, -\frac{{\bm
    u}_0^{\top}}{4\Delta t^2}, 0, \dots, 0\right)^{\top}$. To see
this, we trace the steps back at the discrete level which led to the
first order system at the continuous level: from \eqref{hybridDis2} we
represent $\{{\bm u}_n\}$ and $\{{\bm v}_n\}$ separately as
\begin{equation*}
\begin{split}
&\begin{cases}
\frac{{\bm u}_{n+1}-{\bm u}_{n-1}}{2\Delta t}={\bm v}_n,~n=1,2,\dots, N_t-1, \\
\frac{{\bm u}_{N_t}-{\bm u}_{N_t-1}}{\Delta t}={\bm v}_{N_t},
\end{cases}\\
&\begin{cases}
\frac{{\bm v}_{n+1}-{\bm v}_{n-1}}{2\Delta t}=A{\bm u}_n,~n=1,2,\dots, N_t-1, \\
\frac{{\bm v}_{N_t}-{\bm v}_{N_t-1}}{\Delta t}=A{\bm u}_{N_t}.
\end{cases}
\end{split}
\end{equation*}
Hence, with the matrix $B$ in \eqref{hybridAAAb}, we have
$$
(B\otimes I_x){\bm U}-{\bm V}={\bm b}_1,\quad (B\otimes I_x){\bm V}-A{\bm U}={\bm b}_2,
$$
where ${\bm V}:=({\bm v}_1^{\top}, \dots, {\bm
  v}_{N_t}^{\top})^{\top}$, ${\bm b}_1:=(\frac{{\bm
    u}_0^{\top}}{2\Delta t}, 0, \dots, 0)^{\top}$ and ${\bm
  b}_2:=(\frac{\tilde{\bm u}_0^{\top}}{2\Delta t}, 0, \dots,
0)^{\top}$. From the first equation, we have ${\bm V}=(B\otimes
I_x){\bm U}-{\bm b}_1$. Substituting this into the second equation
gives $(B\otimes I_x)^2{\bm U}-A{\bm U}={\bm b}_2+(B\otimes I_x){\bm
  b}_1$. A routine calculation then yields ${\bm b}_2+(B\otimes I_x){\bm
  b}_1={\bm b}$, and combining this with $(B\otimes I_x)^2=B^2\otimes
I_x$ gives the all-at-once system \eqref{hybridAAA3}.

We now compare the ParaDiag I method with geometric time stepping to
the one with BVM discretization applied to the wave
equation \eqref{WaveEquation1d} with homogeneous Dirichlet boundary
conditions and $T=0.5$, discretized in space using centered finite
differences with $\Delta x=\frac{1}{40}$. The errors for $N_t=2^{2}$
to $2^{8}$ are shown in Figure \ref{ParaDiag1Fig4} on the left.
\begin{figure}
\centering
\includegraphics[width=2.3in,height=1.85in,angle=0]{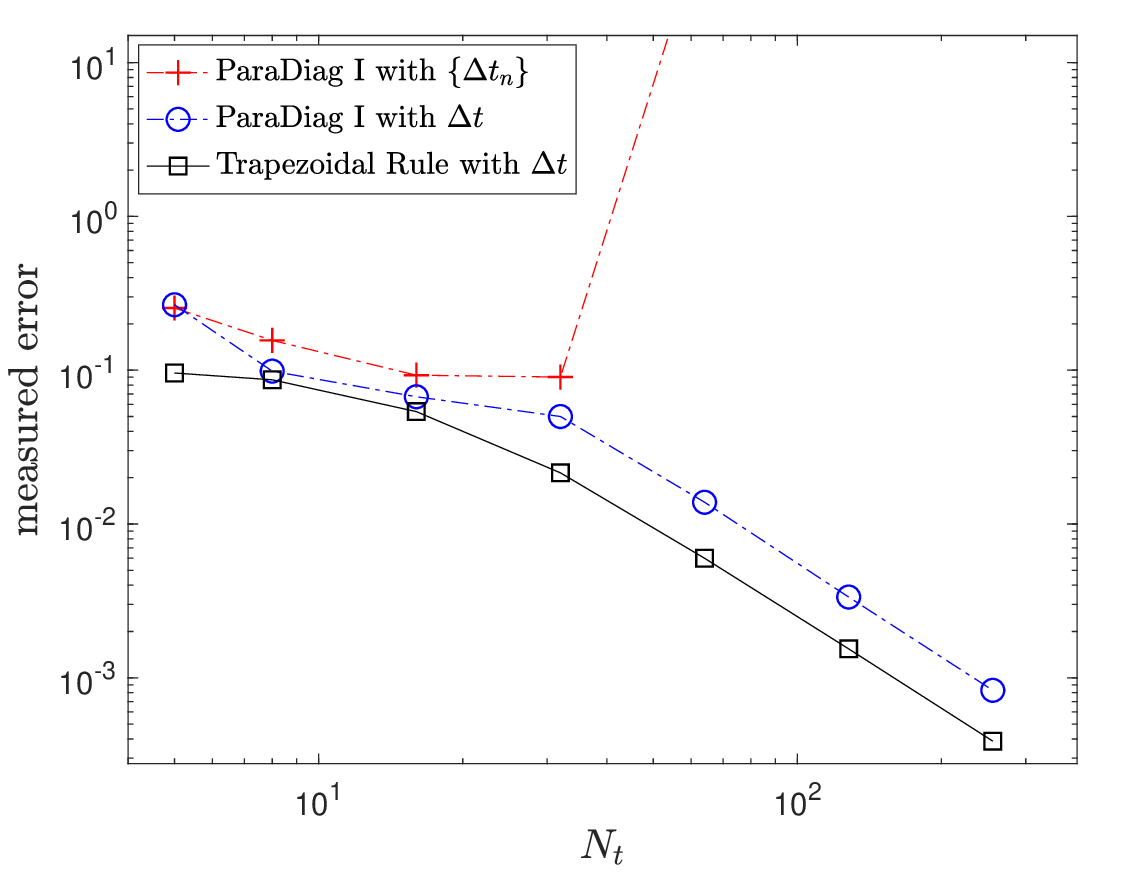}
\includegraphics[width=2.3in,height=1.85in,angle=0]{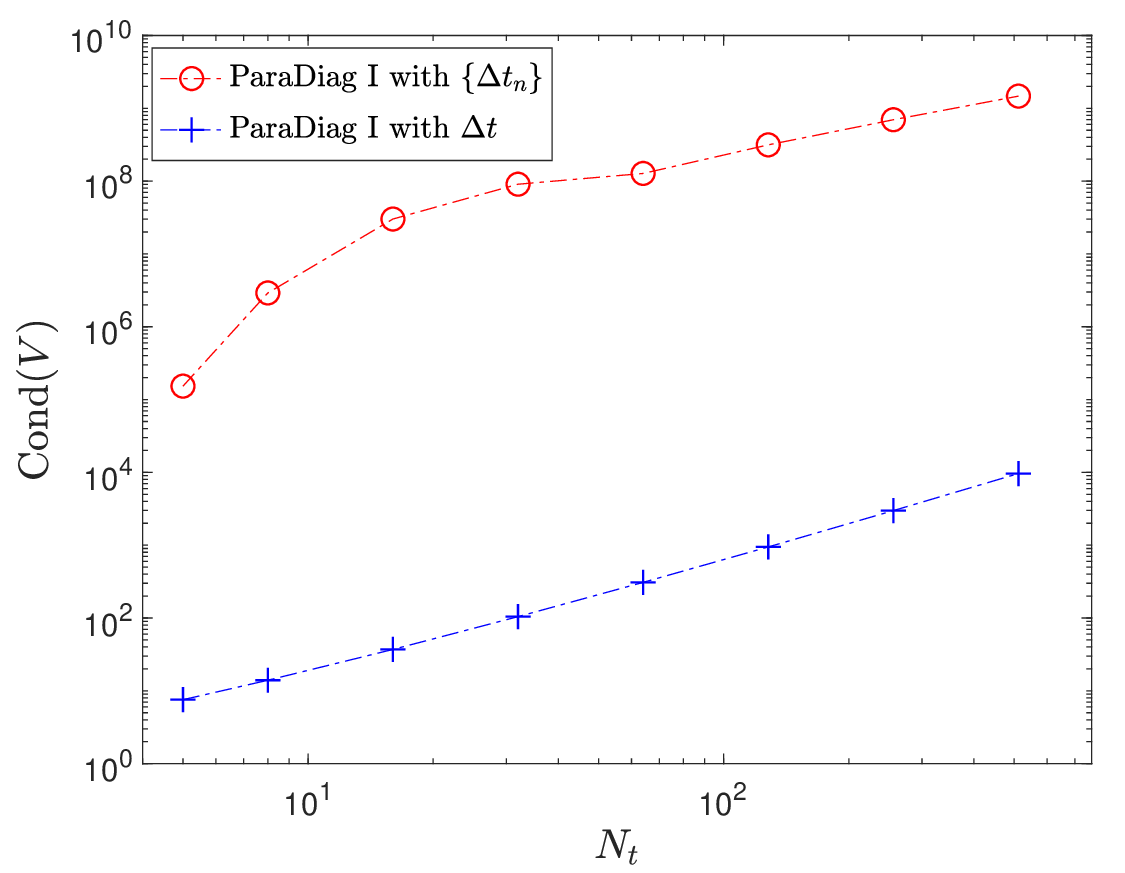}
\caption{Left: measured error for the wave equation
  \eqref{WaveEquation1d} using ParaDiag I with the Trapezoidal Rule
  and geometric step sizes $\{\Delta t_n\}$ and the BVM discretization
  \eqref{hybridDis} with uniform time step $\Delta t=T/N_t$. Right:
  condition number of the eigenvector matrix of the time-stepping
  matrix.}
\label{ParaDiag1Fig4}
\end{figure}
We see that the error of ParaDiag I with geometric time stepping shows
the typical deterioration due to roundoff around $N_t=32$, whereas
ParaDiag I with BVM discretization is of order $\CO(\Delta t^2)$
without any deterioration, like the Trapezoidal Rule. In Figure
\ref{ParaDiag1Fig4} on the right, we show the corresponding condition
number of the eigenvector matrix of the time-step matrix, which shows
that ParaDiag I with BVM discretization has a much lower condition
number and explains why we do not observe any deterioration.

To conclude this section, we show how to apply ParaDiag I to
nonlinear problems. We consider the first-order nonlinear system of
ODEs \eqref{nonlinearODE}, i.e., ${\bm u}'(t)=f({\bm u}(t), t)$ with
initial value ${\bm u}(0)={\bm u}_0$, nonlinear second-order problems
can be treated similarly. As in the linear case, the all-at-once
system for this nonlinear problem is
\begin{equation}\label{nonlinearAAA}
(B\otimes I_x){\bm U}-F({\bm U})={\bm b},
\end{equation}
where $F({\bm U}):=(f^{\top}({\bm u}_1, t_1), f^{\top}({\bm u}_2,
t_2),\dots, f^{\top}({\bm u}_{N_t}, t_{N_t}))^{\top}$, and ${\bm b}$
is a suitable right hand side vector containing the initial condition
and possible terms not depending on the solution. The time step matrix
$B$ is either the one given in \eqref{AAA1b} using variable time
steps, or the one given in \eqref{hybridAAAb} corresponding to the BVM
discretization \eqref{hybridDis}.

{Since the problem \eqref{nonlinearAAA} is non-linear, we apply
  Newton's method,
\begin{equation*}
(B\otimes I_x-\nabla F({\bm U}^k))({\bm U}^{k+1}-{\bm U}^k)={\bm b}-((B\otimes I_x){\bm U}^k- F({\bm U}^k)),
\end{equation*}
which can be simplified to 
\begin{subequations}
\begin{equation}\label{NewtonIt_a}
(B\otimes I_x- \nabla F({\bm U}^k)){\bm U}^{k+1}={\bm b}-\left(\nabla F({\bm U}^k){\bm U}^k-F({\bm U}^k)\right),
\end{equation}
where $k\geq0$ is the Newton iteration index, and
\begin{equation}\label{NewtonIt_b}
\nabla F({\bm U}^k)={\rm blkdiag}(\nabla f({\bm u}_1^k, t_1), \dots, \nabla f({\bm u}_{N_t}^k, t_{N_t})),
\end{equation}
\end{subequations}
with $\nabla f({\bm u}_n^k, t_n)$ being the Jacobian matrix of $f({\bm
  u}, t_n)$ with respect to the first variable ${\bm u}$. To make the
\dg technique still applicable, we have to replace (or approximate)
all the blocks $\{\nabla f({\bm u}_n^k, t_n)\}$ by a single matrix
$A_k$. Inspired by the idea in \cite{GH17}, we consider an
averaged Jacobian matrix,
\begin{equation}\label{Ak}
  A_k:=\frac{1}{N_t}{\sum}_{n=1}^{N_t}\nabla f({\bm u}_n^k, t_n)\quad
  \text{or}\quad A_k:=\nabla f\left (\frac{1}{N_t}{\sum}_{n=1}^{N_t} {\bm u}_n^k, \frac{T}{N_t} \right ).
\end{equation}
Then, we get a {simple Kronecker-product} approximation of $\nabla F({\bm U}^k)$ as
$$
\nabla F({\bm U}^k)\approx I_t\otimes {A_k}.
$$
By substituting this into \eqref{NewtonIt_a}, we obtain the quasi
Newton method
\begin{equation}\label{SNI}
(B\otimes I_x-I_t\otimes {A_k}){\bm U}^{k+1}={\bm b}-\left((I_t\otimes {A_k}){\bm U}^k-F({\bm U}^k)\right).
\end{equation}
Convergence of such quasi Newton methods is well-understood; see,
e.g., \cite[Theorem 2.5]{D04} and \cite{OR00}.

In this quasi Newton method \eqref{SNI}, the Jacobian system can also
be solved parallel in time: with the \dg $B=VDV^{-1}$, we solve ${\bm
  U}^{k+1}$ in \eqref{SNI} again in three steps,
\begin{equation}\label{nonlinear3step}
\begin{cases}
{\bm U}^a=(V^{-1}\otimes I_x){\bm r}^k, &\text{(step-a)}\\
(\lambda_n I_x-A_k){\bm u}^b_n={\bm u}_n^a, ~n=1,2,\dots, N_t, &\text{(step-b)}\\
{\bm U}^{k+1}=(V\otimes I_x){\bm U}^b, &\text{(step-c)}
\end{cases}
\end{equation}
where ${\bm r}^k:={\bm b}- ((I_t\otimes {A_k}){\bm U}^k-F({\bm
  U}^k)$. In the linear case, i.e., $f({\bm u}, t)=A{\bm u}+{\bm
  g}(t)$, we have $A_k=A$ and ${\bm r}^k={\bm b}$ and
\eqref{nonlinear3step} reduces to \eqref{3steps}.

The convergence rate of the quasi Newton method depends on the
accuracy of the approximation of the average matrix $A_k$ to all $N_t$
Jacobian blocks $\nabla f({\bm u}_n^k, t_n)$. One can imagine that if
$\nabla f(u_n^k, t_n)$ changes dramatically for $n=1,2,\dots, N_t$,
any single matrix cannot be a good approximation. In this case, we can
divide the time interval $[0, T]$ into multiple smaller windows and
apply ParaDiag I to these time windows sequentially. We tested this
approach by combining ParaDiag I with the BVM discretization
\eqref{hybridDis} for Burgers' equation \eqref{Burgers} with periodic
boundary conditions. We discretized in space using centered finite
differences with mesh size $\Delta x=0.01$. In time, both the time
step size $\Delta t$ and the length of the time interval $T$ were
varied simultaneously, maintaining a fixed number of time steps,
$N_t=\frac{T}{\Delta t}=200$. In Figure \ref{ParaDiagI_BurgersFig},
\begin{figure}
\centering
\includegraphics[width=2.3in,height=1.75in,angle=0]{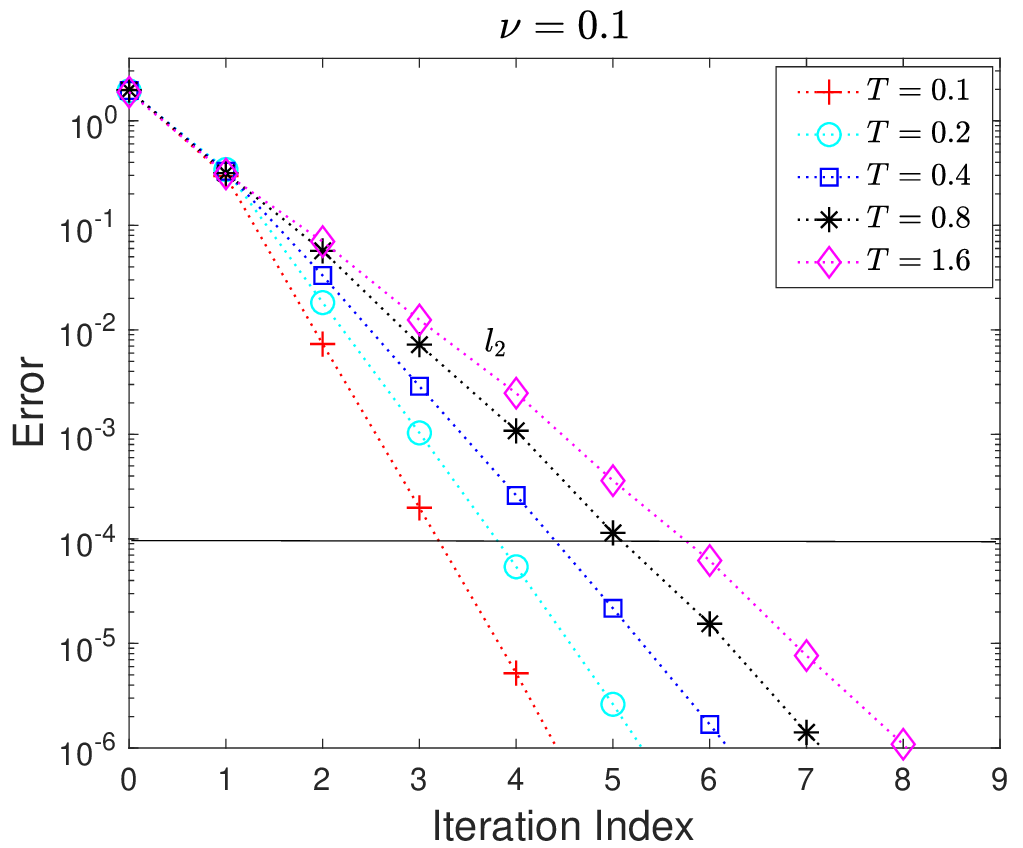}
\includegraphics[width=2.3in,height=1.75in,angle=0]{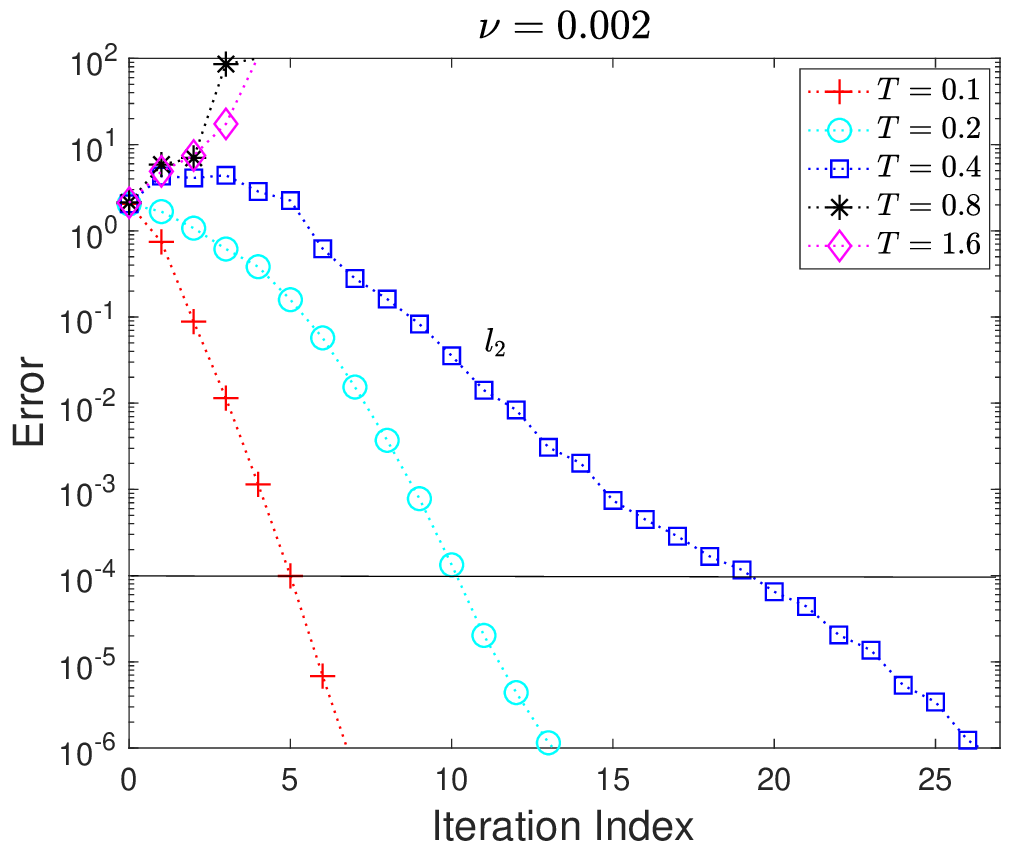}
\caption{Error of ParaDiag I with BVM discretization \eqref{hybridDis}
  for Burgers' equation \eqref{Burgers} with two values of the
  diffusion parameter $\nu$. For each value of $T$, the number of time
  steps is fixed to $N_t=\frac{T}{\Delta t}=200$. The horizontal line
  denotes the approximate space-time discretization error
  $\max\{\Delta t^2, \Delta x^2\}=10^{-4}$.}
  \label{ParaDiagI_BurgersFig}
\end{figure}
we present the convergence histories for two values of the diffusion
parameter $\nu$ and several values of $T$.  Note the dependence of the
convergence rate on $T$, especially when $\nu$ is small. For
$\nu=0.1$, ParaDiag I has similar convergence rates when $T$
increases, indicating that the Jacobian matrix $\nabla f(u, t)$ has
smaller variations for $(t, u)\in\{(t_1, u_1), (t_2, u_n), \dots,
(t_{N_t}, u_{N_t})\}$.

The major computation in ParaDiag I is solving the $N_t$ independent
Jacobian systems in step-b of \eqref{nonlinear3step}. Assuming the
method reaches the stopping criterion after $k$ iterations, the total
number of Jacobian solves is $k$, given that we have access to $N_t$
processors and each of them handles one Jacobian system in step-b. On
the other hand, in a time-stepping mode, one would need to solve
$\sum_{n=1}^{N_t}{\rm It}_{n}$ Jacobian systems, where ${\rm It}_n$
represents the number of Newton iterations performed at the $n$-th
time step. Table \ref{Tab_JacNum}
\begin{table}
\caption{Number of total Jacobian solves for the sequential
  Trapezoidal Rule and ParaDiag I with BVM discretization in
  parallel}\label{Tab_JacNum} \centering
\begin{tabular}{l|c|ccccc}
\hline
\multirow{3}{*}{$\nu=0.1$} &{\small$T$}   &{\small 0.1}        &{\small 0.2}           &{\small 0.4}           &{\small 0.8}      &{\small1.6}  \\  
\cline{2-7}
&{\small Trapezoidal Rule}    &401   &401   &403   &419   &443     \\
&{\small ParaDiag I}          &5     &5     &6     &7     &7\\  
\hline
\hline
\multirow{3}{*}{$\nu=0.002$} &{\small$T$}   &{\small 0.1}        &{\small 0.2}           &{\small 0.4}           &{\small 0.8}      &{\small1.6}  \\  
\cline{2-7}
&{\small Trapezoidal Rule}              &400   &446   &476   &460   &526     \\
&{\small ParaDiag I}           &7    &12    &22   &$\times$   &$\times$\\
\hline
\end{tabular}
\end{table}
shows a comparison of the total number of parallel Jacobian solves in
ParaDiag I with BVM discretization, and when using the trapezoidal
rule sequentially. We see the clear computational advantage of
ParaDiag I with BVM discretization, especially when convergence is
rapid.

A recently proposed idea in \cite[Section 3.3]{LWu22} to accelerate
nonlinear ParaDiag II which we will see in the next section can also
be used to accelerate nonlinear ParaDiag I: instead of using a single
matrix $A_k$ to approximate all the blocks $\{\nabla f({\bm u}_n^k,
t_n)\}$ (cf. \eqref{Ak}), we approximate $\nabla F({\bm U}^k)$ by
using a tensor structure matrix $\Phi_k\otimes A_k$ with a diagonal
matrix $\Phi_k$ determined by minimizing
\begin{equation}\label{eq3.17}
\min_{\Phi_k={\rm diag}(\phi_1,\phi_2,\dots, \phi_{N_t})}\| \nabla F({\bm U}^k)- \Phi_k\otimes A_k\|,
\end{equation}
where $A_k$ is the averaging matrix given in \eqref{Ak}.  For the
Frobenius norm $\|\cdot\|_F$, the solution of this minimization
problem is known as the \textit{Nearest Kronecker product
  Approximation} (NKA), given by \cite[Theorem 3]{VLP93}
\begin{equation}  \label{NKA1}
\phi_n=\frac{{\rm trace}((\nabla f({\bm u}_n^k, t_n))^\top A_k)}{{\rm trace}(A_k^\top A_k)},\quad n=1, 2, \dots, N_t,
\end{equation} 
 under the assumption that ${\rm trace}(A_k^{\top}A_k)>0$.  This
leads to the quasi Newton iteration
$$
(B\otimes I_x-\Phi_k\otimes  {A_k}){\bm U}^{k+1}={\bm b}- \left((\Phi_k\otimes  {A_k}){\bm U}^k-F({\bm U}^k)\right), 
$$
which,  after multiplying  both sides by  the matrix $B^{-1}\otimes I_x$,  can be represented as   
\begin{equation*} 
{\small
\begin{split}
(I_t\otimes I_x-B^{-1}\Phi_k\otimes  {A_k}){\bm U}^{k+1}=(B^{-1}\otimes I_x)({\bm b}+F({\bm U}^k))- (B^{-1}\Phi_k\otimes  {A_k}){\bm U}^k.
\end{split}}
\end{equation*}
By diagonalizing $B^{-1}\Phi_k$ as $V{\rm
  diag}(\lambda_1,\lambda_2,\dots, \lambda_{N_t})V^{-1}$ we can solve
${\bm U}^{k+1}$ via the 3-step diagonalization procedure
\eqref{nonlinear3step} as well, where for step-b we now have to solve the
linear systems
$$
(I_x-\lambda_n A_k){\bm u}^b_n={\bm u}_n^a, ~n=1,2,\dots, N_t. 
$$
So far there is no theory for the diagonalization of $B^{-1}\Phi_k$,
but in practice this matrix is often diagonalizable and $V$ is
generally well conditioned.

 In practice, it is better not to compute the scaling factors
  $\{\phi_n\}$ for each Newton iteration, which can be rather
  expensive due to the matrix-matrix multiplications in \eqref{NKA1}.
  It suffices  to compute these quantities only once before starting
the Newton iteration (i.e., as an offline task) by using a {\em
  reduced} model. For time-dependent PDEs, such a model could be a
semi-discretized system of ODEs obtained by using a coarse space grid.
We now show this idea for Burgers' equation \eqref{Burgers} with
periodic boundary conditions, discretized by a centered finite
difference scheme with mesh size $\Delta x=\frac{1}{200}$. The scaling
factors $\{\phi_n\}$ were determined by the Trapezoidal Rule with a
coarse mesh size $\Delta X=\frac{1}{20}$.  In Figure
\ref{ParaDiagI_NKA_BurgersFig}
\begin{figure}
\centering
\includegraphics[width=2.3in,height=1.75in,angle=0]{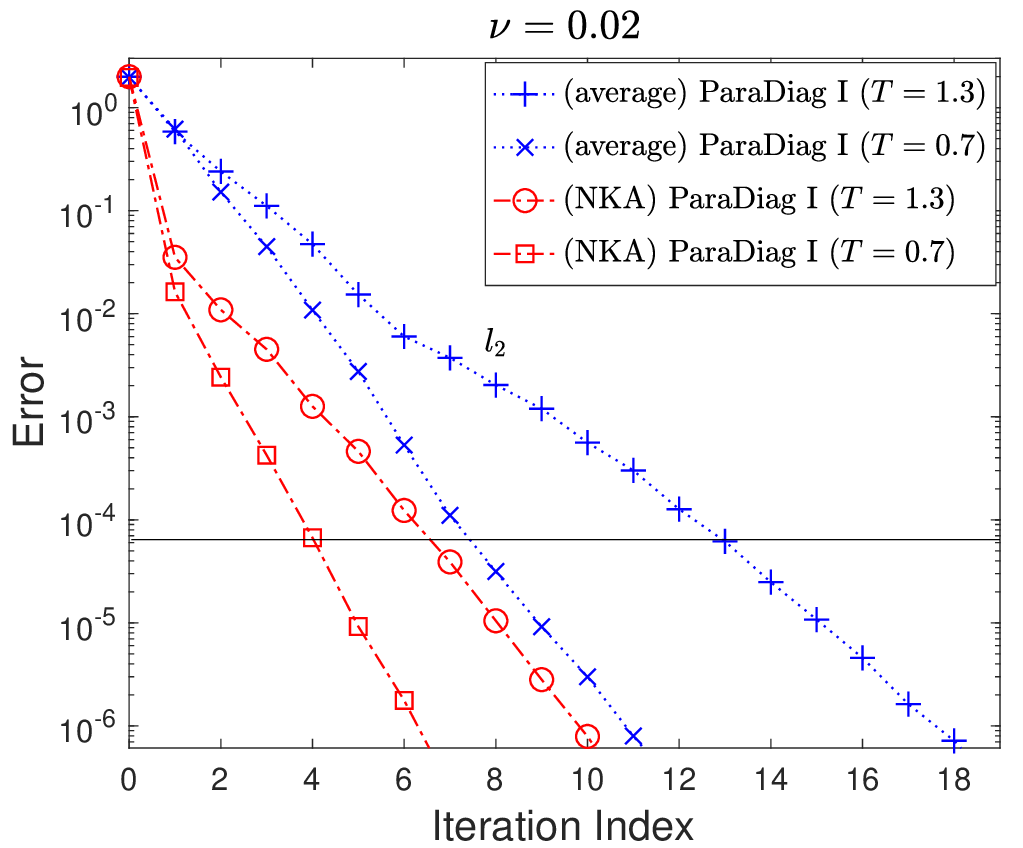}
\includegraphics[width=2.3in,height=1.75in,angle=0]{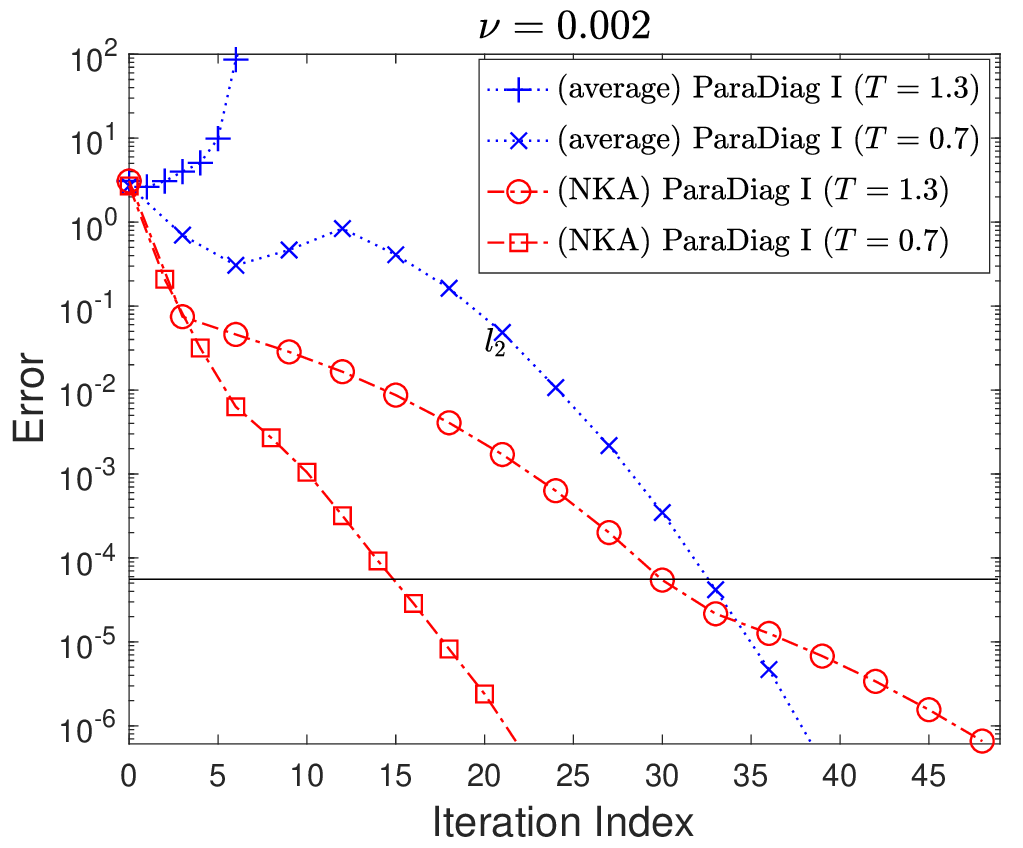}
\caption{Error of the two quasi Newton versions of ParaDiag I with BVM
  discretization \eqref{hybridDis} for Burgers' equation
  \eqref{Burgers} with two values of the diffusion parameter $\nu$. }
\label{ParaDiagI_NKA_BurgersFig}
\end{figure}
we show for two values of the diffusion parameter $\nu$ how the error
decays for the two quasi Newton versions of ParaDiag I with BVM
discretization \eqref{hybridDis}. For each $\nu$ we consider two time
interval lengths, $T=0.7$ and $T=1.3$. Clearly, the quasi Newton
version with NKA technique improves the convergence rate, especially
when $T=1.3$.

\subsubsection{Iterative ParaDiag Methods (ParaDiag II)}\label{sec3.6.2}

It is difficult to generalize ParaDiag I to high-order
time-integrators like multistage Runge-Kutta methods. This is not the
case for the class of ParaDiag II methods, which use approximations of the
time stepping matrices in order to make them diagonalizable, and then
solve the all-at-once system by iteration.

A first member in the ParaDiag II class of methods was proposed in
\cite{mcdonald2018preconditioning}  at the discrete level, and
  independently in \cite{gander2019convergence} at the continuous
  level. Although these two papers essentially describe the same
method, the descriptions themselves are quite different. In
\cite{mcdonald2018preconditioning}, the authors consider the
approximate solution of the first-order linear system of ODEs
\eqref{linearODE} using a linear multistep method with $m$-steps,
$$
{\sum}_{l=0}^{m}a_l{\bm u}_{n-l}=\Delta t{\sum}_{l=0}^{m}b_l(A{\bm u}_{n-l})+\bar{\bm g}_n, ~n=1,\dots, N_t,
$$
where one assumes that the first $m$ initial values $\{{\bm
  u}_{-(m-1)}, {\bm u}_{-(m-2)},\dots, {\bm u}_{0}\}$ are given. The
all-at-once system of these $N_t$ difference equations is $\CK{\bm
  U}={\bm b}$ with ${\bm U}:=({\bm u}_1^\top, \dots, {\bm
  u}_{N_t}^\top)^\top$ and $\CK:=B_1\otimes I_x-B_2\otimes (\Delta
tA)$, where ${\bm b}$ is a vector depending on the initial values and
the source term ${\bm g}(t)$, and
{\small
  $$
  B_1:=\begin{bmatrix} a_0 & &
  & & &\\ a_1 &a_0 & & & &\\ \vdots &\ddots &\ddots & & &\\ a_{m} &
  &~\ddots &\ddots & &\\ &\ddots & &~~~a_1 &~a_0 &\\ & &a_m &\dots
  &a_1 &a_0
\end{bmatrix},~B_2:=\begin{bmatrix}
b_0 & & & & &\\
b_1 &b_0 & & & &\\
\vdots &\ddots &\ddots & & &\\
b_{m} & &~\ddots &\ddots & &\\
&\ddots & &~~~b_1 &~b_0 &\\
& &a_m &\dots &b_1 &b_0
\end{bmatrix}.
$$}
In \cite{mcdonald2018preconditioning}, this all-at-once system is
solved  with GMRES  using a \pd $\CP$ for $\CK$ obtained by
replacing the two time stepping matrices $B_1$ and $B_2$ with two
circulant matrices of Strang type, i.e.,
$$
\CP:=C_1\otimes I_x-C_2\otimes (\Delta tA),
$$
where
{\small $$
C_1 :=\begin{bmatrix}
a_0 & & a_m &\dots & a_1 & a_0\\
a_1 &a_0 & & & & a_1\\
\vdots &\ddots &\ddots & &\ddots &\vdots\\
a_{m} & &~\ddots &\ddots & & a_m\\
&\ddots & &~~~a_1 &~a_0 &\\
& &a_m &\dots &a_1 &a_0
\end{bmatrix},~ C_2 :=\begin{bmatrix}
b_0 & & b_m &\dots & b_1 & b_0\\
b_1 &b_0 & & & & b_1\\
\vdots &\ddots &\ddots & &\ddots &\vdots\\
b_{m} & &~\ddots &\ddots & & b_m\\
&\ddots & &~~~b_1 &~b_0 &\\
& &b_m &\dots &b_1 &b_0
\end{bmatrix}.
$$} For a theoretical understanding, or if the preconditioner
  $\CP$ is very good (like multigrid for Poisson problems), one can
  use it directly in the stationary iteration
\begin{equation}\label{ParaDiagII}
\CP \Delta {\bm U}^k={\bm r}^k:={\bm b}-\CK{\bm U}^k,~{\bm U}^{k+1}={\bm U}^k+\Delta {\bm U}^k,~k=0,1,\dots,
\end{equation}
and the asymptotic convergence is fast if $\rho(\CP^{-1}\CK)\ll 1$. If
convergence is not fast, this process can be accelerated using the
preconditioner $\CP$ within a Krylov method, i.e. solving the
preconditioned linear system $\CP^{-1}\CK{\bm U}=\CP^{-1}{\bm b}$ with a
Krylov method, see \cite[Section 4.1]{ciaramella2022iterative} for a
simple introduction. This can even work when $\rho(\CP^{-1}\CK)\geq1$,
and is advantageous when the spectrum $\sigma(\CP^{-1}\CK)$ is
clustered.

The first advantage of using the block-circulant matrix $\CP$ as \pd
is that, similar to ParaDiag I, for each iteration, the
preconditioning step $\CP^{-1}{\bm r}^k$ can be solved via the \dg
procedure, because any two circulant matrices
$C_1$ and $C_2$ are commutative and therefore can be diagonalized
simultaneously \cite[Chapter 4]{Ng04}, i.e. 
$$
C_l={\rm F}^* D_l{\rm F}, ~l=1,2,
$$
where ${\rm F}$ is the discrete Fourier matrix defined as
(${\rm F}^*$ is the conjugate transform of ${\rm F}$) 
\begin{equation}\label{MatF}
{\rm F}:=\frac{1}{\sqrt{N_t}}
\begin{bmatrix}
1 &1 & \dots &1\\
1 &\omega & \cdots &\omega^{N_t-1}\\
\vdots &\vdots &\cdots &\vdots\\
1 &\omega^{N_t-1} &\cdots &\omega^{(N_t-1)^2}
\end{bmatrix},~\omega:={\rm exp}\left(\frac{2\pi {\rm i}}{N_t}\right),
\end{equation}
and $D_{l}:={\rm diag}(\lambda_{l,1}, \lambda_{l,2},\dots,\lambda_{l,
  N_t})$ contains the eigenvalues of $C_l$, i.e.,
\begin{equation}\label{MatD}
D_l={\rm diag}\left(\sqrt{N_t}{\rm F} C_l(:,1)\right),~l=1, 2.
\end{equation}
Then, according to the property of the Kronecker product, we can
factor $\CP=({\rm F}^*\otimes I_x)(D_1\otimes I_x-D_2\otimes (\Delta
tA))({\rm F}\otimes I_x)$ and thus we can compute $\CP^{-1}{\bm r}^k$
by performing again three steps:
\begin{equation}\label{ParaDiagII3step}
\begin{cases}
{\bm U}^a=({\rm F}\otimes I_x){\bm r}^k, &\text{(step-a)}\\
(\lambda_{1,n} I_x-\lambda_{2,n}\Delta tA){\bm u}^b_n={\bm u}^a_n, ~n=1,2,\dots, N_t, &\text{(step-b)}\\
{\bm U}=({\rm F}^*\otimes I_x){\bm U}^b. &\text{(step-c)}
\end{cases}
\end{equation}
Here the first and last steps can be computed efficiently using the
Fast Fourier Transform (FFT), with $\CO(N_xN_t\log N_t)$ operations,
and as in all ParaDiag methods, step-b can be computed in parallel,
since all linear systems are completely independent of each other at
different time points.

In \cite[Section 3]{mcdonald2018preconditioning}, an important result
about the clustering of eigenvalues of the preconditioned matrix
$\CP^{-1}\CK$ was obtained when this ParaDiag II techniques is used to
precondition a system for its solve by a Krylov method.
\begin{theorem}\label{ParaDiagII_Pro1}
\emph{When $A\in\mathbb{R}^{N_x\times N_x}$ is symmetric
  negative definite, the preconditioned matrix $\CP^{-1}\CK$ has at most $mN_x$
  eigenvalues not equal to 1.}
\end{theorem}
This implies that GMRES converges in at most \(mN_x+1\) steps for the
all-at-once system \(\CK{\bm U}={\bm b}\) using \(\CP\) as
preconditioner. Note however that when \(N_x\) is large, this result
does not guarantee fast convergence of GMRES, and moreover, if \(A\)
is not symmetric, the clustering of \(\sigma(\CP^{-1}\CK)\) becomes
worse. To illustrate this, we consider three examples: the heat
equation \eqref{heatequation}, the advection-diffusion equation
\eqref{ADE} with two values of the diffusion parameter \(\nu\), and
the second-order wave equation \eqref{WaveEquation1d}. We use
homogeneous Dirichlet boundary conditions and the initial condition
\(u(x,0)=\sin(2\pi x)\) for all PDEs, and for the wave equation, we
set \(\partial_tu(x,0)=0\).

The semi-discrete system of ODEs using centered finite differences is
of the form \eqref{linearODE} for the first order parabolic problems,
and for the second-order wave equation we get
\begin{equation}\label{sec3_6_2ndODE}
{\bm u}''(t)=A{\bm u}(t),~{\bm u}(0)={\bm u}_0,~{\bm u}'(0)=0, ~t\in(0, T],
\end{equation}
where \(A={\rm Tri}[1~~-2~~1]/\Delta x^2\). We solve the first-order
system of ODEs \eqref{linearODE} using the Trapezoidal Rule, and the
second order system of ODEs \eqref{sec3_6_2ndODE} using a parametrized
Numerov-type method \cite{Chawla83},
\begin{equation}\label{Numerov}
\begin{cases}
\tilde{\bm u}_n-{\bm u}_n+\gamma\Delta t^2A({\bm u}_{n+1}-2{\bm u}_n+{\bm u}_{n-1})=0,\\
{\bm u}_{n+1}-2{\bm u}_n+{\bm u}_{n-1}-\frac{\Delta t^2A}{12}({\bm u}_{n+1}+10\tilde{\bm u}_n+{\bm u}_n)=0,
\end{cases}
\end{equation}
where \(\gamma>0\) is a parameter. For \(\gamma = 0\), \eqref{Numerov}
reduces to the classical Numerov method, which is a fourth-order
method but only conditionally stable. With \(\gamma
\geq\frac{1}{120}\), this method is unconditionally stable and still
fourth order.

Let \(T=2\), \(\Delta t=\frac{1}{50}\), \(\Delta x=\frac{1}{100}\),
and \(\gamma=\frac{1}{100}\). We show in Figure
\ref{ParaDiagII_alpha_equal_1}
\begin{figure}
\centering
\includegraphics[width=2.3in,height=1.75in,angle=0]{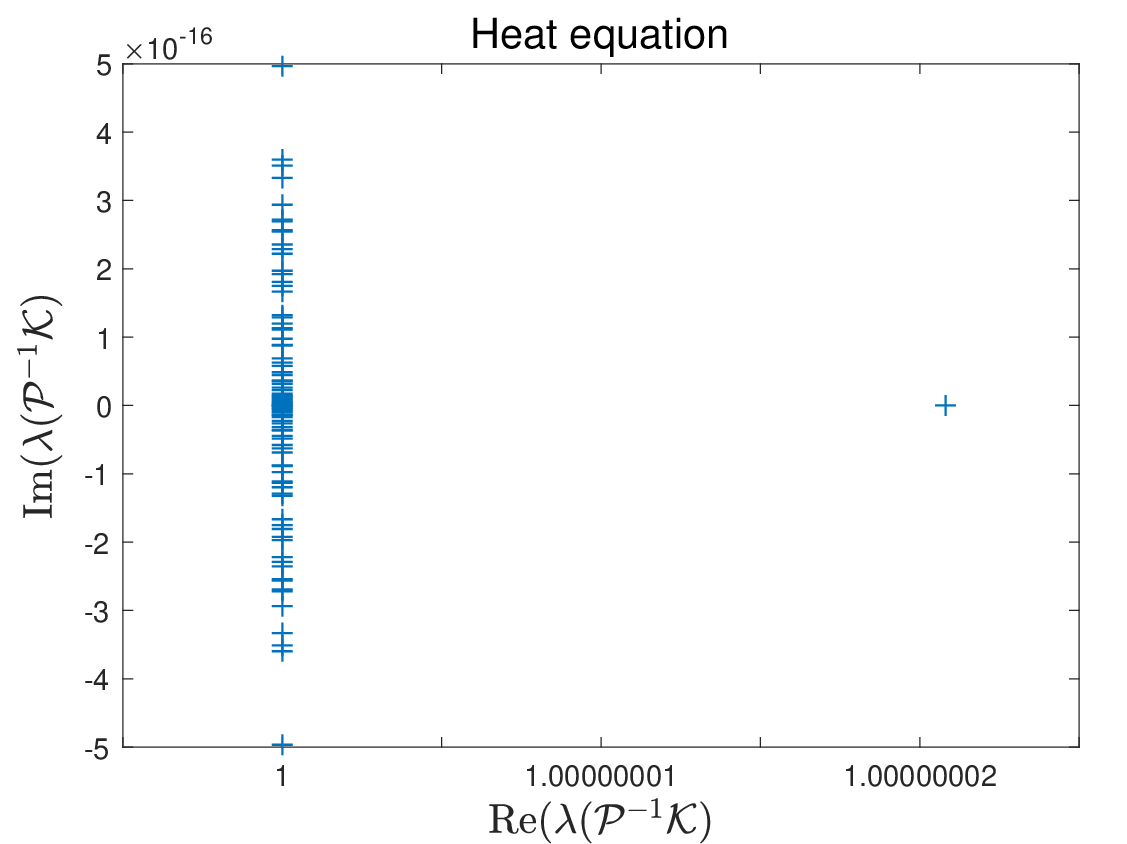}~~
\includegraphics[width=2.3in,height=1.75in,angle=0]{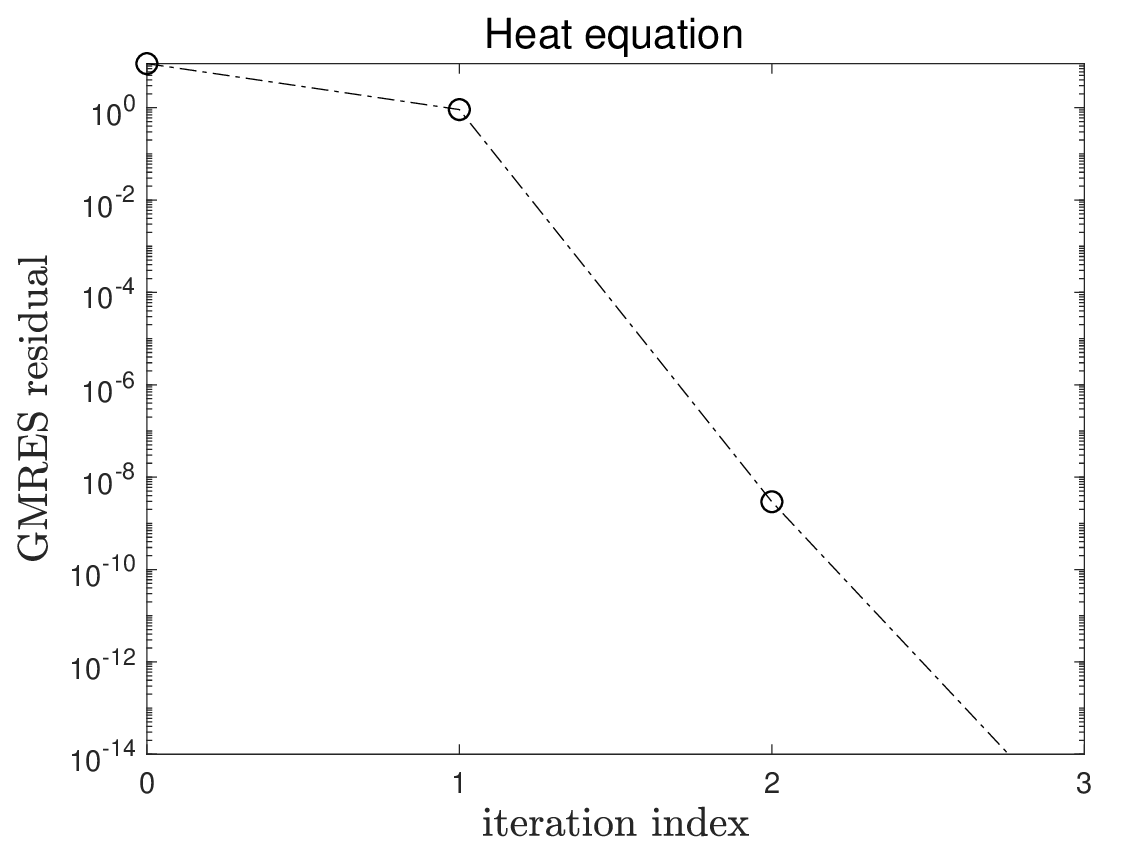} \\
\includegraphics[width=2.3in,height=1.75in,angle=0]{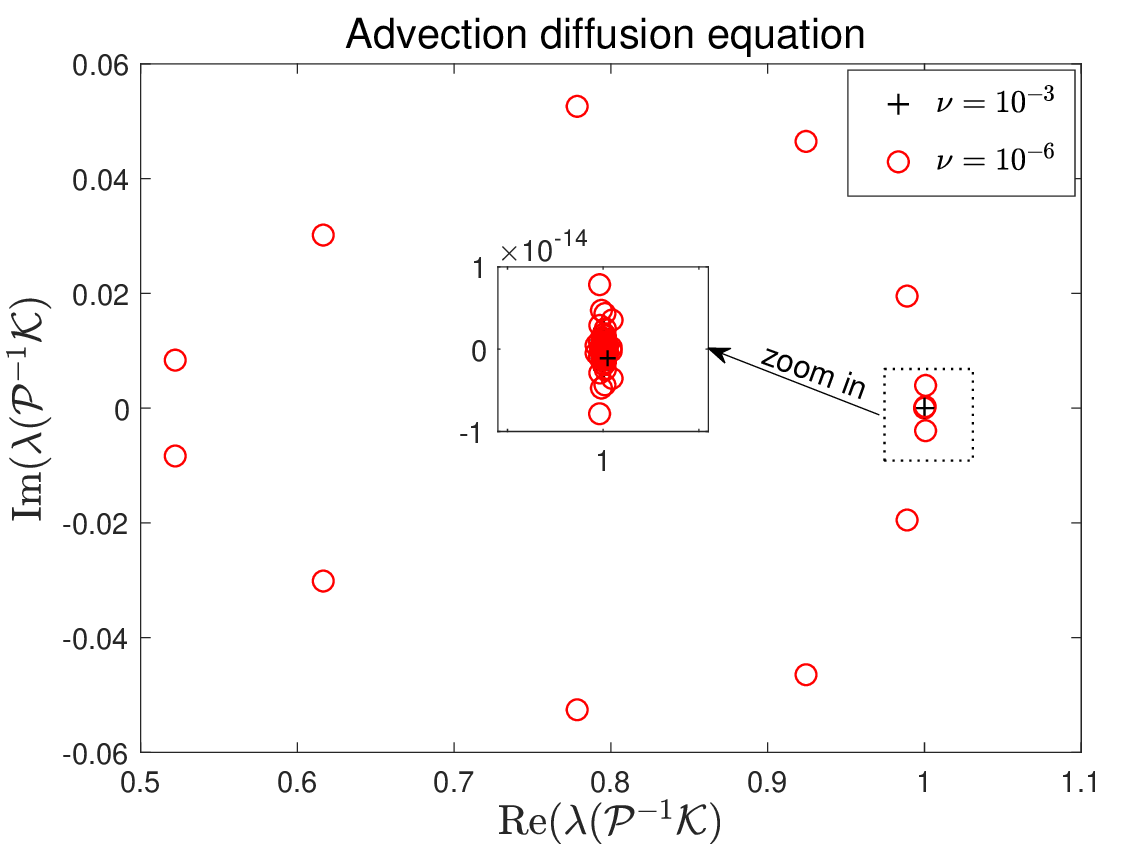}~~
\includegraphics[width=2.3in,height=1.75in,angle=0]{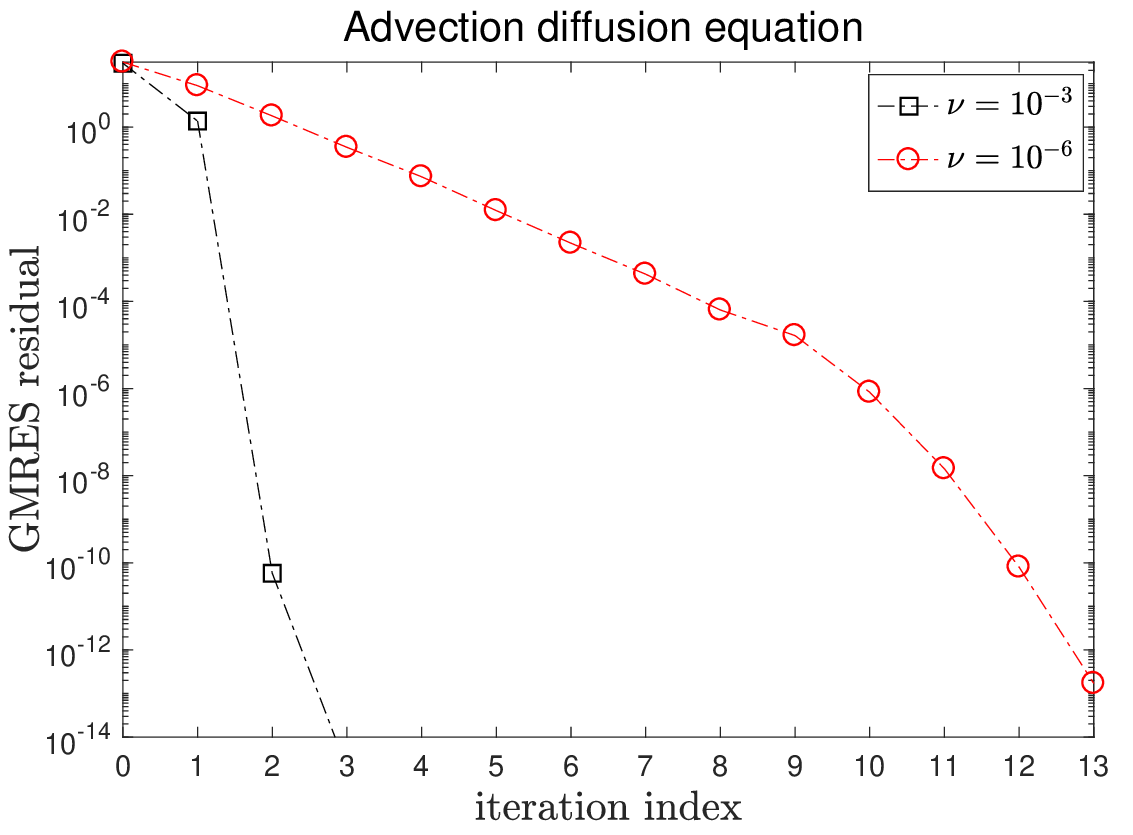} \\
\includegraphics[width=2.3in,height=1.75in,angle=0]{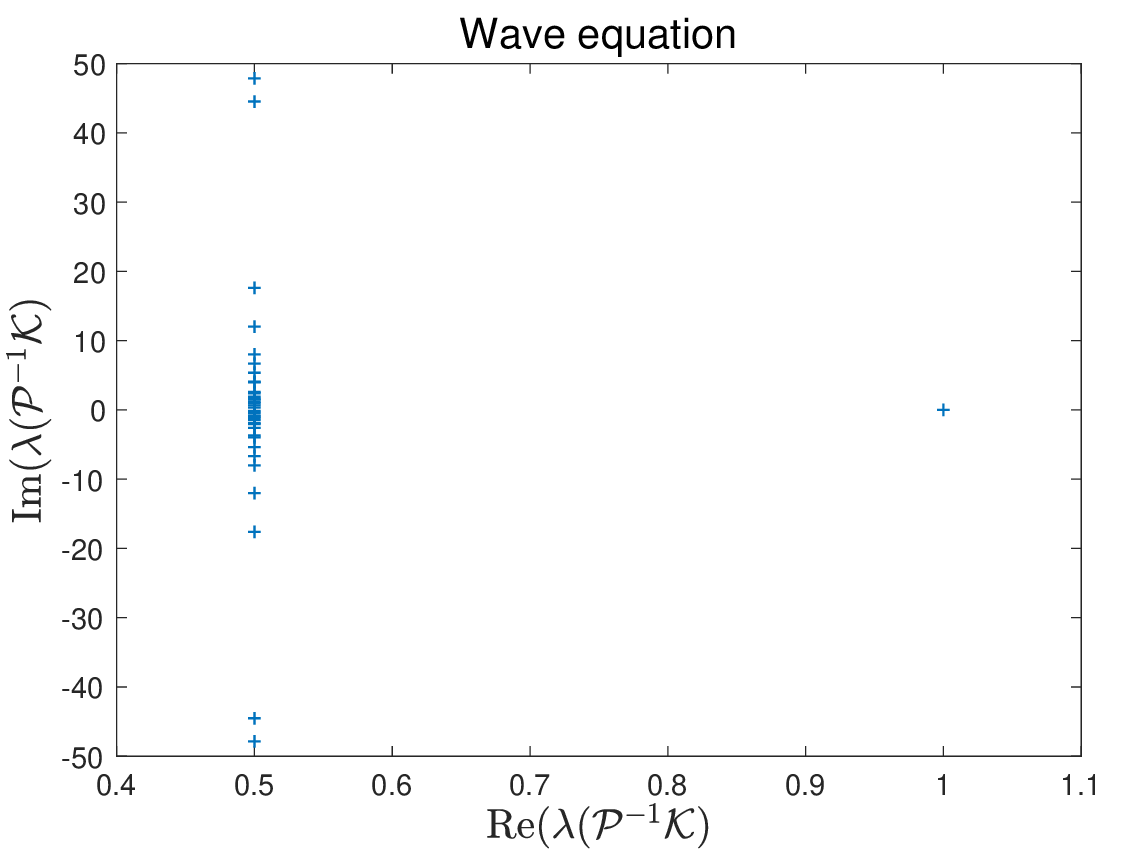}~~
\includegraphics[width=2.3in,height=1.75in,angle=0]{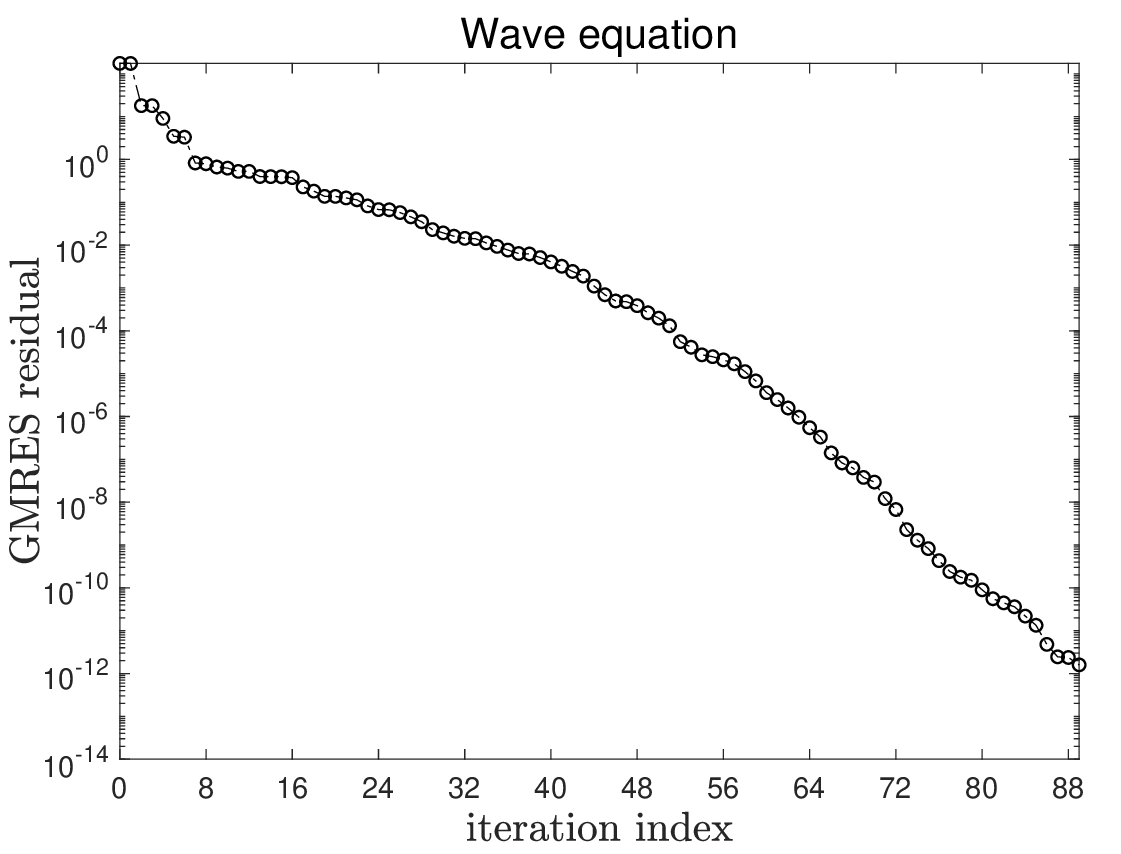}
\caption{Spectra of the preconditioned matrix \(\CP^{-1}\CK\) (left)
  and the measured convergence of preconditioned GMRES (right)
  for the three representative PDEs.}
\label{ParaDiagII_alpha_equal_1}
\end{figure}
the eigenvalues of the preconditioned matrix \(\CP^{-1}\CK\) for these
three PDEs, and the decay of the residual as function of the iteration
number of the preconditioned GMRES solver. We see that for a wide
range of problems, the block-circulant matrix \(\CP\) is a good
preconditioner, even for the advection-dominated diffusion equation
with a small diffusion parameter \(\nu=10^{-3}\). If we continue
however to reduce \(\nu\), the middle row of Figure
\ref{ParaDiagII_alpha_equal_1} shows that the preconditioner \(\CP\)
becomes worse, and ultimately, it loses its power when we switch to
the hyperbolic problem represented by the wave equation.

During the same time, and independently of the work in
\cite{mcdonald2018preconditioning}, another diagonalization-based time
parallel method was proposed in \cite{gander2019convergence} within
the framework of {\em waveform relaxation}. It is constructed at the
continuous level by using a {\em head-tail} coupled condition,
\begin{equation}\label{ParaDiagII_WR}
{\bm u}_t^k(t)=A{\bm u}^k(t)+{\bm g}(t),~{\bm u}^k(0)=\alpha [{\bm u}^k(T)-{\bm u}^{k-1}(T)]+{\bm u}_0,
\end{equation}
where $\alpha\in\mathbb{C}$ is a free parameter. Upon convergence, we
recover the solution to the initial value problem \eqref{linearODE},
i.e., ${\bm u}_t(t)=A{\bm u}(t)+{\bm g}(t)$ with ${\bm u}(0)={\bm
  u}_0$. For the second-order problem \eqref{WaveEquation1d}, the
iteration \eqref{ParaDiagII_WR} can be defined similarly by reducing
the problem to a first order system. The advantage of this iteration
is that we can solve each iterate ${\bm u}^k(t)$ independently for all
the time steps. To see this, we need to take a closer look at the
structure of the discrete system stemming from
\eqref{ParaDiagII_WR}. Suppose we discretize \eqref{ParaDiagII_WR}
using a one-step time-integrator specified by two matrices $r_1(\Delta
tA)$ and $r_2(\Delta tA)$,
\begin{equation}\label{discreteWR}
\begin{cases}
r_1(\Delta tA){\bm u}^k_{n}=r_2(\Delta tA){\bm u}^k_{n-1}+\tilde{\bm g}_n, ~n=1,\dots, N_t,\\
{\bm u}^k_0=\alpha ({\bm u}^k_{N_t}-{\bm u}^{k-1}_{N_t})+{\bm u}_0,
\end{cases}
\end{equation}
where $N_t=T/\Delta t$, ${\bm u}_0$ is a given initial value and
$\tilde{\bm g}_n\in\mathbb{R}^{N_x}$ is a vector coming from the
source term ${\bm g}(t)$. Examples for $r_1(\Delta tA)$ and $r_2(\Delta tA)$
are
\begin{equation}\label{r1r2}
\begin{cases}
r_1=I_x-\Delta tA,~r_2=I_x, &\text{Backward  Euler},\\
r_1=I_x-\frac{1}{2}\Delta tA,~r_2=I_x+\frac{1}{2}\Delta tA, &\text{Trapezoidal Rule}.
\end{cases}
\end{equation}
By replacing ${\bm u}^k_0$ with $\alpha ({\bm u}^k_{N_t}-{\bm
  u}^{k-1}_{N_t})+{\bm u}_0$ for $n=1$, we can unfold
\eqref{discreteWR} as
$$
\begin{cases}
r_1(\Delta tA){\bm u}^k_1-\alpha r_2(\Delta tA) {\bm u}^k_{N_t}=\alpha r_2(\Delta tA){\bm u}^{k-1}_{N_t}+r_2(\Delta tA){\bm u}_0+\tilde{\bm g}_1,\\
r_1(\Delta tA){\bm u}^k_2-r_2(\Delta tA){\bm u}^k_1=\tilde{\bm g}_2,\\
r_1(\Delta tA){\bm u}^k_3-r_2(\Delta tA){\bm u}^k_2=\tilde{\bm g}_3,\\
\vdots\\
r_1(\Delta tA){\bm u}^k_{N_t}-r_2(\Delta tA){\bm u}^k_{N_t-1}=\tilde{\bm g}_{N_t}.
\end{cases}
$$
We see that all the discrete unknowns ${\bm u}_1, {\bm u}_2,\dots,
{\bm u}_{N_t}$ are coupled together and therefore we have to solve
them in one shot. To this end, we represent these $N_t$ equations as
\begin{subequations}
\begin{equation}\label{ParaDiagII_AAA_WR_a}
\CP_{\alpha}{\bm U}^k={\bm b}^k,
\end{equation}
where ${\bm U}^k:=(({\bm u}_1^k)^\top, \dots, ({\bm u}_{N_t}^k)^\top)^\top$ and ${\bm b}^k$ is a vector consisting of
\begin{equation}\label{ParaDiagII_AAA_WR_b}
{\bm b}^k:={\bm b}-\alpha
\begin{bmatrix}
r_2(\Delta tA){\bm u}^{k-1}_{N_t}\\
0\\
\vdots\\
0
\end{bmatrix},
~{\bm b}:=
\begin{bmatrix}
r_2(\Delta tA){\bm u}_0+\tilde{\bm g}_1\\
\tilde{\bm g}_2\\
\vdots\\
\tilde{\bm g}_{N_t}
\end{bmatrix}.
\end{equation}
The matrix $\CP_\alpha$ is given by
\begin{equation}\label{ParaDiagII_AAA_c}
\begin{split}
&\CP_\alpha
:=
\begin{bmatrix}
r_1(\Delta tA) & & &-\alpha  r_2(\Delta tA)\\
- r_2(\Delta tA) & r_1(\Delta tA) & &\\
&\ddots &\ddots &\\
& &- r_2(\Delta tA) & r_1(\Delta tA)
\end{bmatrix}\\
&\hspace{1.5em}=I_t\otimes  r_1(\Delta tA)-C_\alpha\otimes  r_2(\Delta tA),\\
&C_\alpha:=
\begin{bmatrix}
0 & & &\alpha\\
1 &0 & &\\
&\ddots &\ddots &\\
& &1 &0
\end{bmatrix}\in\mathbb{R}^{N_t\times N_t}.
\end{split}
\end{equation}
\end{subequations}
The matrix $C_\alpha$ is known as an $\alpha$-circulant matrix, which,
similarly to the standard circulant matrix where $\alpha=1$, can be
diagonalized by an eigenvector matrix $V_\alpha$ that depends on
$\alpha$ only. Specifically, according to \cite[Theorem 2.10]{BLM05},
for an arbitrary $\alpha$-circulant matrix $C_\alpha$ of Strang type,
we have the spectral decomposition
\begin{subequations}
\begin{equation}\label{diag_Calp_a}
C_\alpha=V_\alpha D_\alpha V^{-1}_\alpha,
\end{equation}
where the diagonal eigenvalue matrix and the eigenvector matrix
$V_\alpha$ are given by
\begin{equation}\label{diag_Calp_b}
\begin{split}
&D_\alpha={\rm diag}(\sqrt{N_t}{\rm F}\Lambda_\alpha C_\alpha(:, 1)),\\
&V_\alpha=\Lambda_\alpha{\rm F}^*,~\Lambda_\alpha:={\rm diag}\left(1, \alpha^{-\frac{1}{N_t}},\dots, \alpha^{-\frac{N_t-1}{N_t}}\right),
\end{split}
\end{equation}
\end{subequations}
and $C_\alpha(:,1)$ represents the first column of
$C_\alpha$. Using the property of the Kronecker product, we can
factor $\CP_\alpha$ as
$$
\CP_\alpha=(V_\alpha\otimes I_x)(I_t\otimes r_1(\Delta tA)-D_\alpha\otimes r_2(\Delta tA))(V^{-1}_\alpha\otimes I_x),
$$
and hence again solve like in all ParaDiag methods for ${\bm U}^k$ in
\eqref{ParaDiagII_AAA_WR_a} using three steps:
\begin{equation}\label{ParaDiagII3step_alp}
\begin{cases}
{\bm U}^a=(V^{-1}_\alpha\otimes I_x){\bm b}^k, &\text{(step-a)}\\
(r_1(\Delta tA)-\lambda_{n}r_2(\Delta tA)){\bm u}^b_n={\bm u}^a_n, ~n=1,2,\dots, N_t, &\text{(step-b)}\\
{\bm U}^k=(V_\alpha\otimes I_x){\bm U}^b. &\text{(step-c)}
\end{cases}
\end{equation}
When $\alpha=1$, the eigenvector matrix becomes the Fourier matrix,
$V_\alpha={\rm F}^*$, and hence this ParaDiag II method obtained from
the discretization of the continuous formulation \eqref{ParaDiagII_WR}
coincides with \eqref{ParaDiagII3step} from
\cite{mcdonald2018preconditioning}, and  since $V_\alpha\otimes
I_x=(\Lambda_\alpha\otimes I_x)({\rm F}^*\otimes I_x)$ and
$V^{-1}_\alpha\otimes I_x=({\rm F}\otimes
I_x)(\Lambda^{-1}_\alpha\otimes I_x)$, one can still use
FFT techniques for the first and last step, also when $\alpha\ne 1$.

In \cite{gander2019convergence}, the convergence of the waveform
relaxation iterations \eqref{ParaDiagII_WR} were examined at the
continuous level, and it was shown that the error ${\bm u}^k(t)-{\bm
  u}(t)$ decays rapidly for both first-order and second-order
problems, with a rate depending on $\alpha$. We illustrate this in
Figure \ref{Fig_ParaDiagII_WR},
\begin{figure}
  \centering
 \includegraphics[width=2.3in,height=1.85in,angle=0]{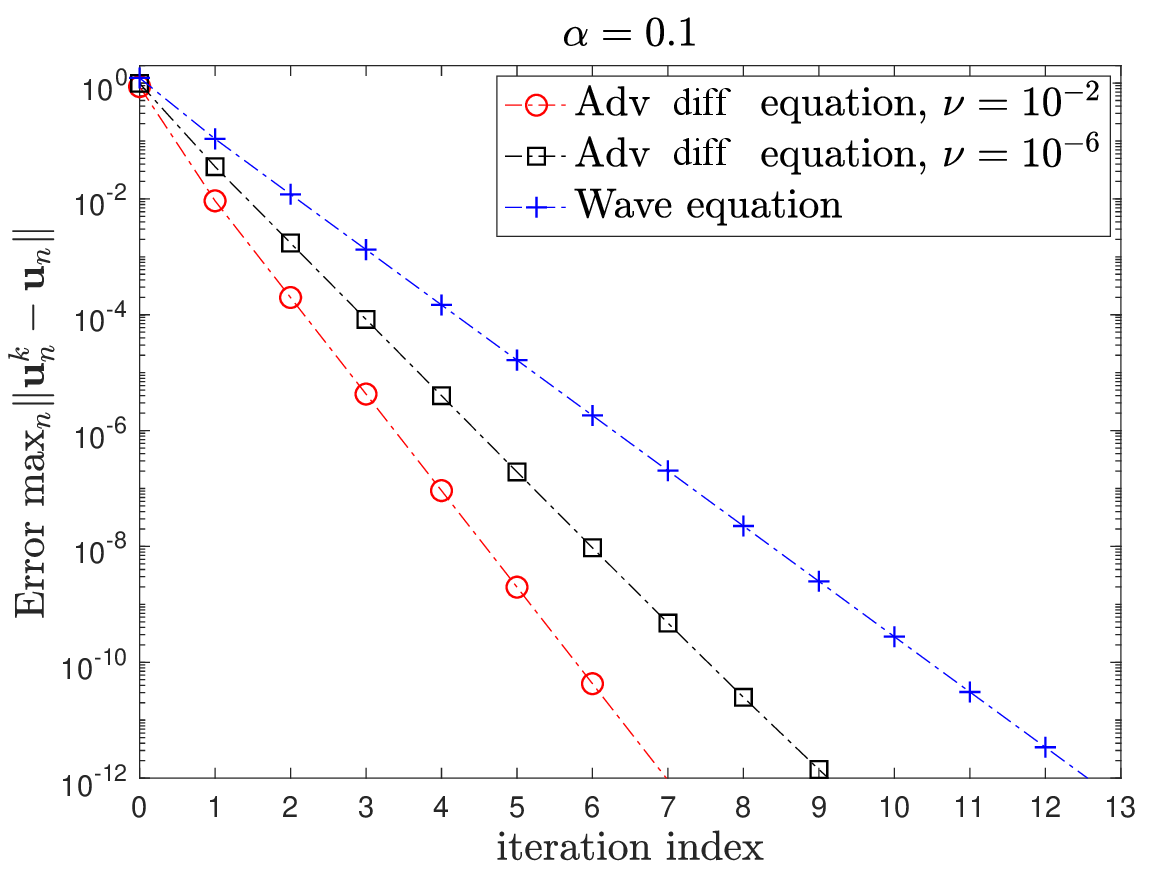}~~
 \includegraphics[width=2.3in,height=1.85in,angle=0]{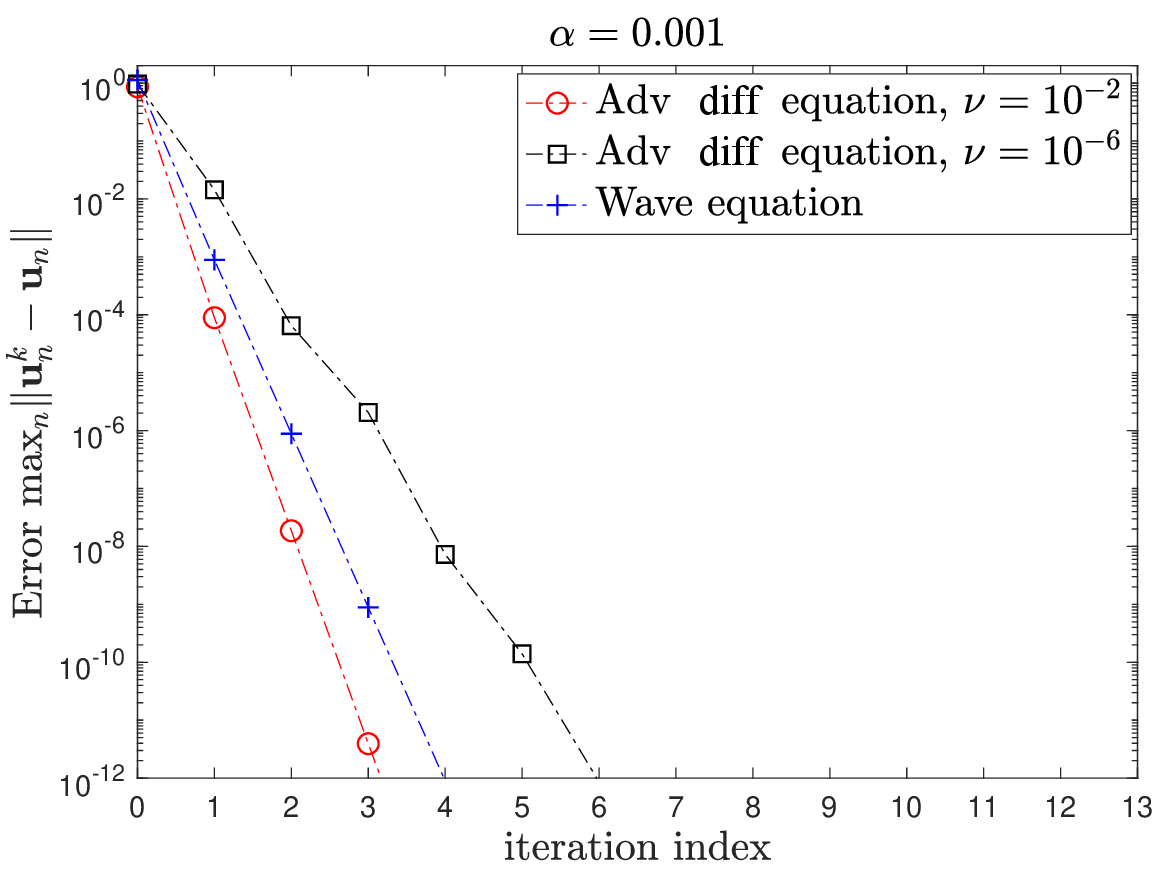}  
  \caption{Error as a function of iteration for the head-tail coupled
    waveform relaxation method \eqref{ParaDiagII_WR} using two values
    of the parameter $\alpha$ and the Trapezoidal Rule.  }
  \label{Fig_ParaDiagII_WR}
\end{figure}
where the data for the two PDEs are the same as those used in Figure
\ref{ParaDiagII_alpha_equal_1}, but with periodic boundary conditions
used here. For this type of boundary conditions, the
  preconditioner $\CP$ with the special choice $\alpha=1$ as
    proposed in \cite{mcdonald2018preconditioning} is singular and
    cannot be used, so we need to use $\alpha<1$. We see that the
introduction of the parameter $\alpha$ makes this ParaDiag II method a
very powerful solver, which works also very well for highly advection
dominated problems, and also the hyperbolic wave equation, and this
without Krylov acceleration!

In \cite{gander2019convergence} two time-integrators were studied,
 Backward  Euler and the Trapezoidal Rule, and it was shown that
the discrete algorithm \eqref{discreteWR} preserves the convergence
rate obtained from the analysis at the continuous level. The proof is
technical and relies on a special representation of $r_1^{-1}(\Delta
tA)r_2(\Delta tA)$ that appears to hold only for these two specific
time-integrators (see \eqref{r1r2} for the formulas of $r_1$ and
$r_2$).

The head-tail coupled waveform relaxation method at the discrete level
\eqref{discreteWR} can be represented as the {\em preconditioned}
stationary iteration
\begin{equation}\label{ParaDiagII_SI}
{\small\begin{split}
\CP_\alpha \Delta {\bm U}^{k-1}={\bm r}^{k-1}:={\bm b}-\CK{\bm U}^{k-1},~{\bm U}^{k}={\bm U}^{k-1}+\Delta {\bm U}^{k-1},~k=1, 2,\dots,
\end{split}}
\end{equation}
where
\begin{subequations}
\begin{equation}\label{ParaDiagII_AAA_a}
\begin{split}
\CK
&:=
\begin{bmatrix}
r_1(\Delta tA) & & & \\
- r_2(\Delta tA) & r_1(\Delta tA) & & \\
&\ddots &\ddots &\\
& &- r_2(\Delta tA) & r_1(\Delta tA)
\end{bmatrix}\\
&=I_t\otimes  r_1(\Delta tA)-B\otimes  r_2(\Delta tA),
\end{split}
\end{equation}
and $B\in\mathbb{R}^{N_t\times N_t}$ is a  Toeplitz  matrix,
\begin{equation}\label{ParaDiagII_AAA_b}
B:=\begin{bmatrix}
0 & & &\\
1 &0 & &\\
&\ddots &\ddots &\\
& &1 &0
\end{bmatrix}.
\end{equation}
\end{subequations}
We thus see that the preconditioned iteration \eqref{ParaDiagII_SI}
with $\alpha=1$ is precisely the method \eqref{ParaDiagII} of
\cite{mcdonald2018preconditioning}. For $\alpha\in(0,1)$ such a preconditioned iteration is the parallel method proposed by Banjai an Peterseim  in 2012 \cite{BP12}. 
To see why the preconditioned
iteration \eqref{ParaDiagII_SI} equals the head-tail coupled waveform
relaxation method \eqref{discreteWR}, we notice that the vector ${\bm
  b}^k$ in \eqref{ParaDiagII_AAA_b} can be represented as ${\bm
  b}^k=(\CP_\alpha-\CK){\bm U}^{k-1}+{\bm b}$, and substituting this
into \eqref{ParaDiagII_AAA_a} gives $\CP_\alpha{\bm
  U}^k=(\CP_\alpha-\CK){\bm U}^{k-1}+{\bm b}$ which leads to
\eqref{ParaDiagII_SI}.

Applying a one-step time-integrator specified by $r_1(\Delta tA)$ and
$r_2(\Delta tA)$ to the initial value problem \eqref{linearODE}, i.e.,
${\bm u}'(t)=A{\bm u}(t)+{\bm g}(t)$ with ${\bm u}(0)={\bm u}_0$,
leads to
\begin{equation}\label{r1r2_1st}
r_1(\Delta tA){\bm u}_{n}=r_2(\Delta tA){\bm u}_{n-1}+\tilde{\bm g}_n, ~n=1,\dots, N_t.
\end{equation}
Therefore, the matrix $\CK$ is an all-at-once representation of
these $N_t$ difference equations, i.e., $\CK{\bm U}={\bm b}$. From
this point of view, $\CP_\alpha$ is a generalized block circulant \pd
for $\CK$.

For second-order problems of the form ${\bm u}''(t)=A{\bm u}(t)+{\bm
  g}(t)$ with initial values ${\bm u}(0)=u_0$ and $\tilde{\bm
  u}(0)=\tilde{\bm u}_0$, we can introduce ${\bm v}(t)={\bm u}'(t)$ to
transform them into a larger first-order system of ODEs, and then
apply ParaDiag II. However, this approach doubles the memory
requirements at each time step, which can be problematic in cases of
very fine spatial mesh sizes or for high-dimensional problems. In that
case, it can be preferable to discretize directly the second order
problem, and we consider the symmetric two-step method
\begin{equation}\label{r1r2_2nd}
{\small
\begin{split}
r_1(\Delta t^2A){\bm u}_{n+1}-r_2(\Delta t^2A){\bm u}_n+r_1(\Delta t^2A){\bm u}_{n-1}=\tilde{\bm g}_n,~n=1,\dots, N_t-1,
\end{split}}
\end{equation}
assuming that the second initial value ${\bm u}_1$ is given.
Examples of the matrices $r_1$ and $r_2$  are
\begin{equation*}
\begin{split}
&r_1(\Delta t^2A)=I_x-\frac{\Delta t^2A}{12}+\frac{10\gamma(\Delta t^2A)^2}{12},\\
&r_2(\Delta t^2A)=
2 I_x+\frac{10\Delta t^2A}{12}+\frac{20\gamma(\Delta t^2A)^2}{12},
\end{split}
\end{equation*}
if we use the Numerov-type method from \eqref{Numerov} as the
time-integrator. For \eqref{r1r2_2nd}, the all-at-once matrix and the
corresponding \pd are
\begin{subequations}
\begin{equation}\label{KP_2nd_a}
\begin{split}
&\CK=\tilde{B}\otimes r_1(\Delta t^2A)-B\otimes r_2(\Delta t^2A),\\
&\CP_\alpha=\tilde{C}_\alpha\otimes r_1(\Delta t^2A)-C_\alpha\otimes r_2(\Delta t^2A),
\end{split}
\end{equation}
where $B$ is the Toeplitz matrix from \eqref{ParaDiagII_AAA_b}, and
$C_\alpha$ is the $\alpha$-circulant matrix of $B$ (see
\eqref{ParaDiagII_AAA_c}). The matrices $\tilde{B}$ and
$\tilde{C}_\alpha$ are defined as
\begin{equation}\label{KP_2nd_b}
\begin{split}
\tilde{B}:=
\begin{bmatrix}
1 & & & &\\
0 &1 & & &\\
1 &0 &1 & &\\
&\ddots &\ddots &\ddots &\\
& &1 &0 &1
\end{bmatrix},\quad
\tilde{C}_\alpha:=
\begin{bmatrix}
1 & & &\alpha &\\
0 &1 & & &\alpha\\
1 &0 &1 & &\\
&\ddots &\ddots &\ddots &\\
& &1 &0 &1
\end{bmatrix}.
\end{split}
\end{equation}
\end{subequations}
According to \eqref{diag_Calp_a}-\eqref{diag_Calp_b}, we can
simultaneously diagonalize $C_\alpha$ and $\tilde{C}_\alpha$. Thus,
for the stationary iteration \eqref{ParaDiagII_SI}, we can solve the
preconditioning step $\CP_\alpha^{-1}{\bm r}^k$ using the
diagonalization procedure (cf. \eqref{ParaDiagII3step_alp}) as well.

The preconditioner $\CP_{\alpha}$ used in the ParaDiag II method
involves substituting the Toeplitz matrix within the all-at-once
matrix $\CK$ with a circulant (or $\alpha$-circulant) matrix, while
keeping the space matrices unchanged. This substitution, which
approximates a pointwise Toeplitz matrix $B$ by a circulant (or
$\alpha$-circulant) matrix $C$, is a natural approach that dates back
to \cite{Strang86}. The spectrum of the preconditioned matrix
$C^{-1}B$ has been extensively examined by researchers over the past
three decades, yielding fruitful results; see the survey paper
\cite{CN96} and the monographs \cite{Ng04,BLM05} for more details.

For blockwise Toeplitz matrices, where all blocks are Toeplitz
(referred to as BTTB matrices), the circulant preconditioner is
obtained by approximating each block by a circulant matrix, analogous
to the approach used in ParaDiag II. Spectral analyses of such
preconditioned matrices can be found in \cite{CN96} and
\cite{Ng04}. However, in the context of ParaDiag II, the blocks (e.g.,
$r_1(\Delta t A)$ and $r_2(\Delta tA)$ in \eqref{r1r2}) are not
Toeplitz. In this scenario, there is a lack of systematic results
regarding the eigenvalues of $\CP_{\alpha}^{-1}\CK$, and the work in
\cite{mcdonald2018preconditioning} explores this for $\alpha=1$.

Since \cite{mcdonald2018preconditioning} and
\cite{gander2019convergence}, a lot of efforts have been put into
analyzing the spectrum of $\CP_{\alpha}^{-1}\CK$. Examples include
\cite{GWJCP20,LN21,WZ21,WZhou21,DSW22,BSM23,HPer2024} for
parabolic problems and \cite{DW21,liu2020fast} for hyperbolic
problems. The analyses are intricate and rely heavily on special
properties of the time-integrator, such as sparsity, Toeplitz
structure, and diagonal dominance of the time stepping matrix.

A comprehensive spectral analysis of the preconditioned matrix
$\CP_{\alpha}^{-1}\CK$ for both first-order and second-order
problems can be found in \cite{WZZ22}, with results that hold for any
stable one-step time-integrator for first order systems of ODEs, and
two-step symmetric time-integrator for second order systems of ODEs.
\begin{theorem}\label{WZZ22}
{\em For the first-order system of ODEs ${\bm u}'(t)=A{\bm
    u}(t)+{\bm g}(t)$, if the one-step time-integrator
  \eqref{r1r2_1st} is stable, i.e., $|r_1^{-1}(z)r_2(z)|\leq 1$ for
  $z\in\sigma(\Delta tA)\subset\mathbb{C}^-$, then the eigenvalues of
  the preconditioned matrix satisfy
\begin{equation}\label{eigBound}
\frac{1}{1-\alpha}\leq|\lambda(\CP_{\alpha}^{-1}\CK)|\leq\frac{1}{1+\alpha},
\end{equation}
where $\CK$ is the all-at-once matrix of the time-integrator
\eqref{r1r2_1st}, and $\CP_{\alpha}$ is the block $\alpha$-circulant
matrix given by \eqref{ParaDiagII_AAA_c} with $\alpha\in(0,
1)$. Similarly, for the second-order system of ODEs ${\bm
  u}''(t)=A{\bm u}(t)+{\bm g}(t)$, if the two-step method
\eqref{r1r2_2nd} is stable, i.e., $|r_1^{-1}(z)r_2(z)|\leq 2$
($\forall z\in\sigma(\Delta t^2A)\subset\mathbb{R}^-$ and
$|r_1^{-1}(z)r_2(z)|=2$ only if $z=0$), then the eigenvalues of the
preconditioned matrix $\CP_{\alpha}^{-1}\CK$ (with $\CP$ and $\CK$
given by \eqref{KP_2nd_a}-\eqref{KP_2nd_b}) also satisfy the bounds in
\eqref{eigBound}.  }
\end{theorem}
For the stationary iteration \eqref{ParaDiagII_SI}, the iteration
matrix $\CM$ is given by
\begin{equation}\label{Iter_Mat}
\CM= \CI- \CP^{-1}_{\alpha}\CK,
\end{equation}
and based on \eqref{eigBound}, we get $\rho(\CM) \leq
\frac{\alpha}{1-\alpha}$. This explains the faster convergence of the
ParaDiag II head-tail waveform relaxation method \eqref{ParaDiagII_WR}
when $\alpha$ is small, as we have seen in Figure
\ref{Fig_ParaDiagII_WR}. The stability of the underlying
time-integrator serves as a sufficient condition for the eigenvalue
bounds of the preconditioned matrix $\CP^{-1}_{\alpha}\CK$ in Theorem
\ref{WZZ22}, or equivalently, the iteration matrix
$\CM$. Numerically, we find that stability is also a necessary
condition. We illustrate this now for the Numerov-type method
\eqref{Numerov} applied to a second-order problem with $A$ being a
centered finite difference discretization of the Laplacian with
Dirichlet boundary conditions, i.e., $\frac{1}{\Delta x^2}A \approx
\partial_{xx}$. Setting $\Delta t = \frac{1}{16}$, $\Delta x =
\frac{1}{128}$, and $\alpha = 0.02$, we show in Figure
\ref{Fig_ParaDiagII_WaveStability}
\begin{figure}
  \centering
 \includegraphics[width=1.5in,height=1.15in,angle=0]{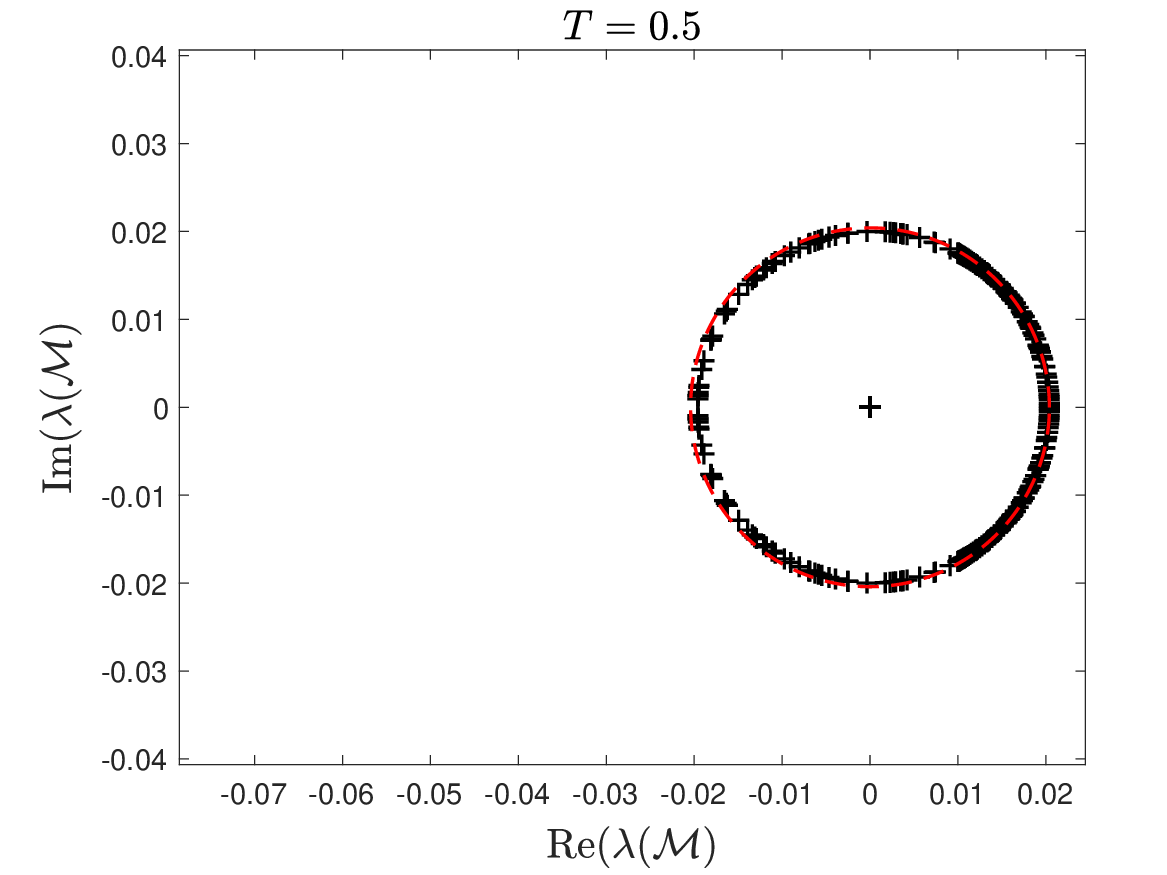}~ \includegraphics[width=1.5in,height=1.15in,angle=0]{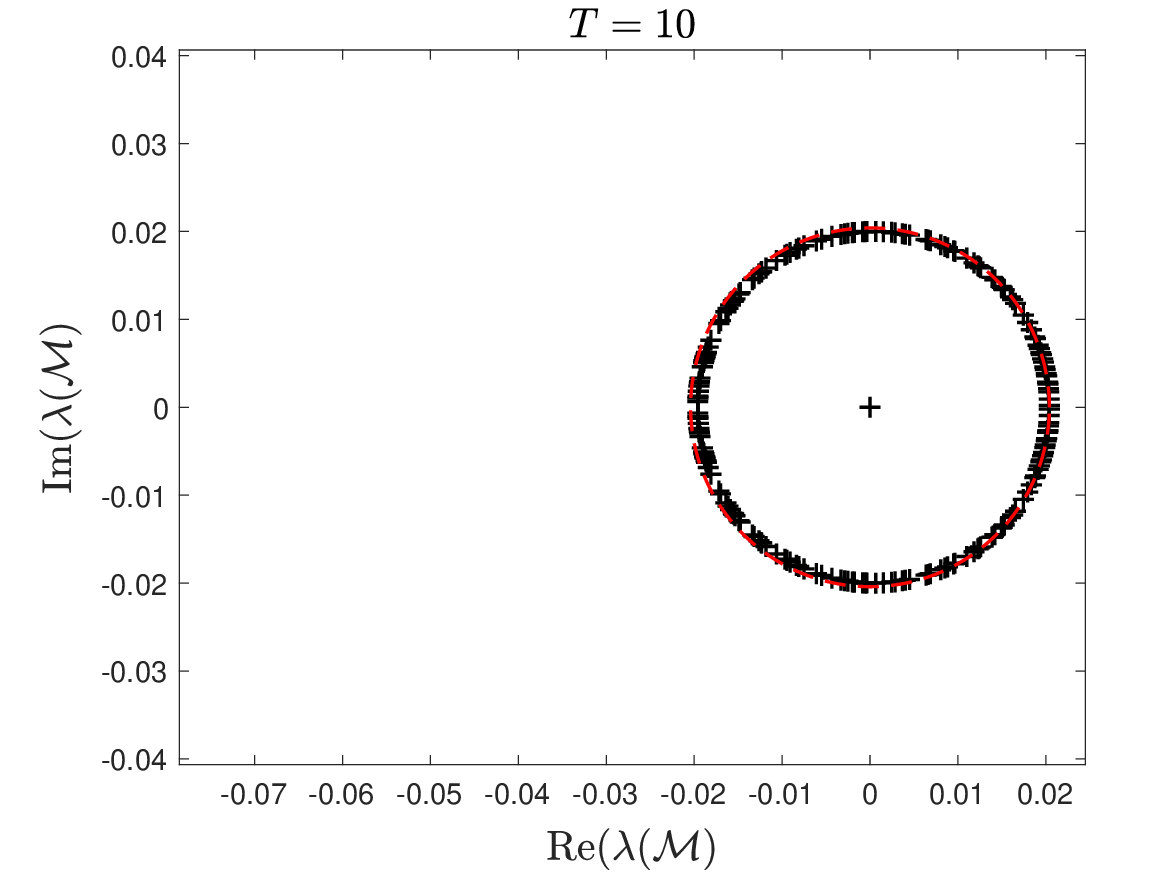}~ \includegraphics[width=1.5in,height=1.15in,angle=0]{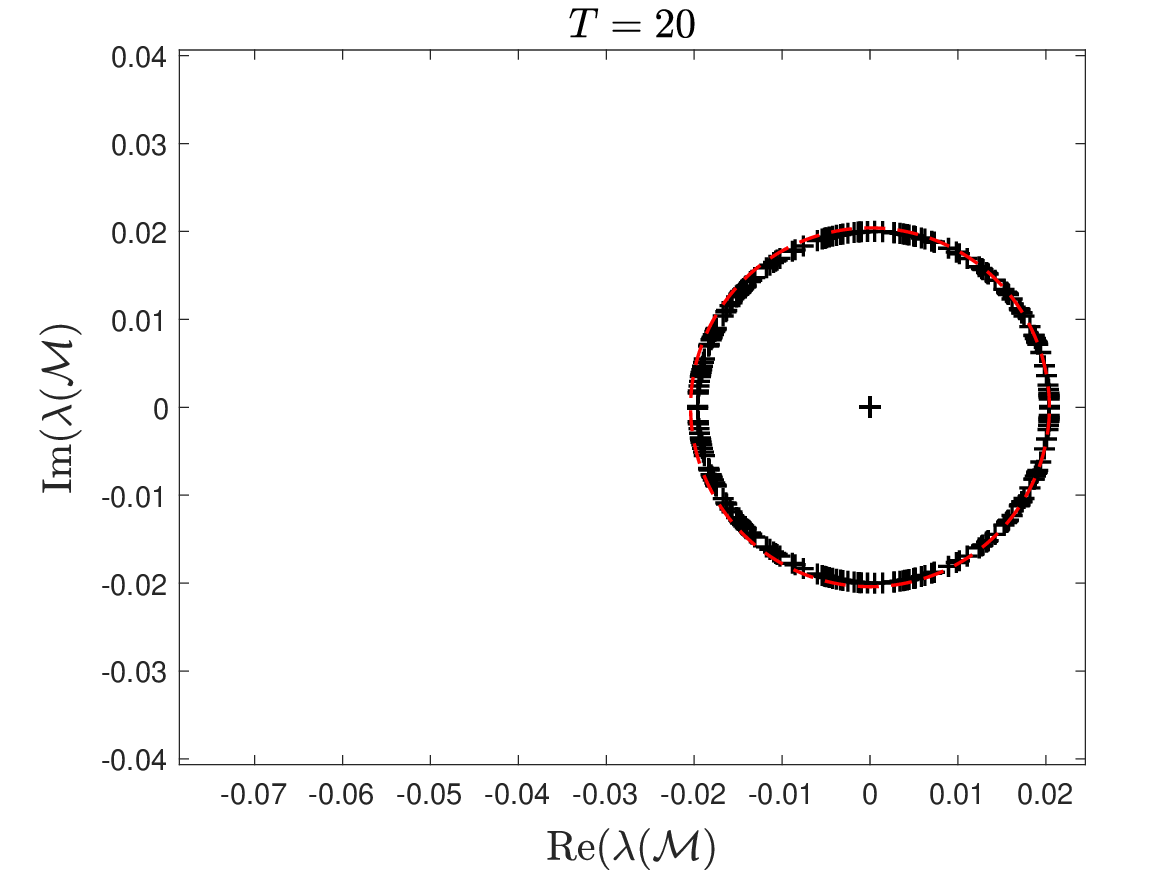}\\
 \includegraphics[width=1.5in,height=1.15in,angle=0]{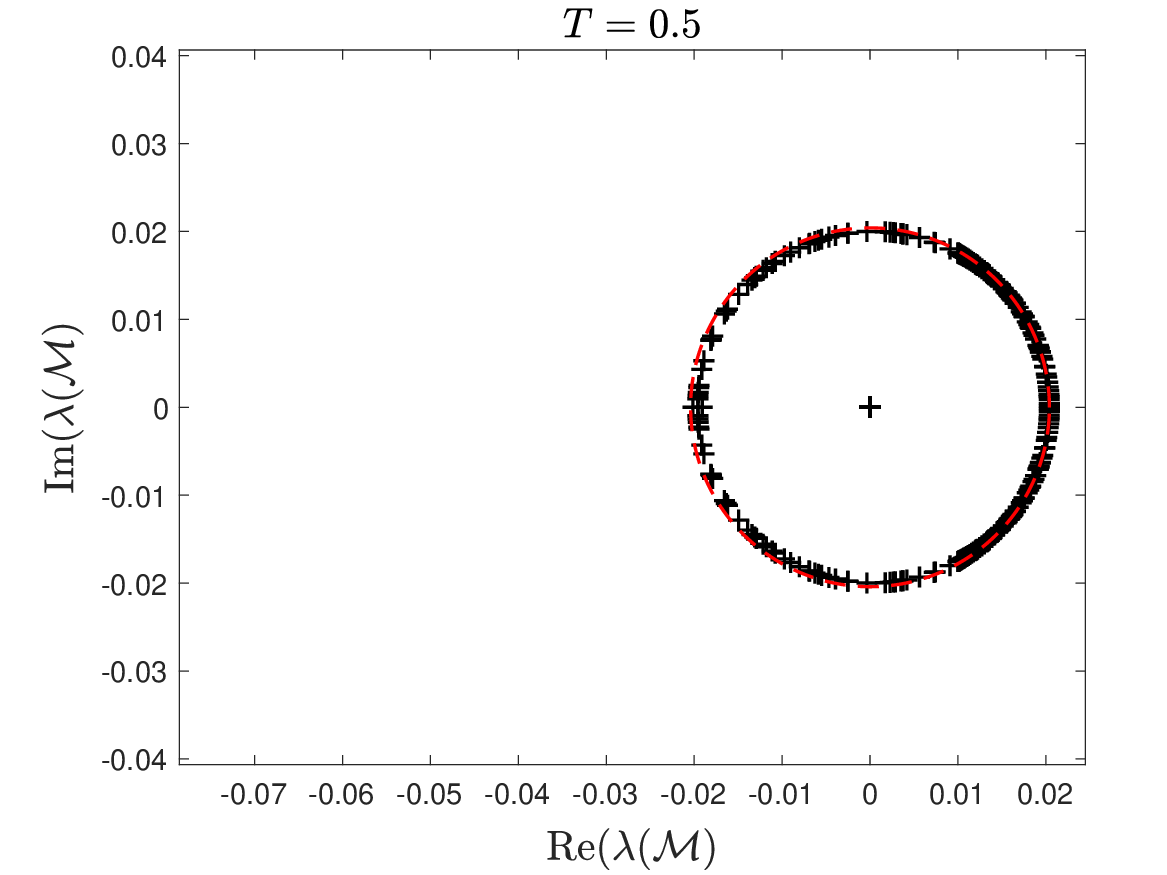}~ \includegraphics[width=1.5in,height=1.15in,angle=0]{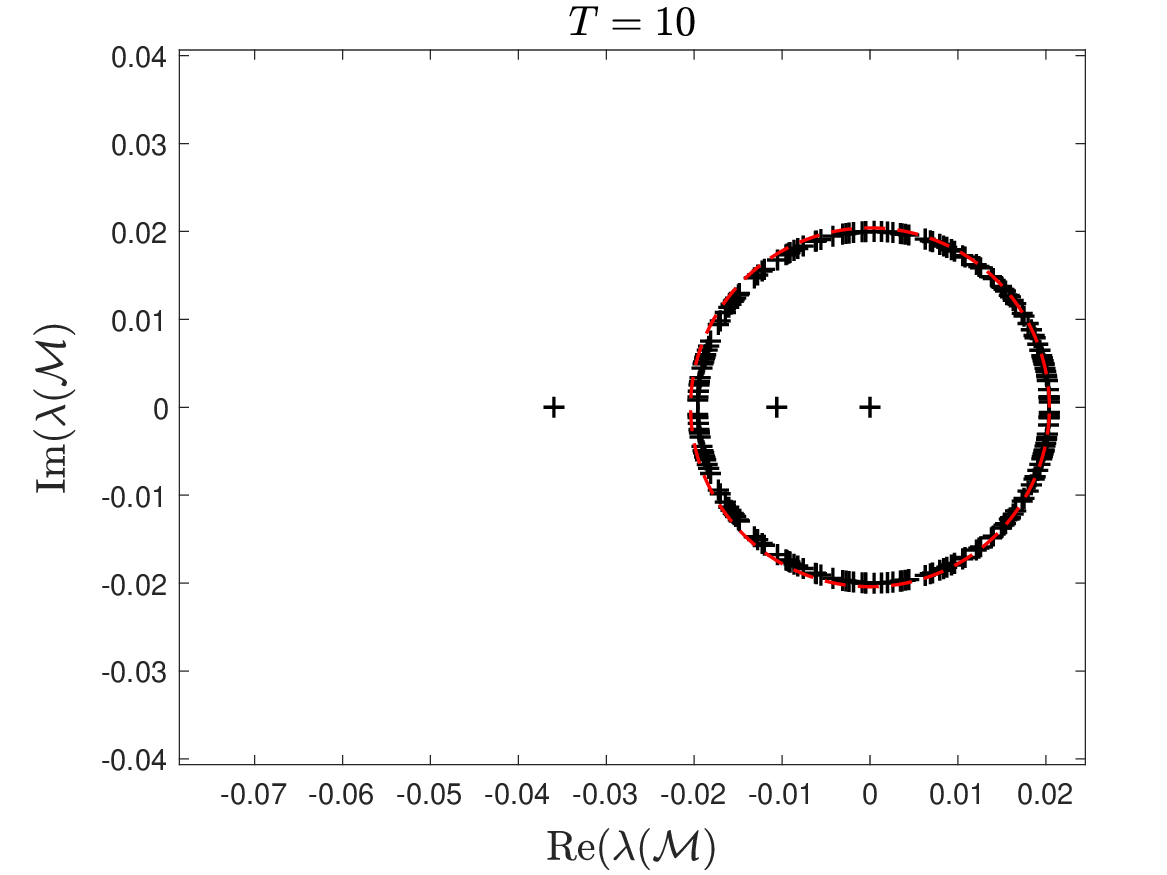}~ \includegraphics[width=1.5in,height=1.15in,angle=0]{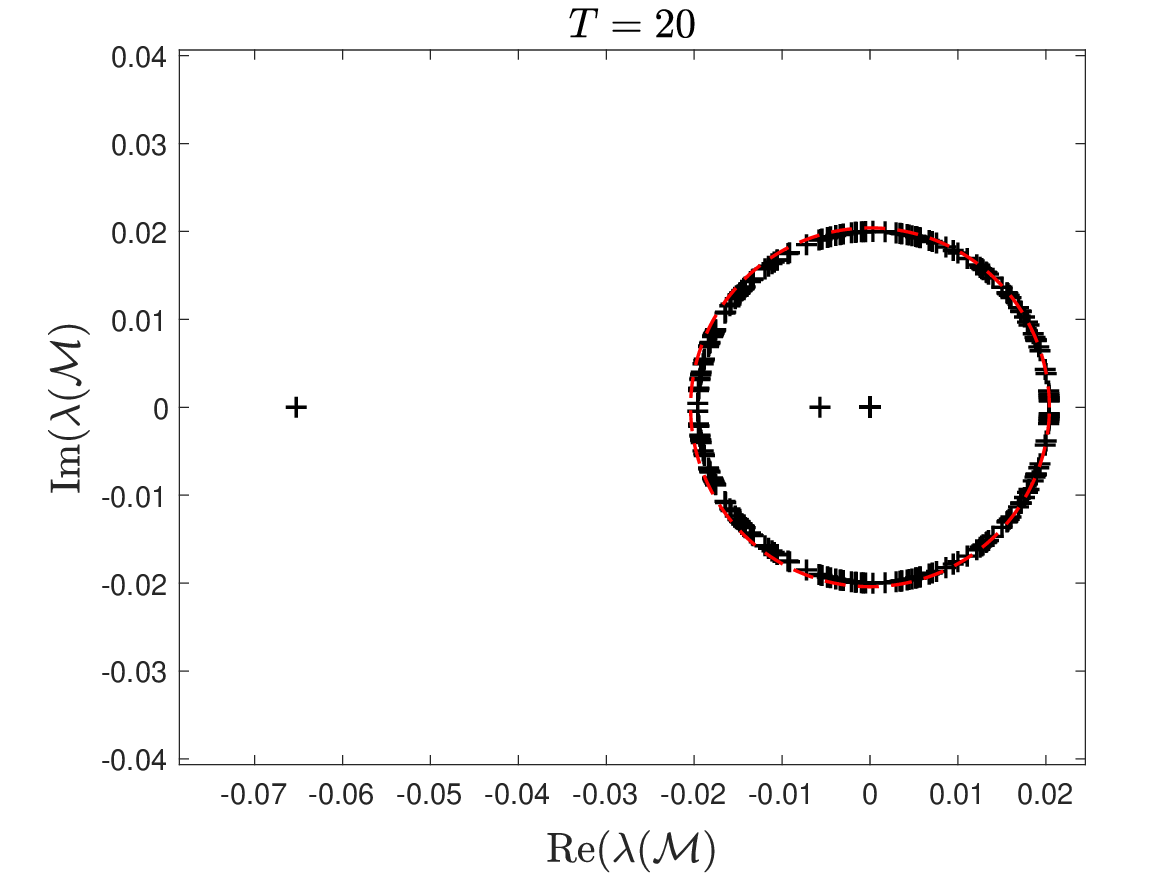}
  \caption{Eigenvalues of the iteration matrix $\CM$
    (cf. \eqref{Iter_Mat}) for the wave equation
    \eqref{WaveEquation1d}. Top row: $\gamma = \frac{1}{120}$. Bottom
    row: $\gamma = \frac{1}{120.01}$. In each panel, the dashed line
    represents the circle with radius $\frac{\alpha}{1-\alpha}$.}
  \label{Fig_ParaDiagII_WaveStability}
\end{figure}
the eigenvalues of the iteration matrix $\CM$ for $T = 0.5, 10,$ and
$20$, using the two values $\gamma = \frac{1}{120}$ and $\gamma =
\frac{1}{120.01}$ for the Numerov-type method, where according to
\cite{Chawla83}, $\gamma = \frac{1}{120}$ represents a stability
threshold for the Numerov-type method. For this threshold value (top
row of Figure \ref{Fig_ParaDiagII_WaveStability}), all eigenvalues of
$\CM$ lie within the theoretically analyzed circle. However, with
$\gamma = \frac{1}{120.01}$ (slightly below the threshold), the
Numerov-type method looses unconditional stability, and the results in
the bottom row clearly indicate that the eigenvalue bounds
\eqref{eigBound} no longer hold for large $T$.

The eigenvalue bounds \eqref{eigBound} indicate that the ParaDiag II
method converges faster when $\alpha$ decreases. This is true within a
certain range of $\alpha$, such as $\alpha\in[10^{-3}, 10^{-1}]$, as
shown earlier (cf. Figure \ref{Fig_ParaDiagII_WR}). However, $\alpha$
cannot be arbitrarily small due to roundoff errors arising from
diagonalizing the $\alpha$-circulant matrix. Specifically, for any
diagonalizable square matrix $P$, floating point operations limit the
precision of its factorization $P \approx VDV^{-1}$. For the
$\alpha$-circulant matrix $C_\alpha$ (cf. \eqref{ParaDiagII_AAA_c}),
even though its eigenvalues and eigenvectors have closed-form
expressions, the difference between $C_\alpha$ and $V_\alpha D_\alpha
V^{-1}_\alpha$ grows linearly as $\alpha$ decreases,
the roundoff error $\text{err}_{\text{ro}}$ behaves like
$$
\text{err}_{\text{ro}} = \mathcal{O}(\epsilon \text{Cond}_2(V_\alpha)) = \mathcal{O}\left(\frac{\epsilon}{\alpha}\right),
$$
where $\epsilon$ is the machine precision (e.g., $\epsilon = 2.2204
\times 10^{-16}$ for double precision), and the equality follows from
\cite{gander2019convergence} and the fact that $V_\alpha =
\Lambda_\alpha \text{F}^*$ (cf. \eqref{diag_Calp_b}) with
$\text{Cond}_2(\text{F}^*) = 1$, implying $\text{Cond}_2(V_\alpha) =
\frac{1}{\alpha}$. Interestingly, this does not necessarily imply a
similar growth in the roundoff error of the ParaDiag II method. The
error behavior depends on the implementation: directly solving for
${\bm U}^{k}$ as in
\eqref{ParaDiagII_AAA_WR_a}-\eqref{ParaDiagII_AAA_WR_b} may lead to
the mentioned growth, while first solving the error equation for
$\Delta {\bm U}^{k-1}$ and then updating ${\bm U}^k$
(cf. \eqref{ParaDiagII_SI}) can significantly mitigate the roundoff
error; see Figure \ref{Fig_ParaDiagII_Roundoff} for illustration.
\begin{figure}
  \centering
 \includegraphics[width=2.3in,height=1.85in,angle=0]{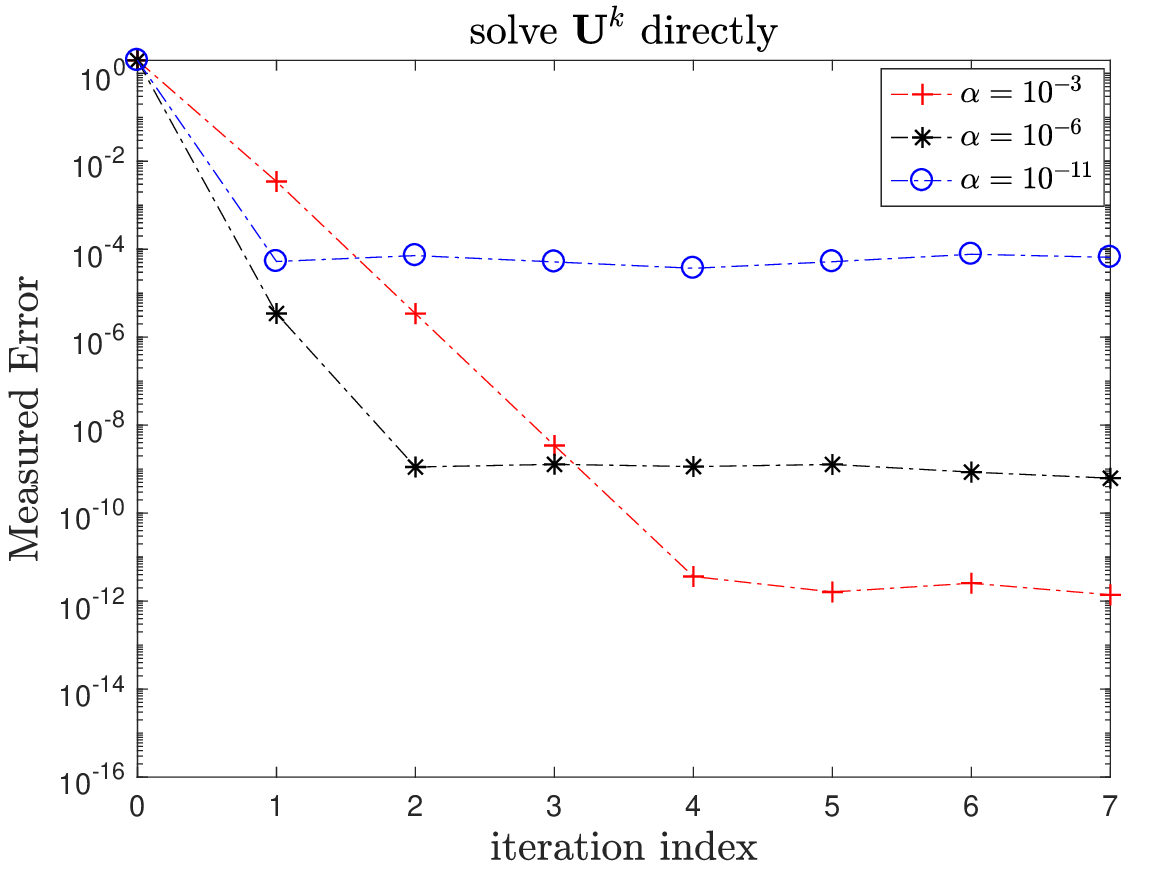}~ \includegraphics[width=2.3in,height=1.85in,angle=0]{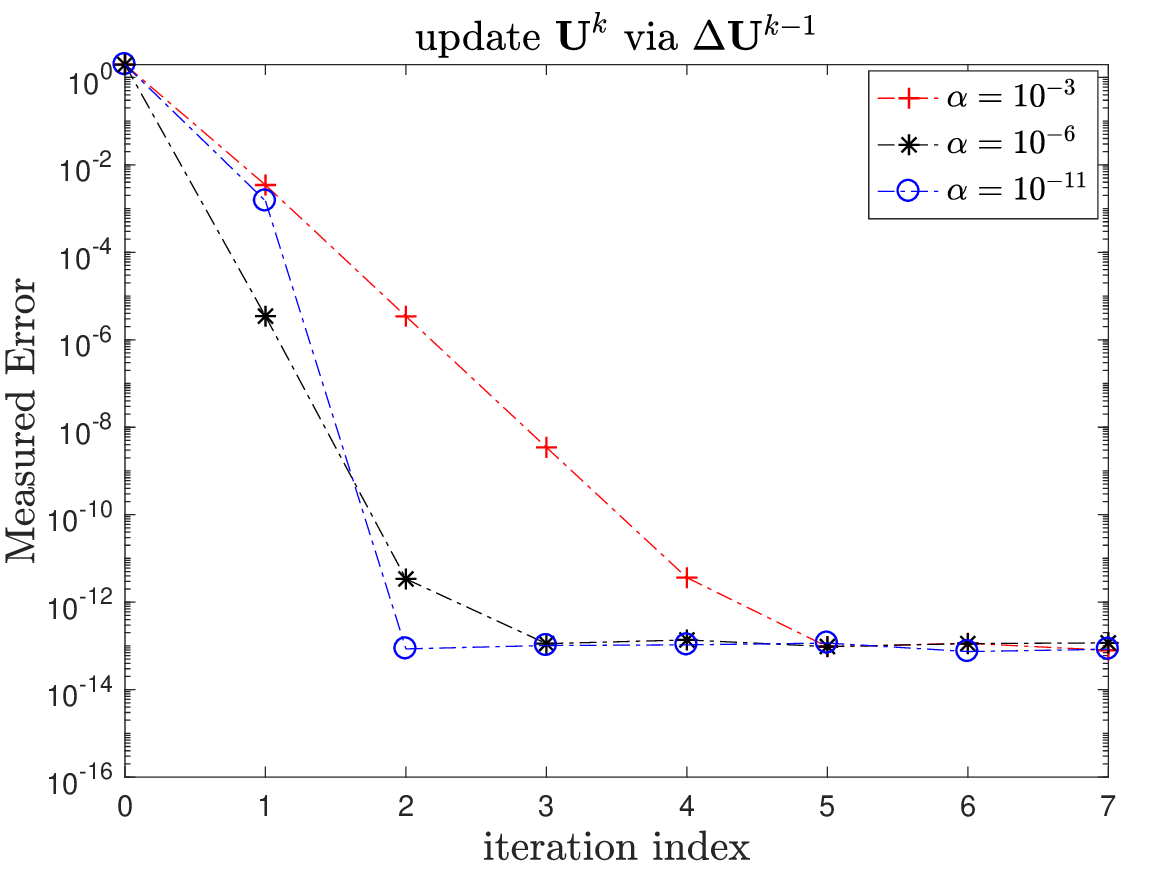} 
 \caption{ Measured error of ParaDiag II for the wave equation
   \eqref{WaveEquation1d} implemented in two modes for three values of
   $\alpha$.  }
  \label{Fig_ParaDiagII_Roundoff}
\end{figure}
A comprehensive study of the roundoff error for ParaDiag II
will appear in \cite{WYZ24}.

So far, we have only considered one-step and symmetric two-step
time-integrators. For general multistep methods, the eigenvalues of
the preconditioned matrix $\CP_\alpha^{-1}\CK$ do not necessarily
satisfy \eqref{eigBound}, yet we can demonstrate analogous
results. For example the case where $\CK=B\otimes I_x-I_t\otimes A$,
with $B$ being a \emph{dense} lower triangular Toeplitz matrix was
studied in \cite{GWJCP20}. This matrix arises in solving Volterra
partial integro-differential equations, and its first column
${\bm\omega}=(\omega_0, \omega_1,\dots, \omega_{N_t})$ is determined
by the quadrature used to handle the integral term. The authors
established a bound for the eigenvalues of the form
$|\lambda(\CP_\alpha^{-1}\CK)|= 1+\mathcal{O}(\alpha)$, provided that
the quantities $\{\omega_n\}$ satisfy certain conditions, such as
positivity and monotonicity.

Turning to nonlinear problems, the application of ParaDiag II closely
resembles that of ParaDiag I. We illustrate this for the first-order
problem ${\bm u}'(t)=f(t,{\bm u}(t))$ discretized using Backward Euler
with a step size $\Delta t$. Initially, we apply Newton's iteration
to the nonlinear all-at-once system,
\begin{equation}\label{ParaDiagII_nonlinear}
\mathcal{J}\Delta{\bm U}^{l}={\bm b}-F({\bm U}^l), \quad {\bm U}^{l+1}={\bm U}^l+\Delta{\bm U}^l,
\end{equation}
where the Jacobian matrix $\mathcal{J} := B\otimes I_x - \nabla F_l$
(cf. \eqref{NewtonIt_a}-\eqref{NewtonIt_b}), with
\begin{equation*}
\begin{split}
\nabla F_l &= {\rm blkdiag}(\nabla f({\bm u}_1^l, t_1), \dots, \nabla f({\bm u}_{N_t}^l, t_{N_t})), \\
B &= \frac{1}{\Delta t}\begin{bmatrix}
1 & & & \\
-1 & 1 & & \\
& \ddots & \ddots & \\
& & -1 & 1
\end{bmatrix}.
\end{split}
\end{equation*}
Then, we solve \eqref{ParaDiagII_nonlinear} with GMRES using
$\CP_\alpha = C_\alpha\otimes I_x - I_t\otimes A_l$ as preconditioner,
where $C_\alpha$ is the $\alpha$-circulant matrix of $B$ and $A_l$ is
the average matrix of $\{\nabla f({\bm u}_n^l, t_n)\}$. In general, we
cannot use the stationary iteration \eqref{ParaDiagII_SI} to solve the
Jacobian system \eqref{ParaDiagII_nonlinear} since
$\rho(\CP_\alpha^{-1}\mathcal{J}) > 1$. However, the eigenvalues of
$\CP_\alpha^{-1}\mathcal{J}$ are clustered, which is good for
GMRES. The eigenvalue distribution of $\CP_\alpha^{-1}\mathcal{J}$ is
influenced by the length of the time interval $T$, with a shorter
$T$ leading to more clustered eigenvalues and thus faster GMRES
convergence. Numerical evidence supporting this aspect can be found in
\cite{gander2019convergence} and \cite{WZZ22}. Additionally, we can
leverage the nearest Kronecker product approximation introduced in
Section \ref{sec3.6.1} to accelerate convergence; see \cite{LWu22}.

\section{PinT methods designed for parabolic problems} \label{Sec4}

We have shown in Section \ref{Sec2} intuitively why realizing PinT
computations for hyperbolic problems is more challenging than for
parabolic problems: parabolic problems tend to have local solutions in
time, except for very low frequency components, whereas hyperbolic
problems have highly non-local solutions in time, and this over all
frequency components, from the lowest to the highest
ones. Nevertheless, we have shown in Section \ref{Sec3} PinT methods
that are effective for hyperbolic problems, and thus tackle all
frequency components in a non-local way in time. Naturally, these
methods perform then often even better when applied to parabolic
problems, since they tackle all frequency components over long time,
which includes the few very low frequency components that are highly
non-local in time in parabolic problems.  The methods we have seen so
far were however often designed for linear problems, where they are
most effective, whereas for nonlinear problems, they all suffer from
certain drawbacks. For example, for OSWR it is not easy to
determine the optimized Robin parameters, and without a reasonable
parameter, the convergence rate can be quite poor. For ParaExp and
ParaDiag I and II, non-linearity also affects the convergence rate of
the Newton iteration used as an outer solver, and in particular,
the Newton iteration may converge slowly or even diverge when the time
interval is large. We show in this section now PinT methods that were
designed for parabolic problems and take advantage of their properties
to be local in time as we have seen in Section \ref{Sec2}, and they work
equally well for linear and non-linear problems. They have entirely
different convergence mechanisms and properties from the methods in
Section \ref{Sec3}, and a direct application of these methods to
hyperbolic problems often leads to slow convergence or even
divergence.

\subsection{Historical development}

The first method we want to introduce is the Parareal algorithm
from \cite{Lions:2001:PTD}, which we have mentioned already in
  Section \ref{Sec3.5} to describe the nonlinear ParaExp variant. Even
  though Parareal was invented independently, it has its roots in
  earlier work on multiple shooting techniques for evolution problems,
  see \cite{Bellen:1989:PAI,Chartier:1993:APS}, and the algorithm was
  presented already in \cite{Saha:1996:API} with a coarse model
  instead of a coarse grid in the context of solar system simulations,
  mentioning a relation to Waveform Relaxation. A very early precursor
  is even \cite{Nievergelt:1964:PMI}, although the method there is not
  iterative.  Parareal, proposed 20 years ago, has attracted
  considerable attention in scientific and engineering
  computations. The convergence of Parareal is very well understood,
  see e.g.
  \cite{gander2007analysis,gander:2008:nca,gander2014analysis,Gander:TPTI:2024}. In
  a sense, Parareal can be regarded as a template for developing more
  efficient PinT methods. There are numerous modifications of Parareal
  in the literature to make it applicable to different problems or for
  different purposes. Interesting examples are the {\em Parallel
    Implicit Time integration Algorithm} (PITA), see
  \cite{farhat2003time,farhat2006time,cortial2009time}, the {\em
    Parallel Full Approximation Scheme in Space-Time} (PFASST), see
    \cite{minion2011hybrid,EM12,minion2015interweaving}, 
    {\em Multigrid Reduction in Time} (MGRiT), see
    \cite{FFK14,dobrev2017two,hessenthaler2020multilevel}, and also
  combinations of Parareal with ParaDiag
  \cite{WSiSC18,gander2020diagonalization}. We present the convergence
  mechanisms and convergence properties of Parareal and its variants
  in this section. The basic feature of PinT methods based on Parareal
  are that they use two grids (or more) for the time discretization,
  while for space \ds they use just one grid.  The idea of using
    multigrid in both space and time is going back to the {\em parabolic
    multigrid method} in \cite{Hackbusch:1984:PMG} with an elegant
    analysis in the form of {\em multigrid waveform relaxation} in
    \cite{Lubich:1987:MGD}. Coarsening in time was not effectively
    possible in this approach, and important improvements using
    multigrid techniques for highly advective problems were proposed
    in
\cite{vandewalle1994space,horton1995space,janssen1996multigrid,van2002multigrid}.
    A new {\em Space-Time MultiGrid} (STMG) method using just standard
    components but as a new main ingredient a block Jacobi smoother in
    time was introduced and analyzed in \cite{Gander:2016:AOANST}, and
    this is currently one of the most powerful PinT algorithms for
    parabolic problems, with excellent strong and weak scalability
    properties, see also \cite{neumueller2019time}. We will introduce
    STMG at the end of this section, and show its effectiveness also
    for non-linear parabolic problems.

\subsection{Parareal}\label{Sec4.1}

The Parareal algorithm proposed in \cite{Lions:2001:PTD} is a
non-intrusive time-parallel solver that is based on multiple shooting,
although it was not invented in this context, but in the context of
virtual control. In Parareal, the Jacobian in Newton's method used to
solve the shooting equations is approximated by a finite difference
across two iterates on a coarser grid or model
\cite{gander2007analysis}. Similar to ParaExp, it is based on a time
decomposition of the interval $(0,T)$ into several smaller time
intervals $0=T_0<T_1<\ldots<T_N=T$, with for example $T_n=T_0+n\Delta
t$. However, in contrast to ParaExp, it uses an iteration, starting
with an initial guess $\vec{U}_n^0$ at $T_n$. For iteration index
$k=0,1,\ldots$, Parareal computes improved approximations using the
update formula
\begin{equation}\label{Parareal}
  {\bm u}_{n+1}^{k+1} =   \CF(T_n, T_{n+1},{\bm u}_n^{k})
  +\CG(T_n, T_{n+1},  {\bm u}_n^{k+1})
  -\CG(T_n, T_{n+1},{\bm u}_n^k).
\end{equation}
Here, $\CF(T_n, T_{n+1}, {\bm u}_n^{k})$ represents an accurate solver
that uses a smaller step size $\Delta t$ for the underlying evolution
problem, with initial condition ${\bm u}_n^{k}$ at time $t=T_n$,
yielding an approximate solution at time $t=T_{n+1}$. Similarly, $\CG$
is a less expensive and less accurate solver that uses a larger step
size, for example $\Delta T$, or a simpler model, and the difference
of the two $\CG$ terms in \eqref{Parareal} represents precisely the
approximation of the Jacobian, see \cite{gander2007analysis}. Note
that in (\ref{Parareal}), all the computationally expensive $\CF$
solves can be performed in parallel, since at iteration $k$, the
approximations ${\bm u}_n^{k}$ are all known. For simplicity, we
consider here uniform fine and coarse time grids, assuming that each
large step size contains $J$ small steps, i.e., $\Delta T/\Delta
t=J\geq2$; see Figure \ref{twoGrids}.
\begin{figure}
  \centering
 \includegraphics[width=4in,height=1in,angle=0]{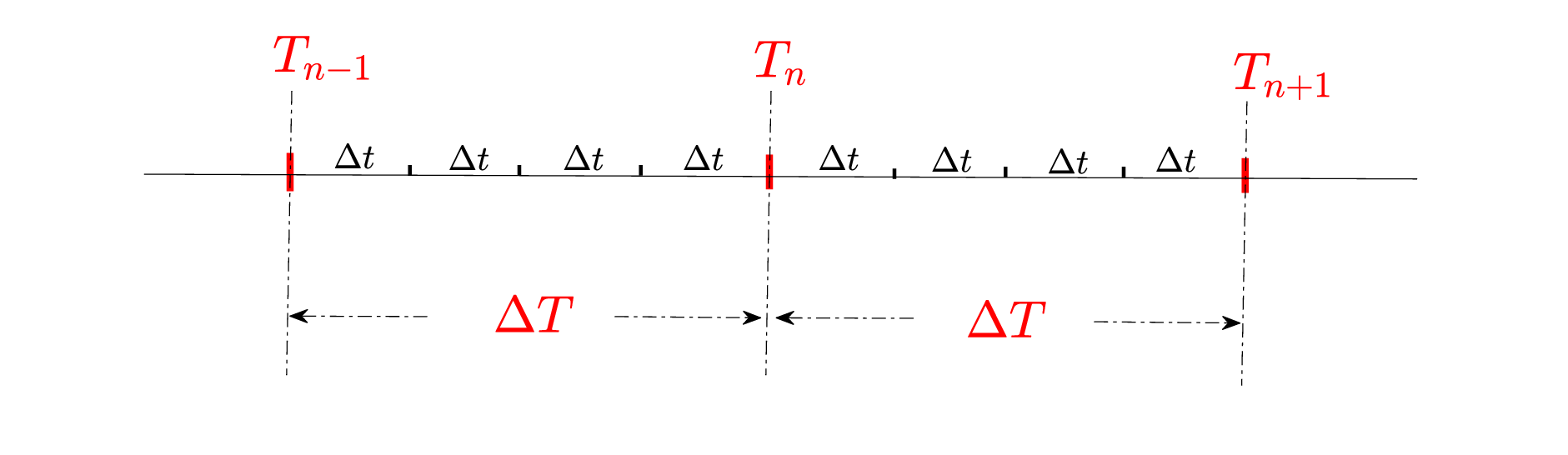} 
  \caption{Parareal uses two time grids, where each large time step size
    $\Delta T$ contains $J$ small time step sizes $\Delta t$.}
  \label{twoGrids}
\end{figure}
However, in principle, it is straightforward to apply
non-uniform time grids in Parareal, see
\cite{gander17,MMjcam20,WZ24}).

The convergence of Parareal is well understood, both for linear
problems \cite{gander2007analysis} and non-linear problems
\cite{gander:2008:nca}: it converges super-linearly on bounded time
intervals and linearly for parabolic problems on arbitrarily long time
intervals. Specifically, for the linear problem \eqref{linearODE},
i.e., ${\bm u}'(t)=A{\bm u}(t)+{\bm g}(t)$ with initial value ${\bm
  u}(0)={\bm u}_0$, and assuming that the fine and coarse solvers
$\CG$ and $\CF$ are  one-step  time integrators with stability
functions ${\rm R}_g(z)$ and ${\rm R}_f(z)$, we have the following
convergence results:

\begin{theorem}\label{PararealLinear}
  {\em Let $\{{\bm u}_{n}\}_{n=1}^{N_t}$ be the solutions computed
    sequentially by the $\CF$ solver, ${\bm u}_{n+1}=\CF(T_n, T_{n+1},
    {\bm u}_n)$. Suppose the matrix $A$ of the linear system of ODEs
    \eqref{linearODE} is diagonalizable, $A=V_ADV_A^{-1}$, and the
    coarse solver $\CG$ is stable, i.e., $|{\rm R}_g(z)|\leq 1$ for
    $z\in\sigma(\Delta TA)$. Then Parareal satisfies the convergence
    estimate
\begin{equation}\label{linearError}
{\small
\begin{split}
\max_{1\leq n\leq N_t}\|V_A({\bm u}_n^k-{\bm u}_n)\|_\infty\leq \max_{z\in\sigma(\Delta TA)}\|M^k(z)\|_\infty \max_{1\leq n\leq N_t}\|V_A({\bm u}_n^0-{\bm u}_n)\|_\infty,
\end{split}}
\end{equation}
where $M(z)$ is a Toeplitz matrix given by two matrices $M_g(z)$ and $M_f(z)$,  
\begin{equation}\label{matMz}
\begin{split}
&M(z):=M_g^{-1}(z)[M_g(z)-M_f(z)],\\
&M_g(z):=
\begin{bmatrix}
1 & & &\\
-{\rm R}_g(z) &1 & &\\
&\ddots &\ddots &\\
& &-{\rm R}_g(z) &1
\end{bmatrix},\\
&M_f(z):=
\begin{bmatrix}
1 & & &\\
-{\rm R}^J_f(z/J) &1 & &\\
&\ddots &\ddots &\\
& &-{\rm R}_f^J(z/J) &1 
\end{bmatrix}. 
\end{split}
\end{equation}
}
\end{theorem}
\begin{proof}
Applying the Parareal iteration \eqref{Parareal} to the system of ODEs
yields
$$
{\bm u}^{k+1}_{n+1}={\rm R}_f^J(\Delta TA/J){\bm u}^{k}_{n}
+{\rm R}_g(\Delta TA){\bm u}^{k+1}_{n}-{\rm R}_g(\Delta TA){\bm u}^{k}_{n}, \ n=0,1,\dots, N_t-1. 
$$
Since the overall fine solution computed sequentially by the $\CF$ solver
satisfies ${\bm u}_{n+1}=\CF(T_n, T_{n+1}, {\bm
  u}_n)$, it also satisfies by adding and subtracting the same term
$$
{\bm u}_{n+1}={\rm R}_f^J(\Delta TA/J){\bm u}_{n}
+{\rm R}_g(\Delta TA){\bm u}_{n}-{\rm R}_g(\Delta TA){\bm u}_{n}. 
$$
The error ${\bm e}_n^k:= {\bm u}_n-{\bm u}^k_n$ thus satisfies for
$k\geq0$ the error equation
$$
{\bm e}_{n+1}^{k+1}={\rm R}_g(\Delta TA){\bm e}_{n}^{k+1}+[{\rm R}_f^J(\Delta TA/J)-{\rm R}_g(\Delta TA)]{\bm e}_{n}^{k}, \quad n=0,1,\dots, N_t-1.
$$
For $n=0$,  ${\bm e}_0^k=0$, since the initial value is known. 
Since $A=V_ADV_A^{-1}$, we have 
$$
{\rm R}_g(\Delta TA)=V_A{\rm R}_g(\Delta TD)V_A^{-1},~{\rm R}_f^J(\Delta TA/J)=V_A{\rm R}_f^J(\Delta TD/J)V_A^{-1}.
$$
Hence, we obtain the error equations in scalar form,
$$
\xi_{n+1}^{k+1}(z)={\rm R}_g(z)\xi_{n}^{k+1}(z)+[{\rm R}_f^J(z/J)-{\rm R}_g(z)]\xi_{n}^{k}(z), \quad n=0,1,\dots, N_t-1,
$$
where $z=\Delta T\lambda$ with $\lambda$ being an arbitrary eigenvalue
of $A$ and $\xi_n^k(z)$ is the element of $V_A{\bm e}_n^k$
corresponding to $\lambda$. Clearly, it holds that
\begin{equation}\label{e_xi}
\|V_A{\bm e}^k_n\|_{\infty}={\max}_{z\in\sigma(\Delta TA)}|\xi_n^k(z)|. 
\end{equation}
Since $\xi^k_0(z)=0$ for $k\geq0$, we have $M_g(z){\bm
  \xi}^{k+1}(z)=[M_g(z)-M_f(z)]{\bm \xi}^k(z)$, which gives
$$
{\bm \xi}^{k+1}(z)=M^{-1}_g(z)[M_g(z)-M_f(z)]{\bm \xi}^k(z),
$$
where ${\bm \xi}^k(z)=(\xi_1^k(z), \xi_2^k(z), \dots, \xi_{N_t}^k(z))^\top$ for $k\geq0$. From \eqref{e_xi} we have
$$
\max_{1\leq n\leq N_t}\|V_A{\bm e}^k_n\|_{\infty}=\max_{z\in\sigma(\Delta TA)}\max_{1\leq n\leq N_t}|\xi_n^k(z)|=
\max_{z\in\sigma(\Delta TA)}\|{\bm \xi}^k(z)\|_\infty,
$$
which completes the proof of \eqref{linearError}. 
\end{proof}
From \eqref{linearError}, we see that the norm $\|M^k(z)\|_\infty$ represents the
convergence factor of the Parareal algorithm when applied to the
Dahlquist test equation $u'(t)=\lambda u(t)+g(t)$, where $\lambda$ is
an arbitrary eigenvalue of $A$.

\begin{remark}\label{ExplainParareal} 
{\em 
From \eqref{matMz} we can interpret the Parareal algorithm from the
perspective of a preconditioner by observing that
$$
 {M(z)=I_t-M_g^{-1}(z)M_f(z)}.
$$
For the Dahlquist test equation $u'(t)=\lambda u(t)+g(t)$, the matrix
$M_f(z)$ corresponds to the all-at-once matrix of the fine solver
$\CF$,
$$
M_f(z)U=b,
$$
where $U=(u_1,u_2,\dots, u_{N_t})^\top$ and $b$ is an appropriate
vector. Parareal can thus be written as
$$
M_g(z)\Delta U^k=r^k:=b-M_f(z)U^k,~U^{k+1}=U^k+\Delta U^k,
$$
and the parallelization stems from computing the residual $r^k$: given
$U^k$ from the previous iteration, all components of $r^k$ can be
computed simultaneously as $r^k_n=b_n-(u^k_n-\CF(T_{n-1}, T_n,
u_{n-1}^k))=b_n-(u^k_n-{\rm R}_f^J(z/J)u_{n-1}^k)$. This understanding
is valuable for designing new variants of Parareal, and we will
revisit this in Section \ref{Sec4.5}.  }
\end{remark}
Using $\|M^k(z)\|_\infty$ to predict the convergence behavior of
Parareal is not convenient, so we introduce the results given in
\cite{gander2007analysis}, which provide a very useful estimate of the
convergence rate. This involves examining the structure of the matrix
$M(z)$. Since
$$
M_g^{-1}(z)=
\begin{bmatrix}
1  & & &\\
{\rm R}_g(z) &1 & &\\
\vdots &\ddots &\ddots &\\
{\rm R}^{N_t-1}_g(z) &\dots &{\rm R}_g(z)  &1 
\end{bmatrix},
$$
we have
\begin{equation*}
\begin{split}
&M(z)=[{{\rm R}^J_f(z/J)-{\rm R}_g(z)}]\tilde{M}({\rm R}_g(z)),\\
&\tilde{M}(\beta):=
\begin{bmatrix}
0  & & & &\\
1 &0 & & &\\
\beta &1 &0 & &\\
\vdots &\ddots &\ddots &\ddots &\\
\beta^{N_t-2}  &\dots  &\beta&1  &0   
\end{bmatrix}. 
\end{split}
\end{equation*}
This implies that
$$
\|M^k(z)\|_\infty=|{{\rm R}^J_f(z/J)-{\rm R}_g(z)}|^k\|\tilde{M}^k({\rm R}_g(z))\|_\infty.
$$
The infinity norm of the matrix $\tilde{M}^k$ was studied
in \cite[Lemma 4.4]{gander2007analysis}, and the main result is
$$
\|\tilde{M}^k({\rm R}_g(z))\|_\infty\leq 
\begin{cases}
 \min\left\{\left(\frac{1-|{\rm R}_g(z)|^{N_t-1}}{1-|{\rm R}_g(z)|}\right)^k, \begin{pmatrix}N_t-1\\ k
 \end{pmatrix}\right\}, &{\rm if}~|{\rm R}_g(z)|<1,\\
  \begin{pmatrix}N_t-1\\ k
 \end{pmatrix}, &{\rm if}~|{\rm R}_g(z)|=1. 
\end{cases}
$$ 
Substituting this into \eqref{linearError} leads to two different
estimates of the convergence rate of the Parareal algorithm.
\begin{theorem}\label{Parareal_rho}
{\em With the same notation and assumptions used in Theorem
  \ref{PararealLinear}, the error of the $k$-th Parareal iteration
  satisfies
\begin{subequations}
 \begin{equation}\label{Error_rho_a}
\begin{split}
&\max_{1\leq n\leq N_t}\|{\bm e}^k_n\|_\infty\leq \max_{z\in\sigma(\Delta TA)}\varrho_s(J,z,N_t,k)\max_{1\leq n\leq N_t}\|{\bm e}_n^0\|_\infty,\\
&\varrho_s(J,z,N_t,k):=\frac{|{\rm R}_g(z)-{\rm R}^J_f(z/J)|^k}{k!}\prod_{j=1}^k(N_t-j), 
\end{split} 
\end{equation}
where  ${\bm e}^k_n=V_A({\bm u}_n^k-{\bm u}_n)$.  If   $|{\rm R}_g(z)|<1 (\forall z\in\sigma(\Delta TA))$,  it holds that 
 \begin{equation}\label{Error_rho_b}
\begin{split}
&\max_{1\leq n\leq N_t}\|{\bm e}^k_n\|_\infty\leq \max_{z\in\sigma(\Delta TA)}\varrho_l^k(J,z)\max_{1\leq n\leq N_t}\|{\bm e}_n^0\|_\infty,\\
&\varrho_l(J, z):=\frac{|{\rm R}_g(z)-{\rm R}^J_f(z/J)|}{1-|{\rm R}_g(z)|}. 
\end{split} 
\end{equation}
\end{subequations}
}
\end{theorem}
The estimate presented in \eqref{Error_rho_a} indicates that Parareal
converges super-linearly and completes iterations in at most $N_t$
steps, since $\rho_s=0$ when $k=N_t$. This estimate is particularly
suitable for {\em short} time intervals where $N_t$ is small. For
larger $N_t$, $\rho_s$ may not provide accurate predictions:
initially, $\rho_s$ increases, but the error actually decreases
uniformly. This is illustrated in Figure
\ref{Fig_Parareal_Heat} for the heat equation \eqref{heatequation}
with periodic boundary conditions, $g(x,t)=0$, and initial value
$u(x,0)=\sin^{2}(2 \pi x)$ for $x\in(0, 1)$.
\begin{figure}
  \centering
 \includegraphics[width=2.3in,height=1.85in,angle=0]{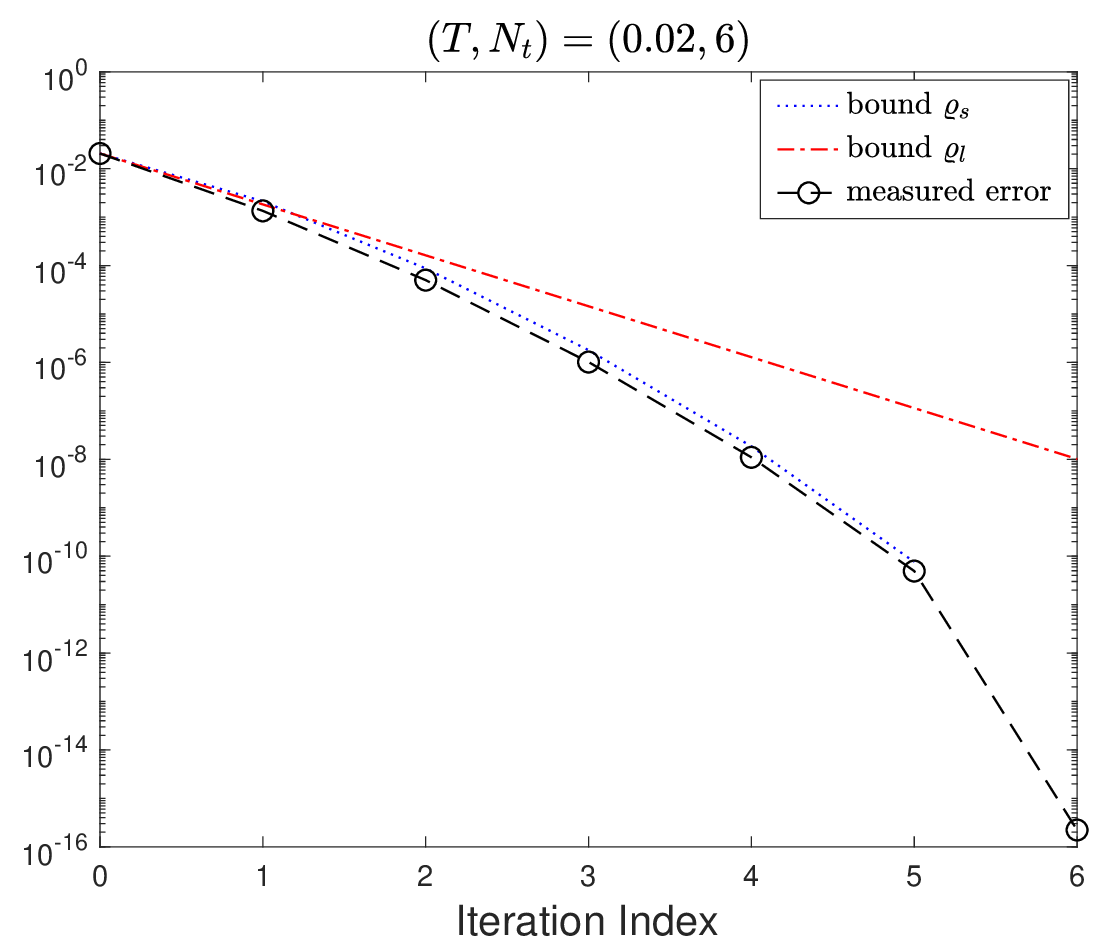}~ \includegraphics[width=2.3in,height=1.85in,angle=0]{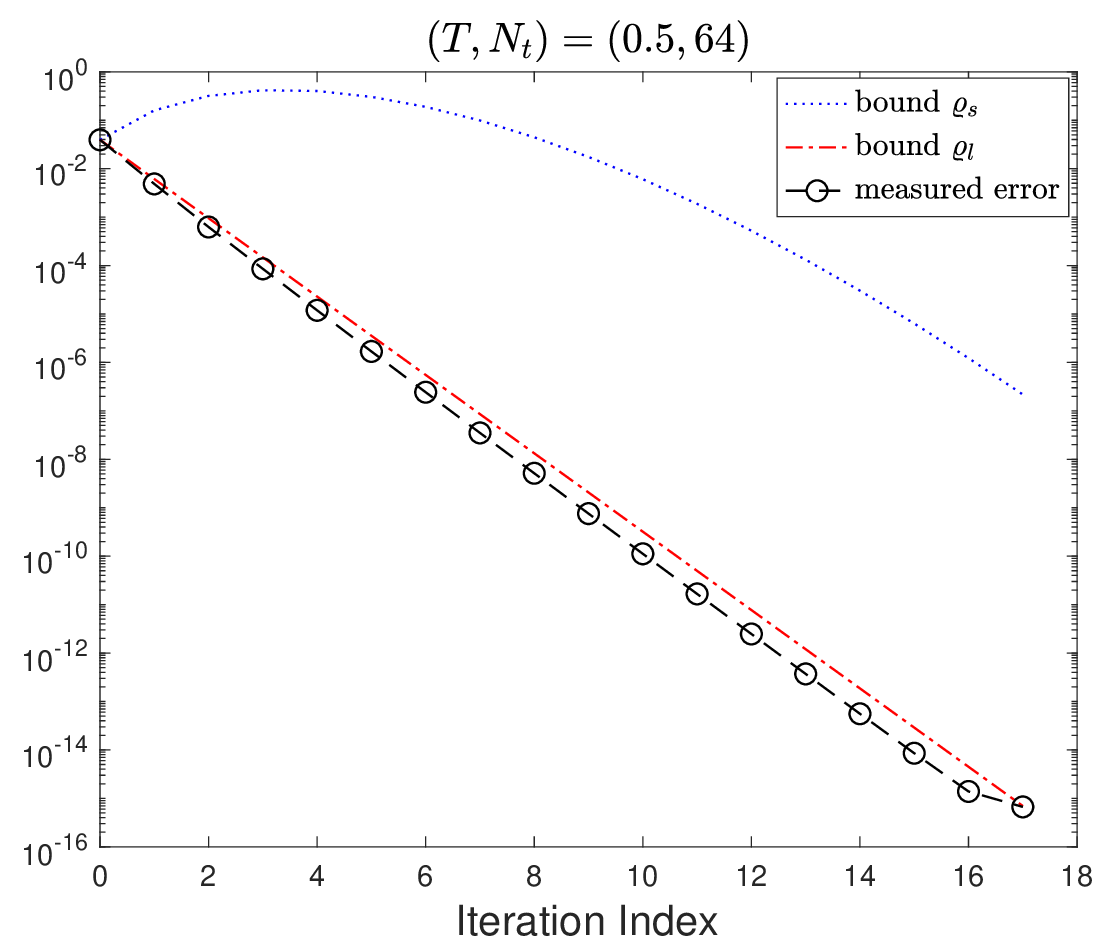} 
  \caption{Parareal convergence for the heat equation showing its
    typical two convergence regimes: superlinear convergence over
    short time intervals and linear convergence over long time
    intervals.}
  \label{Fig_Parareal_Heat}
\end{figure}
We use in space a very large mesh size $\Delta x=\frac{1}{5}$, and
both $\CF$ and $\CG$ use Backward  Euler, with a
coarsening factor $J=10$. For $T=0.02$ and $N_t=6$, the error
decreases at a superlinear rate, and $\varrho_s$ accurately predicts
this decrease. However, for a larger $T$ and $N_t$, the error
decreases linearly, and the prediction by $\varrho_s$ is
inaccurate. Note that for finer meshes in space, (e.g., $\Delta
x=\frac{1}{8}$), Parareal converges linearly.

A convergence analysis of Parareal for nonlinear systems of
ordinary differential equations using generating functions can be
found in \cite[Theorem 1]{gander:2008:nca}, see also \cite[Theorem
  2.6]{Gander:TPTI:2024}.
\begin{theorem}\label{pro_Nonlinear_Parareal}
{\em Let $\CF$ be the exact propagator and $\CG$ be a time-integrator
  of order $p$ with its local truncation error bounded by $C_3\Delta
  T^{p+1}$. Assume that $\CG$ satisfies the Lipschitz condition
$$
\|\CG(T_n, T_n+\Delta T,{\bm v})-\CG(T_n, T_n+\Delta T, {\bm w})\|\leq (1+C_2\Delta T)\|{\bm v}-{\bm w}\|,
$$
and the difference between $\CG$ and $\CF$ can be expressed, for small
$\Delta T$, as
$$
\CF(T_n, T_{n+1}, {\bm v})-\CG(T_n, T_{n+1}, {\bm v})=c_{p+1}({\bm v})\Delta T^{p+1}+c_{p+2}({\bm v})\Delta T^{p+2}+\cdots,
$$
where the coefficients $c_{p+1}, c_{p+2},\ldots$ are continuously
differentiable functions of ${\bm v}$. Then, the error of Parareal
at iteration $k$ is bounded by
\begin{equation}\label{Eq_Parareal_nonlinear}
\|{\bm u}(T_n)-{\bm u}^k_n\|\leq \frac{C_3\Delta T^{p+1}(C_1\Delta T^{p+1})^{k+1}}{(k+1)!}(1+C_2\Delta T)^{n-k-1}
\prod_{j=0}^k(n-j),
\end{equation}
where $n=1,2,\dots, N_t$ and $C_1$ is a constant related to the difference between $\CF$ and $\CG$.
}
\end{theorem}
The error estimate \eqref{Eq_Parareal_nonlinear} has a similar
consequence as the linear error estimate for short time intervals and
small $N_t$ (cf. \eqref{Error_rho_a}): the product term includes a
factor of zero, resulting in an error bound of zero, indicating
convergence in at most $N_t$ steps. A detailed convergence analysis
for Parareal applied to Hamiltonian systems using backward error
analysis can be found in \cite{gander2014analysis}.

Parareal is highly effective for {\em diffusive} problems, such as
the heat equation shown in Figure \ref{Fig_Parareal_Heat}.  In
particular, for linear systems of ODEs of the form ${\bm u}'(t)=A{\bm
  u}(t)+{\bm g}(t)$, where $A$ is a negative semi-definite matrix, it
can be shown that Parareal has a constant convergence factor around
0.3 for arbitrarily large $T$ and $N_t$, provided one uses Backward
Euler\footnote[1]{Note that using  Backward  Euler for $\CG$ in
  Parareal is justified due to the need for a cheap and stable coarse
  grid correction.} for $\CG$ and $\CF$ is an L-stable Runge-Kutta
method.
\begin{theorem}\label{pro_rho_03}
{\em If $\CG$ is  Backward  Euler and $\CF$ is an
  L-stable Runge-Kutta method, then
\begin{equation}\label{eq_Lstable}
\max_{z\in\mathbb{R}^-}\varrho_l(J,z)\approx 0.3, \quad \forall J\geq J_{\min},
\end{equation}
where $J_{\min}=\CO(1)$.
}
\end{theorem}
\begin{proof}
  For the case where $\CF$ is Backward Euler, this result was
  established in \cite{MSS10}. When $\CF$ is the Trapezoidal Rule or
  BDF2 (i.e., Matlab's \texttt{ode23s} solver) or two singly diagonal
  implicit Runge-Kutta (SDIRK) methods, proofs can be found in
  \cite{WuIMA2015} and \cite{WZhou15}. For a general L-stable $\CF$,
  the proof can be found in \cite{YYZ23}.
\end{proof}
The origins of this result go back to a result already shown in
\cite[Table 5.1]{gander2007analysis} at the continuous level, also for
other coarse propagators, and contraction can even be much better,
e.g. $\approx 0.068$ for Radau IIA.

If $\CF$ is only A-stable (not L-stable), for instance, the
Trapezoidal Rule, Parareal does not always have a constant convergence
factor. However, for large coarsening factors $J$, a similar result
holds, namely
\begin{equation}\label{eq_Astable}
\max_{z\in[0,z_{\max}]}\varrho_l(J,z)\approx 0.3, \quad \forall J\geq J_{\min}  =\CO(\log^2(z_{\max})),
\end{equation}
which was proved for the Trapezoidal Rule and a 4th-order Gauss
Runge-Kutta method in \cite{WZhou15}. This differs significantly from
the scenario where $\CF$ is assumed to be the exact solution
propagator (i.e., $\CF=\exp(\Delta TA)$), where Parareal converges
with a rate around 0.3 for $J\geq2$.

We now illustrate this constant convergence factor by applying
Parareal to the heat equation with periodic boundary conditions and
discretization and problem parameters $\Delta x=\frac{1}{256}$, $\Delta
T=0.1$, $T=4$, and $\nu=0.1$ (the diffusion coefficient). We use for
$\CF$ the Trapezoidal Rule and two SDIRK methods which are given 
by the Butcher tableau
\begin{equation}\label{twoSDIRK}
\underbrace{\begin{array}{r|cc}
\gamma &\gamma &0     \\
1-\gamma &1-\gamma  &\gamma\\
\hline
  &1-\gamma &1-\gamma
\end{array}}_{{\rm SDIRK22}, ~\gamma=\frac{2-\sqrt{2}}{2}},~~~
\underbrace{\begin{array}{r|cc}
\gamma &\gamma &0     \\
1-\gamma &\frac{-1}{\sqrt{3}}   &\gamma\\
\hline
  &\frac{1}{2} &\frac{1}{2}
\end{array}}_{{\rm SDIRK23},~\gamma=\frac{3+\sqrt{3}}{6}}.
\end{equation}
Here, `SDIRK$sp$' denotes an $s$-stage SDIRK method of order $p$. For
SDIRK22, \eqref{eq_Lstable} holds for $J_{\min}=2$ \cite{WuIMA2015},
and for SDIRK23, $J_{\min}=4$ \cite{WZhou15}. In Figure
\ref{Fig_Constant_rho_Heat},
\begin{figure}
  \centering
  \includegraphics[width=1.5in,height=1.25in,angle=0]{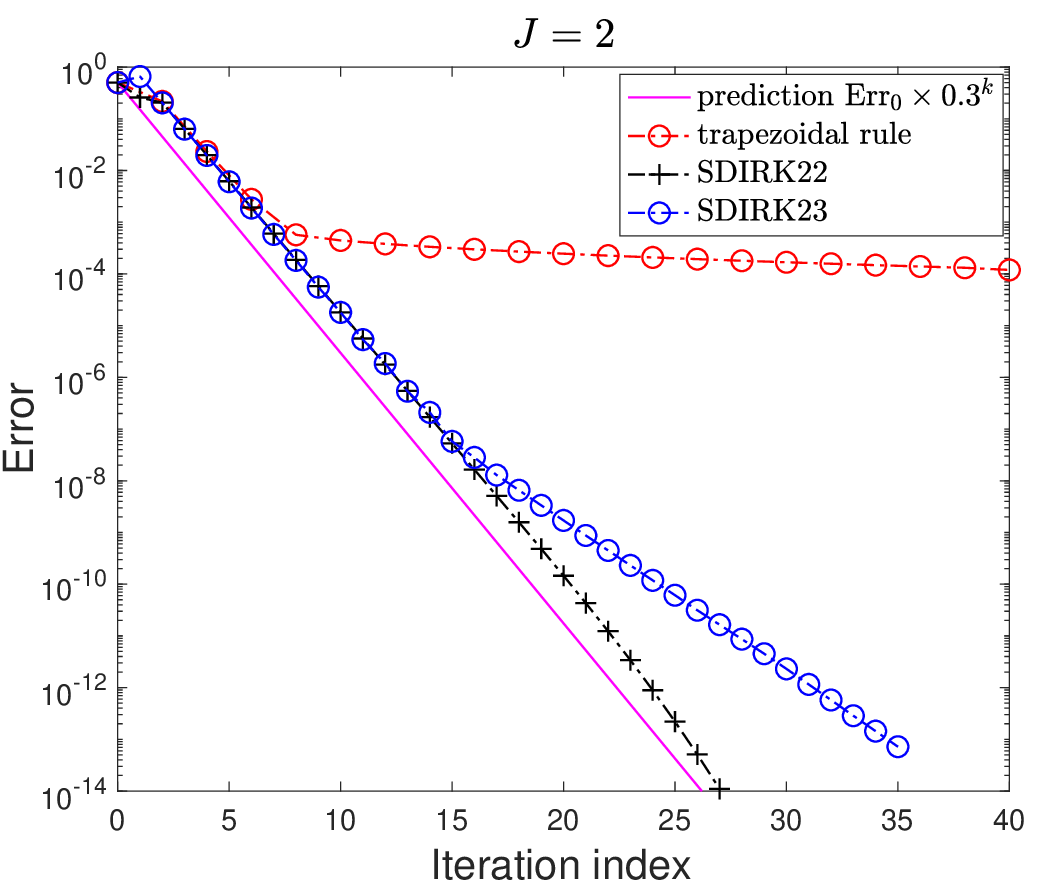}~ 
  \includegraphics[width=1.5in,height=1.25in,angle=0]{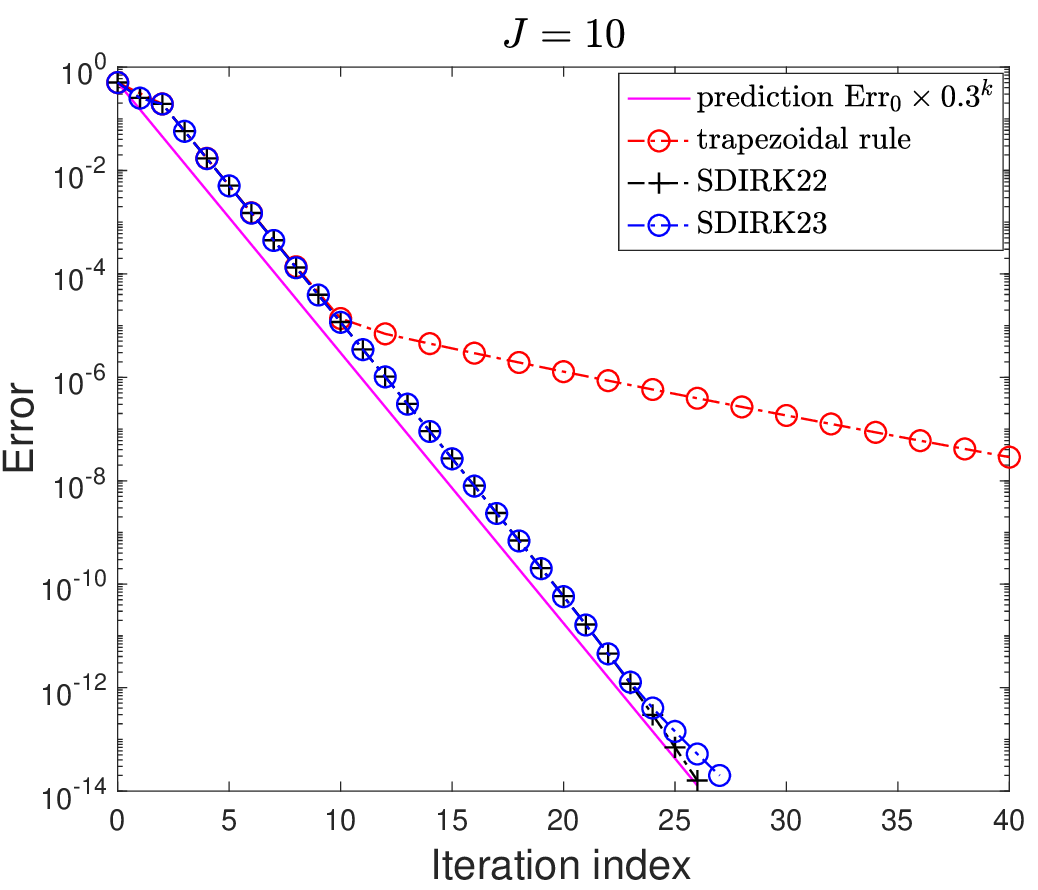}~
  \includegraphics[width=1.5in,height=1.25in,angle=0]{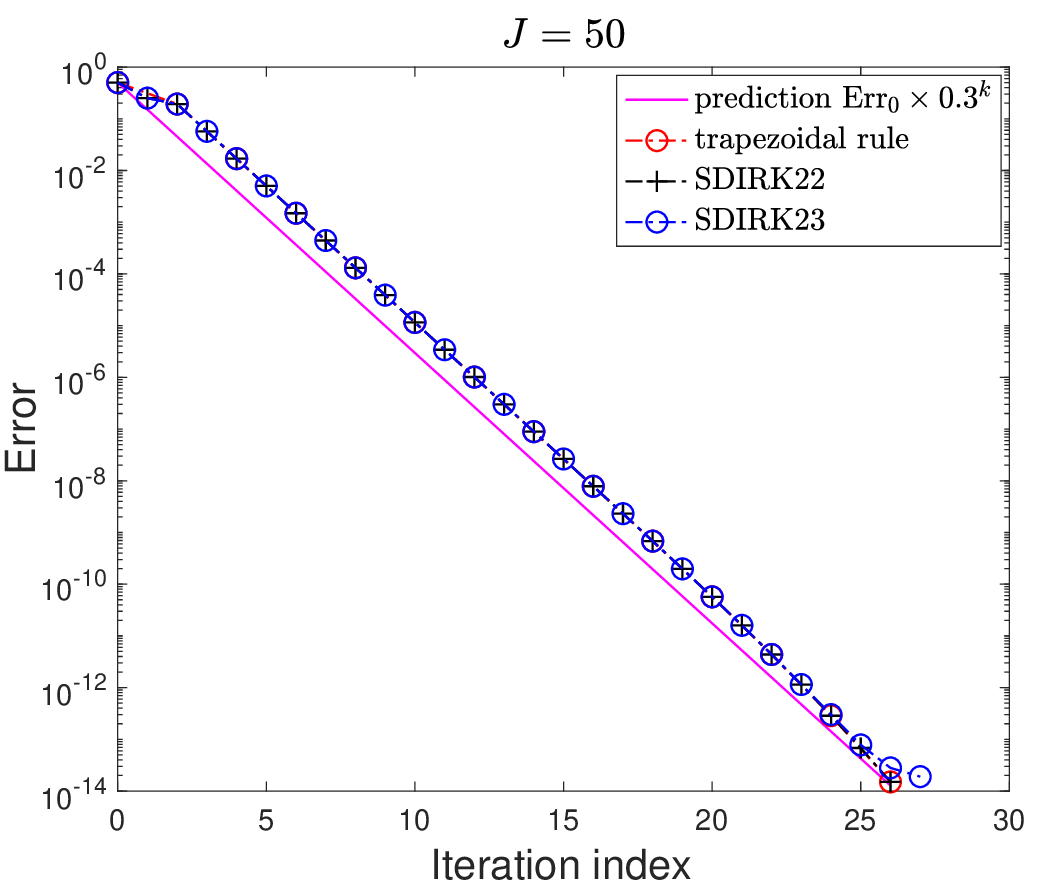}
  \caption{Different choices of the fine solver $\CF$ lead to different convergence rates for Parareal.}
  \label{Fig_Constant_rho_Heat}
\end{figure}
we show the measured error at each iteration for three values of the
coarsening factor $J$. We observe that for small $J$, these three
time-integrators indeed lead to different convergence rates,
especially for the Trapezoidal Rule and the SDIRK23 method, where
Parareal converges more slowly. When $J$ is large, say $J=50$, 
Parareal  converges at the similar rate closed to 0.3 for
all three time-integrators.  This can be intuitively understood by the
fact that for small $J$ the fine integrator which is not $L$-stable is
not accurately enough resolving the physics for high frequencies, and
Parareal tries to converge to this incorrect solution with Backward
Euler as the coarse propagator which represents the correct physics
for high frequencies. 

While Parareal converges very well for the heat equation and more
generally diffusive problems, Parareal is not well-suited for
problems that are only weakly diffusive, since its convergence rate
continuously deteriorates as the diffusion weakens. We illustrate
this for the advection-diffusion equation \eqref{ADE} and Burgers'
equation \eqref{Burgers} with periodic boundary conditions,
$g(x,t)=0$, and initial condition $u(x,0)=\sin(2\pi x)$. We use $T=4$,
$\Delta T=0.1$, $\Delta x=\frac{1}{128}$, and $J=32$ for the problem
and discretization parameters. The coarse solver is  Backward 
Euler, and the fine solver is SDIRK22. For three values of the
diffusion parameter $\nu$, we show in Figure
\ref{Fig_Parareal_rho_ADE}
\begin{figure}
  \centering
  \includegraphics[width=1.5in,height=1.25in,angle=0]{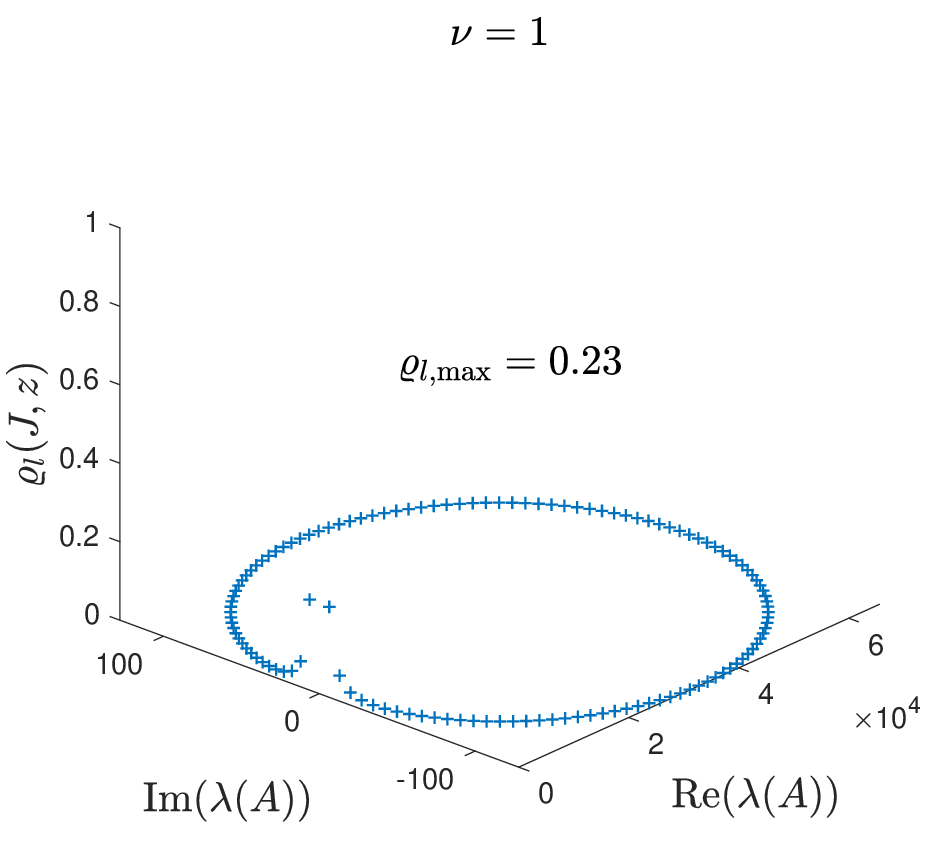}~ 
  \includegraphics[width=1.5in,height=1.25in,angle=0]{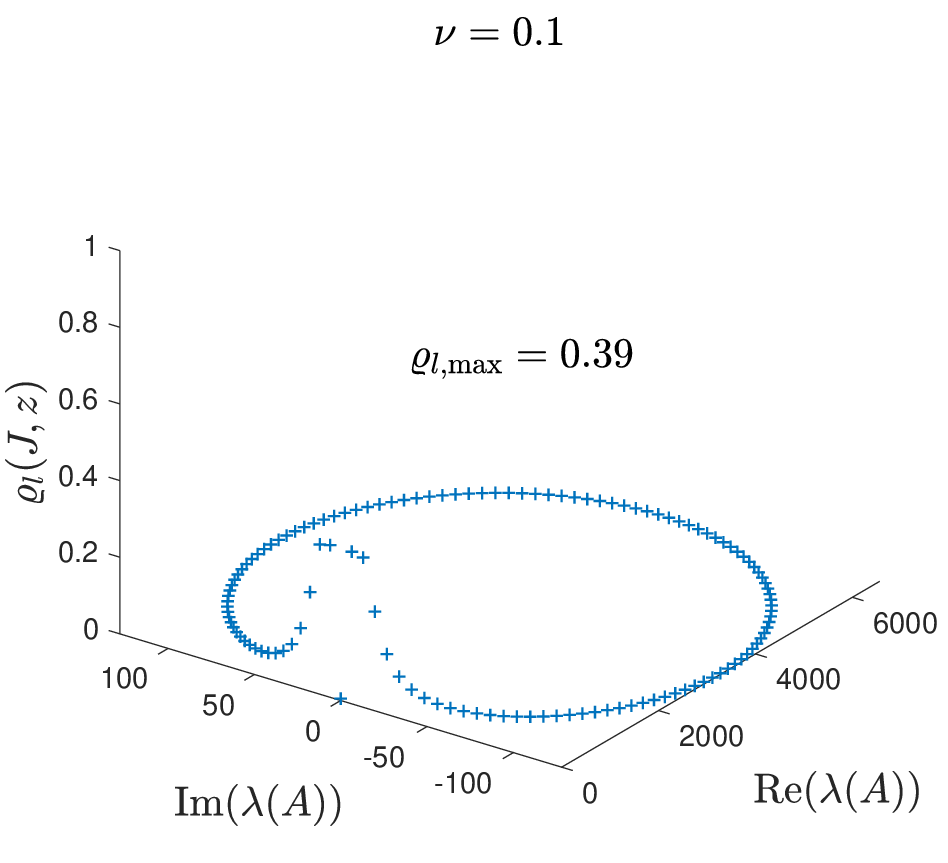}~
  \includegraphics[width=1.5in,height=1.25in,angle=0]{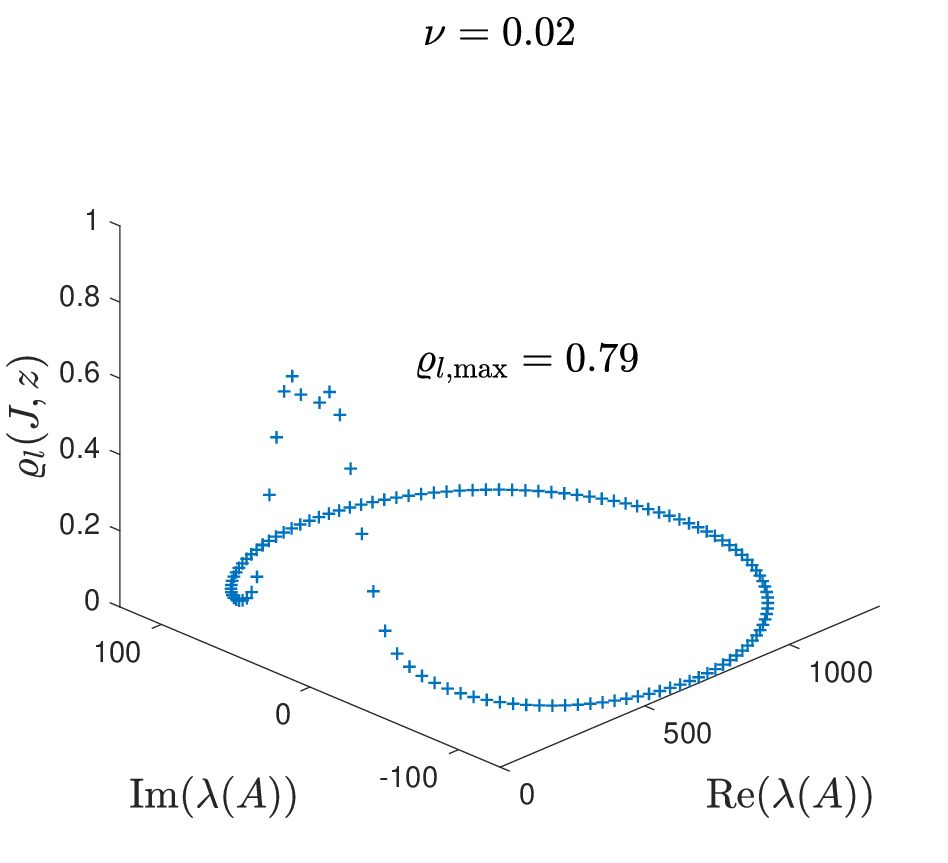}
  \caption{The quantity $\varrho_l(J,z)$ for each $z=\Delta
    T\lambda(A)$ for the advection-diffusion equation with three
    values of the diffusion parameter $\nu$. As $\nu$ decreases, the
    maximum of $\varrho_l$ approaches 1.}
  \label{Fig_Parareal_rho_ADE}
\end{figure}
the quantity $\varrho_l(J,z)$ for each $z=\Delta T\lambda(A)$. As
$\nu$ decreases, the maximum of $\varrho_l$ grows, indicating that
Parareal converges more slowly. This is confirmed in Figure
\ref{Fig_Parareal_err_ADE_Burgers} on the left where we run Parareal
\begin{figure}
  \centering
 \includegraphics[width=2.3in,height=1.85in,angle=0]{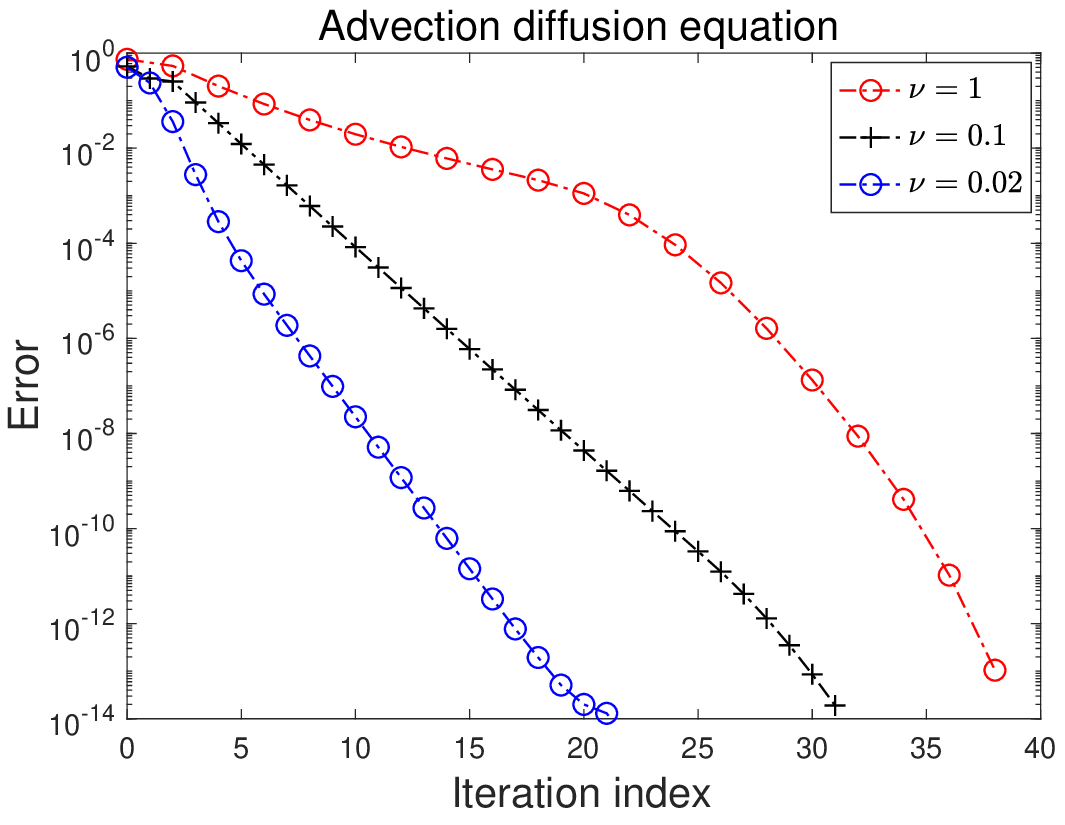} 
  \includegraphics[width=2.3in,height=1.85in,angle=0]{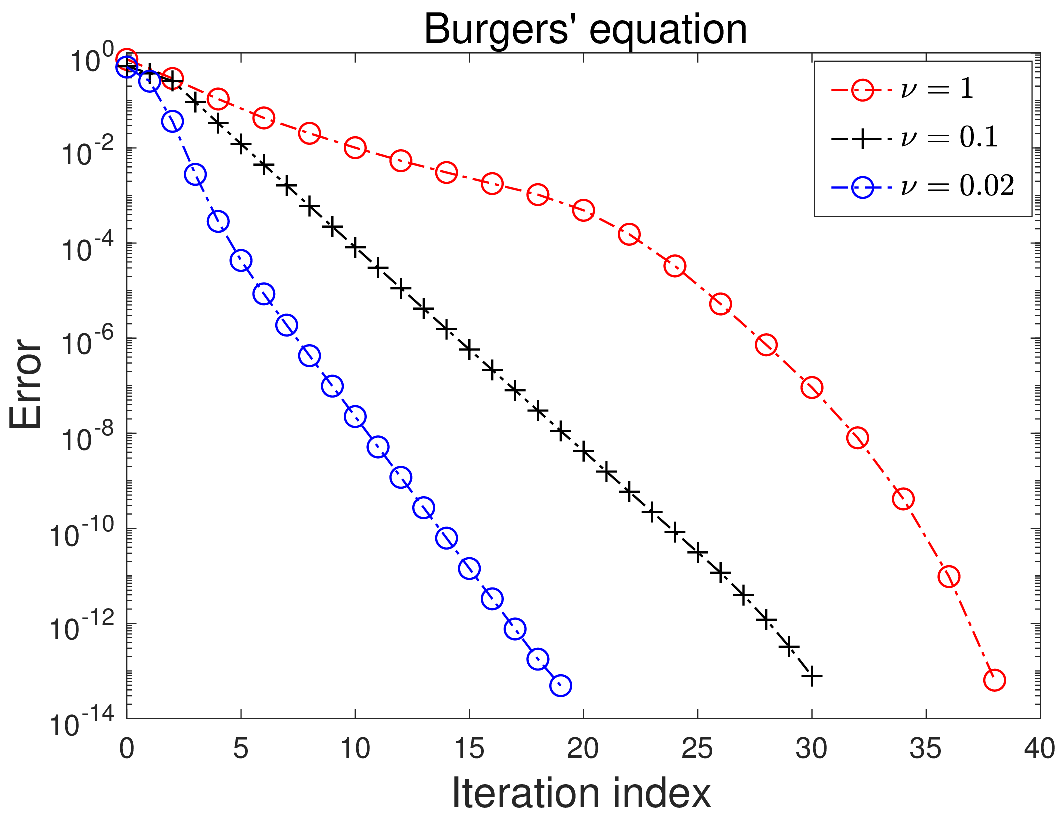} 
  \caption{Convergence of Parareal applied to the
    advection-diffusion equation (left) and Burgers' equation
    (right) with three values of the diffusion parameter $\nu$. }
  \label{Fig_Parareal_err_ADE_Burgers}
\end{figure}
on the corresponding problem. For Burgers' equation, we do not have an
as precise theoretical analysis as for the advection-diffusion
equation in Figure \ref{Fig_Parareal_rho_ADE}, but the results in Figure
\ref{Fig_Parareal_err_ADE_Burgers} on the right show that Parareal
also converges more and more slowly for small $\nu$.  If we continue
to reduce the parameter $\nu$, meaning the advection term becomes
increasingly dominant, Parareal eventually diverges, approximately
when $\nu\leq10^{-3}$, except that the finite step convergence still
holds if one iterates long enough. 

For hyperbolic problems, such as the second-order wave equation
\eqref{WaveEquation1d}, Parareal is also not convergent, as was
already shown in \cite{gander2007analysis}; see also
\cite{gander2020toward,gander2020reynolds,gander2023convergence,gander2023unified}
for more recent and detailed analyses. This degeneration can be
attributed to the fact that for hyperbolic problems, as illustrated in
Figure \ref{WaveExampleFig}, arbitrarily small high frequency
components, i.e. small oscillations, propagate arbitrarily far in both
space and time. Consequently, it becomes very challenging to achieve
high-accuracy solutions in the coarse solver $\CG$ that are comparable
to the fine solver $\CF$, both in space and time. If we strive for
high accuracy in $\CG$, the coarse grid correction becomes rather
time-consuming, and we fail to achieve any speedup.

In the MGRiT community (MGRiT is a multilevel generalization of
Parareal, see Section \ref{Sec4.4}), considerable research effort has
been directed toward making MGRiT work for advection equations; see
\cite{howse2019parallel,de2021optimizing,de2023efficient,de2023fast}
and references therein. The idea is to design an optimized coarse
solver through the so-called semi-Lagrangian discretization. This
technique performs well for linear advection equations, while research
on the nonlinear case is still ongoing, as the semi-Lagrangian
discretization is a characteristics-based method that is not easily
realized for nonlinear problems. Another idea, proposed in
\cite{gander2020diagonalization}, also aims to make Parareal (and
MGRiT) work for hyperbolic problems. In this approach, it is
relatively easy to handle nonlinear problems, as we will see
in Section \ref{Sec4.5}.

\subsection{PFASST}\label{Sec4.3} 

In this and the next two subsections, we present three variants of the
Parareal algorithm. We begin by introducing the Parallel Full
Approximation Scheme in Space-Time (PFASST), which was proposed in
\cite{EM12}. The concept of this method emerged two years earlier
\cite{Min10}, where the author replaced the fine solver with one
iteration of SDC \cite{DGR00}, in order to reduce the computational
cost of one Parareal iteration. PFASST has been successfully applied
to several problems \cite{EM12,SRK12,SRE14}, but a clear description
and theoretical analysis of this method are rather
challenging. Recently, \cite{BMS17} described PFASST as a time
multigrid method based on an algebraic representation of SDC
introduced in \cite{Min15}, and provided a convergence analysis in
\cite{BMS18}.

In the formalism of block iterations, PFASST was precisely described
 and studied for a model problem in Gander et al. (2023b). In
   particular, for the system of ODEs (2.1), a two level variant of
   PFASST can be described as follows: we first partition the time
   interval $(0, T)$ into $N_t$ large subintervals $[T_0,
     T_1]\cup[T_1,T_2]\cup\dots\cup[T_{N_t-1},T_{N_t}]$ with $T_0=0,
   T_{N_t}=T$ and $T_n=n\Delta t$.  For each subinterval, e.g., the
   $n$-th subinterval $[T_n, T_{n+1}]$, we define $M_f$ and $M_c$ time
   points
$$
\begin{cases}
\{t^f_{n,m}:=T_n+\tau^f_m\Delta t\}, ~m=0,1,\dots, M_f ~{\rm and},~ \tau^f_0=0,~\tau^f_{M_f}=1, \\
\{t^c_{n, m}:=T_n+\tau_m^c\Delta t\}, ~m=0,1,\dots, M_c~{\rm and},~\tau^c_0=0,~\tau^c_{M_c}=1, \\
\end{cases}
$$
where $M_f>M_c$.  Here we use the superscript `$f$' and `$c$' to
denote the fine and the coarse time grids.  Then, we solve (2.1) for
$t\in[T_n, T_{n+1}]$ by numerical quadrature as
\begin{equation}\label{quadrature}
\bm u_{n,m}=\bm u_{n,0}+\Delta t{\sum}_{j=1}^Mq_{m,j}(A\bm u_{n,j}+{\bm g}(t_{n,j})),~m=1,2,\dots, M,
\end{equation}
where $\bm u_{n,j}$ is an approximation of $\bm u$ at
$t=t_{n,j}$. Here $M=M_f$ or $M_c$; $t_{n,j}=t^f_{n,j}$ or
$t^c_{n,j}$.  We represent \eqref{quadrature} as
\begin{equation*}
 {\bm u}_n=\Delta t(Q\otimes A){\bm u}_n+{\bm \chi}{\bm u}_{n-1}+\Delta t\bm b_n,
 \end{equation*}
 where $Q:=(q_{m,j})$, ${\bm u}_n:=(u_{n,1}^\top, u_{n,2}^\top, \dots, u_{n, M}^\top)^\top$, $\bm b_n:=(Q\otimes I_x)\bm g_n$   and  ${\bm \chi}$ is the  block `copying' matrix  ${\bm\chi}:=\chi\otimes I_x$ with 
  $$
  \chi :=
  \begin{bmatrix}
  0   &\cdots &0 &1\\
    0   &\cdots &0 &1\\
    \vdots &\cdots &\vdots &\vdots\\
      0   &\cdots &0 &1
  \end{bmatrix}  \in\mathbb{R}^{M\times M},\quad\bm g_n:=\begin{bmatrix}
  {\bm g}(t_{n,1})\\
 {\bm g}(t_{n,2})\\
  \vdots\\
{\bm  g}(t_{n, M})
  \end{bmatrix}. 
  $$
 Hence, for the fine and coarse time grids we have 
\begin{equation*}
\begin{split}
& {\bm u}^f_n={\bm \phi}_f^{-1}({\bm \chi}_f{\bm u}^f_{n-1}+\Delta t\bm b^f_n),~ {\bm u}^c_n={\bm \phi}_c^{-1}({\bm\chi}_c{\bm u}^c_{n-1}+\Delta t\bm b^c_n),\\
&{\bm \phi}_f:={\bm I}_f-\Delta tQ_f\otimes A,~{\bm \phi}_c:={\bm I}_c-\Delta tQ_c\otimes A, 
 \end{split}
\end{equation*}
where $n=1,2,\dots, N_t$, ${\bm I}_c=I_{M_c}\otimes I_x$ and ${\bm
  I}_f=I_{M_f}\otimes I_x$.
 
For PFASST, we need transfer matrices ${\bf T}^{c\rightarrow f}$ and
${\bf T}^{f\rightarrow c}$ which prolongate and restrict vectors
defined on the coarse and fine time grids. These transfer matrices are
defined via Lagrange interpolation,
\begin{equation*}
\begin{split}
&p^{c}(\tau; {\bm u}^c)=\sum_{m=1}^{M_c}u^c_mL_m^{c}(\tau),~L_m^{c}(\tau):=\frac{\prod^{M_c}_{j=1, j\neq m}(\tau-\tau^c_m)}{\prod^{M_c}_{j=1, j\neq m}(\tau^c_j-\tau^c_m)},\\
&p^{f}(\tau; {\bm u}^f)=\sum_{m=1}^{M_f}u^f_mL_m^{f}(\tau),~L_m^{f}(\tau):=\frac{\prod^{M_f}_{j=1, j\neq m}(\tau-\tau^f_m)}{\prod^{M_f}_{j=1, j\neq m}(\tau^f_j-\tau^f_m)},\\
\end{split}
\end{equation*}
where $L_m^{c}$ and $L_m^{f}$ are the $m$-th basis function specified
by the coarse and fine interpolation nodes.  The function $p^{c}$ is
evaluated at the fine nodes $\{\tau^f_m\}$ and the function $p^{f}$ is
evaluated at the coarse nodes $\{\tau^c_m\}$.  Specifically,
      \begin{equation*}
\begin{split}
\begin{bmatrix}
p^{c}(\tau_1^f; {\bm u}^c)\\
p^{c}(\tau_2^f; {\bm u}^c)\\
\vdots\\
p^{c}(\tau_{M_f}^f; {\bm u}^c)\\
\end{bmatrix}=
\underbrace{\left(\begin{bmatrix}
L_1^{c}(\tau^f_1) &L_2^{c}(\tau^f_1) &\cdots &L_{M_c}^{c}(\tau^f_1)\\
L_1^{c}(\tau^f_2) &L_2^{c}(\tau^f_2) &\cdots &L_{M_c}^{c}(\tau^f_2)\\
\vdots &\vdots &\cdots &\vdots \\
L_1^{c}(\tau^f_{M_f}) &L_2^{c}(\tau^f_{M_f}) &\cdots &L_{M_c}^{c}(\tau^f_{M_f})\\
\end{bmatrix}\otimes I_x\right)}_{=:{\bf T}^{c\rightarrow f}\in\mathbb{R}^{M_fN_x\times M_cN_x}}{\bm u}^c.
\end{split}
\end{equation*}
  The matrix ${\bf T}^{f\rightarrow c}\in\mathbb{R}^{M_cN_x\times M_fN_x}$ is defined similarly. 
  
With the above notation, according to Gander et al. (2023b), PFASST can
be written as
$$
\bm u_{n+1}^{k+1}={\bf B}_1^0\bm u_{n+1}^k+{\bf B}_{0}^1 ({\bm \chi}\bm u_{n}^{k+1}+\Delta t\bm b_n^f)+{\bf B}_0^0\bm ({\bm \chi}{\bm u}_n^k+\Delta t\bm b_n^f),
$$
where 
\begin{equation*}
\begin{split}
&{\bf B}_1^0=[{\bm I}_f-{\bf T}^{c\rightarrow f}{\bm \phi}_c^{-1}{\bf T}^{f\rightarrow c}{\bm \phi}_f]({\bm I}_f-\tilde{\bm \phi}_f^{-1}{\bm \phi}_f),\\
&{\bf B}_{0}^1={\bf T}^{c\rightarrow f}{\bm \phi}_c^{-1}{\bf T}^{f\rightarrow c},\\
&{\bf B}_{0}^0=[{\bm I}_f-{\bf T}^{c\rightarrow f}{\bm \phi}_c^{-1}{\bf T}^{f\rightarrow c}{\bm \phi}_f]\tilde{\bm \phi}_f^{-1},
 \end{split}
\end{equation*}
and $\tilde{\bm \phi}_f$  is an approximation of ${\bm \phi}_f$.  In practice, we construct
$\tilde{\bm \phi}_f$ by using an implicit Euler method on the time points $\{t_{n,m}^f\}$: 
 \begin{equation}\label{LUQf}
\frac{{\bm u}_{n, m+1}-{\bm u}_{n,m}}{\Delta t(\tau^f_{m+1}-\tau^f_m)}=A{\bm u}_{n,m+1}+{\bm g}(t^f_{n,m+1}),m=0,1,\dots, M_f-1
 \end{equation}
 i.e., 
 $$
 \tilde{\bm \phi}_f=\begin{bmatrix}
 1 & & &\\
 -1 &1 & &\\
 &\ddots &\ddots &\\
 & &-1 &1
 \end{bmatrix}\otimes I_{x}-\Delta t\begin{bmatrix}
\tau^f_{1}-\tau^f_0 & & &\\
  &\tau^f_2-\tau^f_1 & &\\
 & &\ddots &\\
 & &  &\tau^f_{M_f}-\tau^f_{M_f-1}
 \end{bmatrix}\otimes A. 
 $$

We now apply PFASST to the heat equation (2.3) and the
advection-diffusion equation (2.5) with  $T=3$, periodic BCs and initial value
$u(x,0)=0$. The source term  $g(x,t)$ for the two PDEs is given by (2.4) with $\sigma=1000$ and the space-time mesh size is $\Delta x=\frac{1}{128}, \Delta t=\frac{1}{64}$.
For the numerical quadrature \eqref{quadrature}, we use the Radau IIA
method for both fine and coarse nodes with $M_f=3$ and $M_c=2$. The
nodes are
$$
\{\tau^f_m\}:=\left\{0,\frac{4-\sqrt{6}}{10},~\frac{4+\sqrt{6}}{10},~1\right\},
 \quad \{\tau^c_m\}:=\left\{0,\frac{1}{3}~1\right\}, 
$$
and the corresponding weight matrices $Q_f$ and $Q_c$ are
$$
  Q_f:=
  \begin{bmatrix}
  \frac{88-7\sqrt{6}}{360} &\frac{296-169\sqrt{6}}{1800} &\frac{-2+3\sqrt{6}}{225}\\
   \frac{296+169\sqrt{6}}{1800}  &  \frac{88+7\sqrt{6}}{360}&\frac{-2-3\sqrt{6}}{225}\\
   \frac{16-\sqrt{6}}{36} &   \frac{16+\sqrt{6}}{36} &\frac{1}{9}
  \end{bmatrix},\quad  Q_c:=
  \begin{bmatrix}
\frac{5}{12} &-\frac{1}{12}\\
\frac{3}{4} &\frac{1}{4}
  \end{bmatrix}. 
  $$
Then, the two transfer matrices are 
$$
{\bf T}^{c\rightarrow f}=\begin{bmatrix}
   1.2674   &-0.2674 \\
   0.5325    &0.4674 \\
                   0   &1
\end{bmatrix}\otimes I_x,\quad {\bf T}^{f\rightarrow c}=\begin{bmatrix}
0.5018    &0.6833 &  -0.1851\\
0 &0 &1
\end{bmatrix}\otimes I_x. 
$$
 
In Figure \ref{Fig_PFASST} we show the measured error of PFASST for
the heat equation and the advection-diffusion equation with 3
diffusion parameters. We see that the convergence rate also
deteriorates when the diffusion in the PDE becomes weak, like in
Parareal and MGRiT we have seen before. 
 \begin{figure}
\centering
\includegraphics[width=4in,height=3in,angle=0]{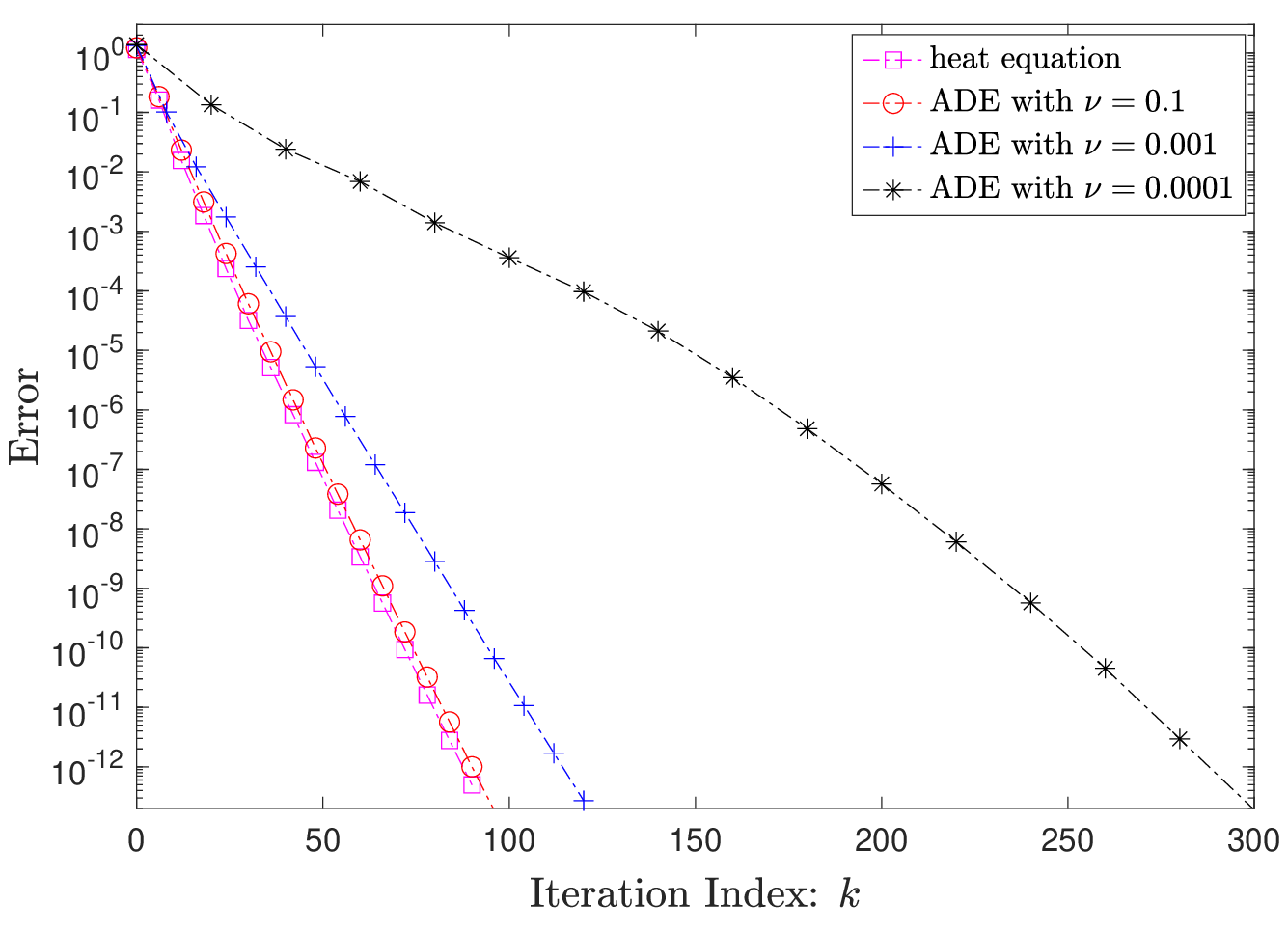} 
 \caption{Measured error of PFASST for the heat equation and the
     advection-diffusion equation (ADE) with 3 diffusion
     parameters.}\label{Fig_PFASST}
\end{figure}

\subsection{MGRiT}\label{Sec4.4}

Multi-grid reduction in time (MGRiT) is another variant of Parareal,
introduced by \cite{FFK14}. MGRiT can be interpreted in different
ways, such as an algebraic multigrid method with the so-called
FCF-relaxation, a block iteration \cite{gander2023unified} and
\cite[Chapter 4.6]{Gander:TPTI:2024}, or as an overlapping Parareal
variant \cite[Theorem 4 and Corollary 1]{GanKZ18}. Here, we present MGRiT as an
overlapping variant of Parareal applied to non-linear
systems of ODEs \eqref{nonlinearODE}, i.e., ${\bm u}'=f({\bm u}, t)$
with initial value ${\bm u}(0)={\bm u}_0$. MGRiT  with two levels
and FCF-relaxation then corresponds to the iteration
\begin{equation}\label{MGRiT}
\begin{array}{rcll}
{\bm u}^{k+1}_0&=&{\bm u}_0,~{\bm u}_1^{k+1}=\CF(T_0, T_1, {\bm u}_0),\\
{\bm u}^{k+1}_{n+1}&=&\CF(T_n, T_{n+1}, \CF(T_{n-1}, T_n, {\bm u}^{k}_{n-1}))+\CG(T_{n}, T_{n+1}, {\bm u}^{k+1}_n)\\
& &-\CG(T_n, T_{n+1}, \CF(T_{n-1}, T_n, {\bm u}^{k}_{n-1})),
\end{array}
\end{equation}
where $n=1, 2, \dots, N_t-1$, and $\CG$ and $\CF$ are the coarse and
fine time propagators used in Parareal (cf. \eqref{Parareal}). We see from
\eqref{MGRiT} that MGRiT with FCF-relaxation costs two fine solves in
each iteration, compared to only one fine solve in Parareal. A geometric
representation of MGRiT with FCF-relaxation is shown in Figure
\ref{Fig_MGRiT},
\begin{figure}
\centering
\includegraphics[width=4.75in,height=0.75in,angle=0]{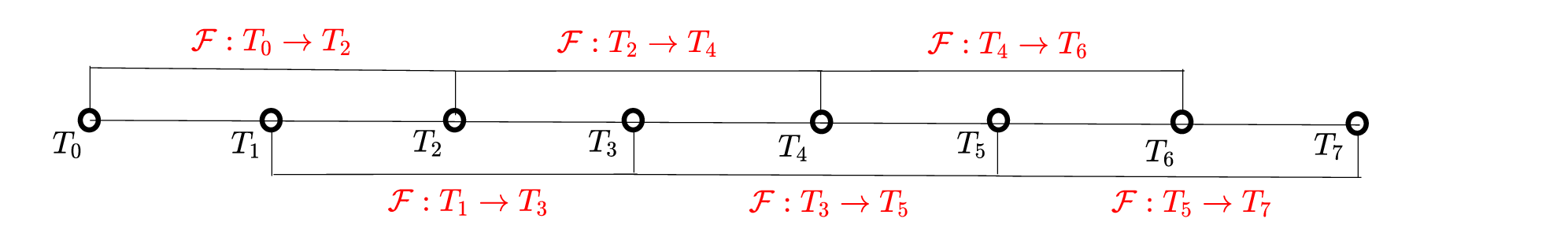}
\caption{Geometric representation of MGRiT with FCF-relaxation as
  an overlapping variant of Parareal. The dark circles represent the
  coarse time points where the coarse solver $\CG$ runs.}
\label{Fig_MGRiT}
\end{figure}
illustrating that two level MGRiT with FCF relaxation is a Parareal
algorithm with overlap size $\Delta T$, and it thus converges also in
a finite number of iterations, i.e., the global error decays to zero
after at most $k=\lceil\frac{N_t}{2}\rceil$ iterations \cite[Theorem
  5]{GanKZ18}.  A superlinear convergence result for two-level MGRiT
applied to non-linear problems with more general
F(CF)$^{\nu}$-relaxation, $\nu=1,2,\ldots$ can be found in
\cite[Theorem 6]{GanKZ18}, and it is shown that this corresponds to
Parareal with $\nu$ coarse time interval $\Delta T$ overlap; see \cite[Corollary
  1]{GanKZ18}.

In the linear case, a linear convergence estimate can be found in
\cite{DKPS17}, which we show now. We consider the linear system of
ODEs \eqref{linearODE}, i.e., ${\bm u}'=A{\bm u}+{\bm g}$ with initial
value ${\bm u}(0)={\bm u}_0$, where we assume that $A$ is
diagonalizable with spectrum $\sigma(A)\subset\mathbb{C}^-$.
\begin{theorem}\cite{DKPS17}\label{MGRiT_Linear}
  {\em With the same notation and assumptions as in Theorem
  \ref{Parareal_rho},  MGRiT with FCF-relaxation 
  satisfies the convergence estimate
\begin{equation}\label{MGRiT_Error_b}
\begin{split}
&\max_{1\leq n\leq N_t}\|{\bm e}^k_n\|_\infty\leq \max_{z\in\sigma(\Delta TA)}\varrho_l^k(J,z)\max_{1\leq n\leq N_t}\|{\bm e}_n^0\|_\infty,\\
&\varrho_l(J, z):=\frac{|{\rm R}^J_f(z/J)| |{\rm R}_g(z)-{\rm R}^J_f(z/J)|}{1-|{\rm R}_g(z)|},
\end{split}
\end{equation}
where ${\rm R}_g$ and ${\rm R}_f$ are the stability functions of the
coarse solver $\CG$ and the fine solver $\CF$, ${\bm
  e}^k_n:=V_A({\bm u}_n^k-{\bm u}_n)$, and the coarse solver $\CG$ is
stable in the sense that $|{\rm R}_g(z)|<1$ for $z\in\sigma(\Delta
TA)$.}
\end{theorem}
The quantity $\varrho_l$ in \eqref{MGRiT_Error_b} is the {\em linear}
convergence factor of MGRiT for long-time computations when applied to
linear problems. Denoting by $\varrho_{l, {\rm parareal}}$ the
convergence factor of Parareal  (cf. \eqref{Error_rho_b})
and $\varrho_{l, {\rm mgrit}}$ the convergence factor of MGRiT, it
holds that
$$
\varrho_{l, {\rm mgrit}}=|{\rm R}^J_f(z/J)|\times \varrho_{l, {\rm parareal}}. 
$$
In practice, the fine solver $\CF$ is stable, i.e., $|{\rm
  R}_f(z)|\leq 1$ $\forall z\in\sigma(\Delta TA)$, and thus each
  additional $CF$ relaxation adds a further contraction to MGRiT
  compared to Parareal, but each such relaxation also costs again an
  expensive fine parallel solve. To illustrate this, we show in
Figure \ref{Fig_Contourf_MGRiT_Parareal_a}
\begin{figure}
  \centering
 \includegraphics[width=1.5in,height=1.25in,angle=0]{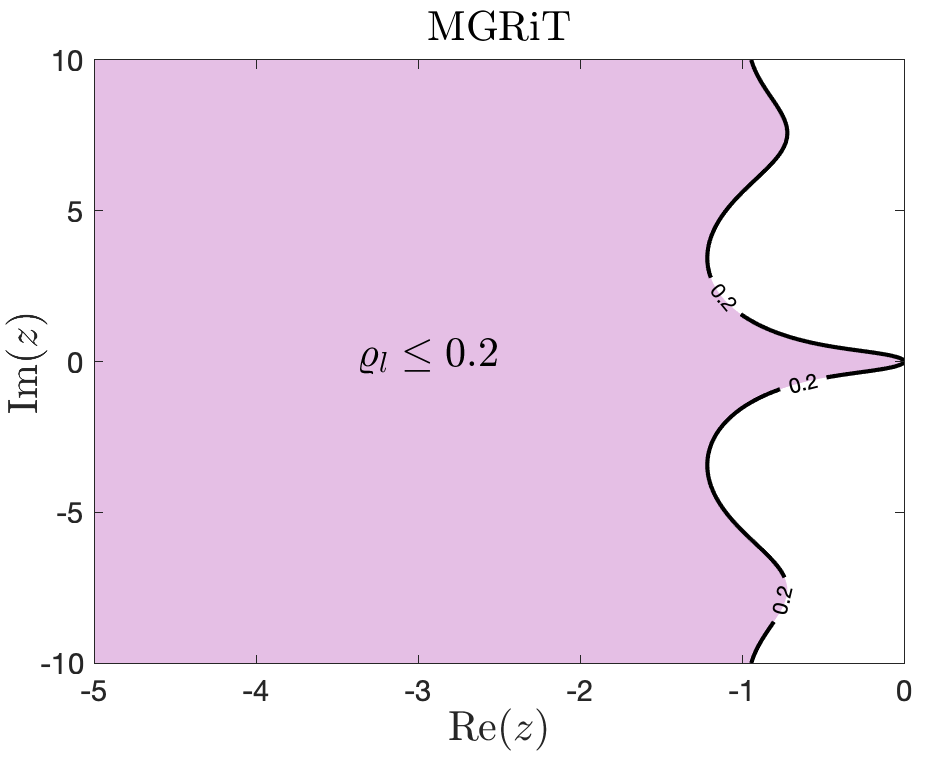}~ \includegraphics[width=1.5in,height=1.25in,angle=0]{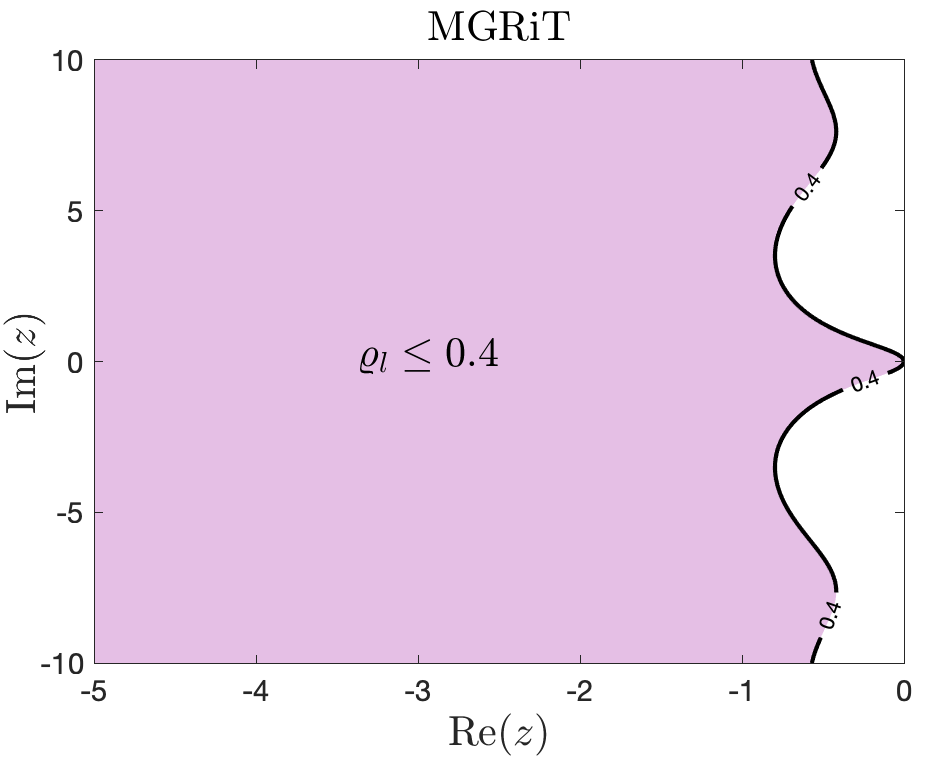} ~
 \includegraphics[width=1.5in,height=1.25in,angle=0]{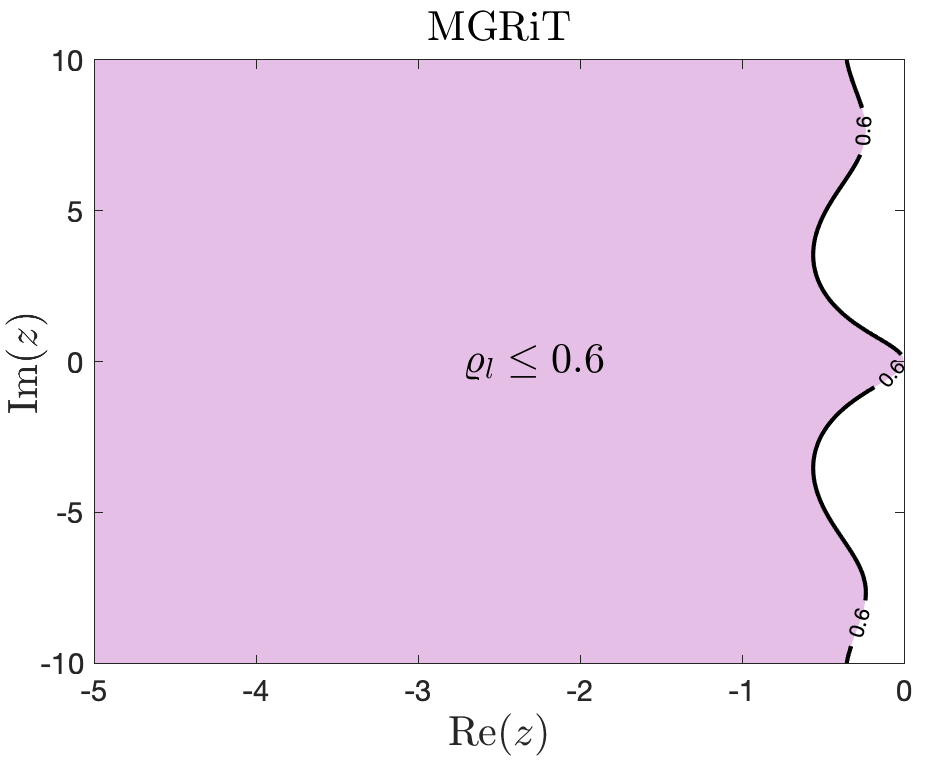}\\
  \includegraphics[width=1.5in,height=1.25in,angle=0]{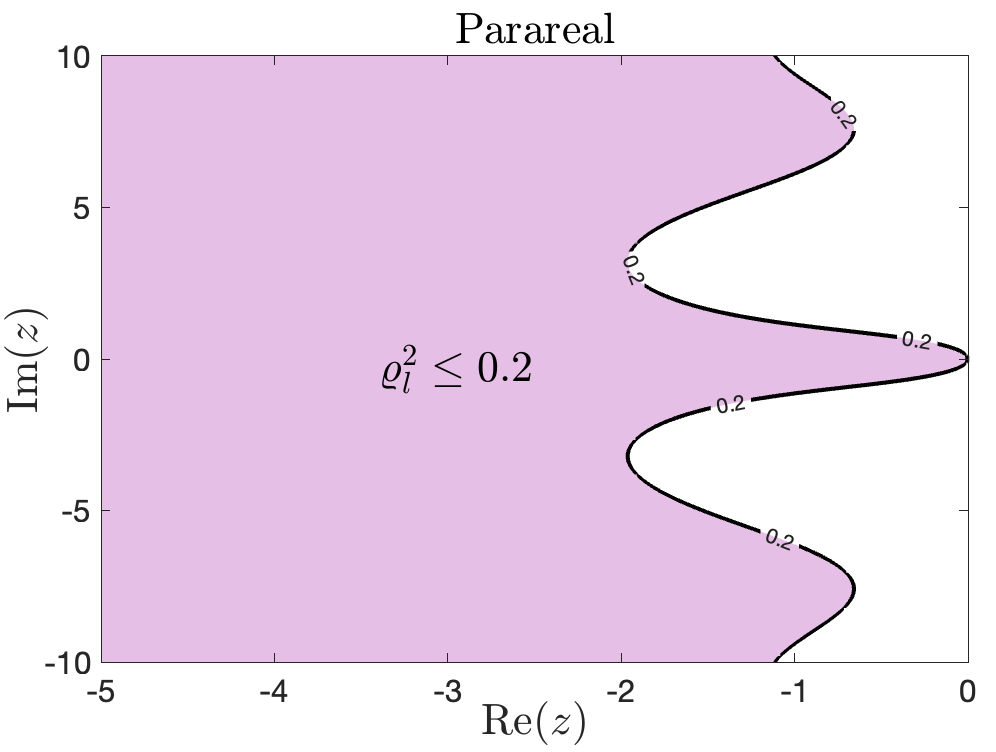}~ \includegraphics[width=1.5in,height=1.25in,angle=0]{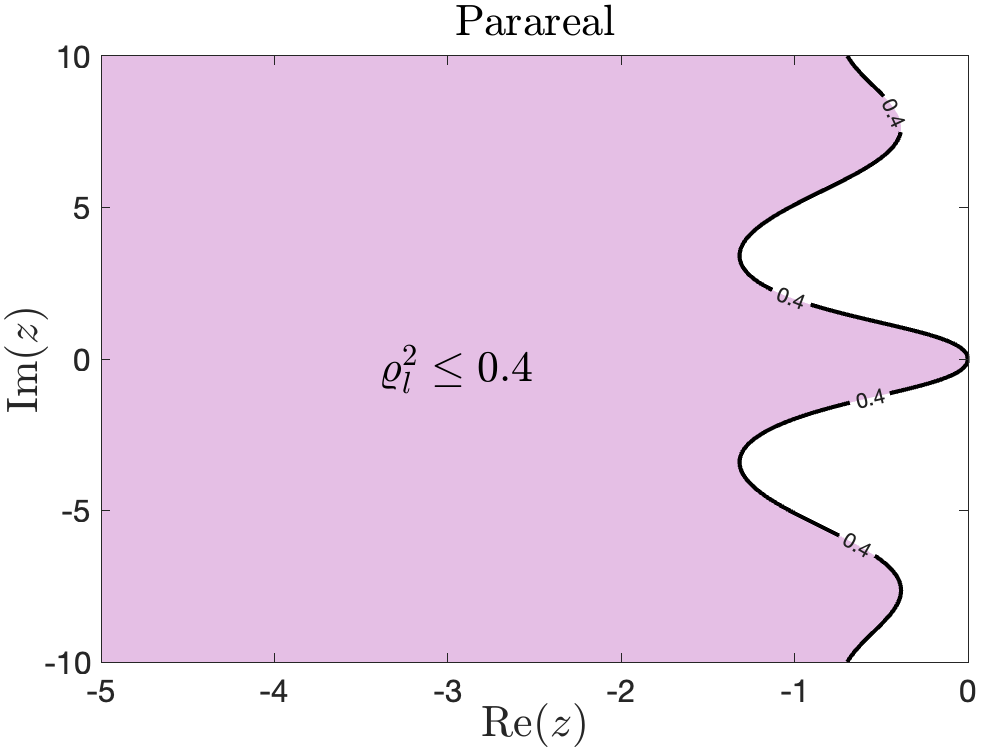} ~
 \includegraphics[width=1.5in,height=1.25in,angle=0]{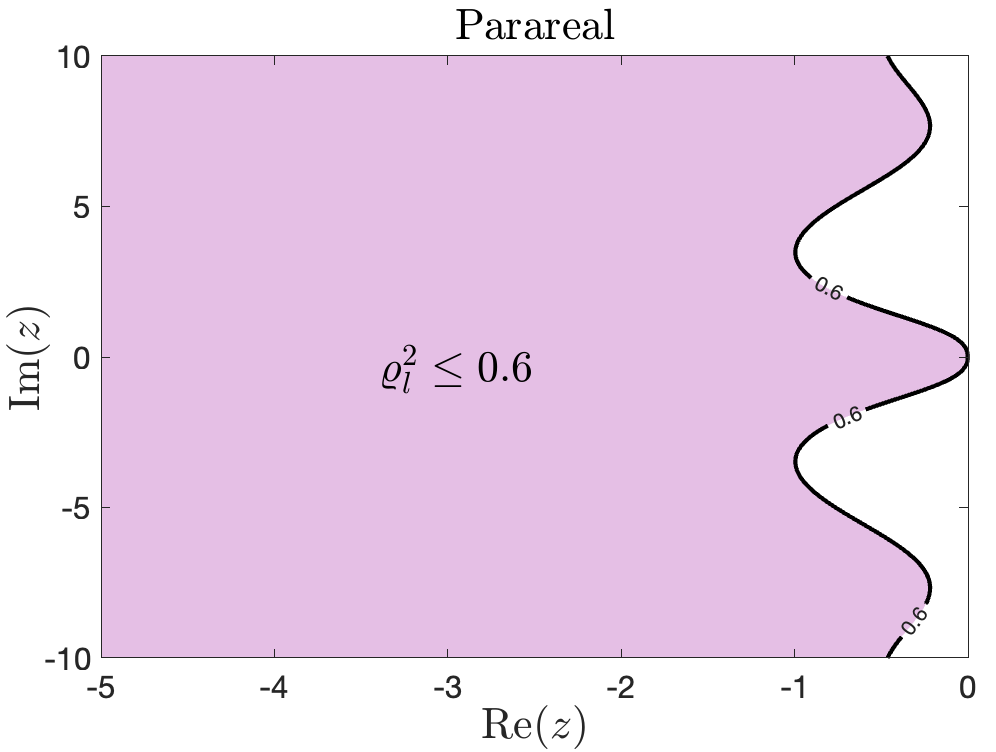}
  \caption{The regions where $\varrho_{l}\leq\hat{\varrho}$ in the
    left half of the complex plane for MGRiT  with FCF-relaxation 
    (top row) and Parareal (bottom row).}
  \label{Fig_Contourf_MGRiT_Parareal_a}
\end{figure}
the regions  where $\varrho_{l,{\rm mgrit}}\leq \hat{\varrho}$ and
  $\varrho_{l,{\rm parareal}}^2\leq \hat{\varrho}$ with
$\hat{\varrho}=0.2, 0.4, 0.6$ in the left half of the complex plane
(i.e., $z\in\mathbb{C}^-$) for Parareal and MGRiT\footnote{For a
    fair comparison, we compare two Parareal iterations with one MGRiT
    iteration with FCF relaxation, since both then use two fine solves.}. Here, we chose
Backward Euler for $\CG$ and the exact time integrator $\CF={\rm
  exp}(\Delta TA)$ for $\CF$. The stability functions of these two
solvers are ${\rm R}_g(z)=\frac{1}{1-z}$ and ${\rm
  R}_f(z)=e^{z}$.  We see that for a given upper bound $\hat{\varrho}$,
the regions where $\varrho_{l,{\rm mgrit}}\leq \hat{\varrho}$ are
comparable to those  where $\varrho_{l,{\rm parareal}}^{s}\leq
\hat{\varrho}$. For other fine time solvers, such as the two SDIRK
methods in \eqref{twoSDIRK}, the results
look similar, and the above conclusion holds as well.

A more quantitative comparison for the case $z\in\mathbb{R}^-$ is
\begin{theorem}\cite{WZSiSC19}\label{MGRiT_rhol_real}
  {\em Suppose we use an L-stable time integrator for $\CF$ and the
    ratio between $\Delta T$ and $\Delta t$, i.e., $J=\frac{\Delta
      T}{\Delta t}$, satisfies $J=\CO(1)$. Then, if we use the
    backward Euler method for $\CG$, it holds that
$$
{\max}_{z\geq0}\varrho_{l}\approx
\begin{cases}
 0.2984, &\mbox{Parareal},\\
 0.1115, &\mbox{MGRiT with FCF-relaxation}. 
\end{cases}
$$
If we use the LIIIC-2 method (i.e., the 2nd-order Lobatto IIIC Runge-Kutta method) for $\CG$,  
$$
{\max}_{z\geq0}\varrho_{l}\approx
\begin{cases}
 0.0817, &\mbox{Parareal},\\
 0.0197, &\mbox{MGRiT with FCF-relaxation}. 
\end{cases}
$$}
\end{theorem} 
Therefore, when using Backward Euler for $\CG$, the convergence factor
of one MGRiT iteration with FCF-relaxation is a bit worse than the
convergence factor of 2 Parareal iterations  ($0.2984^2=0.0890<0.1115$
and $0.0817^2=0.0067<0.0197$), and the cost in fine solves of one
MGRiT iteration with FCF-relaxation is the same as the cost of 2
Parareal iterations. 

We now compare the convergence rates of Parareal and MGRiT by applying
them to the heat equation \eqref{heatequation} and the
advection-diffusion equation \eqref{ADE}. We impose homogeneous
Dirichlet boundary conditions for the heat equation and periodic
boundary conditions for ADE. For both PDEs the initial condition is
$u(x, 0)=\sin^2(8\pi (1-x)^2)$ for $x\in(0, 1)$. The problem and \ds
parameters are $T=5$, $J=20$, $\Delta T=\frac{1}{8}$, and $\Delta
x=\frac{1}{160}$. For $\CG$, we use Backward Euler, and for $\CF$
SDIRK22 from \eqref{twoSDIRK}. We show in Figure
  \ref{Fig_rho_MGRiT_Parareal}
\begin{figure}
  \centering
 \includegraphics[width=1.5in,height=1.25in,angle=0]{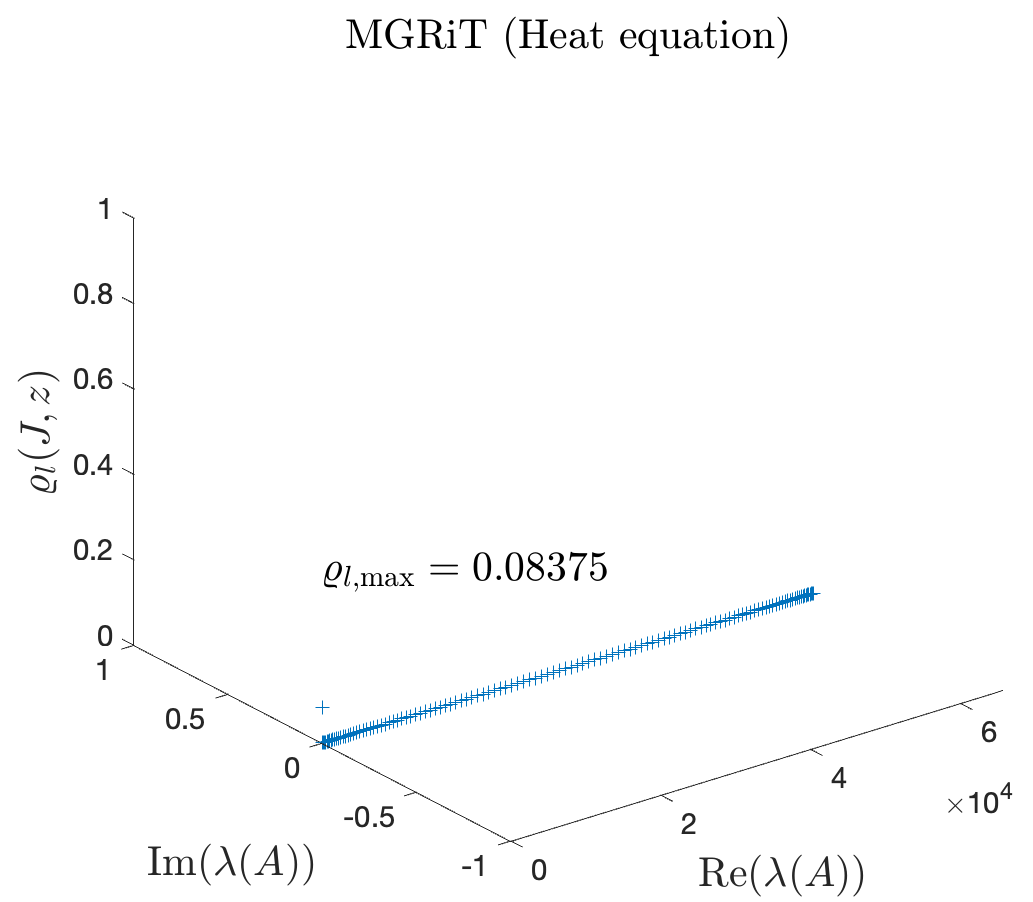}~ \includegraphics[width=1.5in,height=1.25in,angle=0]{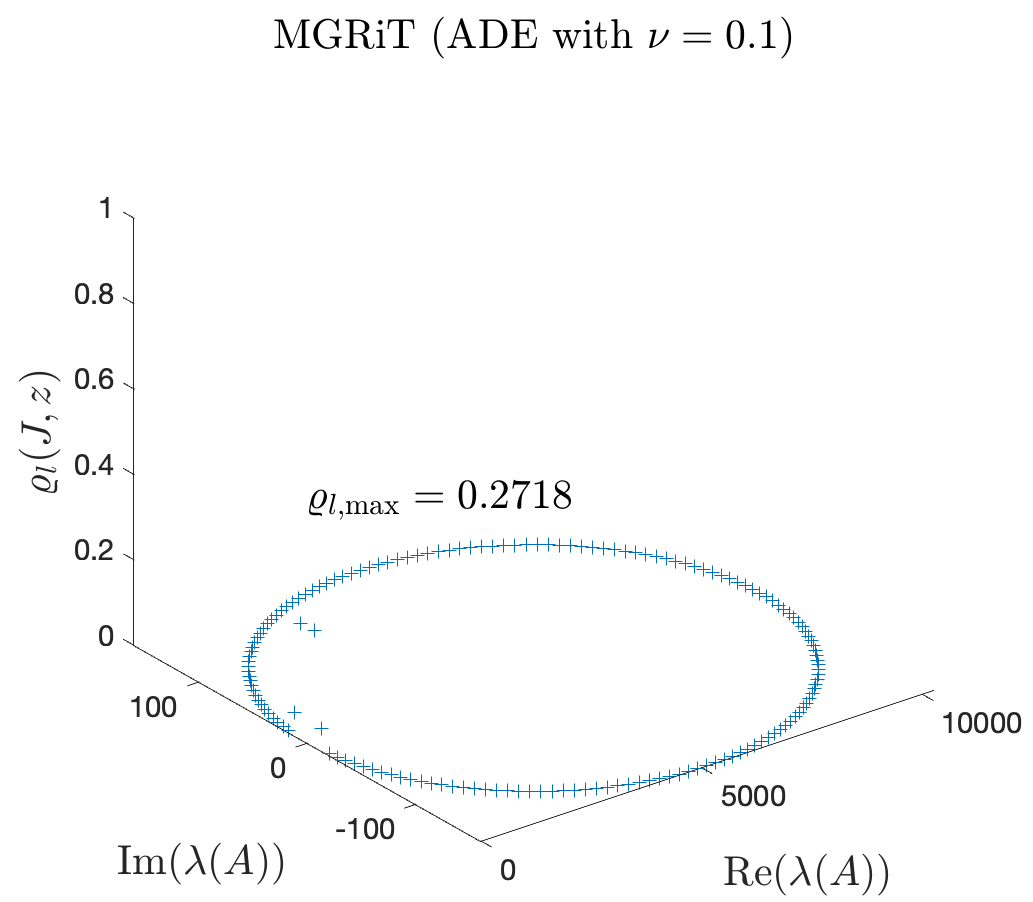} ~
 \includegraphics[width=1.5in,height=1.25in,angle=0]{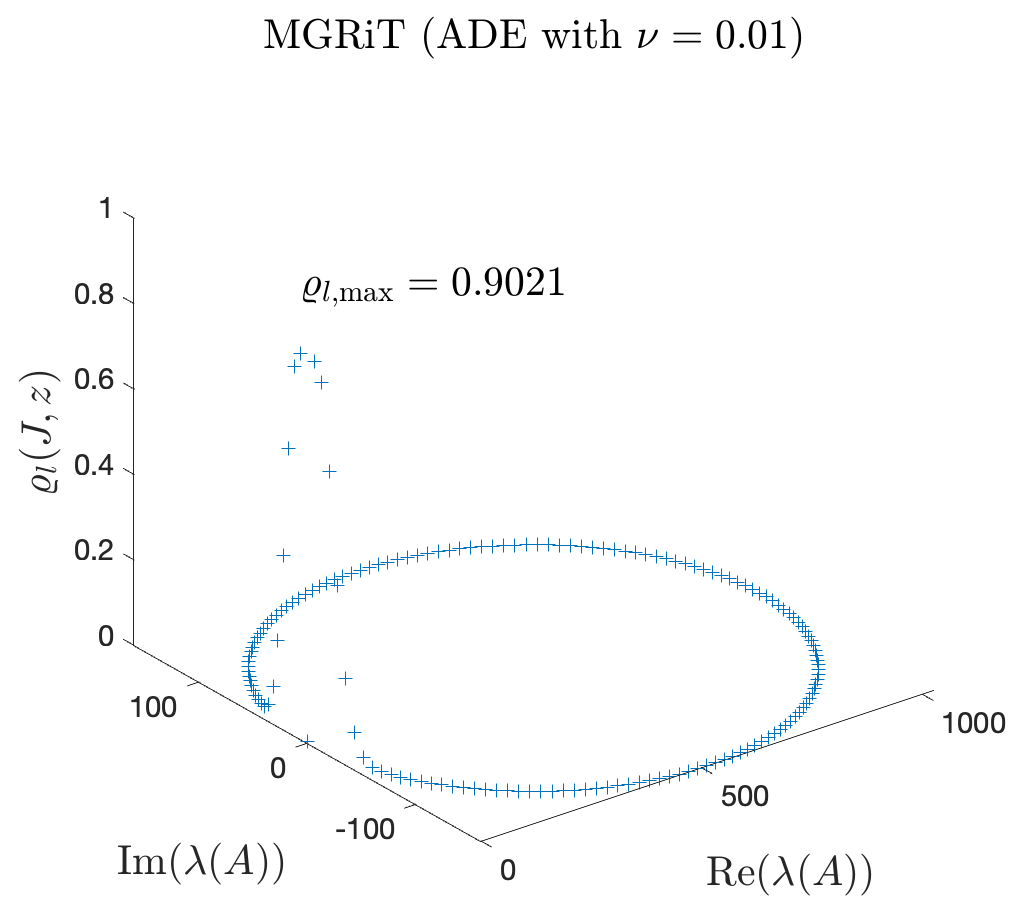}\\
 \includegraphics[width=1.5in,height=1.25in,angle=0]{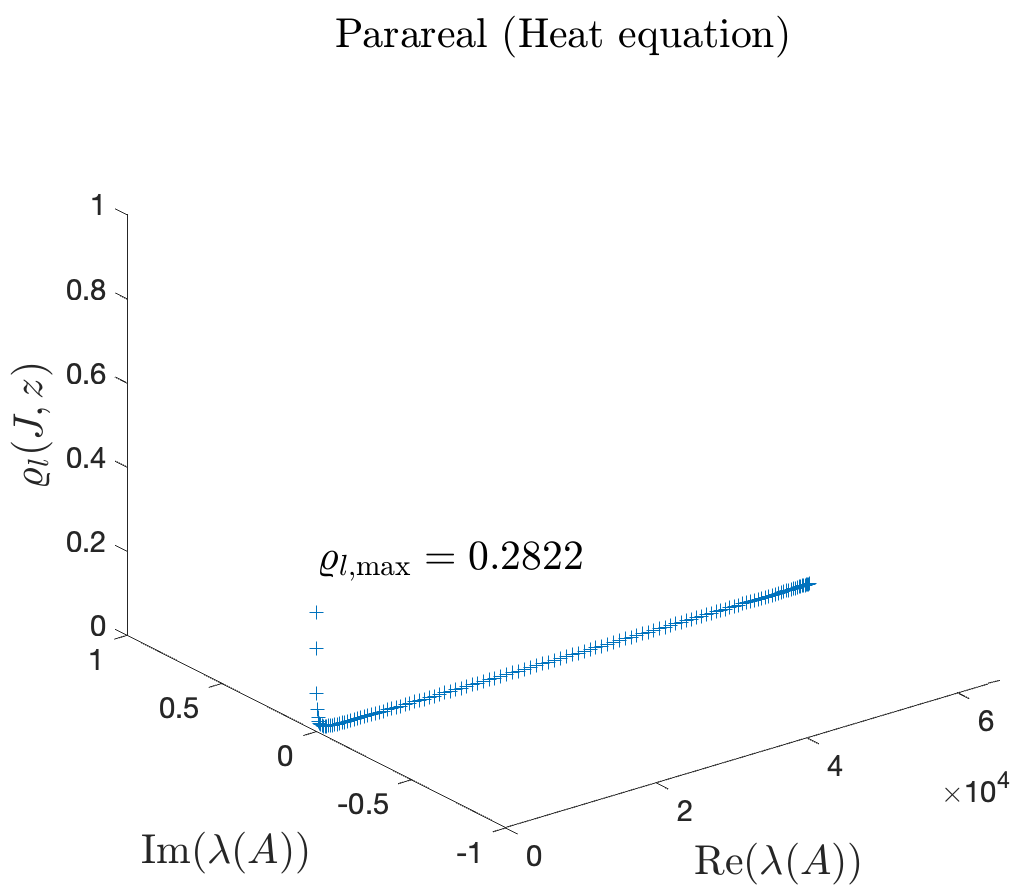}~ \includegraphics[width=1.5in,height=1.25in,angle=0]{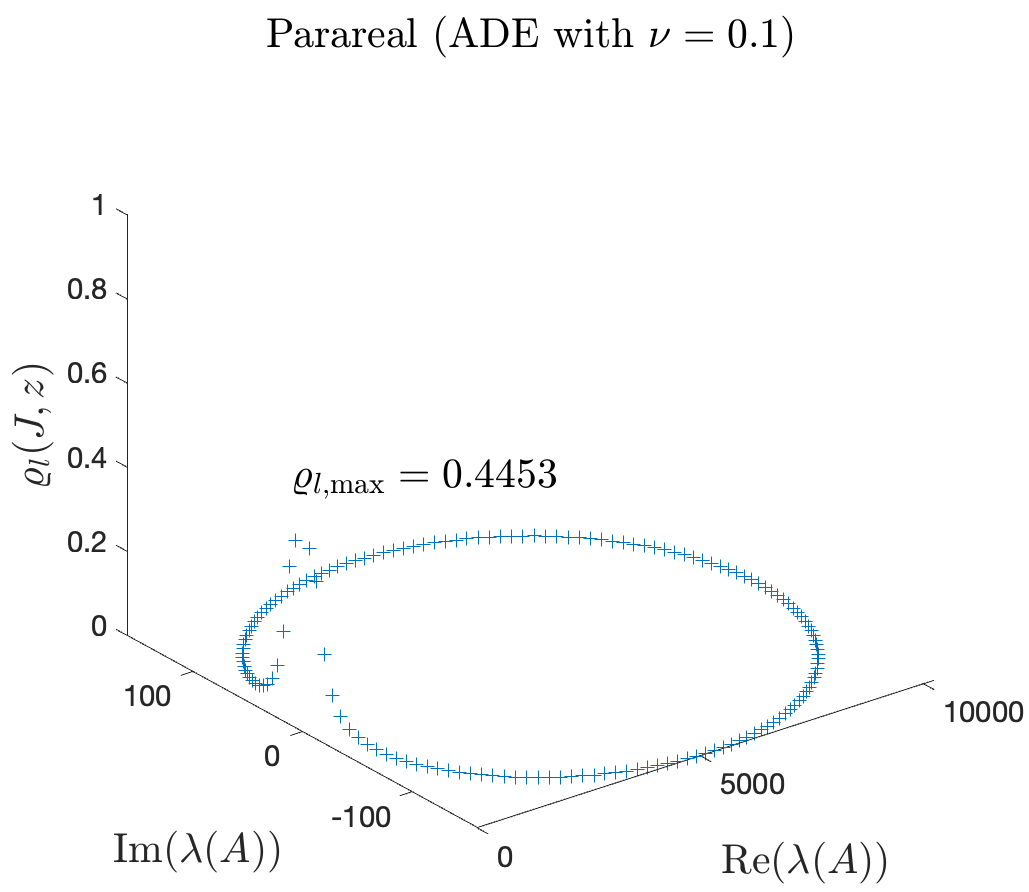} ~
 \includegraphics[width=1.5in,height=1.25in,angle=0]{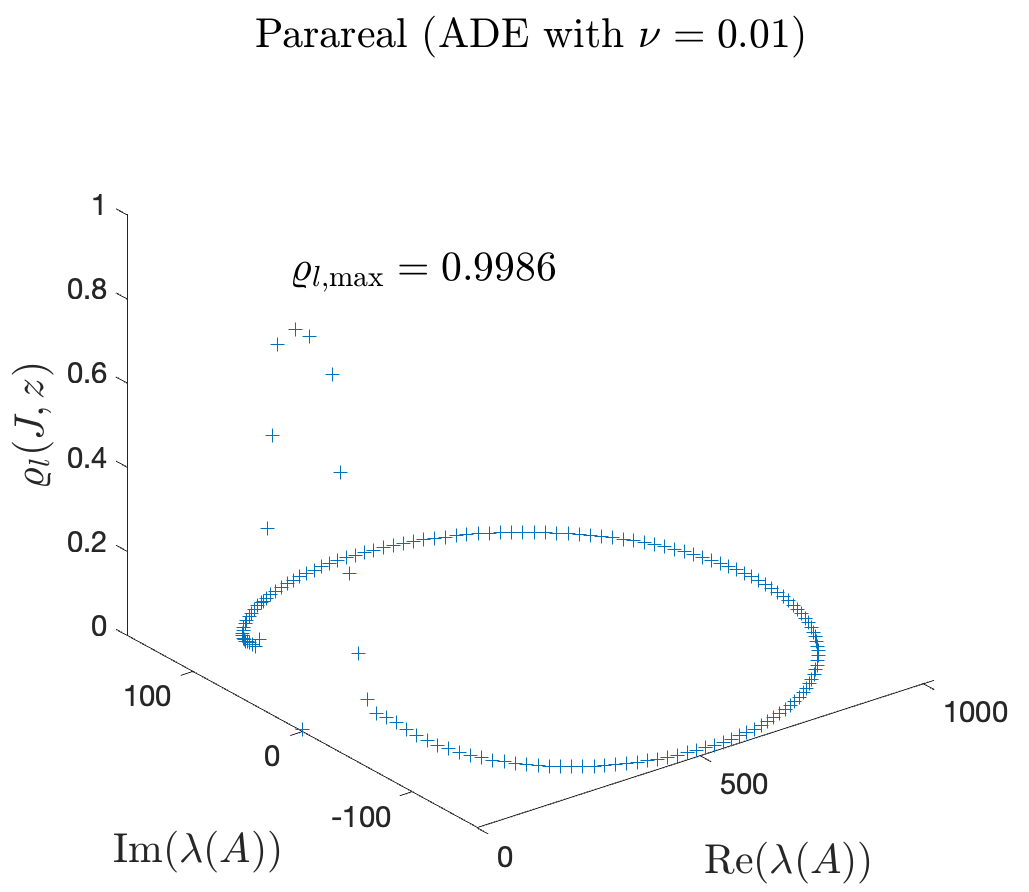}\\
  \caption{The distribution of $\varrho_{l}(J, z)$ for
    $z\in\sigma(\Delta TA)$ for the heat equation and the
    advection-diffusion equation (ADE) with two values of the
    diffusion parameter $\nu$.  Top row: MGRiT; Bottom row:
    Parareal. In each panel, $\varrho_{l,\max}=\max_{z\in\sigma(\Delta
      TA)}\varrho_{l}(z)$.}
  \label{Fig_rho_MGRiT_Parareal}
\end{figure}
for both the heat equation and the advection-diffusion equation with
two different values of the diffusion parameter $\nu$ the quantity
$\varrho_{l}(J, z)$ for $z\in\sigma(\Delta TA)$. The maximum, denoted
by $\varrho_{l,\max}:=\max_{z\in\sigma(\Delta TA)}\varrho_l(J,z)$,
represents the convergence factor for the two methods. It is evident
that for the heat equation and the ADE with $\nu=0.1$ the convergence
factor of MGRiT with FCF-relaxation approximately equals to the square
of that of Parareal, indicating that MGRiT with FCF-relaxation
converges twice as fast as Parareal,  but also at twice the cost,
  since it uses two $\CF$ solves, and Parareal only one, which is
  consistent with Theorem \ref{MGRiT_rhol_real}.  For ADE with
$\nu=0.01$, the convergence factor of both MGRiT and Parareal are
close to 1 indicating very slow converge for both. This is due to
  the fact that the coarse propagator is not good enough any more for
  small $\nu$ when advection dominates. Note also that this has a
  greater impact for Parareal, which uses the coarse propagator after
  each fine solve, see \eqref{Parareal}, whereas MGRiT with
  FCF-relaxation performs two consecutive fine solves without coarse
  solve in between, see \eqref{MGRiT}.  In Figure
\ref{Fig_Error_MGRiT_Parareal},
\begin{figure}
  \centering
 \includegraphics[width=2.3in,height=1.85in,angle=0]{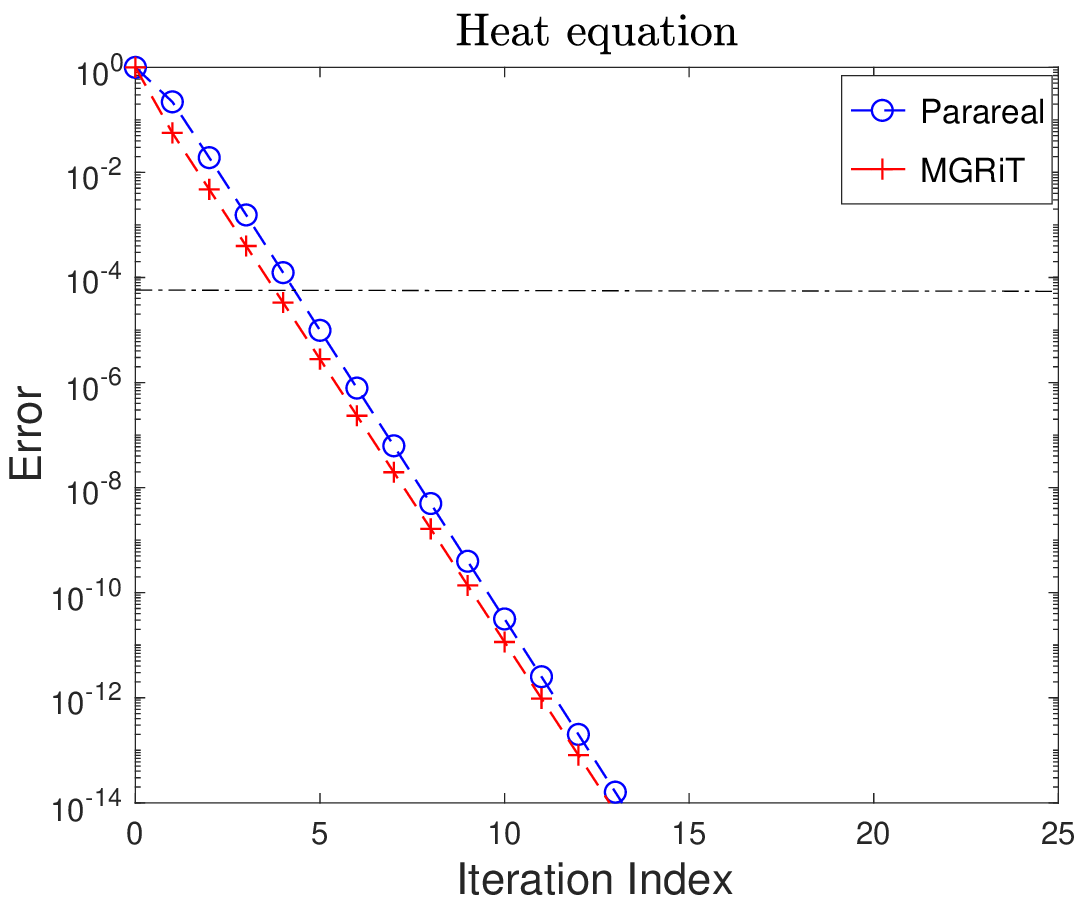}~ \includegraphics[width=2.3in,height=1.85in,angle=0]{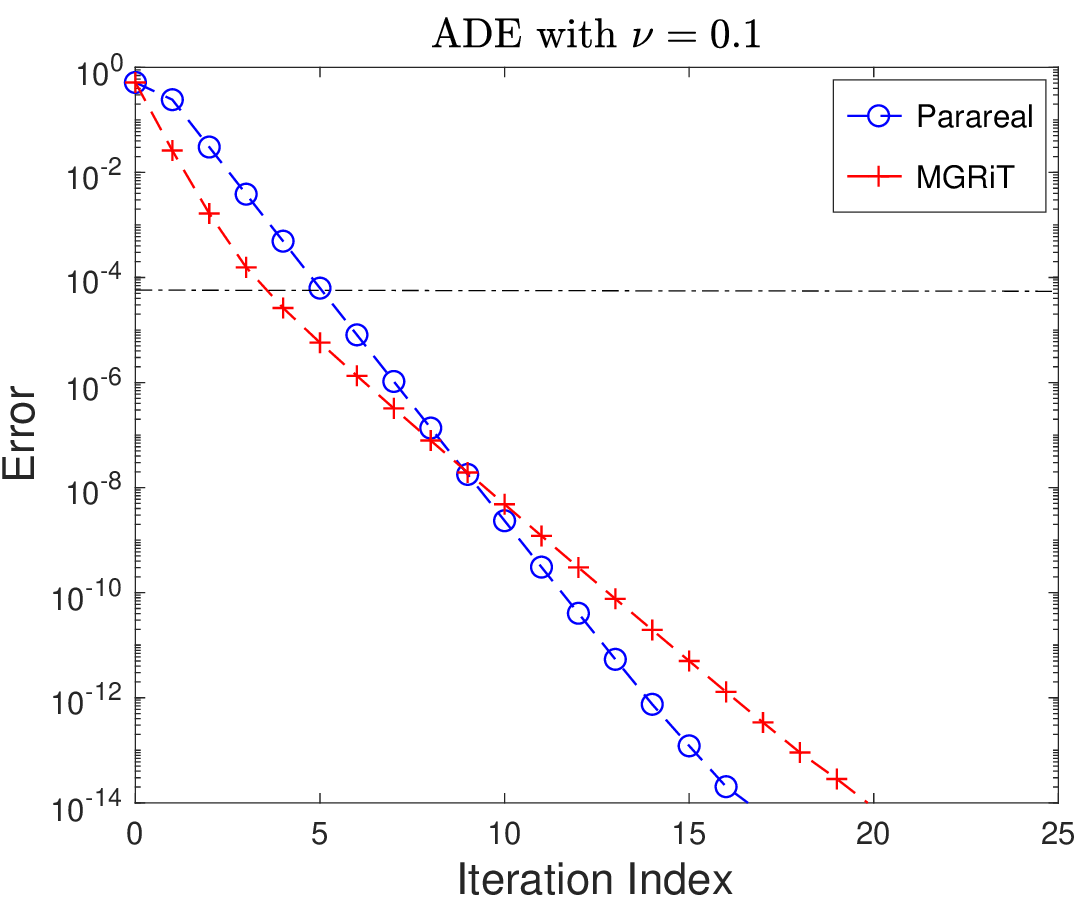} \\
 \includegraphics[width=2.3in,height=1.85in,angle=0]{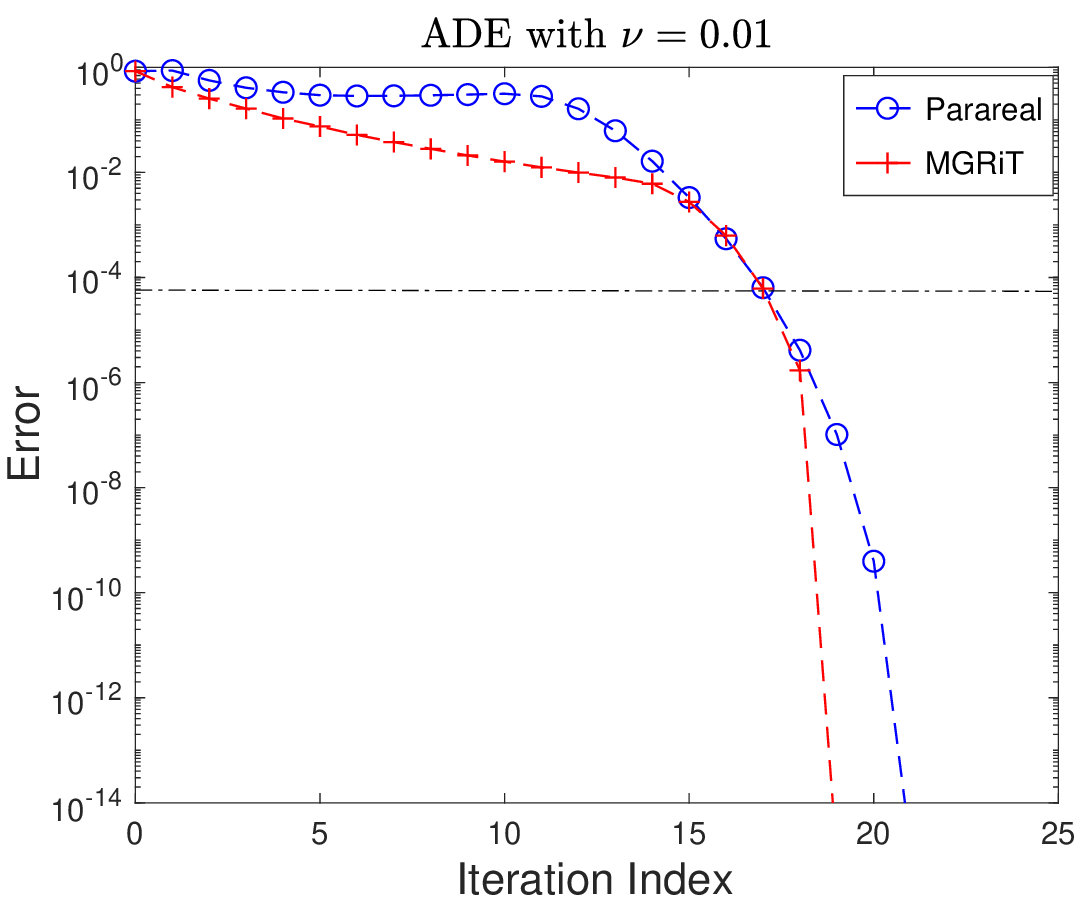}
  \includegraphics[width=2.3in,height=1.85in,angle=0]{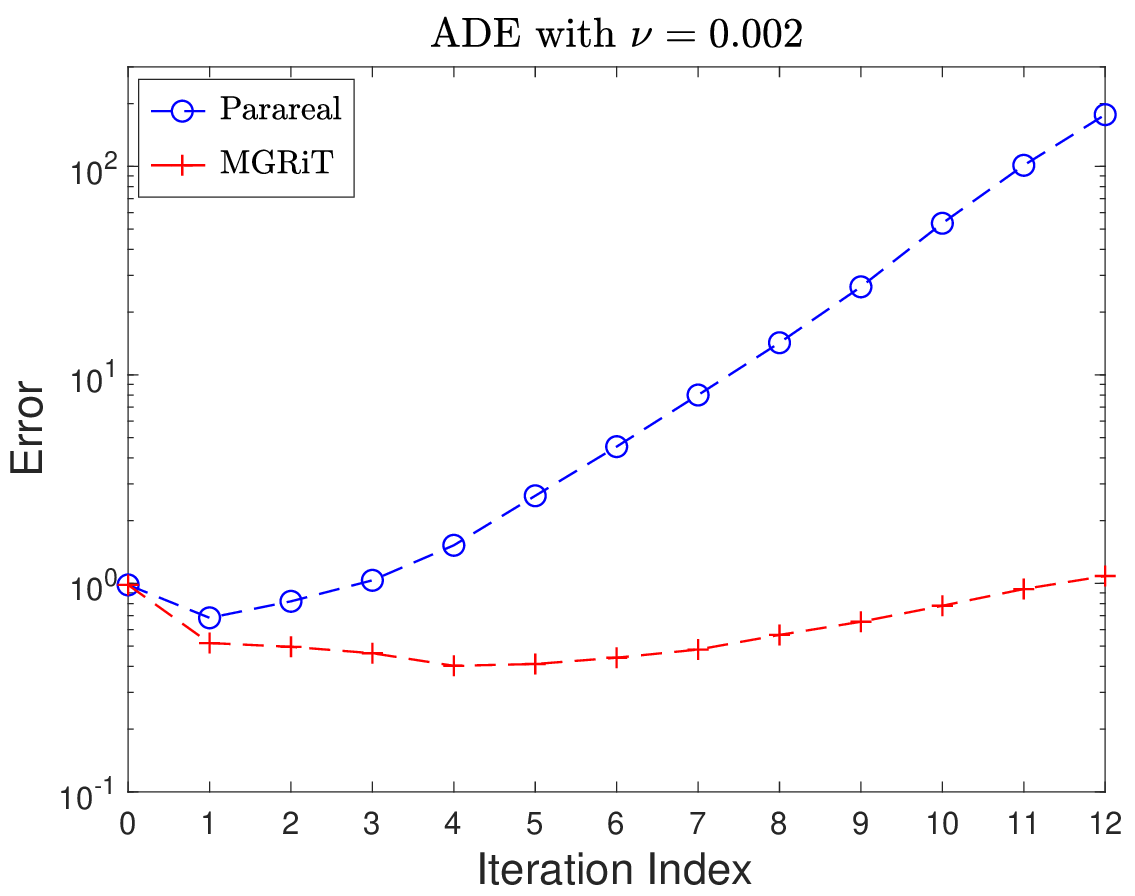}
 \caption{Measured errors for Parareal and MGRiT for the heat equation
   and the advection-diffusion equation (ADE) with three values of the
   diffusion parameter $\nu$. The dash-dotted line in each panel
     indicates the order of the truncation error $\max\{\Delta t^2,
     \Delta x^2\}$ of the discretization, beyond which one would not
     iterate in practice.}
  \label{Fig_Error_MGRiT_Parareal}
\end{figure}
we present the measured errors for the two methods, where we plot
  for Parareal each double iteration as one, in order to to use two
  fine solves, like in MGRiT with FCF-relaxation. We see that for the
  heat equation and ADE with $\nu=0.1$ convergence of Parareal and
  MGRiT is very similar, as expected, but with advection convergence
  is worse. This is confirmed for ADE with $\nu=0.01$, where both
  MGRiT with FCF-relaxation and Parareal now converge very slowly, and
  we observe that Parareal deterioration is more pronounced due to the
  use of the ineffective coarse solve after each fine solve, as
  indicated by the analysis. Finally for $\nu=0.002$ both Parareal and
  MGRiT diverge, the corresponding values are $\varrho_{l,\max}=
  1.4211$ for Parareal and $\varrho_{l,\max}=1.2812$ for MGRiT. These
  results clearly show convergence problems of both methods when
  approaching the hyperbolic regime.
  
Like Parareal, MGRiT can also be applied to nonlinear problems, where
$\CG$ and $\CF$ require the use of some nonlinear solver. A
convergence analysis of MGRiT for nonlinear cases can be found in
\cite{GanKZ18}, under the assumption of certain Lipschitz conditions
for $\CG$, $\CF$, and their difference. The main conclusion is as
follows: one MGRiT iteration with FCF-relaxation (thus using two fine
solves) contracts similarly to two Parareal iterations (also using two
fine solves) as long as the coarse solver $\CG$ is reasonably
accurate. To illustrate this, we apply MGRiT with FCF-relaxation and
Parareal to Burgers' equation \eqref{Burgers} with homogeneous
Dirichlet boundary conditions and initial condition
$u(x,0)=\sin^2(8\pi (1-x)^2)$ for $x\in(0, 1)$, $T=5$ and
discretization parameters $\Delta T=\frac{1}{16}$, $\Delta
x=\frac{1}{160}$, and $J=10$. We use centered finite differences for
the space discretization, and for the time discretization we use
Backward Euler for $\CG$ and SDIRK22 \eqref{twoSDIRK} for $\CF$. In
Figure \ref{Fig_Error_MGRiT_Parareal_Burgers},
\begin{figure}
  \centering
 \includegraphics[width=1.5in,height=1.25in,angle=0]{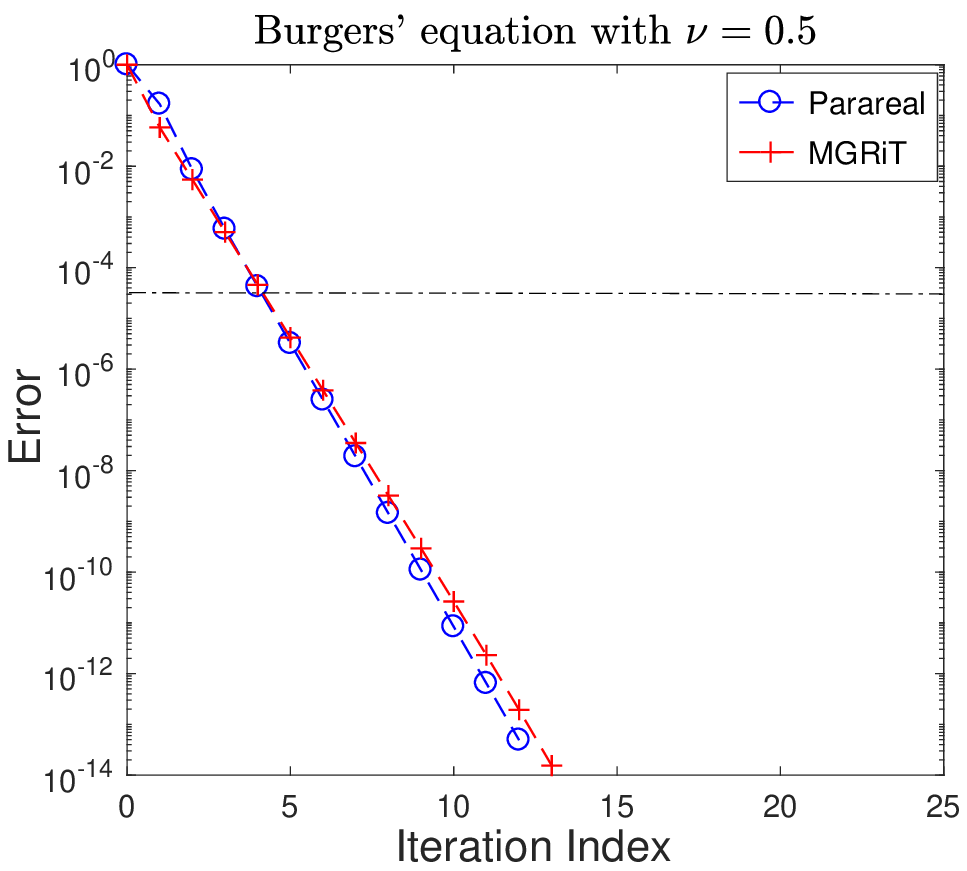}~ 
 \includegraphics[width=1.5in,height=1.25in,angle=0]{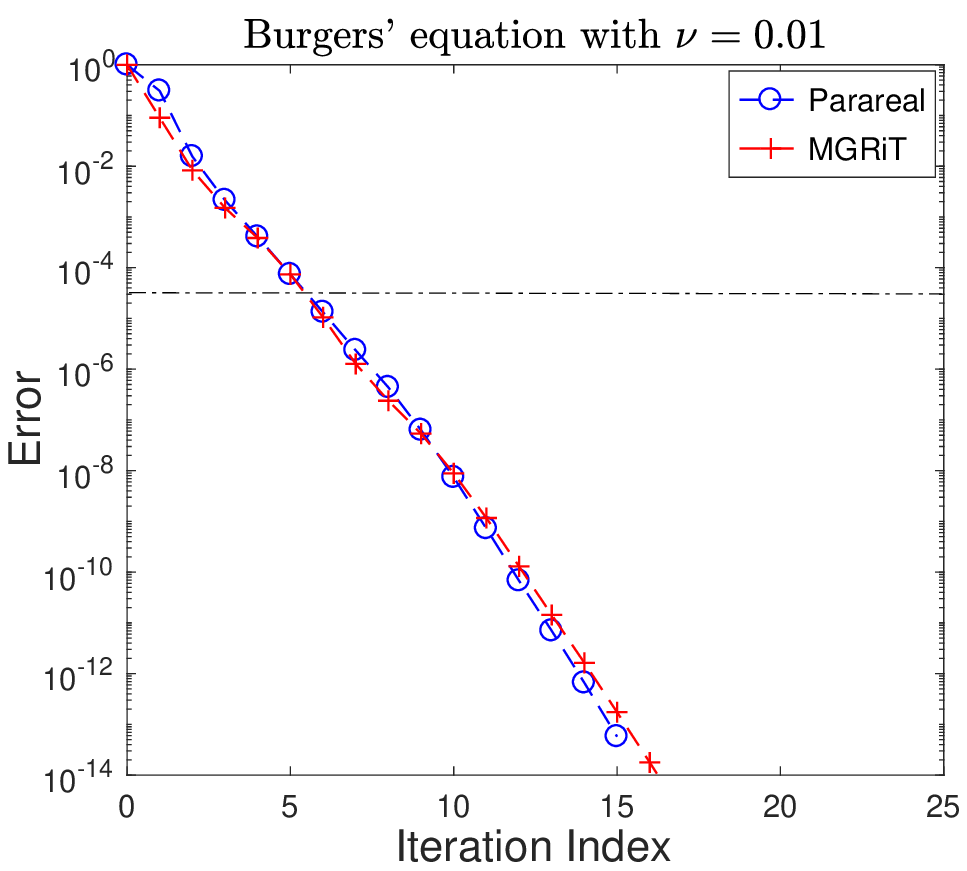} ~
 \includegraphics[width=1.5in,height=1.25in,angle=0]{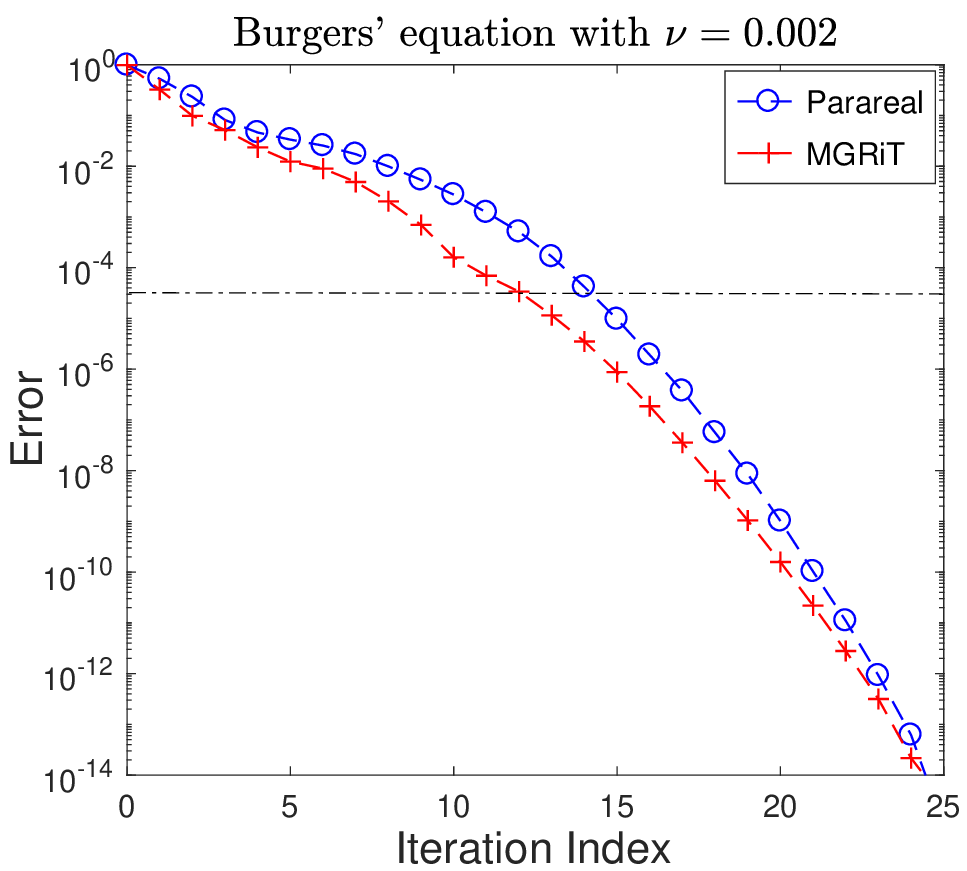} 
  \caption{Error of Parareal and MGRiT for Burgers' equation
    \eqref{Burgers} with three values of the diffusion parameter
    $\nu$.}
  \label{Fig_Error_MGRiT_Parareal_Burgers}
\end{figure}
we show the errors for three values of the diffusion parameter $\nu$,
plotting again a double iteration of Parareal for one iteration of
MGRiT with FCF-relaxation to measure the same number of fine
solves. We see that for each $\nu$, MGRiT converges like two steps of
Parareal again, like for the linear problems in Figure
\ref{Fig_Error_MGRiT_Parareal}.

\subsection{Diagonalization-based Parareal}\label{Sec4.5}

The third variant of Parareal we want to explain uses ParaDiag in the
coarse grid correction (CGC). There are two approaches: the first one
uses a head-tail coupled condition to permit the CGC to be solved in
one shot with ParaDiag; see \cite{WSiSC18} and \cite{WZSiSC19}. The
second one involves designing a special coarse solver that closely
approximates the fine solver, but can be applied at low cost for each
large time interval $[T_n, T_{n+1}]$ using ParaDiag (see
\cite{gander2020diagonalization}). The two Parareal variants have
distinct mechanisms, convergence properties, and scopes of
application.

\subsubsection{Diagonalization-based CGC}\label{Sec4.5.1}

For the initial-value problem ${\bm u}'=f({\bm u})$ with ${\bm
  u}(0)={\bm u}_0$, the standard CGC for Parareal follows a sequential
procedure. Specifically, we solve for $\{{\bm u}_{n}^{k+1}\}$
step-by-step as
\begin{equation}\label{seq_CGC}
{\bm u}_{n+1}^{k+1}=\CG(T_n, T_{n+1}, {\bm u}_n^{k+1})+{\bm b}_{n+1}^k, \quad n=0,1,\dots, N_t-1,
\end{equation}
starting from the initial condition ${\bm u}_0^{k+1}={\bm u}_0$, where
${\bm b}_{n+1}^k=\CF(T_n, T_{n+1}, {\bm u}_n^k)-\CG(T_n, T_{n+1}, {\bm
  u}_n^{k})$ is determined from the previous iteration. The idea in
\cite{WSiSC18} was to impose a  head-tail type coupling condition
${\bm u}_0^{k+1}=\alpha {\bm u}_{N_t}^{k+1}+{\bm u}_0$ for the
CGC, which is different from the more natural head-tail condition
  we have seen in ParaDiag II in \eqref{ParaDiagII_WR} that appeared a
  year later in \cite{gander2019convergence}, but turned out to work
  equally well in the present context. To use this heat-tail type
  coupling condition, we must redefine ${\bm
  b}_{n+1}^k$ as
$$
{\bm b}_{n+1}^k=\CF(T_n, T_{n+1}, \tilde{\bm u}_n^k)-\CG(T_n, T_{n+1}, {\bm u}_n^{k}),
$$
where
$$
\tilde{\bm u}_n^k =
\begin{cases}
{\bm u}_n^k, & \text{if } n\geq1,\\
{\bm u}_0, & \text{if } n=0.
\end{cases}
$$
This redefinition is necessary to ensure that the iterates converge
to the solution of the underlying ODEs. In summary, the Parareal
variant proposed in \cite{WSiSC18}  with this head-tail type coupling
condition is 
\begin{equation}\label{para_CGC}
\begin{split}
&{\bm u}_{n+1}^{k+1}=\CG(T_n, T_{n+1}, {\bm u}_n^{k+1})+\CF(T_n, T_{n+1}, \tilde{\bm u}_n^k)-\CG(T_n, T_{n+1}, {\bm u}_n^{k}), \\
&{\bm u}_0^{k+1}=\alpha {\bm u}_{N_t}^{k+1}+{\bm u}_0,
\end{split}
\end{equation}
where $n=0,\dots, N_t-1$.

We first explain how to implement this variant of Parareal for the
system of linear ODEs ${\bm u}'=A{\bm u}$ with initial condition ${\bm
  u}(0)={\bm u}_0$ and $t\in(0, T)$, the nonlinear case will be
addressed at the end of this section. We use Backward
Euler for the coarse solver $\CG$ and an arbitrary one-step
time integrator for the fine solver $\CF$. By noting that $\CG(T_n,
T_{n+1}, {\bm u}_n^{k+1})=(I_x-\Delta TA)^{-1}{\bm u}_n^{k+1}$, the
new CGC \eqref{para_CGC} involves solving the  $N_t$ linear equations
\begin{equation*}
\begin{cases}
(I_x-\Delta TA){\bm u}_{1}^{k+1}={\bm u}_0^{k+1}+(I_x-\Delta TA){\bm b}^{k}_{1},\\
(I_x-\Delta TA){\bm u}_{2}^{k+1}={\bm u}_1^{k+1}+(I_x-\Delta TA){\bm b}^{k}_{2},\\
\vdots\\
(I_x-\Delta TA){\bm u}_{N_t}^{k+1}={\bm u}_{N_t-1}^{k+1}+(I_x-\Delta TA){\bm b}^{k}_{N_t},\\
{\bm u}_0^{k+1}=\alpha {\bm u}_{N_t}^{k+1}+{\bm u}_0,
\end{cases}
\end{equation*}
where ${\bm b}_{n+1}^k=\CF(T_n, T_{n+1}, \tilde{\bm u}_n^k)-\CG(T_n,
T_{n+1}, {\bm u}_n^{k})$ is known from the previous iteration. Note
that these linear systems cannot be solved one by one due to the
head-tail coupling condition ${\bm u}_0^{k+1}=\alpha {\bm
  u}_{N_t}^{k+1}+{\bm u}_0$. Substituting this condition into the
first equation leads to the all-at-once system
\begin{equation}\label{para_CGC_LinearAAA}
(C_\alpha\otimes I_x- I_t\otimes {\Delta T}A){\bm U}^{k+1}={\bm g}^{k},
\end{equation}
where ${\bm U}^{k+1}=(({\bm u}_1^{k+1})^\top, ({\bm u}_2^{k+1})^\top, \dots, ({\bm u}_{N_t}^{k+1})^\top)^\top$ and
$$
C_\alpha=\begin{bmatrix}
1 & & &-\alpha\\
-1 & 1 & &\\
& \ddots & \ddots &\\
& & -1 & 1
\end{bmatrix}\in\mathbb{R}^{N_t\times N_t}, \quad {\bm g}^k=
\begin{bmatrix}
{\bm u}_0+(I_x-\Delta TA){\bm b}^{k}_{1}\\
(I_x-\Delta TA){\bm b}^{k}_{2}\\
\vdots\\
(I_x-\Delta TA){\bm b}^{k}_{N_t}
\end{bmatrix}.
$$
Similar to the ParaDiag II methods introduced in Section
\ref{sec3.6.2}, we can solve for ${\bm U}^{k+1}$ using three steps,
\begin{equation}\label{Para_CGC_3step}
\begin{cases}
{\bm U}^{a, k+1}=({\rm F}\otimes I_x){\bm g}^k, &\text{(step-a)}\\
(\lambda_{n} I_x-{\Delta T}A){\bm u}^{b,k+1}_n={\bm u}^{a,k+1}_n, ~n=1,2,\dots, N_t, &\text{(step-b)}\\
{\bm U}^{k+1}=({\rm F}^*\otimes I_x){\bm U}^{b,k+1}, &\text{(step-c)}
\end{cases}
\end{equation}
where $\{\lambda_n\}$ are the eigenvalues of $C_\alpha$
(cf. \eqref{MatD}), ${\rm F}$ is the discrete Fourier matrix
(cf. \eqref{MatF}), and ${\bm U}^{a,k+1}:=(({\bm
  u}_1^{a,k+1})^\top,\dots, ({\bm u}_{N_t}^{a,k+1})^\top)^\top$ along
with ${\bm U}^{b,k+1}:=(({\bm u}_1^{b,k+1})^\top,\dots, ({\bm
  u}_{N_t}^{b,k+1})^\top)^\top$. Through the diagonalization procedure
\eqref{Para_CGC_3step}, the new CGC \eqref{para_CGC} can be solved in
parallel across the coarse time grid.

From \eqref{para_CGC} we see that the head-tail coupled CGC simplifies
to the standard CGC \eqref{seq_CGC} when $\alpha\rightarrow0$, and hence,
one can expect that for sufficiently small $\alpha$, this Parareal variant
converges as fast as the original Parareal algorithm. However, due to the
roundoff error stemming from the diagonalization of $C_\alpha$ (as
discussed in Section \ref{sec3.6.2}), an arbitrarily small $\alpha$ is
impractical, particularly when the working precision is low, e.g.,
single or half precision. Fortunately, it is not necessary to use an
extremely small $\alpha$ for the diagonalization-based CGC to achieve
the same convergence rate as the standard CGC.
\begin{theorem}\cite{WSiSC18}\label{pro_Para_CGC}
  {\em Let $\rho$ denote the convergence factor of standard Parareal
    \eqref{seq_CGC}, and $\rho_{\rm new}$ the convergence factor of
    the new Parareal variant \eqref{para_CGC}, where the coarse solver
    $\CG$ is a stable time integrator. Then it holds that
$$
\rho_{\rm new}=\rho,~\text{if}~\alpha\leq \frac{\rho}{1+\rho}. 
$$}
\end{theorem}
This result was proved for linear problems ${\bm u}'=A{\bm
  u}+g$ where $A$ has negative real eigenvalues. For other scenarios,
such as when $A$ has complex eigenvalues, numerical results suggest
its validity as well. Since the roundoff error resulting from the
diagonalization procedure grows as $\alpha$ decreases, it is clear
that the optimal choice for $\alpha$ is $\alpha=\frac{\rho}{1+\rho}$.
In practice, $\rho=\CO(10^{-1})$ and hence
$\alpha=\frac{\rho}{1+\rho}=\CO(10^{-1})$ as well. The roundoff error
incurred with a parameter $\alpha=\CO(10^{-1})$ is negligible.

To illustrate the convergence of the new Parareal variant
\eqref{para_CGC}, we consider the heat equation \eqref{heatequation}
and the advection-diffusion equation (ADE) \eqref{ADE} with periodic
boundary conditions and an initial condition ${u}(x,0)=\sin(2\pi x)$
for $x\in(0,1)$. We use Backward Euler as the coarse solver $\CG$ and
SDIRK22 from \eqref{twoSDIRK} for the fine solver $\CF$. The data used
here is $T=4$, $J=10$, $\Delta T=0.1$, and $\Delta x=\frac{1}{128}$.
In Figure \ref{Fig_Parareal_PCGC_Heat_ADE},
\begin{figure}
  \centering
 \includegraphics[width=2.3in,height=1.85in,angle=0]{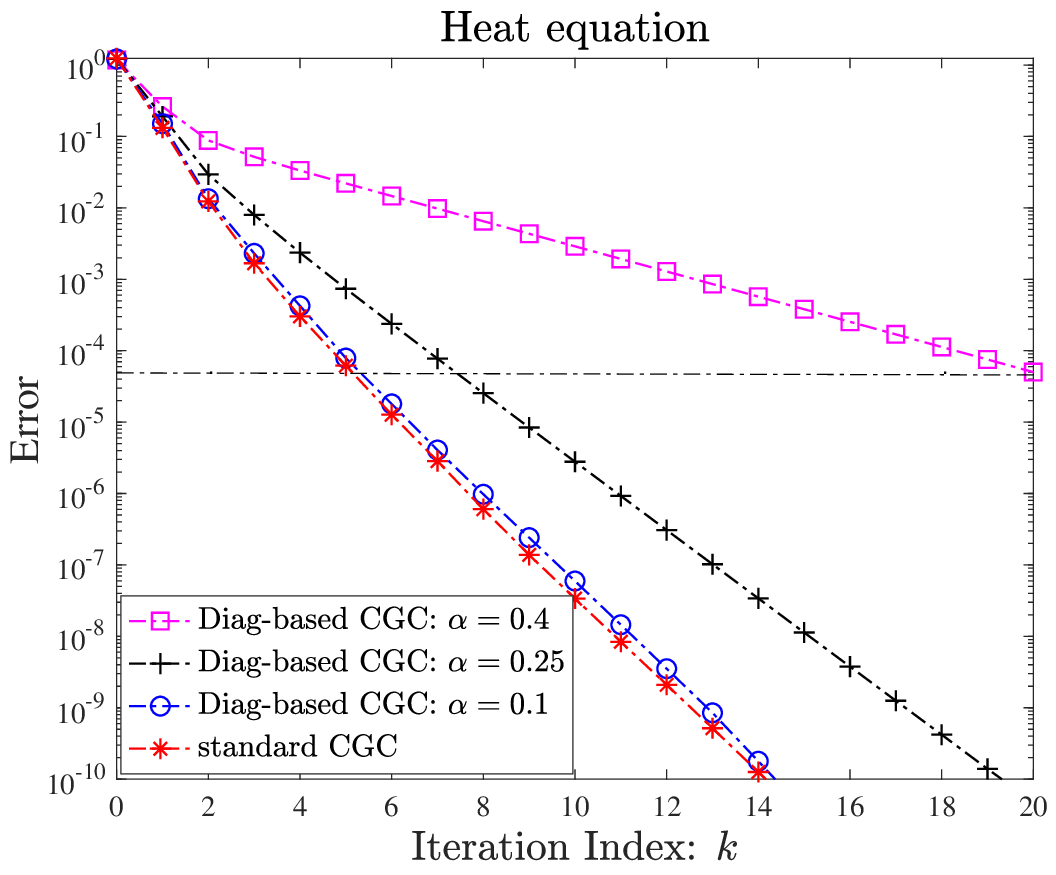}~ \includegraphics[width=2.3in,height=1.85in,angle=0]{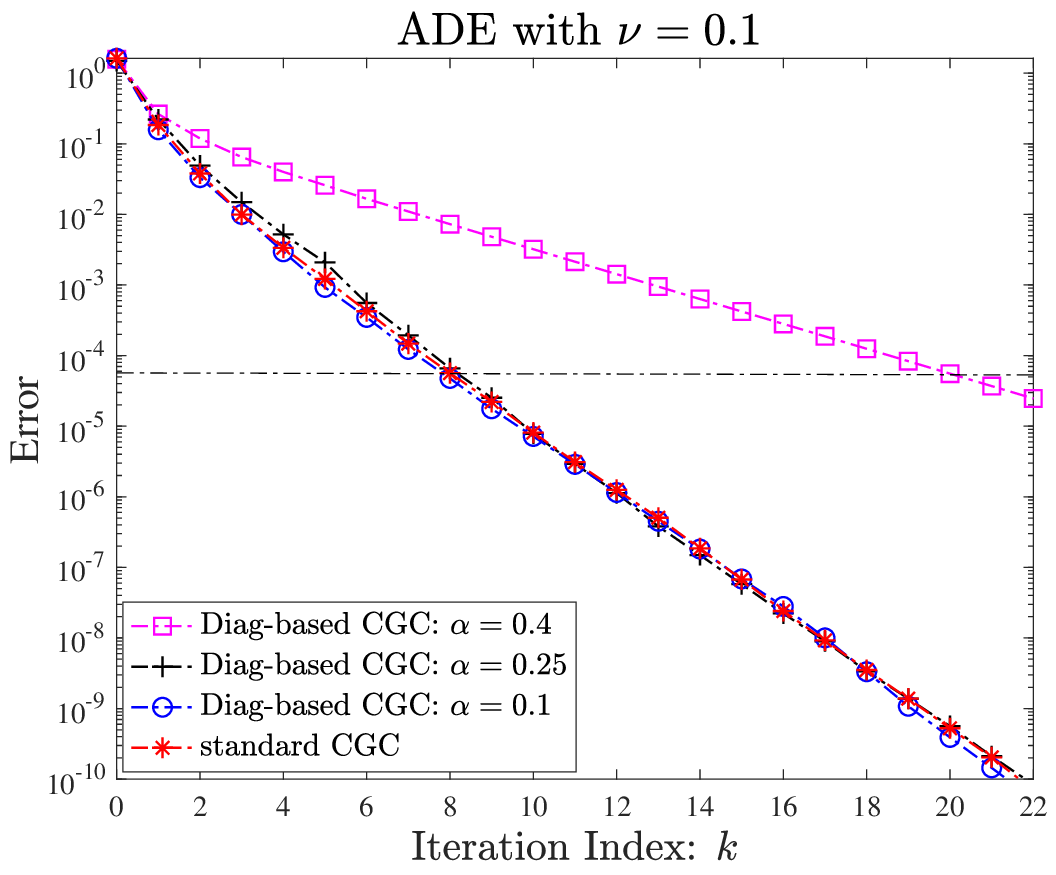}  
  \caption{Measured error of Parareal using the diagonalization-based
    CGC (cf. \eqref{para_CGC}) and the standard CGC (i.e., $\alpha=0$
    in \eqref{seq_CGC}).}
  \label{Fig_Parareal_PCGC_Heat_ADE}
\end{figure}
we present the measured error of Parareal using both the
diagonalization-based and standard CGC. For the diagonalization-based
CGC, we use three values of the parameter $\alpha$ to illustrate how
the convergence rate depends on this parameter. For the heat equation,
the convergence factor of Parareal with standard CGC is
$\rho\approx0.22$, and thus, according to Theorem \ref{pro_Para_CGC},
the threshold for $\alpha$ such that the diagonalization-based CGC
achieves the same convergence rate is
$\frac{\rho}{1+\rho}\approx0.18$. If $\alpha$ exceeds this threshold,
the diagonalization-based CGC results in a slower convergence
rate. Thus the theoretical analysis accurately predicts the numerical
results shown in the left panel. For ADE with $\nu=0.1$,
$\rho\approx0.39$ and the threshold for $\alpha$ is
$\frac{\rho}{1+\rho}\approx0.28$. Hence, as seen in Figure
\ref{Fig_Parareal_PCGC_Heat_ADE} on the right, $\alpha=0.25$ suffices
to let the diagonalization-based Parareal converge as fast as the
standard Parareal.

We next address how to adapt the diagonalization-based CGC to
nonlinear problems of the form ${\bm u}'=f({\bm u})$ with an initial
value ${\bm u}(0)={\bm u}_0$. We continue to use backward Euler for
the coarse solver $\CG$. Then, with ${\bm b}^{k}_{n+1}:=\CF(T_n,
T_{n+1}, \tilde{\bm u}_n^k)-\CG(T_n, T_{n+1}, {\bm u}_n^{k})$ the
initial part of the CGC algorithm \eqref{para_CGC}, i.e., ${\bm
  u}_{n+1}^{k+1}=\CG(T_n, T_{n+1}, {\bm u}_n^{k+1})+{\bm b}^k_{n+1}$,
can be reformulated as
$$
\frac{{\bm u}_{n+1}^{k+1}-{\bm b}_{n+1}^k- {\bm u}_n^{k+1}}{\Delta T}=f({\bm u}_{n+1}^{k+1}-{\bm b}_{n+1}^k),~n=0,1,\dots, N_t-1. 
$$
This, together with the head-tail coupling condition ${\bm
  u}_0^{k+1}=\alpha{\bm u}_{N_t}^{k+1}+{\bm u}_0$, leads to the
nonlinear all-at-once system
\begin{equation}\label{para_CGC_nonLinearAAA}
(C_\alpha\otimes I_x){\bm U}^{k+1}-   \Delta T F({\bm U}^{k+1})= {\bm g}^k, 
\end{equation}  
where the definitions for ${\bm U}^{k+1}$ and $C_\alpha$ remain the same as in \eqref{para_CGC_LinearAAA}, and
$$
 F({\bm U}^{k+1}):=
  \begin{bmatrix}
f({\bm u}_{1}^{k+1}-{\bm b}_{1}^k)\\
f({\bm u}_{2}^{k+1}-{\bm b}_{2}^k)\\
\vdots\\
f({\bm u}_{N_t}^{k+1}-{\bm b}_{N_t}^k)
  \end{bmatrix},\quad{\bm g}^k:=
    \begin{bmatrix}
{\bm b}_{1}^k+{\bm u}_0\\
{\bm b}_2^k\\
\vdots\\
{\bm b}_{N_t}^k
  \end{bmatrix}. 
$$
We solve the nonlinear system \eqref{para_CGC_nonLinearAAA} by a
quasi Newton method as previously for the nonlinear
ParaDiag method (cf. Section \ref{sec3.6.1}),
\begin{subequations}  
\begin{equation}\label{para_CGC_SNI_a}
\begin{split}
&\CP_{\alpha} ^{k+1,l}\Delta {\bm U}^{k+1,l}={\bm g}^k-
(C_\alpha\otimes I_x){\bm U}^{k+1,l}+  \Delta TF({\bm U}^{k+1, l}),\\
&{\bm U}^{k+1,l+1}={\bm U}^{k+1,l}+\Delta{\bm U}^{k+1,l}, 
\end{split}
 \end{equation}
where $l=0,1,\dots, l_{\max}$ denotes the Newton iteration index. The matrix $\CP_{\alpha}^{k+1,l}$ is a block $\alpha$-circulant matrix given by
\begin{equation}\label{para_CGC_SNI_b}
 \CP_{\alpha}^{k+1,l}:=C_\alpha\otimes I_x- I_t\otimes \Delta TA^{k+1, l}, 
 \end{equation} 
\end{subequations} 
where $A^{k+1, l}$ is the average of the Jacobi matrices,
$$
A^{k+1, l}:=\frac{1}{J}{\sum}_{j=1}^J\nabla f({\bm u}_n^{k+1,l}-{\bm b}_n^k). 
$$ 
The Kronecker tensor product $I_t\otimes A^{k+1, l}$ serves as an
approximation of the Jacobi matrix $\nabla F({\bm U}^{k+1,l})$. We
note that the nearest Kronecker product approximation (NKA),
introduced in Section \ref{sec3.6.1}, could also be used to obtain a
better approximation of $\nabla F({\bm U}^{k+1,l})$, but for
simplicity, we do not explore this option here further.

The matrix $\CP_{\alpha}^{k+1,l}$ has the same structure as the
coefficient matrix in \eqref{para_CGC_LinearAAA}, which allows us
to solve for the increment $\Delta {\bm U}^{k+1,l}$ using the
diagonalization procedure outlined in \eqref{Para_CGC_3step}. The
convergence analysis of the new Parareal variant \eqref{para_CGC} in
the nonlinear context is detailed in \cite[Section 4]{WSiSC18}, where
it is shown that the convergence rate mirrors that of Parareal with
standard CGC when $\alpha$ is chosen appropriately small. We illustrate
this by applying Parareal with both CGCs to Burgers'
equation \eqref{Burgers}, using the same problem setup and
discretization parameters as in the previously discussed heat and
advection-diffusion equation case.  In Figure
\ref{Fig_Parareal_PCGC_Burgers},
\begin{figure}
  \centering
 \includegraphics[width=2.3in,height=1.85in,angle=0]{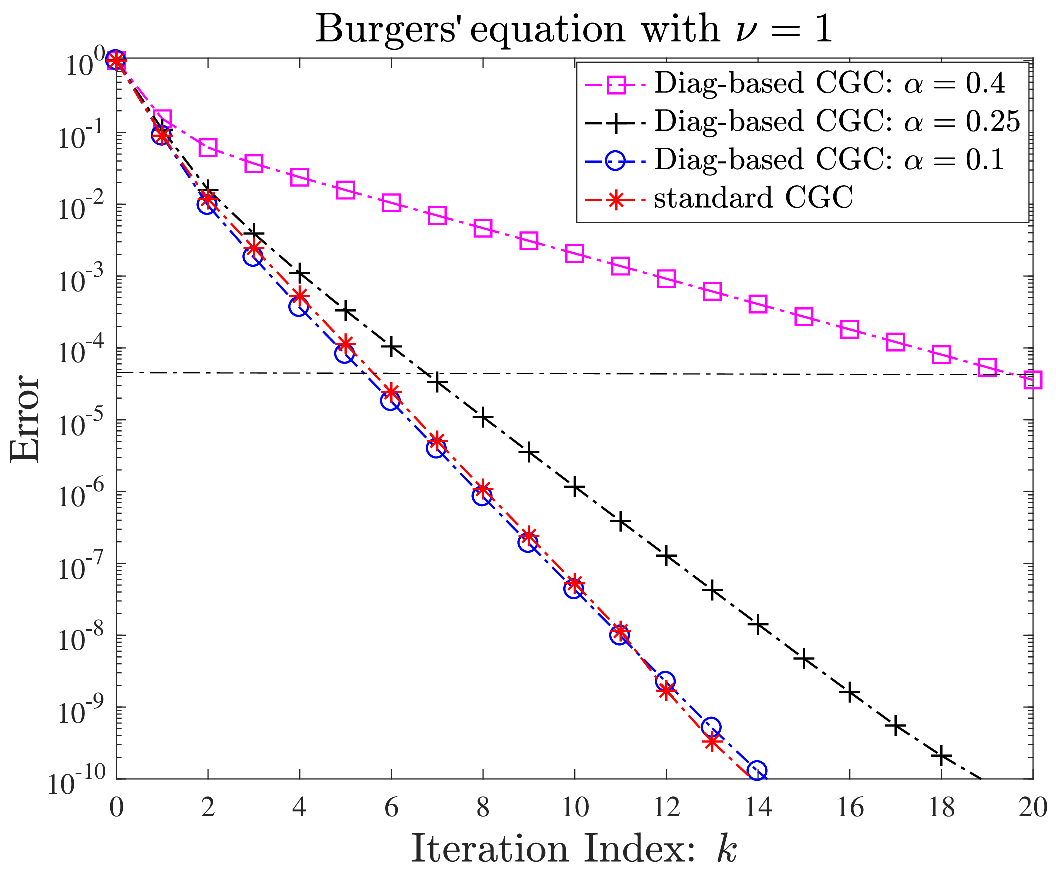}~ \includegraphics[width=2.3in,height=1.85in,angle=0]{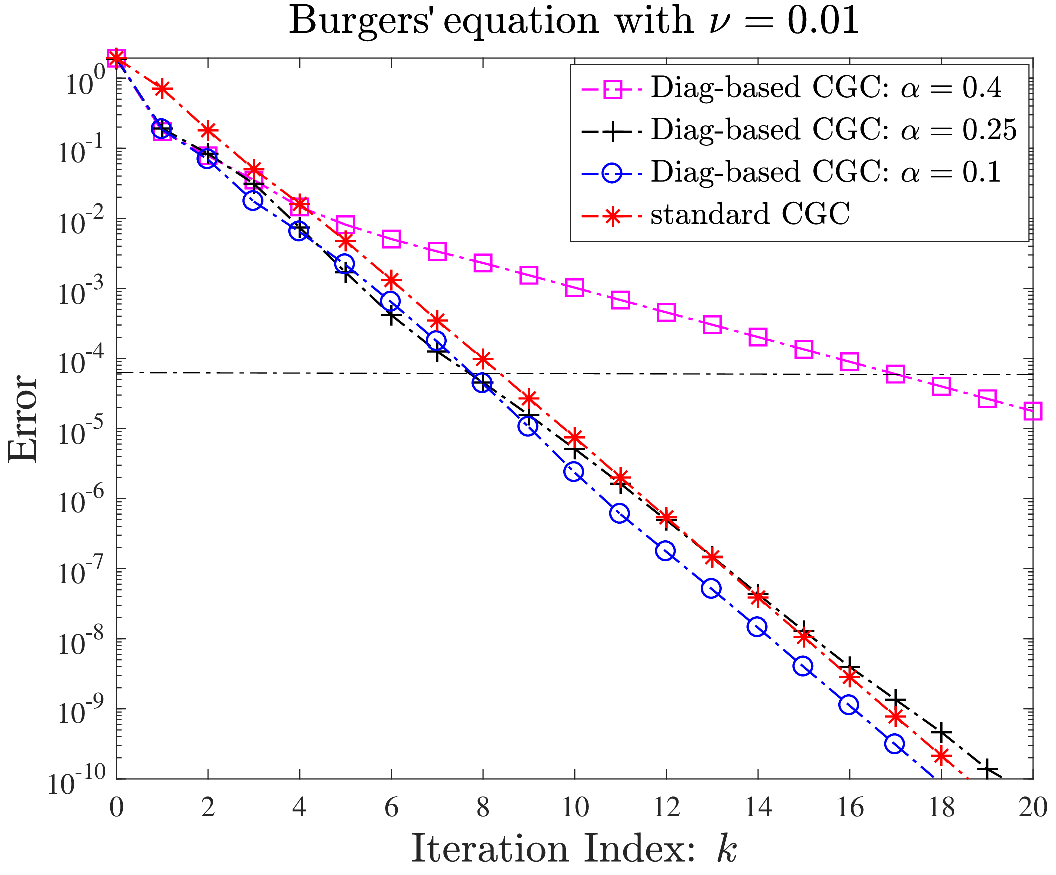}  
  \caption{Measured error of Parareal using the two CGCs for
    Burgers' equation with two values of the diffusion parameter
    $\nu$.}
  \label{Fig_Parareal_PCGC_Burgers}
\end{figure}
we present the measured errors for two values of the diffusion
parameter $\nu$. We see that also for this nonlinear problem, the
influence of $\alpha$ on the convergence rate remains as in the linear
case.

\begin{remark}[Extension to MGRiT]
  \emph{The fundamental mechanism of MGRiT \eqref{MGRiT} aligns with
    that of Parareal, and specifically, the CGC can be similarly
    represented as in \eqref{seq_CGC}. However, a direct extension of
    the diagonalization-based CGC devised for Parareal in
    \eqref{para_CGC} to MGRiT leads to divergence regardless of
      the choice of $\alpha$. This is because the head-tail type
      coupling condition ${\bm u}_1^{k+1}=\alpha {\bm
        u}_{N_t}^{k+1}+{\bm u}_1$ used is less natural than the one
      from \cite{gander2019convergence} that appeared a year later,
      namely ${\bm u}_{1}^{k+1}=\alpha({\bm u}_{N_t}^{k+1}-{\bm
        u}_{N_t}^{k})+{\bm u}_1$. Using this more natural head-tail
      condition which is consistent at convergence for MGRiT was
      proposed in \cite{WZSiSC19}, leading to the convergent MGRiT
    variant
 \begin{equation}\label{MGRiT_para_CGC}
   \begin{split}
&{\bm u}^{k+1}_0={\bm u}_0,~{\bm
       u}_{1}^{k+1}=\alpha({\bm u}_{N_t}^{k+1}-{\bm u}_{N_t}^{k})+{\bm
       u}_1, \\ &{\bm u}^{k+1}_{n+1}=\CG(T_{n}, T_{n+1}, {\bm
       u}^{k+1}_n)+\tilde{\bm b}_{n+1}^k,~n=1,2, \dots, N_t-1,\\
  \end{split}
\end{equation} 
where $\tilde{\bm b}_{n+1}^k=\CF(T_n, T_{n+1}, \tilde{\bm s}^k_n)-\CG(T_n, T_{n+1}, \tilde{\bm s}^k_n)$, $\tilde{s}^k_n:=\CF(T_{n-1}, T_n, \tilde{\bm u}^{k}_{n-1})$ and
$$
\tilde{\bm u}^k_n=
\begin{cases}
{\bm u}_n,~&n=0, 1,\\
{\bm u}_n^k,~&n\geq2. 
\end{cases}
$$
This variant converges at the same rate as the original MGRiT method
\eqref{MGRiT}, provided $\alpha$ is suitably small, Theorem
\ref{pro_Para_CGC} applies to MGRiT in an analogous manner. For
Parareal, one can also use the more natural head-tail condition ${\bm
  u}_{0}^{k+1}=\alpha({\bm u}_{N_t}^{k+1}-{\bm u}_{N_t}^{k})+{\bm
  u}_0$ which is consistent at convergence, instead of
\eqref{para_CGC}, and one obtains the same convergence rate as the one
in Theorem \ref{pro_Para_CGC} shown for the less natural one. }
\end{remark}

\subsubsection{Diagonalization-based Coarse Solver}\label{Sec4.5.2}

A fundamentally distinct idea from the ParaDiag CGC method in Section
\ref{Sec4.5.1} that combines ParaDiag with Parareal was introduced in
\cite{gander2020diagonalization}. The key innovation lies in using the
same time integrator and time step size for both the coarse and fine
solvers, but implementing the coarse solver through a diagonalization
procedure. This approach can work also for hyperbolic problems,
  since it transports all frequency components also in the coarse
  propagator over very long time. 

We illustrate this concept for the nonlinear system of ODEs ${\bm
  u}'=f({\bm u})$ with initial value ${\bm u}(0)={\bm u}_0$. For the
time discretization, we use the linear-$\theta$ method (a
generalization to the $s$-stage Runge-Kutta method can be found in the
Appendix of \cite{gander2020diagonalization}).  On each large time
interval $[T_n, T_{n+1}]$, the fine solver $\CF(T_n, T_{n+1}, {\bm
  u}_n)$ computes the solution at $t=T_{n+1}$ by performing $J$ steps
of the linear-$\theta$ method sequentially, i.e., $\CF(T_n, T_{n+1},
{\bm u}_n)={\bm v}_J$ with ${\bm v}_J$ being the final solution of
\begin{equation}\label{F_solver}
{\bm v}_{j+1}-{\bm v}_{j}=\Delta t[\theta f({\bm v}_{j+1})+(1-\theta)f({\bm v}_{j})],~j=0,1,\dots, J-1,
\end{equation}
with initial condition ${\bm v}_0={\bm u}_n$, where $\Delta
t=\frac{\Delta T}{J}=\frac{T_{n+1}-T_n}{J}$ and $\theta=1$ or
$\theta=\frac{1}{2}$.  For the coarse solver, denoted by
$\CF_\alpha^*(T_n, T_{n+1}, {\bm u}_n)$, we solve the
nonlinear system with a head-tail coupled condition,
\begin{equation}\label{F_solver_alp}
\begin{split}
&{\bm v}_{j+1}-{\bm v}_{j}=\Delta t[\theta f({\bm v}_{j+1})+(1-\theta)f({\bm v}_{j})],~j=0,1,\dots, J-1, \\
&{\bm v}_0=\alpha{\bm v}_J+(1-\alpha){\bm u}_n.
\end{split}
\end{equation}
This system can be recast as the nonlinear all-at-once system
\begin{equation}\label{aaa_F_solver}
\underbrace{(C_\alpha\otimes I_x){\bm V}-\Delta t F({\bm V})}_{:=\CK({\bm V})}={\bm b}({\bm u}_n),
\end{equation}
where ${\bm V}:=({\bm v}^\top_{1}, {\bm v}^\top_{2},\dots, {\bm v}^\top_{J})^\top$, ${\bm b}({\bm u}_n):=((1-\alpha){\bm u}^\top_n, 0,\dots, 0)^\top$, and
\begin{equation}\label{BtB_theta}
\begin{split}
&C_\alpha:= \begin{bmatrix}
  1 & & &-\alpha\\
  -1 &1 & &\\
  &\ddots &\ddots &\\
  & &-1 &1
  \end{bmatrix},\\
  &F({\bm V}):=\begin{bmatrix}
\theta f({\bm v}_{1})+(1-\theta)f(\alpha{\bm v}_{J}+(1-\alpha){\bm u}_n)\\
\theta f({\bm v}_{2})+(1-\theta)f({\bm v}_{1})\\
\vdots\\
\theta f({\bm v}_{J})+(1-\theta)f({\bm v}_{J-1})
  \end{bmatrix}.
\end{split}
\end{equation}
We solve \eqref{aaa_F_solver} with the quasi Newton method 
\begin{subequations}  
\begin{equation}\label{Sec4_4_SNI_a}
\CP_{\alpha}({\bm V}^{l}) \Delta {\bm V}^l = {\bm b}({\bm u}_n) -\CK({\bm V}^l), \quad {\bm V}^{l+1} = {\bm V}^l + \Delta{\bm V}^l, 
\end{equation}
where $l = 0, 1, \dots, l_{\max}$, and the matrix $\CP_{\alpha}({\bm
  V}^{l})$ is a block $\alpha$-circulant matrix serving as an
approximation to the Jacobi matrix $\nabla \CK({\bm V}^l) = C_\alpha
\otimes I_x - \Delta t (\tilde{C}_{\theta, \alpha} \otimes I_x) \nabla
F({\bm V}^l)$. It is defined by
\begin{equation}\label{Sec4_4_SNI_b}
\CP_{\alpha}({\bm V}^{l}) = C_\alpha \otimes I_x - \Delta t \tilde{C}_{\alpha,\theta} \otimes \overline{\nabla \mathbf{f}}({\bm V}^l), 
\end{equation}
where $\tilde{C}_{\alpha, \theta}$ is an $\alpha$-circulant matrix given by
$$
\tilde{C}_{\theta, \alpha} := \begin{bmatrix}
\theta & & & (1-\theta)\alpha \\
1-\theta & \theta & & \\
& \ddots & \ddots & \\
& & 1-\theta & \theta
\end{bmatrix},
$$
and $\overline{\nabla {f}}({\bm V}^l)$ represents the average of the
$J$ Jacobi blocks of $\nabla F({\bm V}^l)$,
$$
\overline{\nabla {f}}({\bm V}^l) := \frac{1}{J} \left[{\sum}_{j=1}^{J-1} \nabla {f}({\bm v}^l_{j}) + \nabla {f}(\alpha{\bm v}^l_{J} + (1-\alpha){\bm u}_n) \right]. 
$$
\end{subequations} 
Using this notation, we can write the parareal algorithm from
\cite{gander2020diagonalization} as
\begin{equation}\label{Diag_parareal_2020}
{\bm u}_{n+1}^{k+1} = \CF^*_\alpha(T_n, T_{n+1}, {\bm u}_n^{k+1}) +
\CF(T_n, T_{n+1}, {\bm u}_n^{k}) -\CF^*_\alpha(T_n, T_{n+1}, {\bm u}_n^{k}), 
\end{equation}
where $n = 0, 1, \dots, N_t - 1$ denotes the time step index. 
 
For linear problems, i.e., $f({\bm u})=A{\bm u}$, the all-at-once
system \eqref{aaa_F_solver} becomes
\begin{equation}\label{Linear_diag_G}
(C_\alpha\otimes I_x- \tilde{C}_{\theta,\alpha}\otimes  \Delta tA){\bm V}={\bm b}({\bm u}_n), 
\end{equation}
with 
$$
 {\bm b}({\bm u}_n)=
([(I_x+\Delta t(1-\theta)A)(1-\alpha){\bm u}_n]^\top, 0,\dots, 0)^\top. 
$$
The coarse solver $\CF^*_\alpha(T_n, T_{n+1}, {\bm u}_n^k)$ is defined by 
$$
\CF^*_\alpha(T_n, T_{n+1}, {\bm u}_n^k):=(H_J\otimes I_x){\bm V}, 
$$
where $H_J:=(0, \dots, 0, 1)\in\mathbb{R}^{1\times J}$. With ${\bm
  V}=({\bm v}^\top_{1}, {\bm v}^\top_{2},\dots, {\bm
  v}^\top_{J})^\top$, it holds that $\CF^*_\alpha(T_n, T_{n+1}, {\bm
  u}_n^k)={\bm v}_{J}$ and the computation of ${\bm V}$ in
\eqref{Linear_diag_G} is equivalent to solving the
all-at-once system
\begin{equation}\label{G_alp_solver}
\begin{split}
&{\bm v}_{j+1}-{\bm v}_{j}=\Delta tA[\theta {\bm v}_{j+1}+(1-\theta){\bm v}_{j}],~j=0,1,\dots, J-1, \\
&{\bm v}_{0}=\alpha {\bm v}_{J}+(1-\alpha){\bm u}^k_n. 
\end{split}
\end{equation}
It is clear that the coarse solver reduces to the fine solver if
$\alpha=0$, and in this limit, Parareal 
\eqref{Diag_parareal_2020} converges in only one iteration, but
without any speedup because we have to solve \eqref{G_alp_solver} for
$\CF^*_\alpha(T_n, T_{n+1}, {\bm u}_n^k)={\bm v}_J$ sequentially. For
$\alpha\in(0, 1)$, we solve \eqref{G_alp_solver} in one shot by
diagonalization, which is parallel for the $J$ fine time points and
thus the computation time is approximately $J$ times less than the
fine solver $\CF(T_n, T_{n+1}, {\bm u}_n^k)$. This Parareal variant
has different convergence rates for parabolic and hyperbolic problems:
\begin{theorem}\cite{gander2020diagonalization}\label{pro_diag_G}
\emph{ For linear initial value problems ${\bm u}'=A{\bm u}+g$ with
  ${\bm u}(0)={\bm u}_0$ and $A\in\mathbb{C}^{N_x\times N_x}$, let
  $\{{\bm u}_n^k\}$ be the $k$-th iterate of the parareal variant
  \eqref{Diag_parareal_2020} and $\{{\bm u}_n\}$ be the converged
  solution. Then, for any stable one-step Runge-Kutta method used for
  $\CF$ and $\CF_\alpha^*$, the global error ${\bm
    e}^k=\max_{n=1,2,\dots, N_t}\|{\bm u}_n-{\bm u}_n^k\|_\infty$
  satisfies the estimate
$$
{\bm e}^k\leq \rho^k{\bm e}^0,
$$
where the convergence factor $\rho$ is given by
\begin{equation}\label{diag_G_rho}
\begin{split}
\rho=
\begin{cases}
\alpha, &\text{if}~\sigma(A)\subset\mathbb{R}^-,\\
\frac{2\alpha N_t}{1+\alpha}, &\text{if}~\sigma(A)\subset {\rm i}\mathbb{R}.
\end{cases}
\end{split}
\end{equation}
}
\end{theorem}
Here, $\sigma(A)$ denotes the spectrum of $A$. When the matrix $A$
arises from semi-discretizing the heat equation, i.e.,
$A\approx\Delta$, it holds that $\sigma(A)\subset\mathbb{R}^-$. In
this case, the convergence factor $\rho=\alpha$ implies that the
parareal variant \eqref{Diag_parareal_2020} converges with a rate
independent of $N_t$. For wave propagation problems, e.g., the
second-order wave equation \eqref{WaveEquation1d} and the
Schr$\ddot{\rm o}$dinger equation, all the eigenvalues of the discrete
matrix $A$ are imaginary, i.e., $\sigma(A)\subset{\rm
  i}\mathbb{R}$. For this kind of problems, the convergence factor
increases linearly in $N_t$. However, this does not necessarily
imply that the convergence rate deteriorates, especially
when $\alpha$ is relatively small and $N_t$ is not too large.

We now illustrate the convergence of the Parareal variant
\eqref{Diag_parareal_2020} for the heat equation \eqref{heatequation}
with homogeneous Dirichlet boundary conditions and initial condition
$u(x,0)=\sin^2(2\pi x)$ with $x\in(0, 1)$. For both $\CF$ and
$\CF^*_\alpha$, we use the Trapezoidal Rule and the discretization
parameters $ \Delta T=\frac{1}{12}$, $J=10$, $\Delta x=\frac{1}{100}$.
For two values of $N_t$ and three parameters $\alpha$, we show
in Figure \ref{Fig_Parareal_Diag_G_Heat}
\begin{figure}
  \centering
 \includegraphics[width=2.3in,height=1.85in,angle=0]{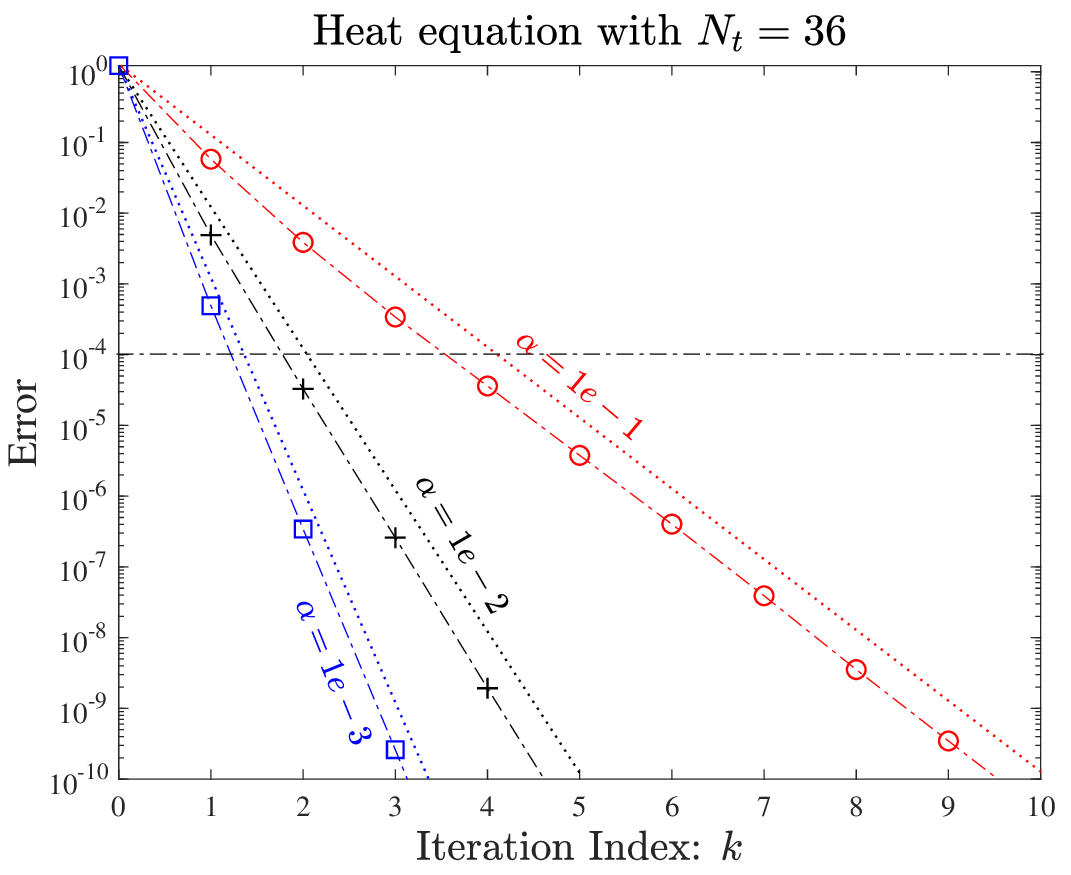}~ \includegraphics[width=2.3in,height=1.85in,angle=0]{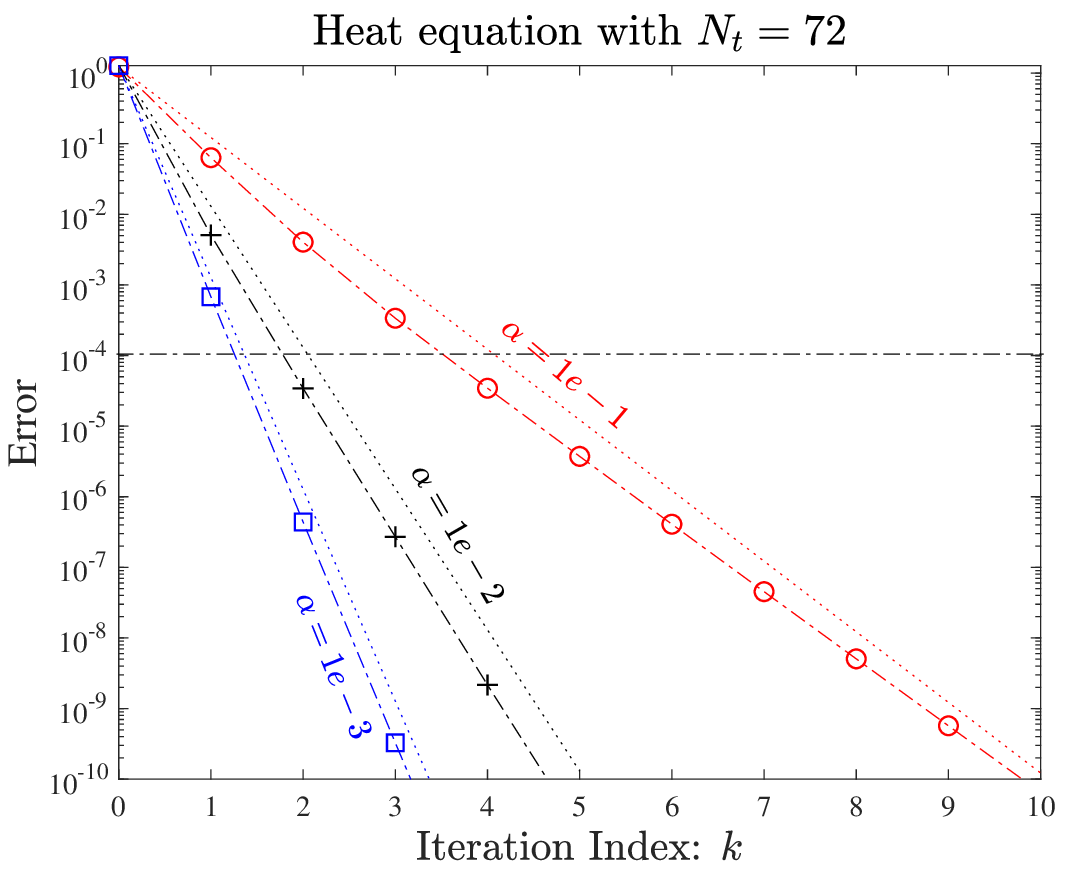}  
  \caption{Error of the Parareal variant \eqref{Diag_parareal_2020} for
    the heat equation. The dotted lines denote the error predicted by
    the theoretical convergence factor $\rho=\alpha$.}
  \label{Fig_Parareal_Diag_G_Heat}
\end{figure}
the measured error, where the error predicted by the convergence
factor $\rho=\alpha$ is plotted as dotted lines. We see that the
theoretical convergence factor is sharp, and the convergence rate is
indeed independent of $N_t$.

We next consider the wave equation \eqref{WaveEquation1d} with
periodic boundary conditions and initial conditions
${u}(x,0)=\sin^2(2\pi x)$ and $\partial_tu(x,0)=0$. After space
discretization, the system of ODEs is given by
$$
{\bm w}'={\bm A}{\bm w},~{\bm w}(0)=
\begin{bmatrix}
\sin^2(2\pi {\bm x}_h)\\0
\end{bmatrix},~{\bm A}:=\begin{bmatrix}
&I_x\\
A &
\end{bmatrix}, 
$$
where ${\bm w}=({\bm u}^\top, ({\bm u}')^\top)^\top$ and
$A\approx\Delta$. All the eigenvalues of ${\bm A}$ are purely
imaginary, and thus according to \eqref{diag_G_rho} the convergence
rate of the Parareal variant \eqref{Diag_parareal_2020} deteriorates
as $N_t$ grows. For relatively large $\alpha$, e.g., $\alpha=0.01$,
this is indeed the case, as shown in Figure
\ref{Fig_Parareal_Diag_G_Wave} (left).
\begin{figure}
  \centering
 \includegraphics[width=2.3in,height=1.85in,angle=0]{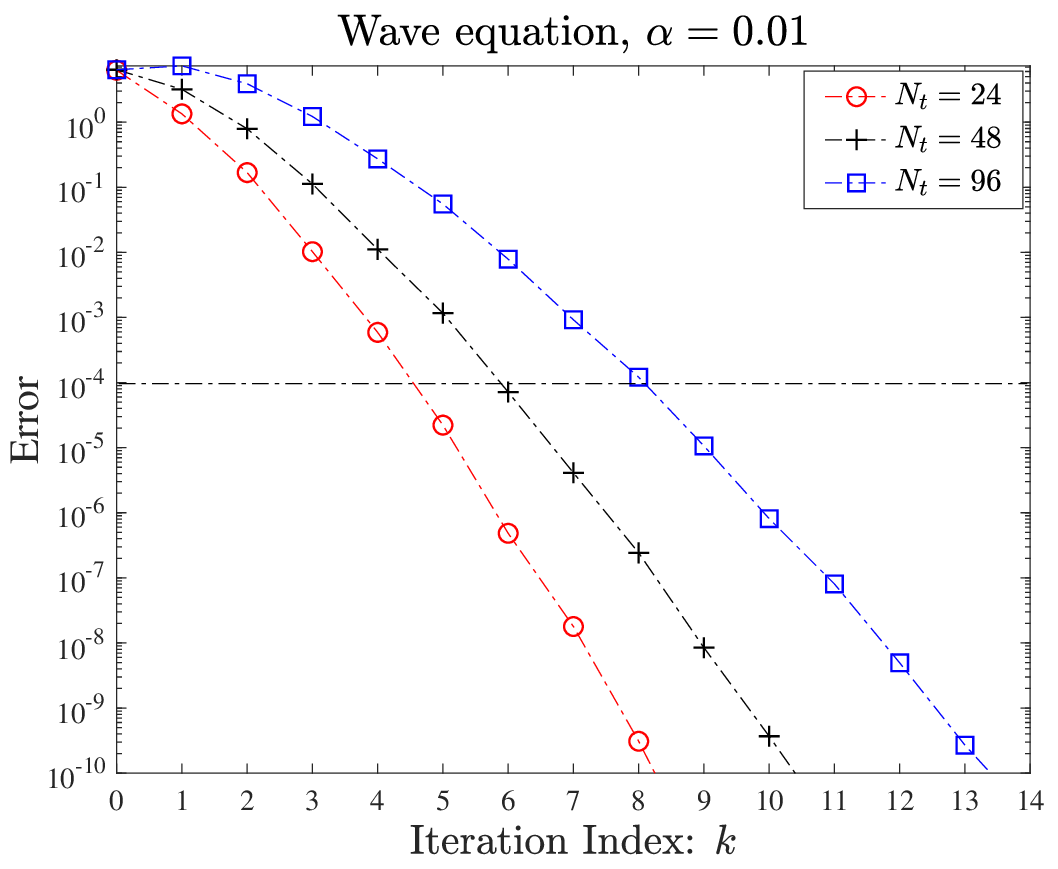}~ \includegraphics[width=2.3in,height=1.85in,angle=0]{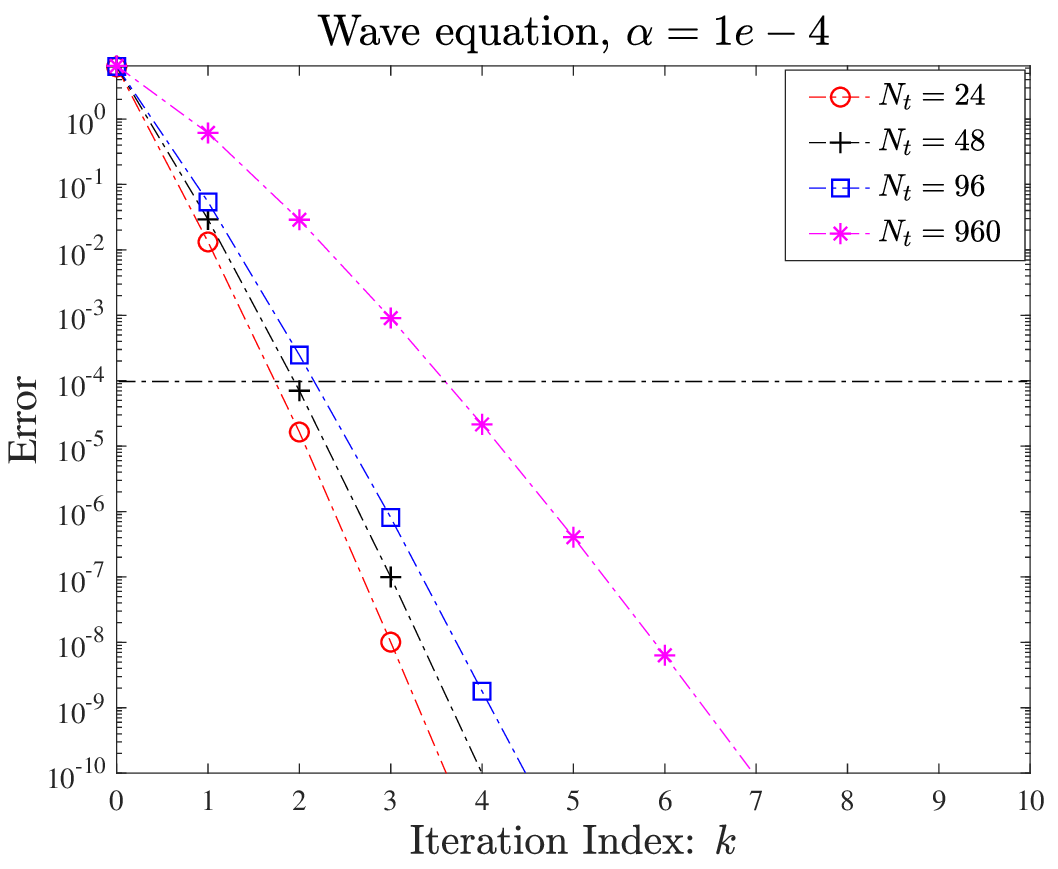}  
  \caption{Measured error of the Parareal variant
    \eqref{Diag_parareal_2020} with two values of the parameter
    $\alpha$ for the wave equation.}
  \label{Fig_Parareal_Diag_G_Wave}
\end{figure}
However, for small $\alpha$, the influence of $N_t$ on the convergence
rate becomes insignificant, as shown in Figure
\ref{Fig_Parareal_Diag_G_Wave} (right). For example, as $N_t$
increases from 24 to 960, we only require an additional 2 iterations
to reach the stopping tolerance, denoted by the horizontal line, i.e.,
the order of the discretization error $\max\{\Delta t^2, \Delta x^2\}$.

In contrast to the heat equation, for wave equations the convergence
factor given in \eqref{diag_G_rho} is not always sharp, depending on
the product $\alpha N_t$. This is illustrated in Figure
\ref{Fig_Parareal_Diag_G_Wave_rho},
\begin{figure}
  \centering
 \includegraphics[width=2.3in,height=1.85in,angle=0]{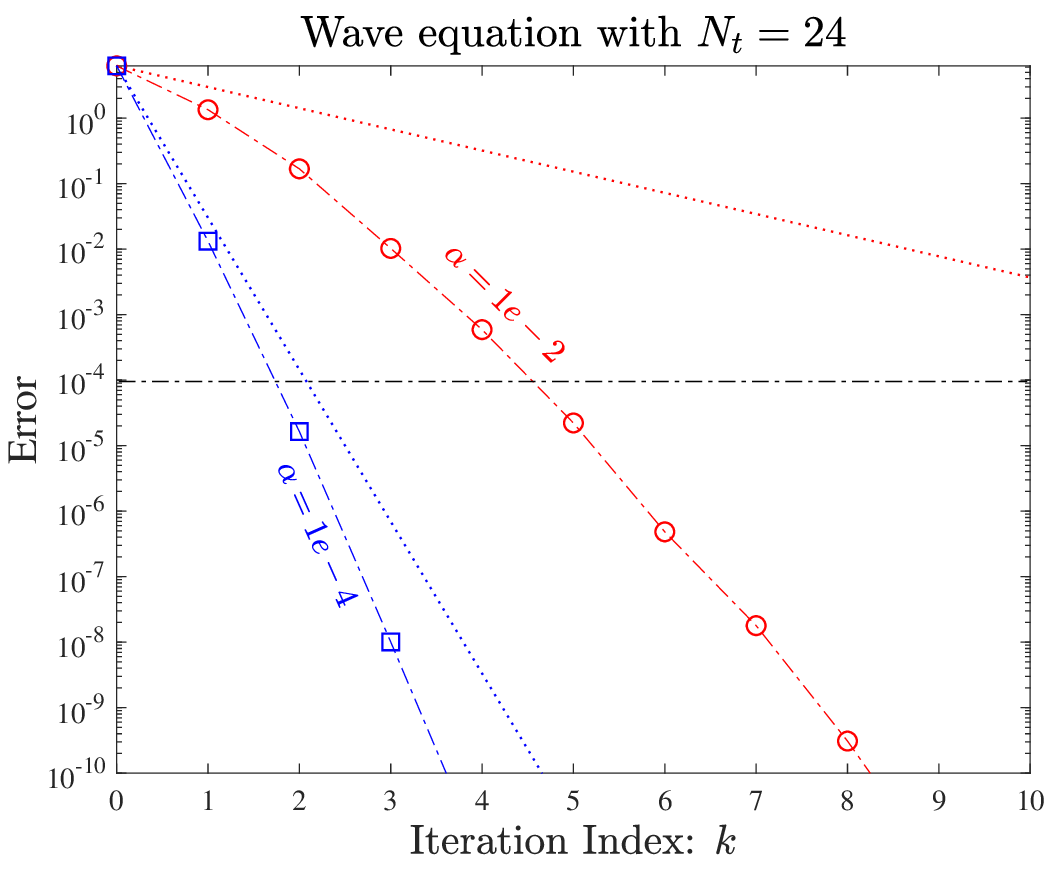}~ \includegraphics[width=2.3in,height=1.85in,angle=0]{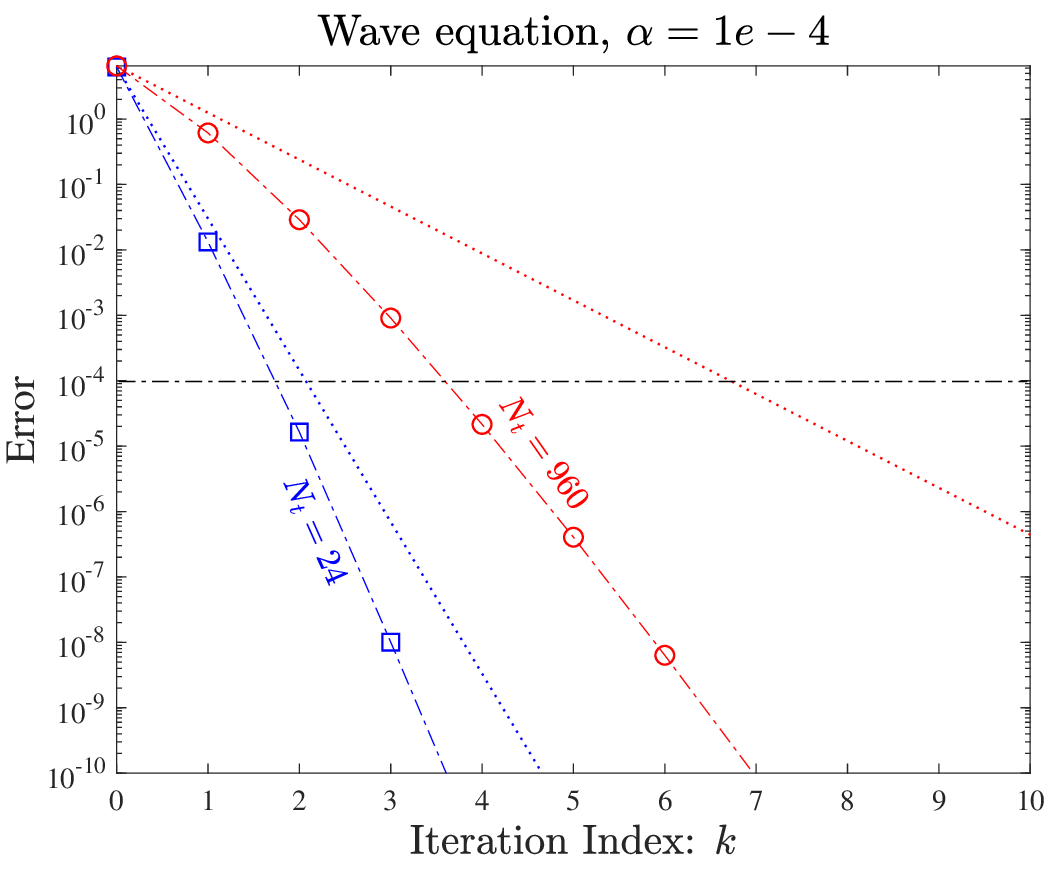}  
  \caption{The convergence factor $\rho$ (dotted line) for the wave
    equation gives quite a sharp bound on the measured error for small
    $\alpha N_t$, while for large $\alpha N_t$ the bound is not
    sharp.}
  \label{Fig_Parareal_Diag_G_Wave_rho}
\end{figure}
where we consider three groups of $(\alpha, N_t)$. For a small
product, i.e., $\alpha=1e-4$ and $N_t=24$, the convergence factor
quite accurately predicts the measured error. For the other two values
of $(\alpha, N_t)$, the linear bound  is not sharp, and we observe
superlinear convergence of the method.

For nonlinear problems, ${\bm u}'=f({\bm u})$, the convergence
analysis of the Parareal variant \eqref{Diag_parareal_2020} can be
found in \cite[Section 4]{gander2020diagonalization}, under the
assumption that the solution of the nonlinear all-at-once system
\eqref{aaa_F_solver} is solved exactly and that the nonlinear function
$f$ satisfies some Lipschitz condition. The main conclusion is that
the method converges with a rate $\rho=\mathcal{O}(\alpha)$ when
$\alpha$ is small, which is similar to the result for the
linear case.  We illustrate this for Burgers' equation
\eqref{Burgers} with periodic boundary conditions and initial condition
$u(x, 0)=\sin^2(2\pi x)$ for $x\in(0, 1)$. Let $\Delta T=0.1$, $J=10$,
and $\Delta x=\frac{1}{100}$. Then, by fixing $N_t=40$, we show in Figure
\ref{Fig_Parareal_Diag_G_Burgers} (left)
\begin{figure}
  \centering
 \includegraphics[width=2.3in,height=1.85in,angle=0]{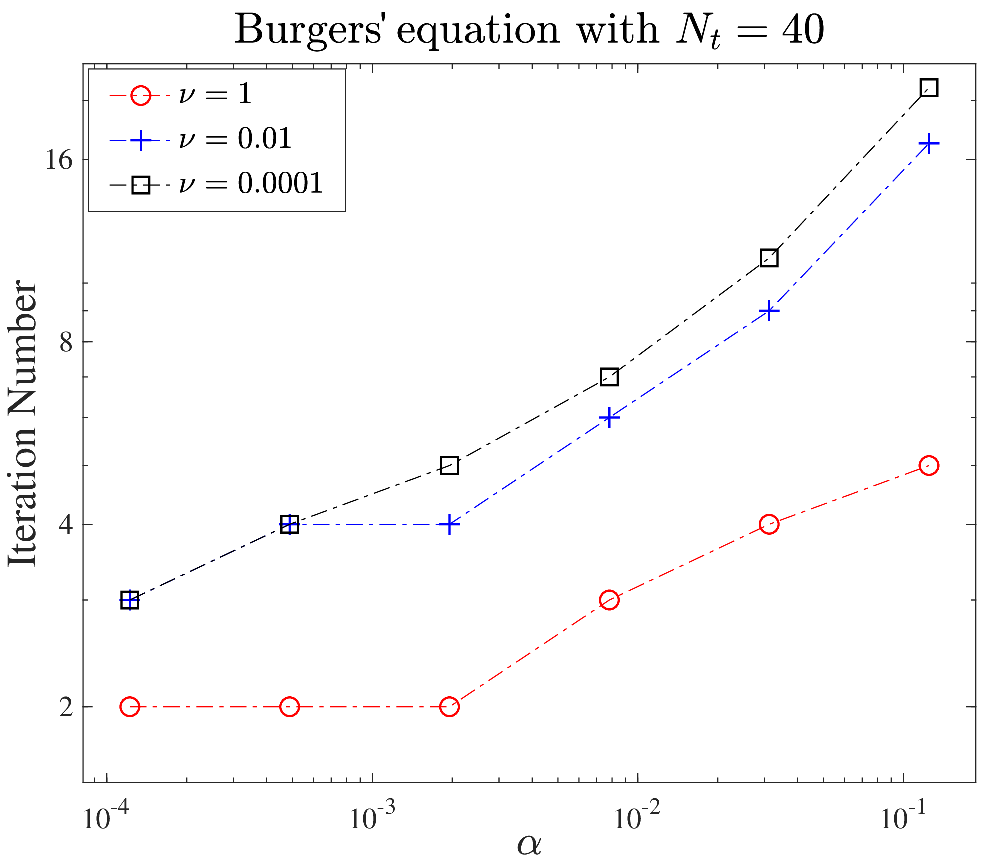}~ \includegraphics[width=2.3in,height=1.85in,angle=0]{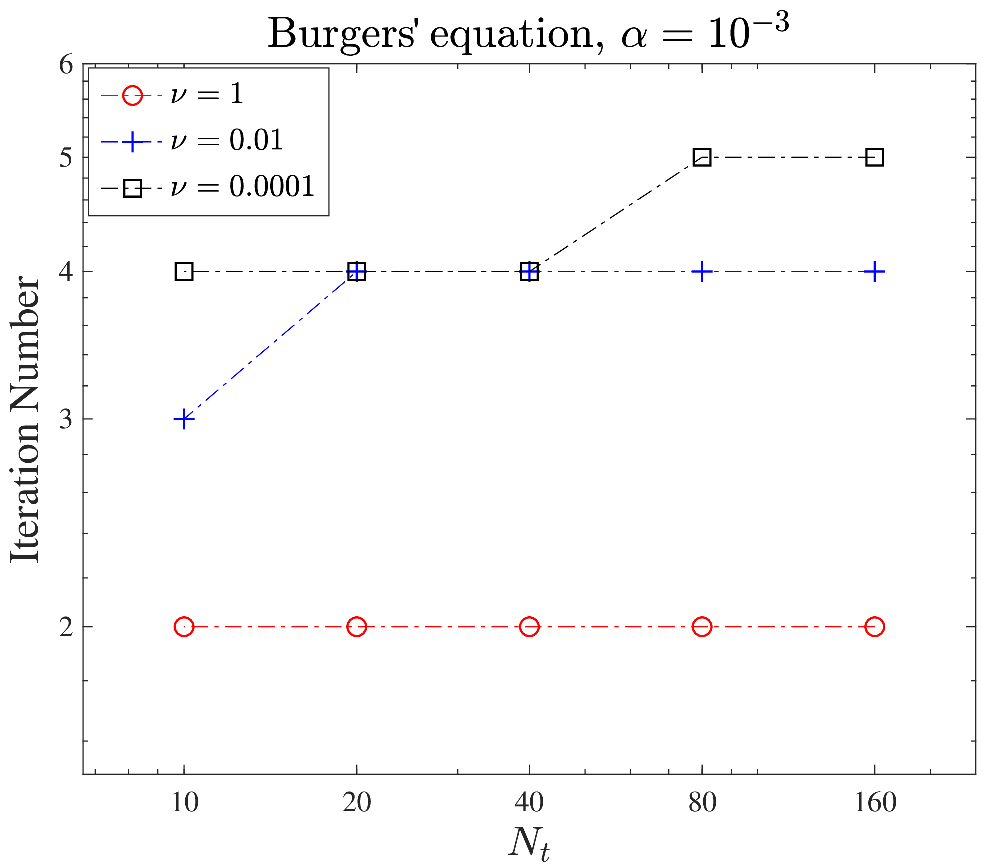}  
  \caption{Iteration numbers of the parareal variant
    \eqref{Diag_parareal_2020} for Burgers' equation when the
    global error reaches $1e-8$, with three values of the diffusion
    parameter $\nu$.}
  \label{Fig_Parareal_Diag_G_Burgers}
\end{figure}
the iteration number for several values of $\alpha$ when the global
error reaches $10^{-8}$. Clearly, a smaller $\alpha$ accelerates
convergence. Concerning the influence of $\nu$, it seems that for
small $\alpha$ it only has a minor influence on the convergence rate,
but for large $\alpha$, the convergence rate deteriorates when $\nu$
decreases.  In the right panel of Figure
\ref{Fig_Parareal_Diag_G_Burgers}, we show the iteration numbers when
$\alpha=10^{-3}$ and $N_t$ varies from 10 to 160, which indicates that the
convergence rate is robust in terms of $N_t$.

The two Parareal variants \eqref{para_CGC} and
\eqref{Diag_parareal_2020} introduced in this section apply ParaDiag
to standard Parareal in different ways. For the former,
diagonalization is used for the $N_t$ coarse time points, changing the
CGC. For the second variant, diagonalization is used for each large
time interval $[T_n, T_{n+1}]$ across the $J$ fine time points,
defining a special coarse solver while keeping the CGC as in the
standard Parareal. These two variants have distinct scopes of
application: the first one, like the standard Parareal, works
primarily for parabolic problems, while the second one is effective
for both parabolic and hyperbolic problems.
 
\subsection{Space-time multigrid (STMG)}\label{Sec4.6}

The final parallel method we wish to introduce is the space-time
multigrid (STMG) method, which is based on using the multigrid (MG)
method both in space and time. After early seminal contributions
\cite{Hackbusch:1984:PMG,horton1995space}, it was recognized that
block Jacobi smoothers in time are a crucial component
\cite{Gander:2016:AOANST}, leading to a method as effective as when MG
is applied to Poisson problems, using only standard
multigrid components in STMG. For the heat equation
\eqref{heatequation} or the advection-diffusion equation \eqref{ADE},
the STMG method can be described as follows:
using a spatial discretization with mesh size $\Delta x$ results in
a system of ODEs ${\bm u}'=A{\bm u}+{\bm f}$, to which we apply a
one-step time-integrator,
\begin{equation}\label{STMG_onestep}
r_1{\bm u}_{n+1}=r_2{\bm u}_n+\tilde{\bm f}_n, \quad n=0,1,\dots, N_t-1, 
\end{equation}
with ${\bm u}_0$ the initial value, and $r_1$ and $r_2$ are
matrix polynomials of $\Delta tA$ (cf. \eqref{r1r2} for Backward
Euler and the Trapezoidal Rule).  The matrix
$A\in\mathbb{R}^{N_{x}\times N_{x}}$ is the discrete matrix of the
Laplacian $\partial_{xx}$ or the advection-diffusion operator
$-\partial_x+\nu\partial_{xx}$ with mesh size $\Delta x$.  As in 
ParaDiag  described in Section \ref{Sec3.6}, we collect the
$N_t$ difference equations in the all-at-once system
\begin{equation}\label{STMG_AAA}
\underbrace{\begin{bmatrix}
r_1 & & &\\
-r_2 &r_1  & &\\
&\ddots  &\ddots  & \\
&  &-r_2 &r_1
\end{bmatrix}}_{:=\CK}\underbrace{\begin{bmatrix}{\bm u}_1\\
{\bm u}_2\\
\vdots\\
{\bm u}_{N_{t}}
\end{bmatrix}}_{:={\bm U}}={\bm b},
\end{equation}
where ${\bm b}$ is a suitable right hand side vector.  

STMG solves for ${\bm U}$ within a multigrid framework, using a damped
block Jacobi iteration as smoother. Starting from an initial
approximation ${\bm U}^{\rm ini}$ of ${\bm U}$, the smoother $\CS$
produces a new approximation ${\bm U}^{\rm new}$ by computing
\begin{equation}\label{STMG_smoother}
{\bm U}^{\rm new}=\CS_\eta({\bm b}, {\bm U}^{\rm ini}, s):
\begin{cases}
{\bm U}^0={\bm U}^{\rm ini},\\
\text{for } j=0, 1, \dots, s-1:\\
\quad (I_{t}\otimes r_1)\Delta{\bm U}^{j}=\eta({\bm b}-\CK{\bm U}^j),\\
\quad {\bm U}^{j+1}={\bm U}^j+\Delta {\bm U}^j,\\
{\bm U}^{\rm new}={\bm U}^{s}, 
\end{cases}
\end{equation}
where $s$ is the number of smoothing iterations and $\eta$ is the
damping parameter. Since $I_t\otimes r_1$ is a block diagonal matrix,
for each smoothing step the $N_t$ subvectors of $\Delta{\bm U}^{j}$
can be solved in parallel, making this a parallel-in-time smoother.
We also need restriction and prolongation operators in space and
time. For illustration, we show these two operators in space with
$N_x=7$,
\begin{equation}\label{STMG_RP}
P_x:=
\begin{bmatrix}
\frac{1}{2} & & \\
1 & &\\
\frac{1}{2} &\frac{1}{2} &\\
& 1 &\\
&\frac{1}{2} &\frac{1}{2}\\
& &1\\
& &\frac{1}{2}
\end{bmatrix}\in\mathbb{R}^{7\times3}, \quad R_x=\frac{1}{2}P_x^\top\in\mathbb{R}^{3\times7}.  
\end{equation}
Similar notations apply to the other two operators $P_t$ and $R_t$ for
the time variable. We can now define the  2-level variant of STMG
from iteration $k$ to $k+1$ as
\begin{equation}\label{STMG}
\begin{cases}
{\bm U}^{k+\frac{1}{3}}=\CS_\eta({\bm b},{\bm U}^k, s_1),\\
{\bm r}={\bm b}-\CK{\bm U}^{k+\frac{1}{3}},~~{\bm r}_{\rm c}=[R_x {\rm Mat}({\bm r})] R_t^\top,\\
{\bm e}_{\rm c}=\CK_{\rm c}^{-1}{\rm Vec}({\bm r}_{\rm c}),~~{\bm e}=[P_x{\rm Mat}({\bm e}_{\rm c})]P_t^\top,\\
{\bm U}^{k+\frac{2}{3}}={\bm U}^{k+\frac{1}{3}}+{\rm Vec}({\bm e}),\\
{\bm U}^{k+1}=\CS_\eta({\bm b}, {\bm U}^{k+\frac{2}{3}}, s_2),
\end{cases}
\end{equation}
where `{Vec}' denotes the vectorization operation from a matrix, and
`{Mat}' the reverse operation, converting a vector to the
corresponding matrix. In practice, we use the \texttt{reshape}
command in Matlab for this. The matrix $\CK_{\rm c}$ is
the all-at-once matrix obtained with larger space and time
discretization parameters $\Delta T=2\Delta t$ and $\Delta X=2\Delta
x$, i.e.,
$$
\CK_{\rm c}=\underbrace{\begin{bmatrix}
r_1^c & & &\\
-r_2^c  &r_1^c  & &\\
&\ddots  &\ddots  & \\
&  &-r_2^c &r_1^c
\end{bmatrix}}_{{N_{t}^c}~\text{blocks}},    
$$
where $r_1^c$ and $r_2^c$ are matrix polynomials of $\Delta TA_c$ with $A_c\in\mathbb{R}^{N_x^c\times N_x^c}$ being the  {\em coarse} discrete matrix of the space derivative(s) with $\Delta X$, e.g.,
$$
\begin{cases}
r_1^c=I_x^c-\Delta TA_c,~r_2^c=I_x^c, &\text{Backward Euler},\\
r_1^c=I_x^c-\frac{1}{2}\Delta TA_c,~r_2^c=I_x^c+\frac{1}{2}\Delta TA_c, &\text{Trapezoidal Rule}.
\end{cases}
$$
In practice, we let $N_x=2^{l_x}-1$ and $N_t=2^{l_t}-1$ with $l_x,
l_t\geq2$ being integers, and thus $N_x^c=2^{l_x-1}-1$ and
$N_t^c=2^{l_t-1}-1$. STMG is obtained naturally by applying the
2-level variant recursively.
 
There is an important difference between STMG
and the parabolic MG method proposed forty years ago
\cite{Hackbusch:1984:PMG}: for parabolic MG, one uses a {\em
  pointwise} Gauss-Seidel iteration as smoother, defined by
\begin{equation}\label{Smoother_GS}
{\bm U}^{\rm new}=\CS_{\rm GS}({\bm b},  {\bm U}^{\rm ini}, s): ~
\begin{cases}
\text{for } n=0, 1, \dots, N_t-1\\
~~~~{\bm u}_{n+1}^0={\bm u}^{\rm ini}_{n+1},\\
~~~~\text{for } j=0, 1, \dots, s-1\\
~~~~~~~~(D+L)\Delta{\bm u}^{j}_{n+1}= \tilde{\bm f}_n+r_2{\bm u}_n^{s}-r_1{\bm u}^j_{n+1},\\
~~~~~~~~{\bm u}^{j+1}_{n+1}={\bm u}^j_{n+1}+\Delta{\bm u}^j_{n+1},\\
~~~~{\bm u}^{\rm new}_{n+1}={\bm u}^{s}_{n+1},
\end{cases}
\end{equation}
where ${\bm u}^{s}_0={\bm u}_0$ and $D$ and $L$ represent the diagonal
and upper triangular parts of $r_1$. Here, ${\bm U}^{\rm ini}:=({\bm
  u}_0^\top, ({\bm u}^{\rm ini}_{1})^\top, \dots, ({\bm u}^{\rm
  ini}_{N_t})^\top)^\top$ and similarly ${\bm U}^{\rm new}$ consists
of the vectors ${\bm u}_0$ and ${\bm u}^{\rm new}_{1}, \dots, {\bm
  u}^{\rm new}_{N_t}$.  This smoother operates sequentially in time:
one must complete the smoothing iteration at time step $n$ to obtain
${\bm u}_n^{s}$, which is necessary for performing the smoothing
iteration at time step $n+1$. After smoothing, one restricts the
residual ${\bm b}-\CK{\bm U}^{\rm new}$ in space-time, like in 
standard multigrid, to a coarser grid. There, one solves a
coarse problem (and repeats this procedure recursively in
practice). Hackbusch in 1984 focused on coarsening in space for this
method and found that for the heat equation, parabolic MG converges
very rapidly. Gander and Lunet \cite{Gander:TPTI:2024} examined the
performance of a 2-level version of the parabolic MG method that
coarsens both in space and time and found that it converges only
slowly then. This slow convergence was already improved upon in
\cite{horton1995space} by using special multigrid components adapted
to the interpretation of the time direction as a strongly advective
term, see also \cite{janssen1996multigrid,van2002multigrid} for
multigrid waveform relaxation variants.
  
Returning to STMG \eqref{STMG}, the smoother plays a crucial role in
achieving good performance. The fundamental concept in designing an
effective smoother is to eliminate as much of the high-frequency error
components as possible within a minimal number of smoothing
iterations. This allows the remaining low-frequency errors to be
well-represented on the coarse grids and then being removed there
through coarse grid correction. A valuable tool for accomplishing this
objective is {\em Local Fourier Analysis} (LFA), which involves
neglecting the initial and boundary conditions of the problem and just
focusing on how the finite difference stencil affects a given Fourier
mode in the error,
\begin{equation}\label{Fourier_Mode}
u_{n,m}^j=C_{\omega, \xi}^je^{{\rm i}\omega n\Delta t} e^{{\rm i}\xi m\Delta x},
\end{equation}
where ${\bm u}_{n}^j:=(u_{n,1}^j, \dots, u_{n,N_x}^j)^\top$ and ${\rm
  i}=\sqrt{-1}$. To apply LFA for the damped Jacobi iteration
\eqref{STMG_smoother}, we consider the 1D heat equation
\eqref{heatequation} discretized using centered finite differences
in space and backward Euler in time. In this case,
$$
A=\frac{1}{\Delta x^2}{\rm tri}(1, -2, 1), \quad r_1=I_x-\Delta tA, \quad r_2=I_x.
$$
For each iteration $j$, the block Jacobi iteration
\eqref{STMG_smoother} consists of the $N_t$ difference equations
\begin{equation}\label{STMG_smoother_Diff}
r_1 ({\bm u}_{n+1}^{j+1}-{\bm u}_{n+1}^{j})=-\eta (r_1{\bm u}_{n+1}^j-r_2{\bm u}_n^j),
\end{equation}
where we set the right-hand side ${\bm b}$ in \eqref{STMG_smoother} to
zero and consider ${\bm U}^j$ as the error at the $j$-th iteration.

To apply LFA to \eqref{STMG_smoother_Diff}, we first consider the
result when the space discrete operator $A$ acts on the Fourier mode
$u_{n+1,m}^{l}$ (with $l=j, j+1$),
\begin{equation}\label{STMG_Au}
\begin{split}
Au_{n+1,m}^{l} &= C_{\omega, \xi}^le^{{\rm i}\omega (n+1)\Delta t} \frac{e^{{\rm i}\xi (m-1)\Delta x}-2e^{{\rm i}\xi m\Delta x}+e^{{\rm i}\xi (m+1)\Delta x}}{\Delta x^2} \\
&= C_{\omega, \xi}^le^{{\rm i}\omega (n+1)\Delta t}e^{{\rm i}\xi m\Delta x} \frac{e^{-{\rm i}\xi\Delta x}-2+e^{{\rm i}\xi\Delta x}}{\Delta x^2} \\
&= \frac{2(\cos(\xi\Delta x)-1)}{\Delta x^2} C_{\omega, \xi}^le^{{\rm i}\omega (n+1)\Delta t}e^{{\rm i}\xi m\Delta x}.
\end{split}
\end{equation}
Hence,
$$
r_1 ({\bm u}_{n+1}^{j+1}-{\bm u}_{n+1}^{j}) = \left(1-\frac{2\Delta t(\cos(\xi\Delta x)-1)}{\Delta x^2}\right)(C_{\omega, \xi}^{j+1}-C_{\omega, \xi}^{j})e^{{\rm i}\omega (n+1)\Delta t}e^{{\rm i}\xi {\bm x}_h},
$$
and
\begin{equation*}
\begin{split}
r_1{\bm u}_{n+1}^j-r_2{\bm u}_n^j &= {\bm u}_{n+1}^j-{\bm u}_n^j-\Delta tA{\bm u}_{n+1}^j \\
&= \left(1-e^{-{\rm i}\omega\Delta t}-\frac{2\Delta t(\cos(\xi\Delta x)-1)}{\Delta x^2}\right)C_{\omega, \xi}^je^{{\rm i}\omega (n+1)\Delta t}e^{{\rm i}\xi {\bm x}_h},
\end{split}
\end{equation*}
where ${\bm x}_h = {\rm Vec}(m\Delta x)$. Now, from \eqref{STMG_smoother_Diff} we have
\begin{equation*}
\begin{split}
&\left(1-\frac{2\Delta t(\cos(\xi\Delta x)-1)}{\Delta x^2}\right)(C_{\omega, \xi}^{j+1}-C_{\omega, \xi}^{j}) \\
&=\eta\left(1-e^{-{\rm i}\omega\Delta t}-\frac{2\Delta t(\cos(\xi\Delta x)-1)}{\Delta x^2}\right)C_{\omega, \xi}^j,
\end{split}
\end{equation*}
i.e., $C_{\omega, \xi}^{j+1} = \rho(\omega,\xi,\eta)C_{\omega, \xi}^j$ with $\rho$ being the convergence factor given by
\begin{equation}\label{STMG_rho_heat}
\rho(\omega,\xi,\eta) = 1-\eta\left(1-\frac{e^{-{\rm i}\omega\Delta t}}{1+\frac{2\Delta t}{\Delta x^2}(1-\cos(\xi\Delta x))}\right),
\end{equation}
where $\omega\Delta t \in (-\pi, \pi)$ and $\xi\Delta x \in (-\pi,
\pi)$. By calculating the maximum of $\rho$ with respect to $\xi$ and
$\omega$ and then minimizing the maximum, the following result was
proved in \cite[Chapter 4]{Gander:TPTI:2024}, see
\cite{Gander:2016:AOANST} for a comprehensive analysis for more
general discretizations:
\begin{theorem}
\emph{For the 1D heat equation discretized with centered finite
  differences and backward Euler, the optimal choice of $\eta$ used in
  the damped Jacobi smoother \eqref{STMG_smoother} always permitting
  time coarsening is $\eta = \frac{1}{2}$. With this choice, all high
  frequencies  in time, $\omega \in \pm \left(\frac{\pi}{2\Delta
      t},\frac{\pi}{\Delta t}\right)$ are damped by at least a factor
    of $\frac{1}{\sqrt{2}}$. If in addition the mesh parameters
    satisfy $\frac{\Delta t}{\Delta x^2} \geq \frac{1}{\sqrt{2}}$,
    then also the high frequencies in space, $\xi \in \pm
    \left(\frac{\pi}{2\Delta x}, \frac{\pi}{\Delta x}\right)$ are
    damped by at least a factor of $\frac{1}{\sqrt{2}}$, and one can
    do space coarsening as well.}
\end{theorem}
A refined analysis concerning the optimality can be found in
\cite{Chaudet2024}.

For the advection-diffusion equation \eqref{ADE}, we can also apply
LFA to  2-level STMG. Here, the discrete matrix obtained
from centered finite differences is $A=\frac{\nu}{\Delta
  x^2}{\rm Tri}(1, -2, 1)+\frac{1}{2\Delta x}{\rm Tri}(-1, 0,
1)$. Similar to \eqref{STMG_Au}, the result when the spatial discrete
operator $A$ acts on the Fourier mode $u_{n+1,m}^{l}$ (with $l=j,
j+1$) is
\begin{equation*} 
\begin{split}
Au_{n+1,m}^{l}&= \left[\frac{2\nu(\cos(\xi\Delta x)-1)}{\Delta x^2}+\frac{e^{{\rm i}\xi  \Delta x}-e^{-{\rm i}\xi  \Delta x}}{2\Delta x}\right]
C_{\omega, \xi}^le^{{\rm i}\omega (n+1)\Delta t}e^{{\rm i}\xi m\Delta x}\\
&= \left[\frac{2\nu(\cos(\xi\Delta x)-1)}{\Delta x^2}+\frac{{\rm i}\sin(\xi\Delta x)}{\Delta x}\right]
C_{\omega, \xi}^le^{{\rm i}\omega (n+1)\Delta t}e^{{\rm i}\xi m\Delta x}.
\end{split}
\end{equation*}
From this, we obtain the convergence factor $\rho$ in Fourier space as 
\begin{equation}\label{STMG_rho_ADE}
\rho(\omega,\xi,\eta)=1-\eta\left(1-\frac{e^{-{\rm i}\omega\Delta t}}{1+\frac{2\nu\Delta t}{\Delta x^2}(1-\cos(\xi\Delta x))+{\rm i}\frac{\Delta t}{\Delta x}\sin(\xi\Delta x)}\right).
\end{equation}
Heuristically, we can still use $\eta=\frac{1}{2}$ as damping
  parameter for coarsening in time;  see the illustration in Figure
\ref{Fig_STMG_ADE_rho}.
\begin{figure}
\centering
\includegraphics[width=1.55in,height=1.4in,angle=0]{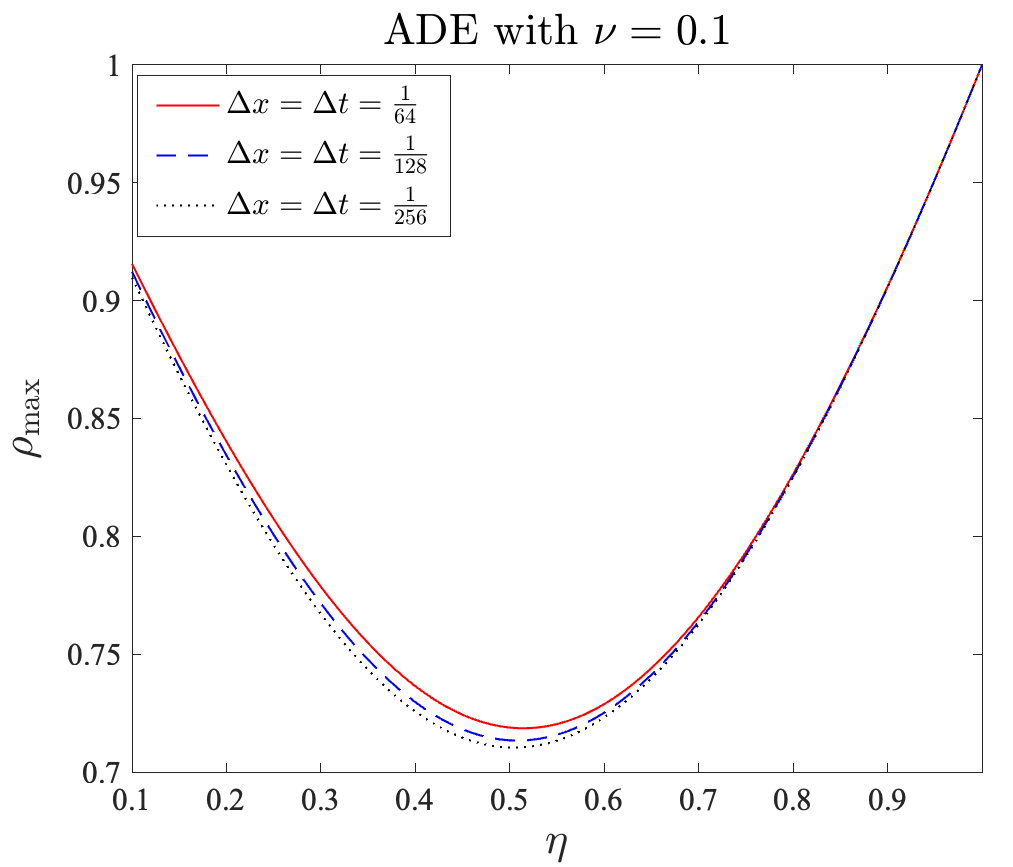} 
\includegraphics[width=1.55in,height=1.4in,angle=0]{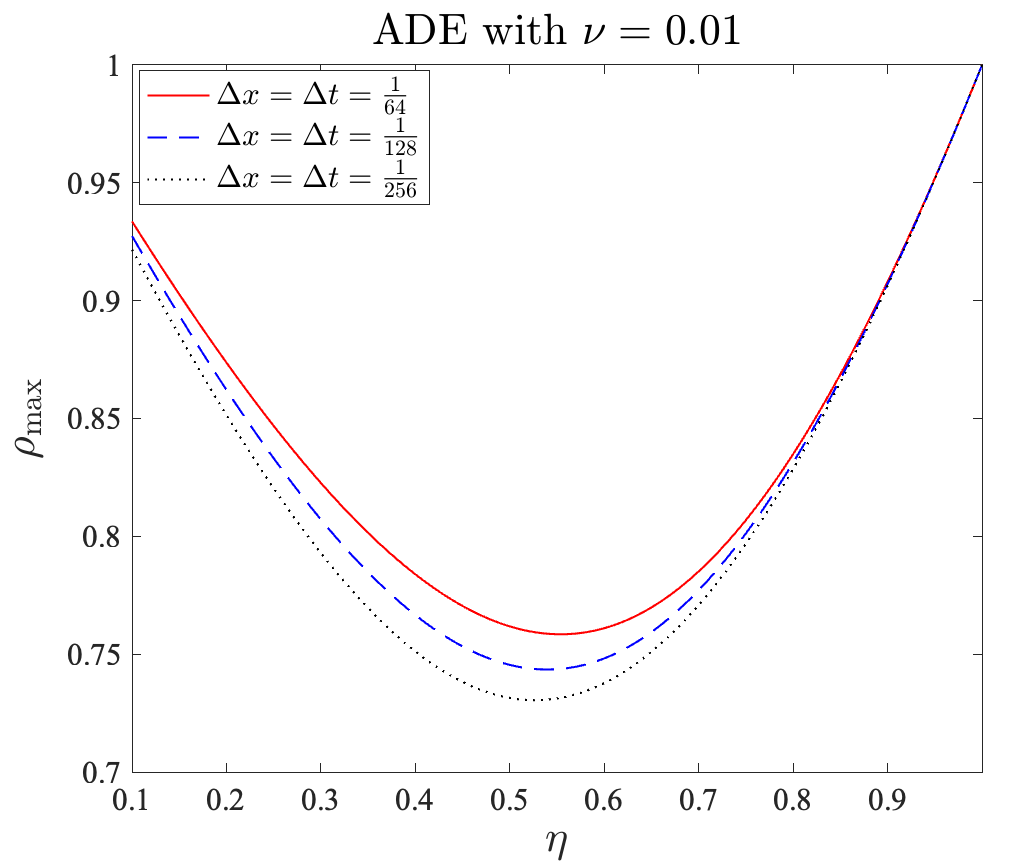}  
\includegraphics[width=1.55in,height=1.4in,angle=0]{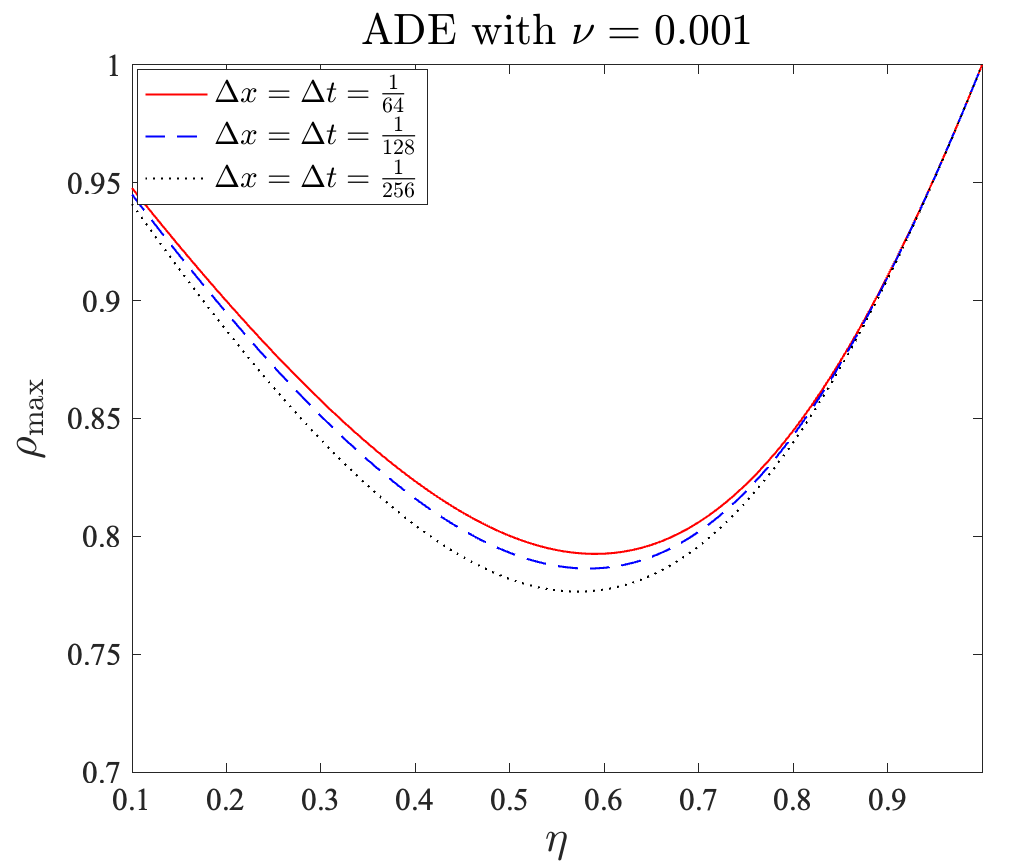} 
\caption{Maximum of the convergence factor over the high frequencies
  for the advection-diffusion equation (ADE) for three values of the
  diffusion parameter $\nu$, i.e., $\rho_{\max}=\max_{(\Delta x\xi,
    \Delta t\omega)\in(-\pi, \pi)\times(\frac{\pi}{2}, \pi)}\rho(\omega,\xi,\eta)$.}
\label{Fig_STMG_ADE_rho}
\end{figure}
The validity of the choice $\eta=\frac{1}{2}$ for the damping
parameter is further illustrated in Figure
\ref{Fig_STMG_heat_ADE_wrt_eta} in a numerical experiment for the
2-level variant of STMG:
\begin{figure}
\centering
\includegraphics[width=2.3in,height=1.85in,angle=0]{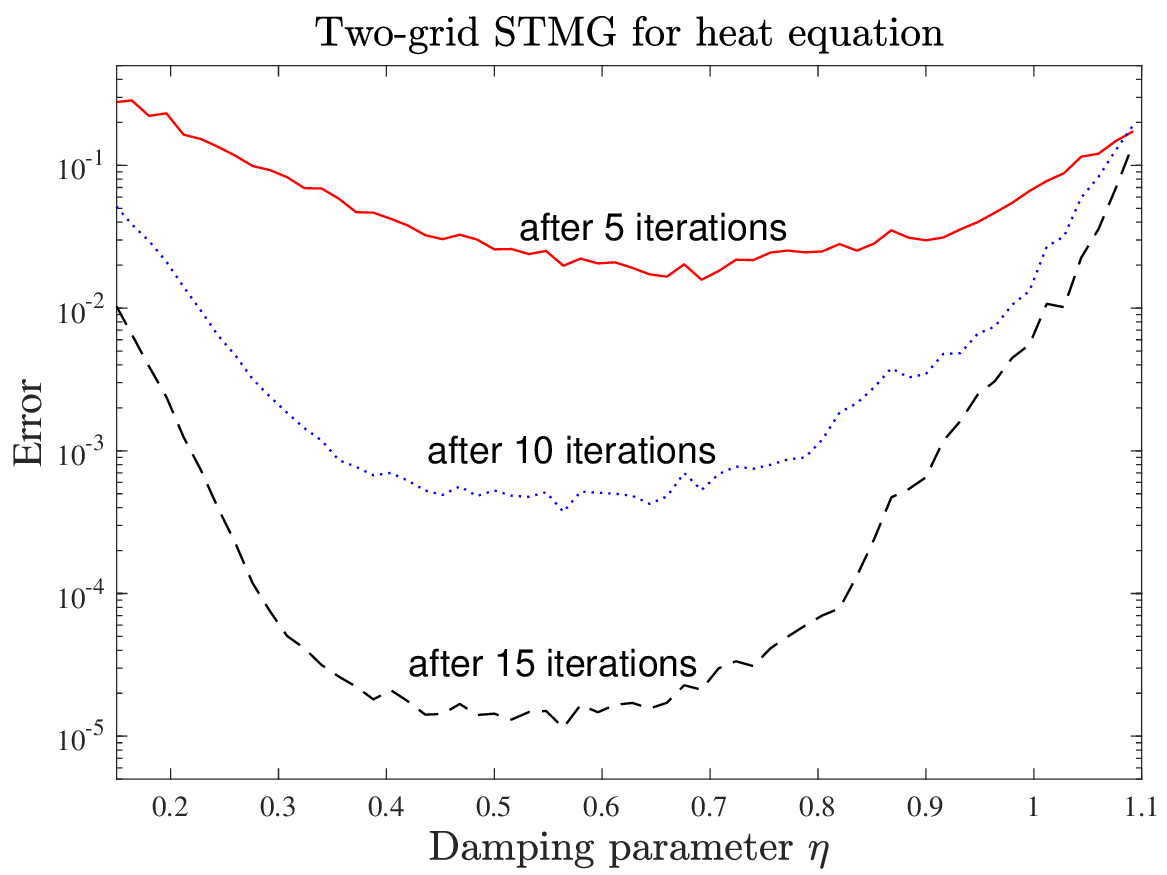} 
\includegraphics[width=2.3in,height=1.85in,angle=0]{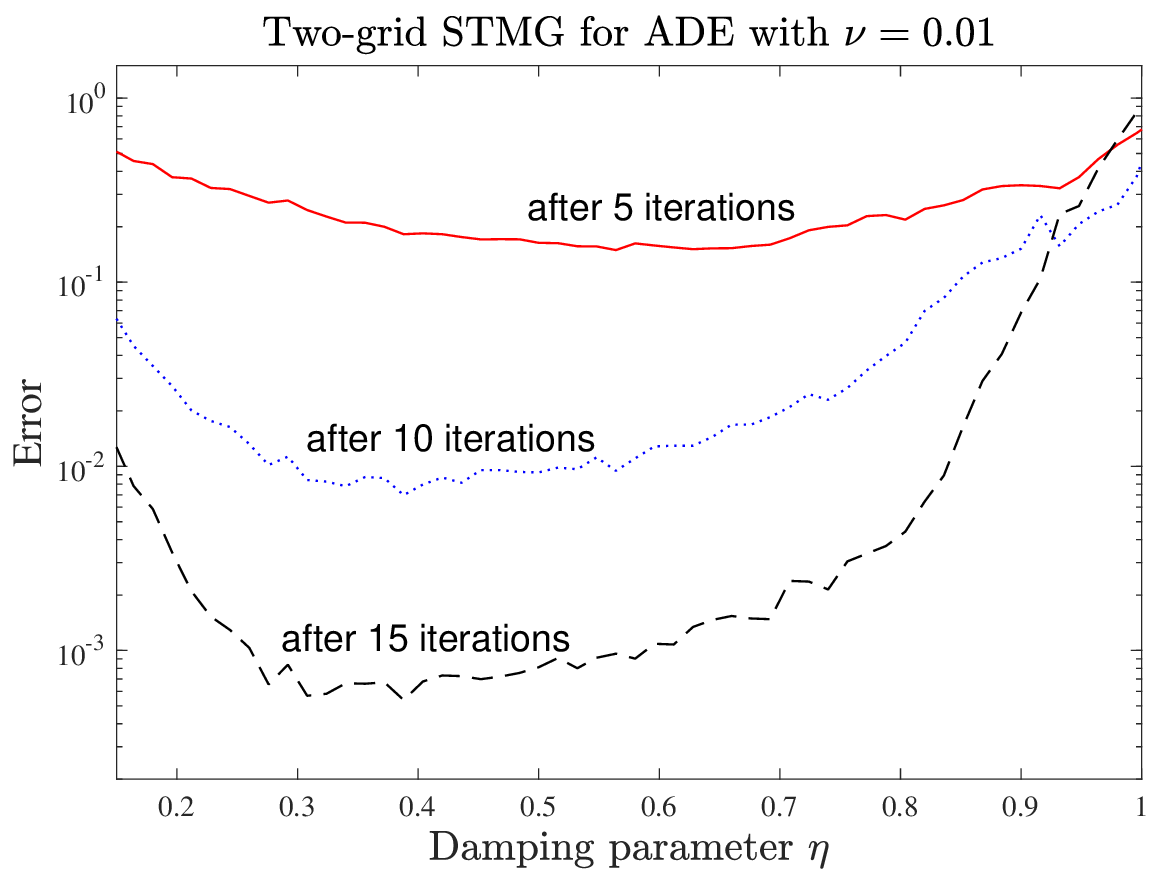}  
\caption{Dependence of the  2-level  STMG error after $k=5, 10, 15$ iterations
  on the choice of the damping parameter $\eta$. Left: heat equation;
  Right: advection-diffusion equation with $\nu=0.01$. Here, we use 1
  block Jacobi smoothing iteration for STMG.}
\label{Fig_STMG_heat_ADE_wrt_eta}
\end{figure}
for both the heat equation and the advection-diffusion equation (ADE),
we show the errors after 5, 10, and 15 iterations
for several values of $\eta$. Clearly, $\eta=\frac{1}{2}$ is a reasonable 
choice to minimize the error in  2-level STMG for both equations.

We show in Figure \ref{Fig_STMG_heat_ADE_error}
\begin{figure}
  \centering 
  \includegraphics[width=2.3in,height=1.85in,angle=0]{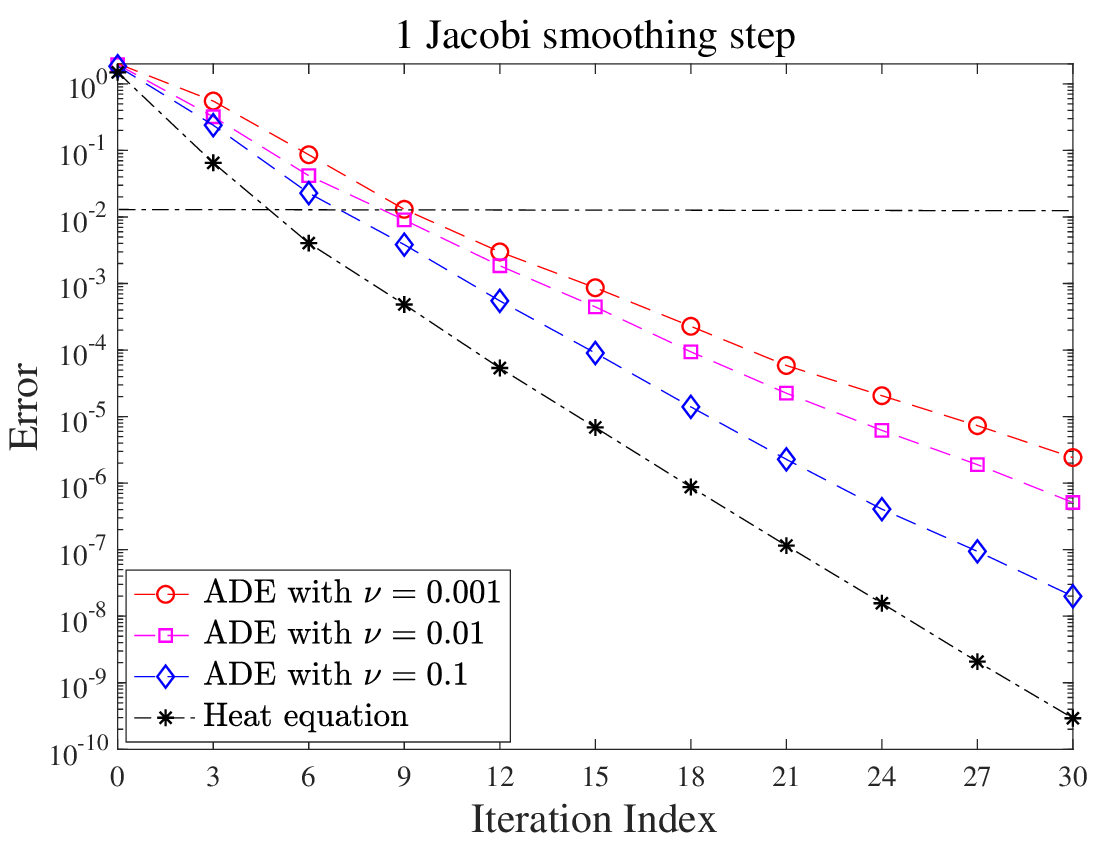}
  \includegraphics[width=2.3in,height=1.85in,angle=0]{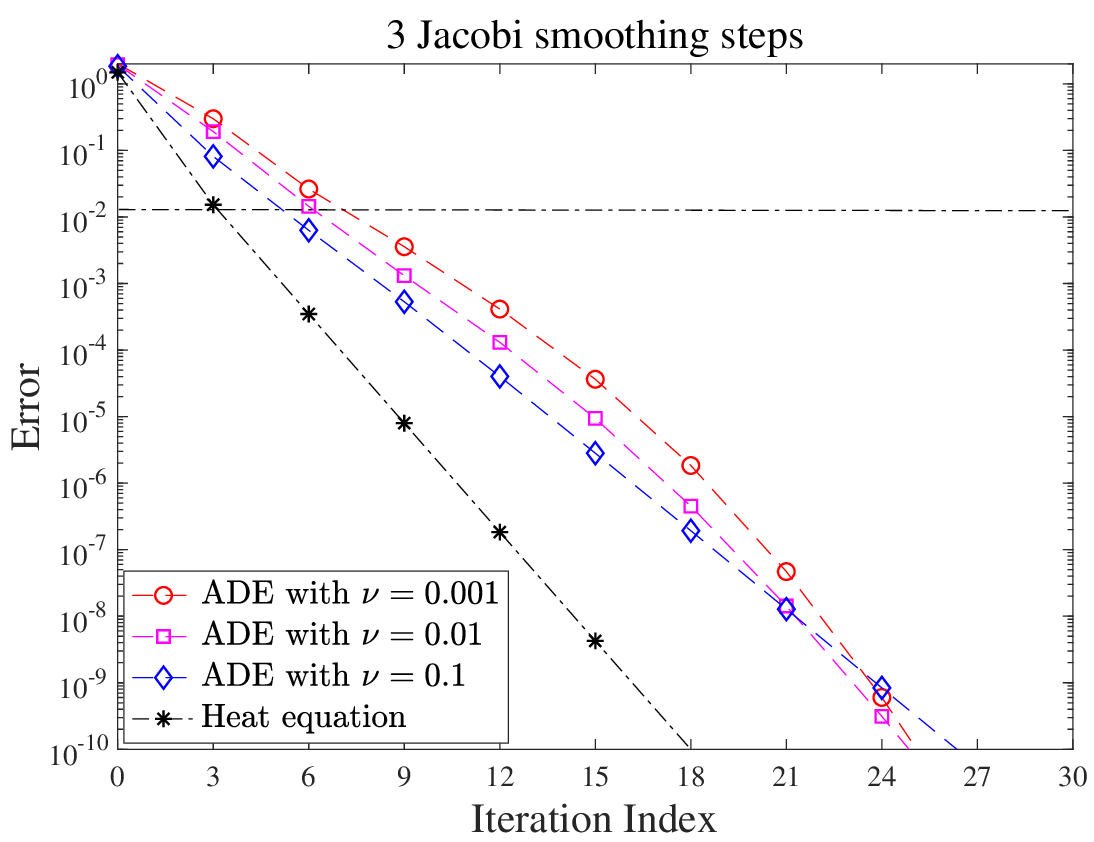} 
  \caption{Measured error of  2-level STMG with 1 and 3  block  Jacobi
    smoothing steps for the advection-diffusion equation (ADE) and the
    heat equation with damping parameter $\eta = \frac{1}{2}$.}
  \label{Fig_STMG_heat_ADE_error}
\end{figure}
the convergence   behavior  of  2-level STMG for both the heat equation
and the advection-diffusion equation with $\eta=\frac{1}{2}$ and three
values of $\nu$. For both equations,  2-level STMG converges
faster when the number of smoothing iterations is increased.  Compared
to the heat equation, the convergence rate is worse for the
advection-diffusion equation, but interestingly it is less sensitive
to $\nu$ when the number of smoothing iterations is large  and a superlinear convergence mechanism sets in, as we see
in the right panel.

The convergence rate of STMG, however, depends on the choice of the
time integrator. The results in Figure \ref{Fig_STMG_Heat_ADE_TR}
\begin{figure}
  \centering
  \includegraphics[width=1.55in,height=1.23in,angle=0]{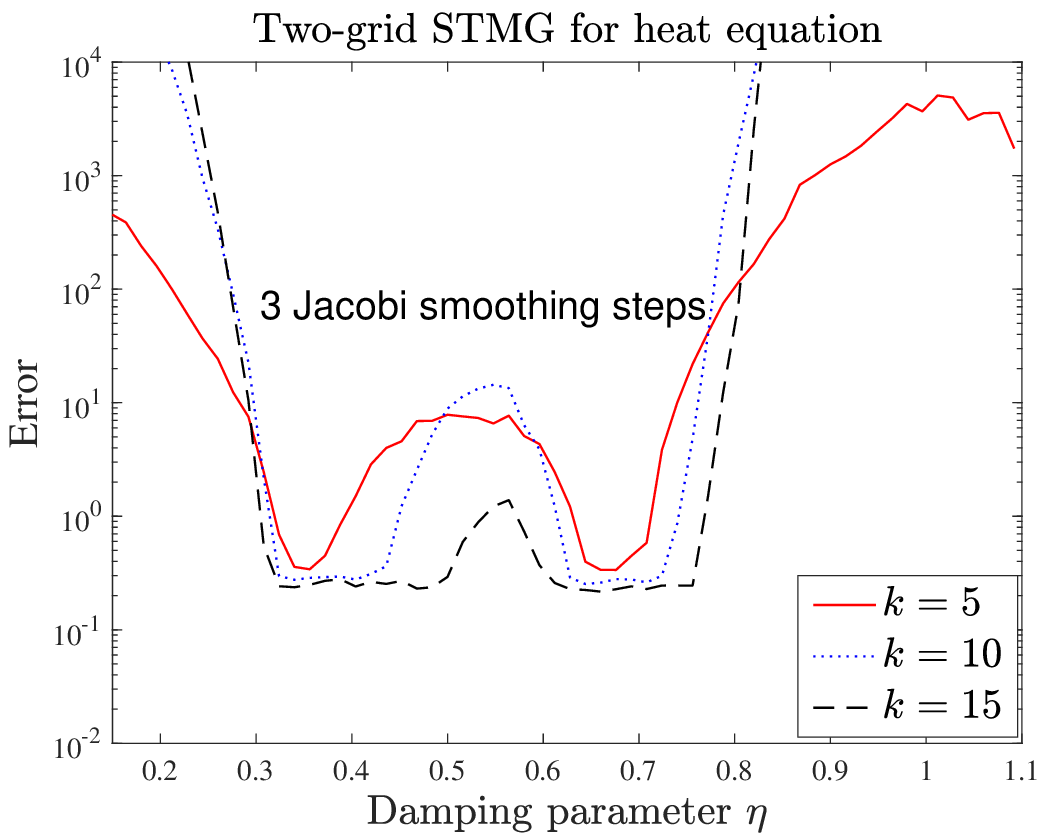} 
  \includegraphics[width=1.55in,height=1.23in,angle=0]{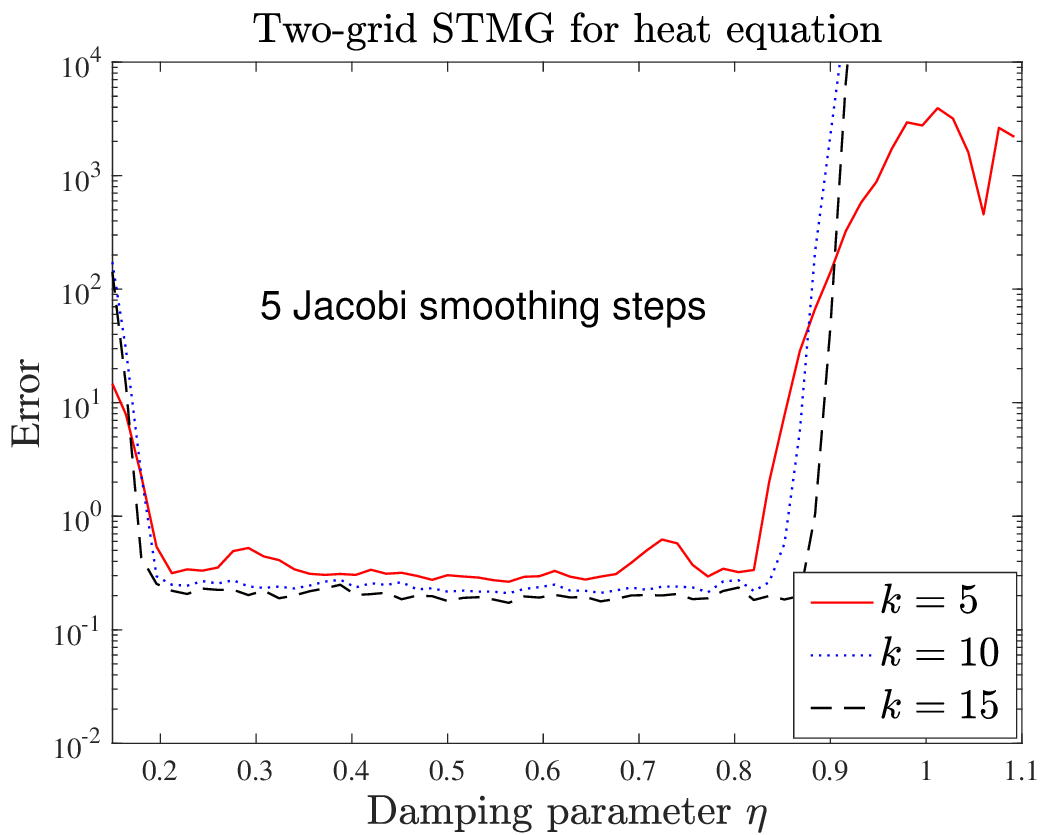}  
  \includegraphics[width=1.55in,height=1.23in,angle=0]{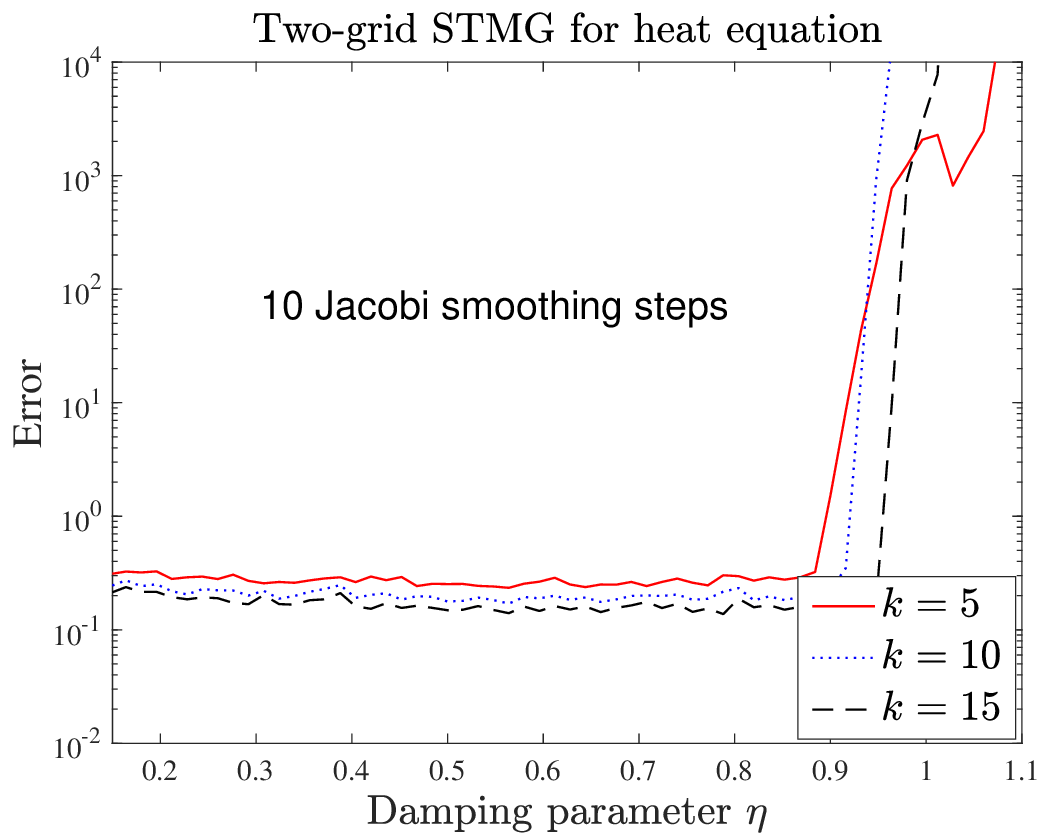} \\
  \includegraphics[width=1.55in,height=1.23in,angle=0]{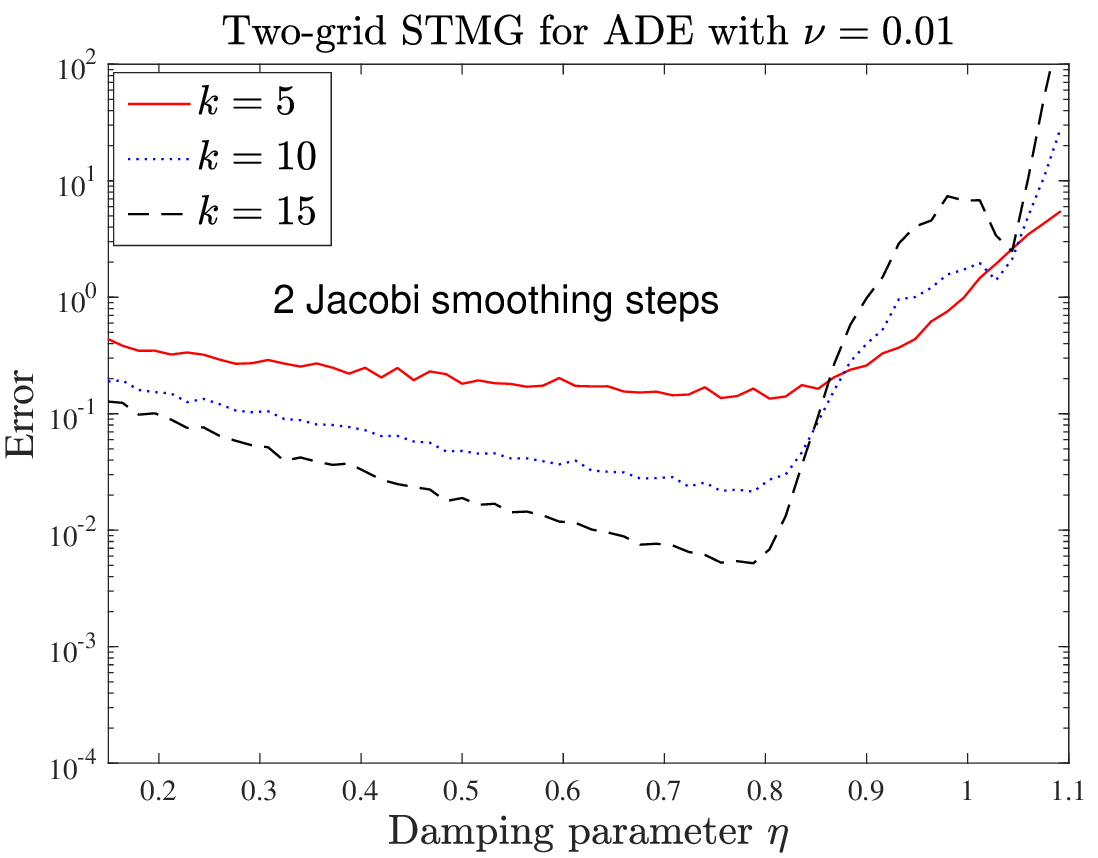} 
  \includegraphics[width=1.55in,height=1.23in,angle=0]{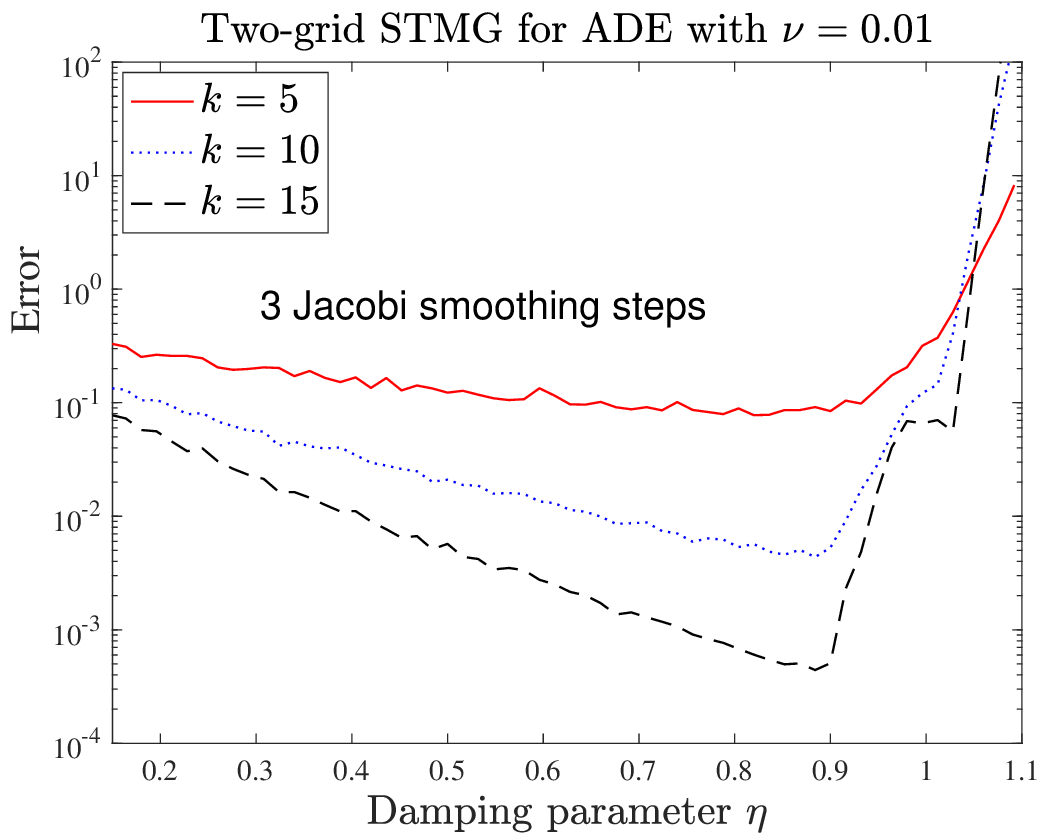}  
  \includegraphics[width=1.55in,height=1.23in,angle=0]{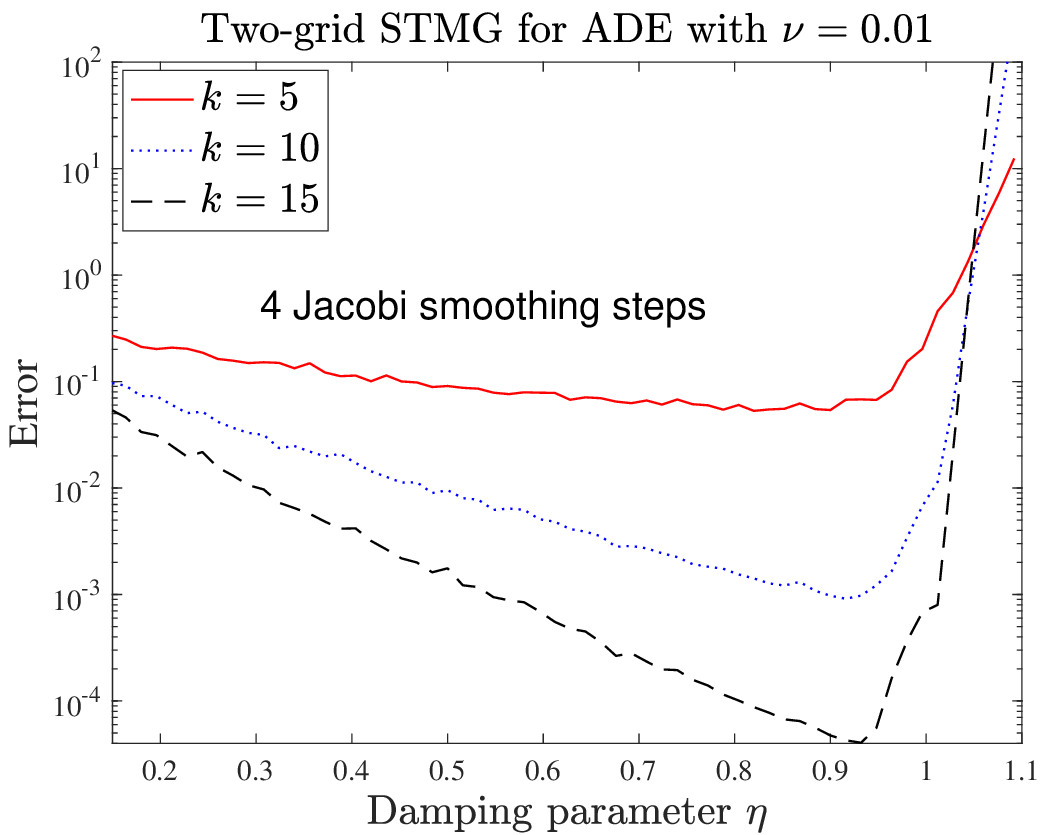} 
  \caption{Error of STMG using the Trapezoidal Rule and different
    numbers of smoothing steps for $\eta \in [0.1, 1.1]$.}
  \label{Fig_STMG_Heat_ADE_TR}
\end{figure}
 for 2-level STMG reveal that using the Trapezoidal Rule as time
 integrator results in a substantially poorer convergence rate
 compared to using Backward Euler. In particular, for the heat
 equation, 2-level STMG appears to have convergence problems 
 regardless of the damping parameter adjustments, even with up to 10
 smoothing steps. Interestingly, 2-level STMG converges for the
 advection-diffusion equation, and doing more smoothing steps enhances
 the convergence rate. Nonetheless, the optimal damping parameter is
 found to be $\eta \approx 0.8$, in contrast to $\eta = \frac{1}{2}$
 we obtained for Backward Euler.

 In real large scale parallel computations on todays super computers,
 excellent weak and strong scaling is achieved with the full STMG
 method, as shown in Table \ref{Tab1}
\begin{table}
  \centering
      \mbox{\scalebox{0.66}{
          \begin{tabular}{|r||r|r|c|r|r|}\hline
        cores & time steps & \multicolumn{1}{c|}{dof} & iter & time & fwd. sub. \\ \hline
        $1$ & $ 2$ & $59\;768$ & $7$ & $28.8$ & $19.0$\\
	$2$ & $ 4$ & $119\;536$ & $7$ & $29.8$ & $37.9$\\
	$4$ & $ 8$ & $239\;072$ & $7$ & $29.8$ & $75.9$\\
	$8$ & $ 16$ & $478\;144$ & $7$ & $29.9$ & $152.2$\\
	$16$ & $ 32$ & $956\;288$ & $7$ & $29.9$ & $305.4$\\
	$32$ & $ 64$ & $1\;912\;576$ & $7$ & $29.9$ & $613.6$\\
	$64$ & $ 128$ & $3\;825\;152$ & $7$ & $29.9$ & $1\;220.7$\\
	$128$ & $ 256$ & $7\;650\;304$ & $7$ & $29.9$ & $2\;448.4$\\
	$256$ & $ 512$ & $15\;300\;608$ & $7$ & $30.0$ & $4\;882.4$\\
	$512$ & $1\;024$ & $30\;601\;216$ & $7$ & $29.9$ & $9\;744.2$\\
	$1\;024$ & $2\;048$ & $61\;202\;432$ & $7$ & $30.0$ & $19\;636.9$\\
	$2\;048$ & $4\;096$ & $122\;404\;864$ & $7$ & $29.9$ & $38\;993.1$\\
	$4\;096$ & $8\;192$ & $244\;809\;728$ & $7$ & $30.0$ & $81\;219.6$\\
	$8\;192$ & $16\;384$ & $489\;619\;456$ & $7$ & $30.0$ & $162\;551.0$\\
	$16\;384$ & $32\;768$ & $979\;238\;912$ & $7$ & $30.0$ & $313\;122.0$\\
	$32\;768$ & $65\;536$ & $1\;958\;477\;824$ & $7$ & $30.0$ & $625\;686.0$\\
	$65\;536$ & $131\;072$ & $3\;916\;955\;648$ & $7$ & $30.0$ & $1\;250\;210.0$\\
	$131\;072$ & $262\;144$ & $7\;833\;911\;296$ & $7$ & $30.0$ & $2\;500\;350.0$\\
	$262\;144$ & $524\;288$ & $15\;667\;822\;592$ & $7$ & $30.0$ & $4\;988\;060.0$\\ \hline
      \end{tabular}
      }
   \scalebox{0.66}{\begin{tabular}{|r|r|c|r|}\hline
        time steps & \multicolumn{1}{c|}{dof} & iter & time \\ \hline
	$512$ & $15\;300\;608$ & $7$ & $7\;635.2$\\
	$512$ & $15\;300\;608$ & $7$ & $3\;821.7$\\
	$512$ & $15\;300\;608$ & $7$ & $1\;909.9$\\
	$512$ & $15\;300\;608$ & $7$ & $954.2$\\
	$512$ & $15\;300\;608$ & $7$ & $477.2$\\
	$512$ & $15\;300\;608$ & $7$ & $238.9$\\
	$512$ & $15\;300\;608$ & $7$ & $119.5$\\
	$512$ & $15\;300\;608$ & $7$ & $59.7$\\
	$512$ & $15\;300\;608$ & $7$ & $30.0$\\
	$524\;288$ & $15\;667\;822\;592$ & $7$ & $15\;205.9$\\
	 $524\;288$ & $15\;667\;822\;592$ & $7$ & $7\;651.5$\\
	 $524\;288$ & $15\;667\;822\;592$ & $7$ & $3\;825.3$\\
	 $524\;288$ & $15\;667\;822\;592$ & $7$ & $1\;913.4$\\
	 $524\;288$ & $15\;667\;822\;592$ & $7$ & $956.6$\\
	 $524\;288$ & $15\;667\;822\;592$ & $7$ & $478.1$\\
	 $524\;288$ & $15\;667\;822\;592$ & $7$ & $239.3$\\
	 $524\;288$ & $15\;667\;822\;592$ & $7$ & $119.6$\\
	 $524\;288$ & $15\;667\;822\;592$ & $7$ & $59.8$\\
	 $524\;288$ & $15\;667\;822\;592$ & $7$ & $30.0$\\ \hline
      \end{tabular}
 }}
      \caption{Weak (left) and strong (right) scaling results of a
        modern Space-Time MultiGrid (STMG) method applied to a 3D heat
        equation. Solution times of classical time stepping with best
        possible parallelization in space only are also shown in the
        column 'fwd. sub.'.}
      \label{Tab1}
\end{table}
for a 3D heat equation model problem, taken from
\cite{Gander:2016:AOANST}.

We next extend STMG to nonlinear problems of the form
\begin{equation}\label{above}
{\bm u}'=f({\bm u}),~{\bm u}(0)={\bm u}_0, ~t\in(0, T), 
\end{equation} 
where ${\bm u}\in\mathbb{R}^{N_x}$ and $f:
\mathbb{R}^{N_x}\rightarrow\mathbb{R}^{N_x}$ is defined by the
discretization of a PDE in space. To describe the idea, we apply the
linear-$\theta$ method to the nonlinear system of ODEs \eqref{above},
leading to the all-at-once system
\begin{equation}\label{STMG_nonlinear_AAA}
\underbrace{(B\otimes I_x){\bm U}-\Delta t(\tilde{B}\otimes I_x)f({\bm U})}_{:=\CK({\bm U})}={\bm b},
\end{equation} 
where ${\bm b}=({\bm u}_0^\top+\Delta t(1-\theta)f^\top({\bm u}_0), 0,\dots, 0)^\top$, ${\bm U}=({\bm u}_1^\top,\dots, {\bm u}_{N_t}^\top)^\top$ and 
$$
B:=
\begin{bmatrix}
1 & & &\\
-1 &1 & &\\
&\ddots &\ddots &\\
& &-1 &1
\end{bmatrix},~\tilde{B}:=
\begin{bmatrix}
\theta & & &\\
1-\theta &\theta & &\\
&\ddots &\ddots &\\
& &1-\theta &\theta
\end{bmatrix},~f({\bm U}):=
\begin{bmatrix}
{\bm u}_1\\
{\bm u}_2\\
\vdots\\
{\bm u}_{N_t}
\end{bmatrix}. 
$$
To formulate STMG \eqref{STMG_nonlinear_AAA}, similar to
\eqref{STMG_smoother}, we first define a nonlinear  block Jacobi 
smoother ${\bm U}^{\rm new}={\CS}_{{\rm non}, \eta}({\bm b}, {\bm
  U}^{\rm ini}, s)$ by
\begin{equation}\label{STMG_smoother_nonlinear}
\begin{split}
\begin{cases}
\tilde{\bm U}^0={\bm U}^{\rm ini},\\
\text{for } j=0, 1, \dots, s-1:\\
~~~~\text{solve } \Delta\tilde{\bm U}^{j}-\Delta t\theta f(\Delta\tilde{\bm U}^j)=\eta({\bm b}-\CK(\tilde{\bm U}^j)),\\
~~~~\tilde{\bm U}^{j+1}=\tilde{\bm U}^j+\Delta \tilde{\bm U}^j,\\
{\bm U}^{\rm new}=\tilde{\bm U}^{s}, 
\end{cases}
\end{split}
\end{equation} 
where the correction term $\Delta\tilde{\bm U}^j$ is solved via an
inner solver, e.g., Newton's iteration.  We can however not obtain a
theoretically optimized estimate  for the damping parameter $\eta$ in
\eqref{STMG_smoother_nonlinear}, since LFA can not be used in
the nonlinear case}.

Following \cite{Brd77}, we now define a non-linear STMG method for
\eqref{STMG_nonlinear_AAA} using the full approximation scheme, 
\begin{equation}\label{STMG_nonlinear}
\begin{cases}
{\bm U}^{k+\frac{1}{3}}=\CS_{{\rm non}, \eta}({\bm b},{\bm U}^k, s_1),\\
{\bm r}={\bm b}-\CK({\bm U}^{k+\frac{1}{3}}),\\
{\bm r}_{\rm c}=[R_x {\rm Mat}({\bm r})]R_t^\top,~{\bm U}^{k+\frac{1}{3}}_{\rm c}=[R_x {\rm Mat}({\bm U}^{k+\frac{1}{3}})]R_t^\top,\\
 {\rm Solve}~ \CK_{\rm c}({\bm U}^{k+\frac{2}{3}}_{\rm c})={\bm r}_{\rm c}+\CK_{\rm c}({\bm U}^{k+\frac{1}{3}}_{\rm c}),\\
{\bm e}_{\rm c}= {\bm U}^{k+\frac{2}{3}}_{\rm c}-{\bm U}^{k+\frac{1}{3}}_{\rm c}, ~{\bm e}=[P_x{\rm Mat}({\bm e}_{\rm c})]P_t^\top,\\
{\bm U}^{k+\frac{2}{3}}={\bm U}^{k+\frac{1}{3}}+{\rm Vec}({\bm e}),\\
{\bm U}^{k+1}=\CS_{{\rm non}, \eta}({\bm b}, {\bm U}^{k+\frac{2}{3}}, s_2).
\end{cases}
\end{equation}
In Figure \ref{Fig_STMG_Burgers}, we show the measured error of this
2-level STMG method for Burgers' equation \eqref{Burgers} with
conditions. We use two values of the diffusion parameter $\nu$, and we
see that the convergence is heavily dependent on this parameter:
  with enough diffusion, STMG works vey well also in the non-linear
  setting, whereas when the diffusion is getting smaller, the
  convergence of STMG deteriorates, as in the linear case. Here, we
used $\eta=\frac{1}{4}$ for the damping parameter, which was found to
be the best choice in our numerical experiments.
\begin{figure}
\centering
\includegraphics[width=2.5in,height=2.1in,angle=0]{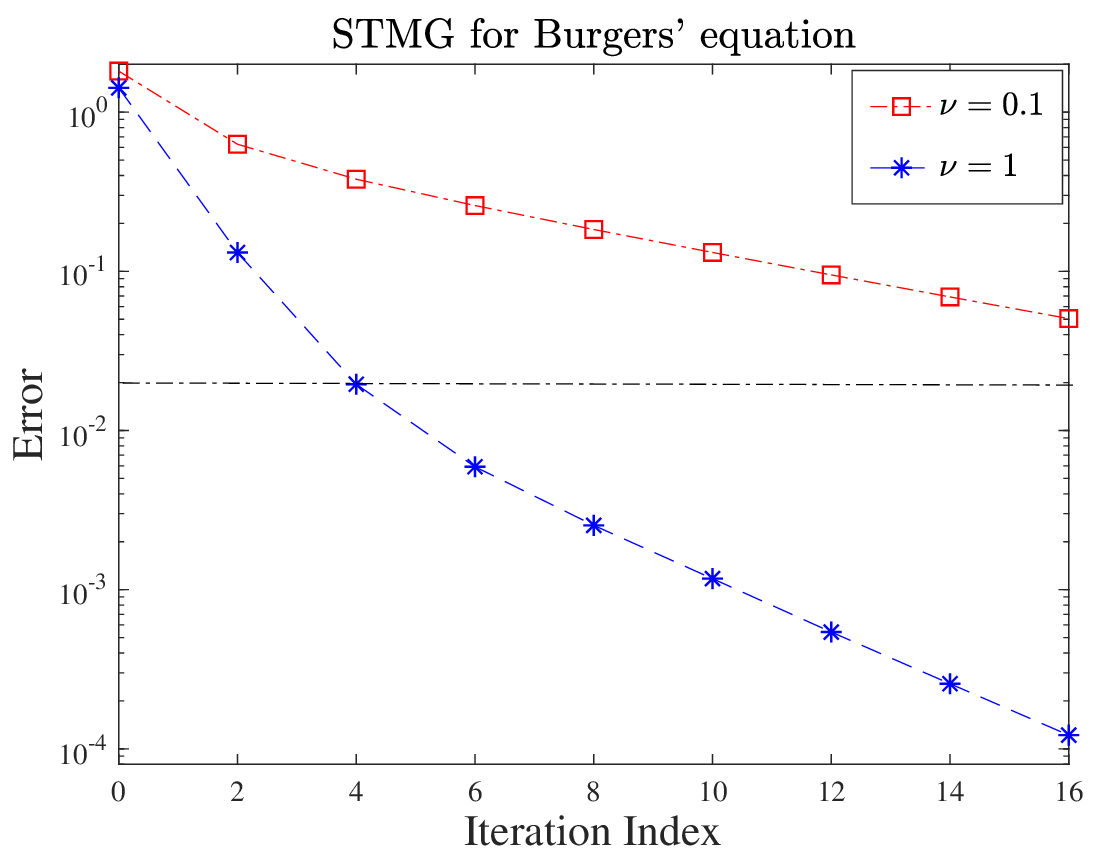}
\caption{Measured error of STMG with  two  block Jacob  smoothing
  steps for Burgers' equation.}
\label{Fig_STMG_Burgers}
\end{figure}

For parabolic problems, STMG presented in this section stands out as
by far the most effective time-parallel solver currently available,
but it is intrusive in nature, unlike the Parareal algorithm. However,
when dealing with hyperbolic problems, as shown in Figures
\ref{Fig_STMG_heat_ADE_error} and \ref{Fig_STMG_Burgers}, STMG
appears to be less efficient, indicating that additional efforts are
required in this domain. Furthermore, as highlighted in the top row of
Figure \ref{Fig_STMG_Heat_ADE_TR}, even for parabolic problems the
convergence rate of STMG depends on the time integrator used, and this
dependency deserves further investigation as well.

\section{Conclusions}
\label{Sec5}

We explained in this paper the important difference for time parallel
time integration, or PinT (Parallel in Time) methods, when applied to
hyperbolic or parabolic problems. For parabolic problems, which tend
to forget a lot of information in time, and thus have solutions that
are local in time, there are many very effective PinT techniques, like
Parareal, Space-Time MultiGrid (STMG), ParaExp and ParaDiag, and
Waveform Relaxation (WR) techniques based on Domain Decomposition
(DD). For hyperbolic problems, which retain very fine solution
features over very long time, only some of these techniques are
effective, like ParaExp, ParaDiag, and Schwarz Waveform Relaxation
(SWR), especially in relation to tent pitching.  For more information,
please take a look at the recent research monograph
\cite{Gander:TPTI:2024}, which contains an up to date treatment of
PinT methods, giving for each the historical content, a simple but
complete and self contained convergence analysis, and also small
Matlab codes that can be directly executed. Codes used to produce the
results in this paper are available from
\url{https://github.com/wushulin/ActaPinT}.
 
\bibliographystyle{actaagsm}

\bibliography{paper.bib}

@ARTICLE{Nievergelt:1964:PMI,
  AUTHOR       = "J{\"o}rg Nievergelt",
  TITLE        = "Parallel methods for integrating ordinary
                  differential equations",
  JOURNAL      = "Comm. ACM",
  YEAR         = 1964,
  VOLUME       = 7,
  PAGES        = "731--733"
}

@ARTICLE{Lions:2001:PTD,
  AUTHOR       = "Jacques-Louis Lions and Yvon Maday and Gabriel
                  Turinici",
  TITLE        = "A parareal in time discretization of {PDE}s",
  JOURNAL      = "C.R. Acad. Sci. Paris, Serie I",
  YEAR         = 2001,
  VOLUME       = 332,
  PAGES        = "661--668"
}

@article{BP12,
author = {Banjai, Lehel and Peterseim, Daniel},
title = {Parallel multistep methods for linear evolution problems},
journal={IMA J. Numer. Anal.},
  volume={32},
  number={3},
  pages={1217--1240},
year = {2012}
}

@incollection{gander50years,
  title={50 years of time parallel time integration},
  author={Gander, Martin J.},
  booktitle={Multiple Shooting and Time Domain Decomposition Methods},
  pages={69--113},
  year={2015},
  publisher={Springer}
}

@book{Gander:TPTI:2024,
  title={Time Parallel Time Integration},
  author={Gander, Martin J. and Lunet, Thibaut},
  publisher={SIAM},
  year={2024}
}

@ARTICLE{Gander:2004:ABC,
  AUTHOR       = "Martin J. Gander and Laurence Halpern",
  TITLE        = "Absorbing Boundary Conditions for the Wave
                  Equation and Parallel Computing",
  JOURNAL      = "Math. Comput.",
  YEAR         = 2004,
  VOLUME       = 74,
  NUMBER       = 249,
  PAGES        = "153--176"
}

@ARTICLE{Gander:2003:OSWW,
  AUTHOR       = "Martin J. Gander and Laurence Halpern and
                  Fr{\'e}d{\'e}ric Nataf",
  TITLE        = "Optimal {S}chwarz Waveform Relaxation for the One
                  Dimensional Wave Equation",
  JOURNAL      = "SIAM J. Numer. Anal.",
  VOLUME       = 41,
  YEAR         = 2003,
  NUMBER       = 5,
  PAGES        = "1643--1681"
}

@ARTICLE{Gander:2007:OSW,
  AUTHOR       = "Martin J. Gander and Laurence Halpern",
  TITLE        = "Optimized {S}chwarz Waveform Relaxation Methods
                  for Advection Reaction Diffusion Problems",
  JOURNAL      = "SIAM J. Numer. Anal.",
  VOLUME       = "45",
  NUMBER       = 2,
  PAGES        = "666--697",
  YEAR         = 2007
}

@article{GSW2017,
author = {Gopalakrishnan, J. and Sch\"{o}berl, J. and Wintersteiger, C.},
title = {Mapped Tent Pitching Schemes for Hyperbolic Systems},
journal = {SIAM J. Sci. Comput.},
volume = {39},
number = {6},
pages = {B1043-B1063},
year = {2017}, 
}

@INPROCEEDINGS{Ciaramella:2023:UTP,
  AUTHOR       = "Ciaramella, Gabriele and Gander, Martin J. and
                  Mazzieri, Ilario",
  TITLE        = "Unmapped tent pitching schemes by waveform relaxation",
  BOOKTITLE    = "27th international Conference of Domain
                  Decomposition Methods",
  YEAR         = 2023,
  PUBLISHER    = "Springer"
}

@inproceedings{gopalakrishnan2020explicit,
  title={An explicit mapped tent pitching scheme for {M}axwell equations},
  author={Gopalakrishnan, Jay and Hochsteger, Matthias and Sch{\"o}berl, Joachim and Wintersteiger, Christoph},
  booktitle={Spectral and high order methods for partial differential equations—ICOSAHOM 2018},
  pages={359--369},
  year={2020}
}

@article{bennequin2009homographic,
  title={A homographic best approximation problem with application to optimized {S}chwarz waveform relaxation},
  author={Bennequin, Daniel and Gander, Martin J. and Halpern, Laurence},
  journal={Math. Comput.},
  volume={78},
  number={265},
  pages={185--223},
  year={2009}
}

@article{gander2008schwarz,
  title={Schwarz methods over the course of time},
  author={Gander, Martin J.},
  journal={Electron. Trans. Numer. Anal},
  volume={31},
  number={5},
  pages={228--255},
  year={2008}
}

@article{gander2013paraexp,
  title={Para{E}xp: A parallel integrator for linear initial-value problems},
  author={Gander, Martin J and G{\"u}ttel, Stefan},
  journal={SIAM J. Sci. Comput.},
  volume={35},
  number={2},
  pages={C123--C142},
  year={2013},
  publisher={SIAM}
}

@article{gander2019direct,
  title={A direct time parallel solver by diagonalization for the wave equation},
  author={Gander, Martin J. and Halpern, Laurence and Rannou, Johann
                  and Ryan, Juliette},
  journal={SIAM J. Sci. Comput.},
  volume={41},
  number={1},
  pages={A220--A245},
  year={2019},
  publisher={SIAM}
}

@article{gander:2021:ParaDiag,
  title={Para{D}iag: Parallel-in-Time Algorithms Based on the
                  Diagonalization Technique},
  author={Gander, Martin J. and Liu, Jun and Wu, Shu-Lin and Yue,
                  Xiaoqiang and Zhou, Tao},
  journal={arXiv preprint arXiv:2005.09158},
  year={2021}
}

@article{gander2020diagonalization,
  title={A diagonalization-based parareal algorithm for dissipative
                  and wave propagation problems},
  author={Gander, Martin J. and Wu, Shu-Lin},
  journal={SIAM J. Numer. Anal.},
  volume={58},
  number={5},
  pages={2981--3009},
  year={2020},
  publisher={SIAM}
}

@article{gander2019convergence,
  title={Convergence analysis of a periodic-like waveform relaxation
                  method for initial-value problems via the
                  diagonalization technique},
  author={Gander, Martin J. and Wu, Shu-Lin},
  journal={Numer. Math.},
  volume={143},
  pages={489--527},
  year={2019},
  publisher={Springer}
}

@article{liu2020fast,
  title={A fast block $\alpha$-circulant preconditoner for all-at-once systems from wave equations},
  author={Liu, Jun and Wu, Shu-Lin},
  journal={SIAM J.  Matrix Anal. Appl.},
  volume={41},
  number={4},
  pages={1912--1943},
  year={2020},
  publisher={SIAM}
}

@article{mcdonald2018preconditioning,
  title={Preconditioning and iterative solution of all-at-once systems for evolutionary partial differential equations},
  author={McDonald, Eleanor and Pestana, Jennifer and Wathen, Andy},
  journal={SIAM J. Sci. Comput.},
  volume={40},
  number={2},
  pages={A1012--A1033},
  year={2018},
  publisher={SIAM}
}

@article{gander2007analysis,
  title={Analysis of the parareal time-parallel time-integration method},
  author={Gander, Martin J. and Vandewalle, Stefan},
  journal={SIAM J. Sci. Comput.},
  volume={29},
  number={2},
  pages={556--578},
  year={2007},
  publisher={SIAM}
}

@InProceedings{gander:2008:nca,
  author	= {Martin J. Gander and Ernst Hairer},
  title		= {Nonlinear convergence analysis for the parareal algorithm},
  editor	= {Olof B. Widlund and David E. Keyes},
  booktitle	= {Domain Decomposition Methods in Science and Engineering
		  XVII},
  year		= 2008,
  pages		= {45-56},
  series	= {Lecture Notes in Computational Science and Engineering},
  volume	= {60},
  publisher	= {Springer}
}

@article{Gander:2016:AOANST,
  title={Analysis of a New Space-Time Parallel Multigrid Algorithm for
  Parabolic Problems},
  author={Martin J. Gander and Martin Neum{\"u}ller},
  journal={SIAM J. Sci. Comput.},
  volume={38},
  number={4},
  pages={A2173--A2208},
  year={2016}
}

@article{horton1995space,
  title={A space-time multigrid method for parabolic partial differential equations},
  author={Horton, Graham and Vandewalle, Stefan},
  journal={SIAM J. Sci. Comput.},
  volume={16},
  number={4},
  pages={848--864},
  year={1995},
  publisher={SIAM}
}

@INPROCEEDINGS{Hackbusch:1984:PMG,
  AUTHOR       = "Wolfgang Hackbusch",
  TITLE        = "Parabolic Multi-Grid Methods",
  BOOKTITLE    = "Computing Methods in Applied Sciences and
                  Engineering, {VI}",
  YEAR         = 1984,
  EDITOR       = "Roland Glowinski and Jacques-Louis Lions",
  PAGES        = "189--197",
  PUBLISHER    = "North-Holland"
}

@article{maday2008parallelization,
  title={Parallelization in time through tensor-product space--time solvers},
  author={Maday, Yvon and R{\o}nquist, Einar M},
  journal={C. R. Math. Acad. Sci. Paris},
  volume={346},
  number={1},
  pages={113--118},
  year={2008},
  publisher={Elsevier}
}

@article{howse2019parallel,
  title={Parallel-in-time multigrid with adaptive spatial coarsening for the linear advection and inviscid {B}urgers equations},
  author={Howse, Alexander J. and Sterck, Hans De and Falgout, Robert D. and MacLachlan, Scott and Schroder, Jacob},
  journal={SIAM J. Sci. Comput.},
  volume={41},
  number={1},
  pages={A538--A565},
  year={2019},
  publisher={SIAM}
}

@article{de2021optimizing,
  title={Optimizing multigrid reduction-in-time and Parareal coarse-grid operators for linear advection},
  author={De Sterck, Hans and Falgout, Robert D. and Friedhoff, Stephanie and Krzysik, Oliver A. and MacLachlan, Scott P.},
  journal={Numer. Linear Alg. Appl.},
  volume={28},
  number={4},
  pages={e2367},
  year={2021},
  publisher={Wiley Online Library}
}

@article{de2023efficient,
  title={Efficient multigrid reduction-in-time for method-of-lines discretizations of linear advection},
  author={De Sterck, H. and Falgout, R. D. and Krzysik, O. A. and Schroder, J. B.},
  journal={J. Sci. Comput.},
  volume={96},
  number={1},
  pages={1},
  year={2023},
  publisher={Springer}
}

@article{de2023fast,
  title={Fast multigrid reduction-in-time for advection via modified semi-Lagrangian coarse-grid operators},
  author={De Sterck, Hans and Falgout, Robert D. and Krzysik, Oliver A.},
  journal={SIAM J. Sci. Comput.},
  volume={45},
  number={4},
  pages={A1890--A1916},
  year={2023},
  publisher={SIAM}
}

@article{gander2020toward,
  title={Toward error estimates for general space-time discretizations of the advection equation},
  author={Gander, Martin J. and Lunet, Thibaut},
  journal={Comput. Visual Sci.},
  volume={23},
  pages={1--14},
  year={2020},
  publisher={Springer}
}

@article{gander2023unified,
  title={A unified analysis framework for iterative parallel-in-time algorithms},
  author={Gander, Martin J. and Lunet, Thibaut and Ruprecht, Daniel and Speck, Robert},
  journal={SIAM J. Sci. Comput.},
  volume={45},
  number={5},
  pages={A2275--A2303},
  year={2023},
  publisher={SIAM}
}

@incollection{gander2023convergence,
  title={Convergence of Parareal for a Vibrating String with Viscoelastic Damping},
  author={Gander, Martin J. and Lunet, Thibaut and Pogo{\v{z}}elskyt{\.e}, Au{\v{s}}ra},
  booktitle={Domain Decomposition Methods in Science and Engineering XXVI},
  pages={435--442},
  year={2023},
  publisher={Springer}
}

@inproceedings{gander2020reynolds,
  title={A Reynolds number dependent convergence estimate for the Parareal algorithm},
  author={Gander, Martin J and Lunet, Thibaut},
  booktitle={Domain Decomposition Methods in Science and Engineering XXV 25},
  pages={277--284},
  year={2020},
  organization={Springer}
}

@article{BSM23,
  title={Diagonalization-based preconditioners and generalized convergence bounds for {P}ara{O}pt},
  author={Bouillon, Arne and Samaey, Giovanni and Meerbergen, Karl},
  journal={SIAM J. Sci. Comput.},
  volume={46},
      number={5},
  pages={S317-S345},
  year={2024},
}

@article{BGGH16,
  title={Optimized {S}chwarz waveform relaxation for advection reaction diffusion equations in two dimensions},
  author={Bennequin, Daniel and Gander, Martin J and Gouarin, Loic and Halpern, Laurence},
  journal={Numer. Math.},
  volume={134},
  pages={513--567},
  year={2016},
  publisher={Springer}
}

@book{BS84,
  title={Defect Correction Methods, Theory and Applications},
  author={K.~B{\"o}hmer and H.~J.~Stetter},
  year={1984},
  publisher={Springer--Verlag, New York}
}

@article{BMS17,
  title={A multigrid perspective on the parallel full approximation scheme in space and time},
  author={Bolten, Matthias and Moser, Dieter and Speck, Robert},
  journal={Numer. Linear Alg. Appl.},
  volume={24},
  number={6},
  pages={e2110},
  year={2017},
  publisher={Wiley Online Library}
}

@article{BMS18,
  title={Asymptotic convergence of the parallel full approximation scheme in space and time for linear problems},
  author={Bolten, Matthias and Moser, Dieter and Speck, Robert},
  journal={Numer. Linear Alg. Appl.},
  volume={25},
  number={6},
  pages={e2208},
  year={2018},
  publisher={Wiley Online Library}
}

@book{BLM05,
  title={Numerical Methods for Structured Markov Chains},
  author={Bini, Dario A and Latouche, Guy and Meini, Beatrice},
  year={2005},
  publisher={Oxford University Press}
}

@book{BT03,
  title={Solving Differential Problems by Multistep Initial and Boundary Value Methods},
  author={Brugnano, L. and Trigiante, D.},
  publisher={Gordon and Breach Science Publ., Amsterdam},
  year={2003}
}

@article{Brd77,
  title={Multi-level adaptive solutions to boundary-value problems},
  author={Brandt, Achi},
  journal={Math. Comput.},
  volume={31},
  number={138},
  pages={333--390},
  year={1977}
}

@article{BMT93,
  title={Parallel implementation of {BVM} methods},
  author={Brugnano, Luigi and Mazzia, Francesca and Trigiante, Donato},
  journal={Appl. Numer. Math.},
  volume={11},
  number={1-3},
  pages={115--124},
  year={1993},
  publisher={Elsevier}
}

@article{Cai91,
  title={Additive {S}chwarz algorithms for parabolic convection-diffusion equations},
  author={Cai, Xiao-Chuan},
  journal={Numer. Math.},
  volume={60},
  number={1},
  pages={41--61},
  year={1991},
  publisher={Springer}
}

@article{Cai94,
  title={Multiplicative {S}chwarz methods for parabolic problems},
  author={Cai, Xiao-Chuan},
  journal={SIAM J. Sci. Comput.},
  volume={15},
  number={3},
  pages={587--603},
  year={1994},
  publisher={SIAM}
}

@article{CN96,
  title={Conjugate gradient methods for {T}oeplitz systems},
  author={Chan, Raymond H and Ng, Michael K},
  journal={SIAM Rev.},
  volume={38},
  number={3},
  pages={427--482},
  year={1996},
  publisher={SIAM}
}

@article{Chawla83,
  title={Unconditionally stable {N}oumerov-type methods for second order differential equations},
  author={Chawla, M M},
  year={1983},
  journal={BIT},
  volume={23},
  pages={541--542}
}

@article{CMO10,
  title={Parallel high-order integrators},
  author={Christlieb, Andrew J. and Macdonald, Colin B and Ong, Benjamin W.},
  journal={ SIAM J. Sci. Comput.},
  volume={32},
  number={2},
  pages={818--835},
  year={2010},
  publisher={SIAM}
}

@article{DSW22,
  title={Space-time block preconditioning for incompressible flow},
  author={Danieli, Federico and Southworth, Ben S. and Wathen, Andrew J.},
  journal={SIAM J. Sci. Comput.},
  volume={44},
  number={1},
  pages={A337--A363},
  year={2022},
  publisher={SIAM}
}

@article{DW21,
  title={All-at-once solution of linear wave equations},
  author={Danieli, Federico and Wathen, Andrew J},
  journal={Numer. Linear Algebra Appl.},
  volume={28},
  number={6},
  pages={e2386},
  year={2021},
  publisher={Wiley Online Library}
}

@book{D04,
  title={Newton Methods for Nonlinear Problems},
  author={Deuflhard, Peter},
  year={2004},
  publisher={Springer, Berlin}
}

@article{DKPS17,
  title={Two-level convergence theory for multigrid reduction in time (MGRIT)},
  author={Dobrev, Veselin A and Kolev, Tz V and Petersson, N Anders and Schroder, Jacob B},
  journal={SIAM J. Sci. Comput.},
  volume={39},
  number={5},
  pages={S501--S527},
  year={2017},
  publisher={SIAM}
}

@article{DGR00,
  title={Spectral deferred correction methods for ordinary differential equations},
  author={Dutt, Alok and Greengard, Leslie and Rokhlin, Vladimir},
  journal={BIT},
  volume={40},
  number={2},
  pages={241--266},
  year={2000},
  publisher={Springer}
}

@article{EM12,
  title={Toward an efficient parallel in time method for partial differential equations},
  author={Emmett, Matthew and Michael L. Minion},
  journal={Communi.  Appl. Math. Comput.  Sci.},
  volume={7},
  number={1},
  pages={105--132},
  year={2012},
  publisher={Mathematical Sciences Publishers}
}

@article{F54,
  title={A note on the numerical integration of first-order differential equations},
  author={Fox, L},
  journal={Quart. J. Mech. Appl. Math.},
  volume={7},
  number={3},
  pages={367--378},
  year={1954},
  publisher={Oxford University Press}
}

@article{FM57,
  title={Boundary-value techniques for the numerical solution of initial-value problems in ordinary differential equations},
  author={Fox, L and Mitchell, A. R},
  journal={Quart. J. Mech. Appl. Math.},
  volume={10},
  number={2},
  pages={232--243},
  year={1957},
  publisher={Oxford University Press}
}

@article{G98,
  title={A waveform relaxation algorithm with overlapping splitting for reaction diffusion equations},
  author={Gander, Martin J},
  journal={Numer. Linear Algebra Appl.},
  volume={6},
  number={2},
  pages={125--145},
  year={1999},
  publisher={Wiley Online Library}
}

@article{GStuart98,
  title={Space-time continuous analysis of waveform relaxation for the heat equation},
  author={Gander, Martin J and Stuart, Andrew M},
  journal={SIAM J. Sci. Comput.},
  volume={19},
  number={6},
  pages={2014--2031},
  year={1998},
  publisher={SIAM}
}

@article{GanKZ18,
  title={Multigrid interpretations of the parareal algorithm leading to an overlapping variant and {MGRIT}},
  author={Gander, Martin J and Kwok, Felix and Zhang, Hui},
  journal={Comput. Visual Sci.},
  volume={19},
  number={3},
  pages={59--74},
  year={2018},
  publisher={Springer}
}

@inproceedings{GGP18,
  title={A nonlinear ParaExp algorithm},
  author={Gander, M J and G{\"u}ttel, Stefan and Petcu, Madalina},
  booktitle={Domain Decomposition Methods in Science and Engineering XXIV 24},
  pages={261--270},
  year={2018},
  organization={Springer}
}

@article{GR2005,
  title={Overlapping {S}chwarz waveform relaxation for convection-dominated nonlinear conservation laws},
  author={Gander, Martin J and Rohde, Christian},
  journal={SIAM J. Sci. Comput.},
  volume={27},
  number={2},
  pages={415--439},
  year={2005},
  publisher={SIAM}
}

@inproceedings{GHR16,
  title={A direct solver for time parallelization},
  author={Gander, Martin J and Halpern, Laurence and Ryan, Juliet and Tran, Thuy Thi Bich},
  booktitle={Domain decomposition methods in science and engineering XXII 22},
  pages={491--499},
  year={2016},
  organization={Springer}
}

@inproceedings{GH17,
  title={Time parallelization for nonlinear problems based on diagonalization},
  author={Gander, Martin J. and Halpern, Laurence},
  booktitle={Domain decomposition methods in science and engineering XXIII 23},
  pages={163--170},
  year={2017},
  organization={Springer}
}

@article{GP24,
  title={A new {P}ara{D}iag time-parallel time integration method},
  author={Gander, Martin J and Palitta, Davide},
  journal={SIAM J. Sci. Comput.},
  volume={46},
  number={2},
  pages={A697--A718},
  year={2024},
  publisher={SIAM}
}

@article{gander17,
  title={Three different multigrid interpretations of the parareal algorithm and an adaptive variant},
  author={Gander, Martin Jakob},
  journal={Workshop on Space-time Methods for Time-dependent Partial Differential Equations},
  year={2017},
}

@article{GanK02,
  title={Space-time domain decomposition for parabolic problems},
  author={Giladi, Eldar and Keller, Herbert B},
  journal={Numer. Math.},
  volume={93},
  pages={279--313},
  year={2002},
  publisher={Springer}
}

@article{GWJCP20,
  title={A parallel-in-time iterative algorithm for {V}olterra partial integro-differential problems with weakly singular kernel},
  author={Gu, Xian Ming and Wu, Shu-Lin},
  journal={J. Comput. Phys.},
  volume={417},
  pages={109576},
  year={2020},
  publisher={Elsevier}
}

@incollection{GT07,
  title={Parallel deferred correction method for {CFD} problems},
  author={Guibert, David and Tromeur-Dervout, Damien},
  booktitle={Parallel Computational Fluid Dynamics 2006},
  pages={131--138},
  year={2007},
  publisher={Elsevier}
}

@article{HPer2024,
  title={Diagonalization-Based Parallel-in-Time Preconditioners for Instationary Fluid Flow Control Problems},
  author={Heinzelreiter, Bernhard and Pearson, John W},
  journal={arXiv preprint arXiv:2405.18964},
  year={2024}
}

@book{Higham2008,
  title={Functions of {M}atrices: {T}heory and {C}omputation},
  author={Higham, Nicholas J},
  year={2008},
  publisher={SIAM, Philadelphia.}
}

@inproceedings{worley1991parallelizing,
  title={Parallelizing across time when solving time-dependent partial differential equations},
  author={Worley, Patrick},
  booktitle={Proc. 5th SIAM Conf. on Parallel Processing for Scientific Computing, D. Sorensen, ed., SIAM},
  year={1991}
}

@article{HVW95,
  title={An algorithm with polylog parallel complexity for solving parabolic partial differential equations},
  author={Horton, Graham and Vandewalle, Stefan and Worley, P.},
  journal={SIAM J. Sci. Comput.},
  volume={16},
  number={3},
  pages={531--541},
  year={1995},
  publisher={SIAM}
}

@article{LRS82,
  title={The waveform relaxation method for time-domain analysis of large scale integrated circuits},
  author={Lelarasmee, Ekachai and Ruehli, Albert E and Sangiovanni-Vincentelli, Alberto L},
  journal={IEEE Trans. CAD IC Syst.},
  volume={1},
  number={3},
  pages={131--145},
  year={1982},
  publisher={IEEE}
}

@article{MMjcam20,
  title={An adaptive parareal algorithm},
  author={Maday, Yvon and Mula, Olga},
  journal={J. Comput. Appl. Math.},
  volume={377},
  pages={112915},
  year={2020},
  publisher={Elsevier}
}

@article{VLP93,
  title={Approximations with {K}ronecker products},
  author={Van Loan, C, F and Pitsianis, N},
  journal={Linear Algebra for Large Scale and Real-Time Apprications},
  pages={293--314},
  year={1993},
  publisher={Springer, Dordrecht, Netherlands}
}

@article{MSS10,
  title={Analysis of block parareal preconditioners for parabolic optimal control problems},
  author={Mathew, Tarek P and Sarkis, Marcus and Schaerer, Christian E},
  journal={SIAM J. Sci. Comput.},
  volume={32},
  number={3},
  pages={1180--1200},
  year={2010},
  publisher={SIAM}
}

@inproceedings{Meu91,
  title={Numerical experiments with a domain decomposition method for parabolic problems on parallel computers},
  author={Meurant, G{\'e}rard A},
  booktitle={Proceedings of the Fourth International Symposium on Domain Decomposition Methods for Partial Differential Equations},
  pages={394-408},
  year={1991},
  publisher={SlAM, Philadelphia, PA}
}

@article{Min10,
  title={A hybrid parareal spectral deferred corrections method},
  author={Minion, Michael L},
  journal={Commun. Appl. Math. Comput. Sci.},
  volume={5},
  number={2},
  pages={265--301},
  year={2010}
}

@article{Min15,
  title={Interweaving {PFASST} and parallel multigrid},
  author={Minion, Michael L and Speck, Robert and Bolten, Matthias and Emmett, Matthew and Ruprecht, Daniel},
  journal={SIAM J. Sci. Comput.},
  volume={37},
  number={5},
  pages={S244--S263},
  year={2015},
  publisher={SIAM}
}

@article{MV02,
  title={Nineteen dubious ways to compute the exponential of a matrix, twenty-five years later},
  author={Moler, Cleve and Van Loan, Charles},
  journal={SIAM Rev.},
  volume={45},
  number={1},
  pages={3--49},
  year={2003},
  publisher={SIAM}
}

@article{Nev89,
  title={Remarks on {P}icard-{L}indel{\"o}f iteration: {P}art {I}},
  author={Nevanlinna, Olavi},
  journal={BIT},
  volume={29},
  number={2},
  pages={328--346},
  year={1989},
  publisher={Springer}
}

@book{Ng04,
  title={Iterative Methods for {T}oeplitz Systems},
  author={Ng, Michael K},
  year={2004},
  publisher={Oxford Academic}
}

@book{OR00,
  title={Iterative Solution of Nonlinear Equations in Several Variables},
  author={Ortega, James M and Rheinboldt, Werner C},
  year={2000},
  publisher={SIAM, Philadelphia, PA, USA}
}

@article{PSW12,
  title={Regularization-robust preconditioners for time-dependent PDE-constrained optimization problems},
  author={Pearson, John W and Stoll, Martin and Wathen, Andrew J},
  journal={SIAM J. Matrix Anal. Appl.},
  volume={33},
  number={4},
  pages={1126--1152},
  year={2012},
  publisher={SIAM}
}

@article{SV00,
  title={Waveform relaxation with fast direct methods as preconditioner},
  author={Simoens, Jo and Vandewalle, Stefan},
  journal={SIAM J. Sci. Comput.},
  volume={21},
  number={5},
  pages={1755--1773},
  year={2000},
  publisher={SIAM}
}

@article{Schwarz1870,
  title={{\"U}ber einen {G}renz{\"u}bergang durch alternierendes {V}erfahren},
  author={Schwarz, H},
  journal={Vierteljahrsschrift der Naturforschenden Gesellschaft in Z{\"u}rich},
  volume={15},
  pages={272--286},
  year={1870},
}

@incollection{SRE14,
  title={A space-time parallel solver for the three-dimensional heat equation},
  author={Speck, Robert and Ruprecht, Daniel and Emmett, Matthew and Bolten, Matthias and Krause, Rolf},
  booktitle={Parallel Computing: Accelerating Computational Science and Engineering (CSE)},
  volume={25},
  pages={263--272},
  year={2014},
  publisher={IOS Press}
}

@inproceedings{SRK12,
  title={A massively space-time parallel {N}-body solver},
  author={Speck, Robert and Ruprecht, Daniel and Krause, Rolf and Emmett, Matthew and Minion, Michael and Winkel, Mathias and Gibbon, Paul},
  booktitle={SC'12: Proceedings of the International Conference on High Performance Computing, Networking, Storage and Analysis},
  pages={1--11},
  year={2012},
  organization={IEEE Computer Society Press}
}

@article{Strang86,
  title={A proposal for {T}oeplitz matrix calculations},
  author={Strang, Gilbert},
  journal={Stud. Appl. Math.},
  volume={74},
  number={2},
  pages={171--176},
  year={1986},
  publisher={Wiley Online Library}
}

@article{WHH12,
  title={Convergence analysis of the overlapping {S}chwarz waveform relaxation algorithm for reaction-diffusion equations with time delay},
  author={Wu, Shu-Lin and Huang, Cheng Ming and Huang, Ting Zhu},
  journal={IMA J. Numer. Anal.},
  volume={32},
  number={2},
  pages={632--671},
  year={2012},
  publisher={OUP}
}

@article{WK14,
  title={Semi-discrete {S}chwarz waveform relaxation algorithms for reaction diffusion equations},
  author={Wu, Shu-Lin and Al-Khaleel, Mohammad D},
  journal={BIT},
  volume={54},
  pages={831--866},
  year={2014},
  publisher={Springer}
}

@article{WuIMA2015,
  title={Convergence analysis of some second-order parareal algorithms},
  author={Wu, Shu-Lin},
  journal={IMA J. Numer. Anal.},
  volume={35},
  number={3},
  pages={1315--1341},
  year={2015},
  publisher={OUP}
}

@article{WZhou15,
  title={Convergence analysis for three parareal solvers},
  author={Wu, Shu-Lin and Zhou, Tao},
  journal={SIAM J. Sci. Comput.},
  volume={37},
  number={2},
  pages={A970--A992},
  year={2015},
  publisher={SIAM}
}

@article{WX17,
  title={Convergence analysis of {S}chwarz waveform relaxation with convolution transmission conditions},
  author={Wu, Shu-Lin and Xu, Yingxiang},
  journal={SIAM J. Sci. Comput.},
  volume={39},
  number={3},
  pages={A890--A921},
  year={2017},
  publisher={SIAM}
}

@article{WSiSC18,
  title={Toward parallel coarse grid correction for the parareal algorithm},
  author={Wu, Shu-Lin},
  journal={SIAM J. Sci. Comput.},
  volume={40},
  number={3},
  pages={A1446--A1472},
  year={2018},
  publisher={SIAM}
}

@article{WZSiSC19,
  title={Acceleration of the two-level {MGRIT} algorithm via the diagonalization technique},
  author={Wu, Shu-Lin and Zhou, Tao},
  journal={SIAM J. Sci. Comput.},
  volume={41},
  number={5},
  pages={A3421--A3448},
  year={2019},
  publisher={SIAM}
}

@article{WWZSIMAX23,
  title={Pin{T} Preconditioner for Forward-Backward Evolutionary Equations},
  author={Wu, Shu-Lin and Wang, Zhiyong and Zhou, Tao},
  journal={SIAM J. Matrix Anal. Appl.},
  volume={44},
  number={4},
  pages={1771--1798},
  year={2023},
  publisher={SIAM}
}

@article{WLSiSC20,
  title={A parallel-in-time block-circulant preconditioner for optimal control of wave equations},
  author={Wu, Shu-Lin and Liu, Jun},
  journal={SIAM J. Sci. Comput.},
  volume={42},
  number={3},
  pages={A1510--A1540},
  year={2020},
  publisher={SIAM}
}

@article{WZ21,
  title={Parallel implementation for the two-stage {SDIRK} methods via diagonalization},
  author={Wu, Shu-Lin and Zhou, Tao},
  journal={J. Comput. Phys.},
  volume={428},
  pages={110076},
  year={2021},
  publisher={Elsevier}
}

@article{WZhou21,
  title={A parallel-in-time algorithm for high-order {BDF} methods for diffusion and subdiffusion equations},
  author={Wu, Shuo Nan and Zhou, Zhi},
  journal={SIAM J. Sci. Comput.},
  volume={43},
  number={6},
  pages={A3627--A3656},
  year={2021},
  publisher={SIAM}
}

@article{WZZ22,
  title={A uniform spectral analysis for a preconditioned all-at-once system from first-order and second-order evolutionary problems},
  author={Wu, Shu-Lin and Zhou, Tao and Zhou, Zhi},
  journal={SIAM J. Matrix Anal. Appl.},
  volume={43},
  number={3},
  pages={1331--1353},
  year={2022},
  publisher={SIAM}
}

@article{WYZ24,
author = {Wu, Shu-Lin and Yang, Zi Hao},
title = {Mixed precision iterative ParaDiag algorithm},
journal={submitted},
year = {2024}
}

@article{WZ24,
author = {Wu, Shu-Lin and Zhou, Tao},
title = {Convergence analysis of the parareal algorithm
with non-uniform fine time grid},
journal={SIAM J. Numer. Anal.},
  volume={62},
  number={5},
  pages={2308--2330},
year = {2024}
}

@article{YYZ23,
  title={Robust convergence of parareal algorithms with arbitrarily high-order fine propagators},
  author={Yang, Jiang and Yuan, Zhao Ming and Zhou, Zhi},
  journal={CSIAM Trans. Appl. Math.},
  volume={4},
  number={3},
  pages={566--591},
  year={2023},
  publisher={Global Science Press}
}

@article{LWWZ22,
  title={A well-conditioned direct {P}in{T} algorithm for first-and second-order evolutionary equations},
  author={Liu, Jun and Wang, Xiang-Sheng and Wu, Shu-Lin and Zhou, Tao},
  journal={Adv. Comput. Math.},
  volume={48},
  number={3},
  pages={16},
  year={2022},
  publisher={Springer}
}

@article{AV85,
  title={Boundary value techniques for initial value problems in ordinary differential equations},
  author={Axelsson, A. O. H and Verwer, Johannes Gerardus},
  journal={Math. Comput.},
  volume={45},
  number={171},
  pages={153--171},
  year={1985}
}

@article{LWu22,
  title={Parallel-in-time preconditioner for the {S}inc-{N}ystr{\"o}m systems},
  author={Liu, Jun and Wu, Shu-Lin},
  journal={SIAM J. Sci. Comput.},
  volume={44},
  number={4},
  pages={A2386--A2411},
  year={2022},
  publisher={SIAM}
}

@article{LN21,
  title={An all-at-once preconditioner for evolutionary partial differential equations},
  author={Lin, Xue Lei and Ng, Michael},
  journal={SIAM J. Sci. Comput.},
  volume={43},
  number={4},
  pages={A2766--A2784},
  year={2021},
  publisher={SIAM}
}

@ARTICLE{LWSiSC22,
  AUTHOR="Jun Liu and Shu-Lin Wu",
  TITLE="Parallel-in-time preconditioner for the Sinc-Nystr{\"o}m systems",
  JOURNAL="SIAM J. Sci. Comput.",
  YEAR         = 2022,
  VOLUME       = 44,
  PAGES        = "A2386--A2411"
}

@article{ong2020applications,
  title={Applications of time parallelization},
  author={Ong, Benjamin W. and Schr{\"o}der, Jacob B.},
  journal={Computing and Visualization in Science},
  volume={23},
  pages={1--15},
  year={2020},
  publisher={Springer}
}

@article{farhat2006time,
  title={Time-parallel implicit integrators for the near-real-time
                  prediction of linear structural dynamic responses},
  author={Farhat, Charbel and Cortial, Julien and Dastillung, Climene
                  and Bavestrello, Henri},
  journal={International journal for numerical methods in engineering},
  volume={67},
  number={5},
  pages={697--724},
  year={2006},
  publisher={Wiley Online Library}
}

@article{farhat2003time,
  title={Time-decomposed parallel time-integrators: theory and
                  feasibility studies for fluid, structure, and
                  fluid--structure applications},
  author={Farhat, Charbel and Chandesris, Marion},
  journal={International Journal for Numerical Methods in Engineering},
  volume={58},
  number={9},
  pages={1397--1434},
  year={2003},
  publisher={Wiley Online Library}
}

@article{cortial2009time,
  title={A time-parallel implicit method for accelerating the solution
                  of non-linear structural dynamics problems},
  author={Cortial, Julien and Farhat, Charbel},
  journal={International Journal for Numerical Methods in Engineering},
  volume={77},
  number={4},
  pages={451--470},
  year={2009},
  publisher={Wiley Online Library}
}

@article{minion2015interweaving,
  title={Interweaving {PFASST} and parallel multigrid},
  author={Minion, Michael L. and Speck, Robert and Bolten, Matthias
                  and Emmett, Matthew and Ruprecht, Daniel},
  journal={SIAM journal on scientific computing},
  volume={37},
  number={5},
  pages={S244--S263},
  year={2015},
  publisher={SIAM}
}

@article{minion2011hybrid,
  title={A hybrid parareal spectral deferred corrections method},
  author={Minion, Michael},
  journal={Communications in Applied Mathematics and Computational Science},
  volume={5},
  number={2},
  pages={265--301},
  year={2011},
  publisher={Mathematical Sciences Publishers}
}

@article{FFK14,
  title={Parallel time integration with multigrid},
  author={Falgout, Robert D. and Friedhoff, Stephanie and Kolev,
                  Tz. V. and MacLachlan, Scott P. and Schroder, Jacob
                  B.},
  journal={SIAM Journal on Scientific Computing},
  volume={36},
  number={6},
  pages={C635--C661},
  year={2014},
  publisher={SIAM}
}

@article{dobrev2017two,
  title={Two-level convergence theory for multigrid reduction in time (MGRIT)},
  author={Dobrev, Veselin A and Kolev, Tz and Petersson, N Anders and Schroder, Jacob B},
  journal={SIAM Journal on Scientific Computing},
  volume={39},
  number={5},
  pages={S501--S527},
  year={2017},
  publisher={SIAM}
}

@article{gander2014analysis,
  title={Analysis for parareal algorithms applied to {H}amiltonian
                  differential equations},
  author={Gander, Martin J. and Hairer, Ernst},
  journal={Journal of Computational and Applied Mathematics},
  volume={259},
  pages={2--13},
  year={2014},
  publisher={Elsevier}
}

@article{hessenthaler2020multilevel,
  title={Multilevel convergence analysis of multigrid-reduction-in-time},
  author={Hessenthaler, Andreas and Southworth, Ben S. and Nordsletten,
                  David and R{\"o}hrle, Oliver and Falgout, Robert D. and
                  Schroder, Jacob B.},
  journal={SIAM Journal on Scientific Computing},
  volume={42},
  number={2},
  pages={A771--A796},
  year={2020},
  publisher={SIAM}
}

@INPROCEEDINGS{Gander:1999:OCO,
  AUTHOR       = "Martin J. Gander and Laurence Halpern and
                  Fr\'ed\'eric Nataf",
  TITLE        = "Optimal Convergence for Overlapping and
                  Non-Overlapping {S}chwarz Waveform Relaxation",
  BOOKTITLE    = "Eleventh international Conference of Domain
                  Decomposition Methods",
  YEAR         = 1999,
  EDITOR       = "C-H. Lai and P. Bj{\o}rstad and M. Cross and O.
                  Widlund",
  PUBLISHER    = "ddm.org"
}

@PHDTHESIS{Bjorhus:1995:DDS,
  AUTHOR       = "Morten Bj{\o}rhus",
  TITLE        = "On Domain Decomposition, Subdomain Iteration and
                  Waveform Relaxation",
  SCHOOL       = "University of Trondheim, Norway",
  YEAR         = 1995
}

@ARTICLE{Gander:1998:WRA,
  AUTHOR       = "Martin J. Gander",
  TITLE        = "A Waveform Relaxation Algorithm with
                  Overlapping Splitting for Reaction Diffusion
                  Equations",
  JOURNAL      = "Numerical Linear Algebra with Applications",
  YEAR         = 1998,
  VOLUME       = 6,
  PAGES        = "125--145"
}

@article{Gander:2014:DNNWR,
  title={Dirichlet-{N}eumann and {N}eumann-{N}eumann Waveform Relaxation
  Algorithms for Parabolic Problems},
  author={Martin J. Gander and Felix Kwok and Bankim Mandal},
  journal={ETNA},
  volume=45,
  pages={424--456},
  year={2016}
}

@article{gander2021dirichlet,
  title={Dirichlet--{N}eumann waveform relaxation methods for parabolic and hyperbolic problems in multiple subdomains},
  author={Gander, Martin J. and Kwok, Felix and Mandal, Bankim C.},
  journal={BIT Numerical Mathematics},
  volume={61},
  pages={173--207},
  year={2021},
  publisher={Springer}
}

@article{bennequin2016optimized,
  title={Optimized {S}chwarz waveform relaxation for advection reaction diffusion equations in two dimensions},
  author={Bennequin, Daniel and Gander, Martin J. and Gouarin, Loic and Halpern, Laurence},
  journal={Numerische Mathematik},
  volume={134},
  pages={513--567},
  year={2016},
  publisher={Springer}
}

@article{vandewalle1994space,
  title={Space-time concurrent multigrid waveform relaxation},
  author={Vandewalle, Stefan and Van de Velde, Eric},
  journal={Annals of Numer. Math},
  volume={1},
  pages={347--363},
  year={1994}
}

@article{merkel2017paraexp,
  title={Para{E}xp Using Leapfrog as Integrator for High-Frequency
                  Electromagnetic Simulations},
  author={Merkel, Melina and Niyonzima, Innocent and Sch{\"o}ps, Sebastian},
  journal={Radio Science},
  volume={52},
  number={12},
  pages={1558--1569},
  year={2017},
  publisher={Wiley Online Library}
}

@article{kooij2017block,
  title={A block Krylov subspace implementation of the time-parallel
                  {ParaExp} method and its extension for nonlinear
                  partial differential equations},
  author={Kooij, Gijs L. and Botchev, Mike A. and Geurts, Bernard J.},
  journal={Journal of computational and applied mathematics},
  volume={316},
  pages={229--246},
  year={2017},
  publisher={Elsevier}
}

@article{janssen1996multigrid,
  title={Multigrid waveform relaxation of spatial finite element meshes: the continuous-time case},
  author={Janssen, Jan and Vandewalle, Stefan},
  journal={SIAM Journal on Numerical Analysis},
  volume={33},
  number={2},
  pages={456--474},
  year={1996},
  publisher={SIAM}
}

@article{van2002multigrid,
  title={Multigrid waveform relaxation for anisotropic partial differential equations},
  author={Van Lent, Jan and Vandewalle, Stefan},
  journal={Numerical Algorithms},
  volume={31},
  pages={361--380},
  year={2002},
  publisher={Springer}
}

@article{neumueller2019time,
  title={Time-parallel iterative solvers for parabolic evolution equations},
  author={Neum{\"u}ller, Martin and Smears, Iain},
  journal={SIAM Journal on Scientific Computing},
  volume={41},
  number={1},
  pages={C28--C51},
  year={2019},
  publisher={SIAM}
}

@ARTICLE{Lubich:1987:MGD,
  AUTHOR       = "C. Lubich and A. Ostermann",
  TITLE        = "Multi-grid dynamic iteration for parabolic
                  equations",
  JOURNAL      = "BIT",
  YEAR         = 1987,
  VOLUME       = 27,
  NUMBER       = 2,
  PAGES        = "216--234"
}

@ARTICLE{Bellen:1989:PAI,
  AUTHOR       = "Alfredo Bellen and Marino Zennaro",
  TITLE        = "Parallel Algorithms for Initial-Value Problems for
                  Difference and Differential Equations",
  JOURNAL      = "J. Comput. Appl. Math.",
  YEAR         = 1989,
  VOLUME       = 25,
  PAGES        = "341--350"
}

@ARTICLE{Chartier:1993:APS,
  AUTHOR       = "Philippe Chartier and Bernard Philippe",
  TITLE        = "A Parallel Shooting Technique for Solving
                  Dissipative {ODEs}",
  JOURNAL      = "Computing",
  YEAR         = 1993,
  VOLUME       = 51,
  PAGES        = "209--236"
}

@ARTICLE{Saha:1996:API,
  AUTHOR       = "Prasenjit Saha and Joachim Stadel and Scott Tremaine",
  TITLE        = "A parallel integration method for solar system dynamics",
  JOURNAL      = "The Astronomical Journal",
  YEAR         = 1997,
  VOLUME       = 114,
  NUMBER       = 1,
  PAGES        = "409--415"
}

@article{gander2018MIP,
  title={Multigrid Interpretations of the Parareal Algorithm Leading
                  to an Overlapping Variant and MGRIT},
  author={Gander, Martin J. and Kwok, Felix and Zhang, Hui},
  journal={Journal of Computing and Visualization in Science},
  volume={19},
  number={3},
  pages={59--74},
  year={2018}
}

@ARTICLE{Miranker:1967:PMF,
  AUTHOR       = "Willard L. Miranker and Werner Liniger",
  TITLE        = "Parallel Methods for the Numerical Integration of
                  Ordinary Differential Equations",
  JOURNAL      = "Math. Comp.",
  YEAR         = 1967,
  VOLUME       = 91,
  PAGES        = "303--320"
}

@ARTICLE{Lelarasmee:1982:WRM,
  AUTHOR       = "Ekachai Lelarasmee and Albert E. Ruehli and Alberto L.
                  Sangiovanni-Vincentelli",
  TITLE        = "The Waveform Relaxation Method for Time-Domain
                  Analysis of Large Scale Integrated Circuits",
  JOURNAL      = "IEEE Trans. on CAD of IC and Syst.",
  YEAR         = 1982,
  VOLUME       = 1,
  PAGES        = "131--145"
}

@article{schreiber2018beyond,
  title={Beyond spatial scalability limitations with a massively
                  parallel method for linear oscillatory problems},
  author={Schreiber, Martin and Peixoto, Pedro S. and Haut, Terry and
                  Wingate, Beth},
  journal={The International Journal of High Performance Computing Applications},
  volume={32},
  number={6},
  pages={913--933},
  year={2018},
  publisher={SAGE Publications Sage UK: London, England}
}

@article{gander2005overlapping,
  title={Overlapping {S}chwarz waveform relaxation for
                  convection-dominated nonlinear conservation laws},
  author={Gander, Martin J. and Rohde, Christian},
  journal={SIAM journal on Scientific Computing},
  volume={27},
  number={2},
  pages={415--439},
  year={2005},
  publisher={SIAM}
}

@article{gander2023non,
  title={Non-overlapping {S}chwarz waveform-relaxation for nonlinear
                  advection-diffusion equations},
  author={Gander, Martin J. and Lunowa, Stephan B. and Rohde, Christian},
  journal={SIAM Journal on Scientific Computing},
  volume={45},
  number={1},
  pages={A49--A73},
  year={2023},
  publisher={SIAM}
}

@article{wu2017optimized,
  title={Optimized overlapping {S}chwarz waveform relaxation for a
                  class of time-fractional diffusion problems},
  author={Wu, Shu-Lin},
  journal={Journal of Scientific Computing},
  volume={72},
  pages={842--862},
  year={2017},
  publisher={Springer}
}

@article{halpern2010optimized,
  title={Optimized and quasi-optimal {S}chwarz waveform relaxation for
                  the one-dimensional {S}chr{\"o}dinger equation},
  author={Halpern, Laurence and Szeftel, J{\'e}r{\'e}mie},
  journal={Mathematical Models and Methods in Applied Sciences},
  volume={20},
  number={12},
  pages={2167--2199},
  year={2010},
  publisher={World Scientific}
}

@article{besse2017schwarz,
  title={Schwarz waveform relaxation method for one-dimensional
                  {S}chr{\"o}dinger equation with general potential},
  author={Besse, Christophe and Xing, Feng},
  journal={Numerical Algorithms},
  volume={74},
  pages={393--426},
  year={2017},
  publisher={Springer}
}

@article{antoine2017analysis,
  title={An analysis of {S}chwarz waveform relaxation domain
                  decomposition methods for the imaginary-time linear
                  {S}chr{\"o}dinger and {G}ross-{P}itaevskii equations},
  author={Antoine, Xavier and Lorin, Emmanuel},
  journal={Numerische Mathematik},
  volume={137},
  pages={923--958},
  year={2017},
  publisher={Springer}
}

@article{martin2009schwarz,
  title={Schwarz waveform relaxation algorithms for the linear viscous
                  equatorial shallow water equations},
  author={Martin, V{\'e}ronique},
  journal={SIAM Journal on Scientific Computing},
  volume={31},
  number={5},
  pages={3595--3625},
  year={2009},
  publisher={SIAM}
}

@article{antoine2016lagrange,
  title={Lagrange--{S}chwarz Waveform Relaxation domain decomposition
                  methods for linear and nonlinear quantum wave
                  problems},
  author={Antoine, Xavier and Lorin, Emmanuel},
  journal={Applied Mathematics Letters},
  volume={57},
  pages={38--45},
  year={2016},
  publisher={Elsevier}
}

@article{thery2022analysis,
  title={Analysis of {S}chwarz waveform relaxation for the coupled
                  {E}kman boundary layer problem with continuously
                  variable coefficients},
  author={Thery, Sophie and Pelletier, Charles and Lemari{\'e},
                  Florian and Blayo, Eric},
  journal={Numerical Algorithms},
  pages={1--37},
  year={2022},
  publisher={Springer}
}

@incollection{courvoisier2013time,
  title={Time domain {M}axwell equations solved with {S}chwarz waveform relaxation methods},
  author={Courvoisier, Yves and Gander, Martin J.},
  booktitle={Domain Decomposition Methods in Science and Engineering XX},
  pages={263--270},
  year={2013},
  publisher={Springer}
}

@article{audusse2010optimized,
  title={Optimized {S}chwarz waveform relaxation for the primitive equations of the ocean},
  author={Audusse, Emmanuel and Dreyfuss, Pierre and Merlet, Benoit},
  journal={SIAM Journal on Scientific Computing},
  volume={32},
  number={5},
  pages={2908--2936},
  year={2010},
  publisher={SIAM}
}

@phdthesis{Gander:1997:PhD,
  title={Analysis of Parallel Algorithms for Time Dependent Partial
                  Differential Equations},
  author={Gander, Martin J.},
  school={Stanford},
  year={1997}
}

@book{ciaramella2022iterative,
  title={Iterative methods and preconditioners for systems of linear equations},
  author={Ciaramella, Gabriele and Gander, Martin J.},
  year={2022},
  publisher={SIAM}
}

@article{gander2021discrete,
  title={Discrete optimization of {R}obin transmission conditions for
                  anisotropic diffusion with discrete duality finite
                  volume methods},
  author={Gander, Martin J. and Halpern, Laurence and Hubert, Florence and Krell, Stella},
  journal={Vietnam Journal of Mathematics},
  volume={49},
  number={4},
  pages={1349--1378},
  year={2021},
  publisher={Springer}
}

@article{gander2018optimized,
  title={Optimized {S}chwarz methods for anisotropic diffusion with
                  Discrete Duality Finite Volume discretizations},
  author={Gander, Martin J. and Halpern, Laurence and Hubert, Florence and Krell, Stella},
  journal={Moroccan Journal of Pure and Applied Analysis},
  volume={7},
  number={2},
  year={2018}
}

@ARTICLE{Womble:1990:ATS,
  AUTHOR       = "David E. Womble",
  TITLE        = "A Time-Stepping Algorithm for Parallel Computers",
  JOURNAL      = "SIAM J. Sci. Stat. Comput.",
  YEAR         = 1990,
  VOLUME       = 11,
  NUMBER       = 5,
  PAGES        = "824--837"
}

@article{Kressner2022,
    AUTHOR = {Kressner, D. and Massei, S. and Zhu, J.},
   TITLE = {Improved parallel-in-time integration via low-rank updates
and interpolation},
      YEAR = {2022},
    JOURNAL = {ArXiv Preprint: 2204.03073},
}

@article{Chaudet2024,
    AUTHOR = {Chaudet-Dumas, Bastien and Gander, Martin J. and
                  Pogozelskyte, Ausra},
   TITLE = {An Optimized Space-Time Multigrid Algorithm for Parabolic
{PDE}s},
     YEAR = {2024},
    JOURNAL = {arXiv preprint arXiv:2302.13881},
}

@InProceedings{gander:2024:PararealNoCoarse,
  author	= {Martin J. Gander and Mario Ohlberger and Stephan Rave},
  title		= {A {P}arareal algorithm without coarse propagator?},
  booktitle	= {Domain Decomposition Methods in Science and Engineering
		  XXVIII},
  year		= 2024,
  publisher	= {submitted}
}

@article{ciaramella2017analysis,
  title={Analysis of the parallel {S}chwarz method for growing chains of fixed-sized subdomains: Part {I}},
  author={Ciaramella, Gabriele and Gander, Martin J.},
  journal={SIAM Journal on Numerical Analysis},
  volume={55},
  number={3},
  pages={1330--1356},
  year={2017},
  publisher={SIAM}
}

@article{ciaramella2018analysis,
  title={Analysis of the parallel {S}chwarz method for growing chains of fixed-sized subdomains: Part {II}},
  author={Ciaramella, Gabriele and Gander, Martin J.},
  journal={SIAM Journal on Numerical Analysis},
  volume={56},
  number={3},
  pages={1498--1524},
  year={2018},
  publisher={SIAM}
}

@article{ciaramella2018analysis3,
  title={Analysis of the parallel {S}chwarz method for growing chains of fixed-sized subdomains: Part {III}},
  author={Ciaramella, Gabriele and Gander, Martin J.},
  journal={Electron. Trans. Numer. Anal},
  volume={49},
  pages={201--243},
  year={2018}
}
    
\end{document}